\documentclass[10pt]{article}
\usepackage{amsmath}
\usepackage{lipsum}

\usepackage{geometry}
\usepackage{amssymb,amsmath,esint}
\usepackage{graphicx}
\usepackage{enumerate}
\usepackage{comment}

\usepackage{colortbl}
\usepackage{xcolor}

\usepackage{amsrefs}

\usepackage{tikz}
\usepackage{multirow}

\usepackage{mathtools}
\allowdisplaybreaks

\colorlet{linkequation}{red}
\usepackage[colorlinks,plainpages=true,pdfpagelabels,hypertexnames=true,colorlinks=true,pdfstartview=FitV,linkcolor=blue,citecolor=red,urlcolor=black]{hyperref}

\usepackage{hyperref}
\usepackage{tocloft}

\usepackage[title,titletoc]{appendix}

\definecolor{dgreen}{rgb}{0,0.5,0}
\definecolor{violet}{rgb}{0.5,0,0.5}
\definecolor{dred}{rgb}{0.7,0,0}
\definecolor{ddred}{rgb}{0.5,0,0}
\definecolor{dblue}{rgb}{0,0,0.5}
\definecolor{ddblue}{rgb}{0,0,0.3}

\newtheorem{defn}{Definition}[section]
\newtheorem{lemma}[defn]{Lemma}
\newtheorem{proposition}[defn]{Proposition}
\newtheorem{theorem}[defn]{Theorem}
\newtheorem{cor}[defn]{Corollary}
\newtheorem{remark}[defn]{Remark}
\numberwithin{equation}{section}

\definecolor{dgreen}{rgb}{0,0.5,0}
\definecolor{violet}{rgb}{0.5,0,0.5}
\definecolor{dred}{rgb}{0.7,0,0}
\definecolor{ddred}{rgb}{0.5,0,0}
\definecolor{dblue}{rgb}{0,0,0.5}
\definecolor{ddblue}{rgb}{0,0,0.3}
\definecolor{llgray}{rgb}{0.9,0.9,0.9}
\definecolor{lgray}{rgb}{0.7,0.7,0.7}

\newcommand{\bq}{\begin{equation}}
\newcommand{\eq}{\end{equation}}

\newcommand{\R}{{ \mathbb{R}  }}

\newcommand{\bbr}{{ \mathbb{R}  }}

\newcommand{\calC}{{ \mathcal C  }}

\newcommand{\calP}{{ \mathcal P  }}
\newcommand{\calS}{{ \mathcal S  }}

\newcommand{\bke}[1]{\left( #1 \right)}

\newcommand{\norm}[1]{\left\Vert #1 \right\Vert}
\newcommand{\abs}[1]{\left| #1 \right|}

\DeclareMathOperator*{\essosc}{ess \ osc}
\DeclareMathOperator*{\esssup}{ess \ sup}
\DeclareMathOperator*{\essinf}{ess \ inf}

\DeclareMathOperator{\loc}{loc}
\DeclareMathOperator{\supp}{supp}

\newcommand{\qed}{\hfill\fbox{}\par\vspace{.2cm}}

\def\XXint#1#2#3{{\setbox0=\hbox{$#1{#2#3}{\int}$}
     \vcenter{\hbox{$#2#3$}}\kern-.5\wd0}}

\newenvironment{proof}{{\bf Proof.}}{\hfill\fbox{}\par\vspace{.2cm}}
\newenvironment{pfprop2.19}{{\par\noindent\bf
            Proof of Proposition \ref{thm-uniqueness} }}{\hfill\fbox{}\par\vspace{.2cm}}
\newenvironment{pfthm2.20}{{\par\noindent\bf
            Proof of Theorem \ref{Corollary : Uniqueness} }}{\hfill\fbox{}\par\vspace{.2cm}}
\newenvironment{pfthm2.21}{{\par\noindent\bf
            Proof of Theorem \ref{thm-jengchik50} }}{\hfill\fbox{}\par\vspace{.2cm}}

\newenvironment{pf-prop3.5}{{\par\noindent\bf
            Proof of Proposition  \ref{P:apriori} }}{\hfill\fbox{}\par\vspace{.2cm}}
\newenvironment{pf-prop3.7}{{\par\noindent\bf
            Proof of Proposition   \ref{P:Energy} }}{\hfill\fbox{}\par\vspace{.2cm}}

\newenvironment{pfthm-2}{{\par\noindent\bf
            Proof of Theorem  \ref{Theorem-weaksolution} }}{\hfill\fbox{}\par\vspace{.2cm}}

\begin{document}
\bibliographystyle{plain}

\title{Existence of weak solutions for Porous medium equation \\
   with a divergence type of drift term}

\author{
Sukjung Hwang
\thanks{Department of Mathematics Education, Chungbuk National University, Cheongju, Chungbuk 28644, Republic of Korea, \texttt{sukjungh@chungbuk.ac.kr}}
\and
Kyungkeun Kang
\thanks{Department of Mathematics, Yonsei University,  Seoul 03722, Republic of Korea, \texttt{kkang@yonsei.ac.kr}}
\and
Hwa Kil Kim
\thanks{Department of Mathematics Education, Hannam University, Daejeon 34430, Republic of Korea, \texttt{hwakil@hnu.kr}}
}

\date{}

\maketitle

\begin{abstract}
We consider degenerate porous medium equations with a divergence type of drift terms. We establish existence of nonnegative $L^{q}$-weak solutions (satisfying energy estimates or even further with moment and speed estimates in Wasserstein spaces), in case the drift term belongs to a sub-scaling (including scaling invariant) class depending on $q$ and $m$ caused by nonlinear structure of diffusion, which is a major difference compared to that of a linear case.
It is noticeable that the classes of drift terms become wider, if the drift term is divergence-free.
Similar conditions of gradients of drift terms are also provided to ensure the existence of such weak solutions.
Uniqueness results follow under an additional condition on the gradients of the drift terms with the aid of methods developed in Wasserstein spaces.
%Besides a priori estimates,
One of our main tools is so called the splitting method to construct a sequence of approximated solutions, which implies, by passing to the limit, the existence of weak solutions satisfying not only an energy inequality but also moment and speed estimates. One of crucial points in the construction is
uniform H\"{o}lder continuity up to initial time for homogeneous porous medium equations, which seems to be of independent interest.
%, which is of independent interest.
As an application, we improve a regularity result for solutions of a repulsive Keller-Segel system of porous medium type.
\end{abstract}

%%%%%%%%%%%%%%%%%%%


%\Addresses

%%%%%%%%%%%%%%%%%%%%%%%%%%%%%%%%%%%%%%%%%%%%%%%%%%%%%%%%%

%{
%  \hypersetup{linkcolor=black}
%  \tableofcontents
%}

\hypersetup{linkcolor=black}
\setcounter{tocdepth}{2}
\tableofcontents

%%%%%%%%%%%%%%%%%%%%%%%%%%%%%%%%%%%%%%%%%%%%%%%%%%%%%%%%%%%%

\section{Introduction}

In this paper, we consider \emph{porous medium equations} (PME) with divergence form of drifts in the form of
\begin{equation}\label{E:Main}
  \left\{
  \begin{array}{ll}
  \partial_t \rho =   \nabla \cdot( \nabla \rho^m  - V\rho), & \text{ in } \ Q_{T}  \\
  \rho(\cdot,0)=\rho_0, & \text{ on } \ \mathbb{R}^d,
  \end{array}
  \right. \quad \text{ for }\ m > 1,
\end{equation}
on a cylindrical domain $Q_{T}:= \mathbb{R}^d \times (0, T]$, $d\geq 2$, $0 <T < \infty$,
where $\rho:Q_T\mapsto \bbr$ and a vector field $V: Q_{T} \to \mathbb{R}^d$.
We note that neither the maximum principle nor the gradient flow structure is, in general, available for the equation \eqref{E:Main} which has been rarely analyzed. 
Nevertheless, for given  nonnegative initial data, we establish mainly the existence of  nonnegative weak solutions of \eqref{E:Main} when the drift belongs to some sub-scaling classes (see Definition~\ref{D:Serrin}).

First, we point out that $\Delta \rho^m \sim \nabla \cdot \left(\rho^{m-1} \nabla \rho\right)$ which clearly shows the different tendencies depending on the constant $m$.
If $m=1$, then the diffusion is linear so \eqref{E:Main} becomes the heat equation with drifts. When $m>1$, the equation is called \emph{porous medium equations} (PME) in which we are interested in this paper. When $0<m<1$,it follows that $\rho^{m-1}$ tends to infinity as $\rho$ approaches to zero and we call such equations \emph{fast diffusion equations}. These are important nonlinear evolution equations describing a number of physical applications: fluid flow, heat radiations in plasma, mathematical biology, and other fields. For a systematic presentation of the mathematical theories, we refer to monographs \cites{Vaz06, Vaz07, DK07, DGV12} and the references therein.

The equations \eqref{E:Main} carry both nonlinear diffusion and drift that are useful in numerous places. For example, this type of PDEs represents a certain form of aggregation-diffusion (reaction-diffusion) equations that arises naturally in systems of interacting components and is used widely to describe pattern-formation phenomena and chemotaxis models in biological, chemical and physical systems.
For example, one can consider the porous medium equation with drift caused by fluid flow,
\begin{equation*}%\label{divfree-10}
\rho_t-\Delta \rho^m+u\cdot \nabla \rho=0,\quad \text{ in } \ Q_T :=\R^d\times (0, T],
\end{equation*}
where $u:Q_T\rightarrow \R^d$ is a divergence-free vector field.
Recently, there has been extensive research on so called enhanced dissipation effects for various PDEs with fluid flows for the linear case $m=1$ (see e.g. \cites{MR2434887, MR3621820, MR4242626, MR4152270}, etc). In particular, there are recent developments of suppression of blow-up for nonlinear equations such as Keller-Segel equations with external fluid flows (see e.g. \cites{MR3730537, MR4156602, MR3544323}, etc).
Concerning free boundary regularity of \eqref{E:Main}, there are recent developments of $\calC^{1,\alpha}$-regularity assuming enough smoothness on $V$ and $\rho_0 \in L^1(\bbr^d)\cap L^{\infty}(\bbr^d)$ in \cites{KZ18, KZ20-arXiv} and references therein for relevant works when $V=0$.
The large-time asymptotics and regularity of nonlinear diffusion equations with drifts are also analyzed as well: for example,  the rate of decay to equilibrium of self-similar solutions in \cites{Carr01}, viscosity solutions,  equivalent to weak solutions, and their asymptotic convergence of the free boundary to the equilibrium in \cite{KL10}, and H\"{o}lder continuous solutions for aggregation-diffusion equations with singular interaction kernels in \cite{Zhang20}. Compared to PME without drifts or linear diffusion-drift equations, fewer results are known for \eqref{E:Main} because of difficulties of handling not only nonlinear diffusion but also drift term.

% describes the behaviour of chemical systems where the diffusion of material competes with the production of that material by some form of chemical reaction \cite{KellerSegel}. Also, the divergence-free condition on the drift is relevant for applications to incompressible fluids \cite{Osada}.

The drift may affect the well-posedness of diffusive parabolic equations. There are works showing counterexamples of lack of regularities under certain conditions on the drift, for example, \cites{SVZ13, SSSZ, Z11} for (fractional) linear diffusion and \cites{KZ18, HZ21} for porous medium equations.
From the perspective of critical conditions on $V$ providing continuity of \eqref{E:Main} (for example, \cites{CHKK17, KZ18, HZ21}), we recall
\begin{equation}\label{linear-Serrin}
V \in L^{q_1, q_2}_{x,t} \quad \text{ where } \quad \frac{d}{q_1} + \frac{2}{q_2} \leq 1.
\end{equation}
The condition \eqref{linear-Serrin} is also observed from the scaling invariant classes in $L^{\infty}$-sense. Moreover, the same conditions are well known from heat (linear) diffusion-drift equations (see e.g. \cites{LSU, Lie96, SVZ13}) and there is a connection with so called Prodi-Serrin-Ladyzhenskaya conditions concerning the existence of smooth solutions of the Navier-Stokes equations (see \cites{MR0236541, MR0150444}).

One of motivations of this research has started from an observation that the nonlinear diffusion factor $m$ yields different scaling invariance depending on conservation of $L^q$-norm in spatial variables.
Naturally, it follows a question of searching minimal conditions on $V$ related to $m$, which allows us to construct weak solutions of \eqref{E:Main} for a given $L^q$-initial data.
From this perspective, we investigate \emph{the scaling invariant classes of $V$ and $\nabla V$ in $L^{q}$-sense for $q \geq 1$} which turns out proper conditions for constructing $L^q$-weak solutions (see Definition~\ref{D:weak-sol}).
Heuristically, we look for conditions of $V$ in terms of scaling invariant classes of solutions
for $L^q$-norm in spatial variables.
It turns out that the $L^q$-norm of $\rho$ is preserved, if the following scaling is taken: for $q \geq 1$ and $\kappa > 0$
\begin{equation*} %\label{lambda-scale}
  \rho_{\kappa}(x,t) = \kappa^{d/q} \rho(\kappa x, \kappa^{\beta}t)
  \quad \text{and} \quad
  V_{\kappa} (x,t) = \kappa^{\alpha} V(\kappa x, \kappa^{\beta}t),
\end{equation*}
where
\begin{equation*}
  \alpha = q_{m,d} + 1, \quad \beta= q_{m,d} + 2 \quad \text{for}\quad q_{m,d} := \frac{d(m-1)}{q}.
\end{equation*}
We then see that $V_{\kappa}$ is scaling invariant in a mixed norm $L^{q_1, q_2}_{x, t}$, i.e. for $1\le q_1, q_2\le \infty$, it holds $\left\| V_{\kappa} (x, t)\right\|_{L^{q_1, q_2}_{x, t}}
 = \left\| V (x,t)\right\|_{L^{q_1, q_2}_{x,t}}=\norm{\norm{V(\cdot,t)}_{L^{q_1}_x}}_{L^{q_2}_t}$
 provided that
$
%\begin{equation*}
  \frac{d}{q_1} + \frac{2+q_{m,d}}{q_2} = 1 + q_{m,d}.$
%\end{equation*}
We also note that similar condition of the scaling invariance for $\nabla V$ is that $\nabla V\in L^{\tilde{q}_1, \tilde{q}_2}_{x,t}$, where  $\tilde{q}_1$ and $\tilde{q}_2$ satisfy
$%\begin{equation*}
  \frac{d}{\tilde{q}_1} + \frac{2+q_{m,d}}{\tilde{q}_2} = 2 + q_{m,d}.
$%\end{equation*}

Now we define the \textit{scaling invariant} and  \textit{sub-scaling classes} of $V$ and $\nabla V$ in accordance with $q$ and $m$.

\begin{defn}\label{D:Serrin}
Let $m, q \geq 1$ and $q_{m,d} := \frac{d(m-1)}{q}$. For constants $q_1, q_2, \tilde{q}_1, \tilde{q}_2 > 0$, we define following spaces.
\begin{itemize}
\item[(i)] The \textit{scaling invariant classes} of $V$ and $\nabla V$ are defined as
\begin{equation}\label{Serrin}
 \mathcal{S}_{m,q}^{(q_1, q_2)}:= \left\{ V: \ \|V\|_{L^{q_1, q_2}_{x,t}} < \infty \ \text{ where } \ \frac{d}{q_1} + \frac{2+q_{m,d}}{q_2} = 1 + q_{m,d} \right\},
\end{equation}
\begin{equation}\label{Serrin-grad}
  \tilde{\mathcal{S}}_{m, q}^{(\tilde{q}_1, \tilde{q}_2)}:= \left\{V: \ \|\nabla V\|_{L^{\tilde{q}_1, \tilde{q}_2}_{x,t}} < \infty  \ \text{ where } \  \frac{d}{\tilde{q}_1} + \frac{2+q_{m,d}}{\tilde{q}_2} = 2 + q_{m,d}\right\}.
\end{equation}
Moreover, let us name $\|V\|_{\mathcal{S}_{m,q}^{(q_1, q_2)}}$ and $\|V\|_{\tilde{\mathcal{S}}_{m,q}^{(\tilde{q}_1, \tilde{q}_2)}}$, as the  \textit{scaling invariant norms} corresponding to each spaces.

\item[(ii)] The \textit{sub-scaling classes} are defined as
\begin{equation}\label{subSerrin}
 \mathfrak{S}_{m,q}^{(q_1, q_2)}:= \left\{ V: \ \|V\|_{L^{q_1, q_2}_{x,t}} < \infty \ \text{ where } \ \frac{d}{q_1} + \frac{2+q_{m,d}}{q_2} \leq 1 + q_{m,d} \right\},
\end{equation}
\begin{equation}\label{subSerrin-grad}
  \tilde{\mathfrak{S}}_{m, q}^{(\tilde{q}_1, \tilde{q}_2)}:= \left\{V: \ \|\nabla V\|_{L^{\tilde{q}_1, \tilde{q}_2}_{x,t}} < \infty  \ \text{ where } \  \frac{d}{\tilde{q}_1} + \frac{2+q_{m,d}}{\tilde{q}_2} \leq 2 + q_{m,d}\right\}.
\end{equation}
Let us name $\|V\|_{\mathfrak{S}_{m,q}^{(q_1, q_2)}}$ and $\|V\|_{\tilde{\mathfrak{S}}_{m,q}^{(\tilde{q}_1, \tilde{q}_2)}}$, as the  \textit{sub-scaling norms} corresponding to each spaces.

\end{itemize}
\end{defn}
%%%%%%%%%%%%%%%%%%%%%%%%%%%%%%%%%%%%%%%%%%

We remark that the main existence results are given under certain assumptions on $V$ in sub-scaling classes. The critical case matters the most and the strict-subcritical case (the case of strict inequality in
\eqref{subSerrin} or \eqref{subSerrin-grad})
 is simpler, and therefore, proofs throughout the paper concern only the case of scaling invariant classes.

\begin{remark}\label{R:S-tildeS}
\begin{enumerate}
    \item[(i)] In Figure 1, we graph the lines of $(\frac{1}{q_1}, \frac{1}{q_2})$ satisfying $\mathcal{S}_{m,q}^{(q_1, q_2)}$ defined in \eqref{Serrin} for $m, q \geq 1$. We easily observe that $q_{m,d}$ becomes $0$ as $m \to 1$ or $q\to \infty$, which corresponds that $\frac{d}{q_1} + \frac{2}{q_2}=1$ as in \eqref{linear-Serrin}.
    The condition \eqref{linear-Serrin} appears in the context of understanding heat equations $(m=1)$  or $L^{\infty}$-sense scaling invariant classes of porous medium equations
    $(m>1, q=\infty)$ (see e.g. \cites{CHKK17, KZ18, HZ21}).  Therefore, our scaling invariant condition \eqref{Serrin} is more comprehensive including the linear case, i.e.  $m=1$  or the case of PME for existence of $L^{\infty}$-solutions.

    \item[(ii)] In Figure 2, we graph lines of $(\frac{1}{\tilde{q}_1}, \frac{1}{\tilde{q}_2})$ satisfying $\tilde{\mathcal{S}}_{m,q}^{(\tilde{q}_1, \tilde{q}_2)}$
    defined in \eqref{Serrin-grad} for $m, q \geq 1$. In fact, by using embedding property, one can supress the norm of $V$ by the norm of $\nabla V$;  that is,
    \begin{equation}\label{norm-V-gradV}
   \|V\|_{\mathcal{S}_{m,q}^{(q_1, q_2)}}  := \| V\|_{L^{q_1, q_2}_{x,t}} \lesssim \|V\|_{\tilde{\mathcal{S}}_{m,q}^{(\tilde{q}_1, \tilde{q}_2)}} := \|\nabla V\|_{L^{\tilde{q}_1, \tilde{q}_2}_{x,t}}
    \end{equation}
   provided
$  \tilde{q}_1 = \frac{dq_1}{d+q_1} \in (1,d)  \ (\text{equivalently } q_1 = \frac{d \tilde{q}_1}{d-\tilde{q}_1} \in ( \frac{d}{d-1}, \infty)),$ and
 $\tilde{q}_2 = q_2$. Hence, the range of $\tilde{q}_2$ is
    \begin{equation}\label{tilde-q2}
    \begin{cases}
    \frac{2+q_{m,d}}{1+{q_{m,d}}} < \tilde{q}_2 \leq \infty, & \text{ if } 1 < m < 1+ \frac{q (d-2)}{d}  \vspace{1 mm}\\
    \frac{2+q_{m,d}}{1+{q_{m,d}}} < \tilde{q}_2 < \frac{2+q_{m,d}}{2-d+q_{m,d}}, & \text{ otherwise}.
    \end{cases}
    \end{equation}
    In Figure 2, the intersection of lines and shaded region is where \eqref{tilde-q2} holds.

    \item[(iii)] Let $1 < q < q^{L}$ and let $ (q_1^L, q_2^L)$ be a pair satisfying $\mathcal{S}_{m,q^L}^{(q_1^L, q_2^L)}$. Concerning $\mathcal{S}_{m,q}^{(q_1^L, q_2)}$ and $\mathcal{S}_{m,q^L}^{(q_1^L, q_2^L)}$, we observe that $q_2^L > q_2$ illustrated in Figure 1. Moreover, the H\"{o}lder inequality in $t\in [0, T]$ yields
    \begin{equation}\label{q-qL}
    \|V\|_{\mathcal{S}_{m,q}^{(q_1^L, q_2)}} \leq T^{ \frac{1}{q_2} - \frac{1}{q_2^L}} \|V\|_{\mathcal{S}_{m,q^L}^{(q_1^L, q_2^L)}}.
    \end{equation}
\end{enumerate}
\end{remark}

%%%%%%%%%%
% Figure for $\mathcal{S}$ and $\tilde{\mathcal{S}}$
\begin{center}
\begin{tikzpicture}[domain=0:16]

\draw (0, -1) node[right] { \scriptsize{Figure 1. $\mathcal{S}_{m,q}^{(q_1, q_2)}$ in Definition~\ref{D:Serrin}.} };

\draw[->] (0,0) node[left] {\scriptsize $0$} -- (6,0) node[right] {\scriptsize $\frac{1}{q_1}$};
\draw[->] (0,0) -- (0,5) node[left] { \scriptsize $\frac{1}{q_2}$};
\draw (0,4) node{\scriptsize $+$} node[left] {\scriptsize $1$} ;

\draw (0,2) node {\scriptsize $+$} node[left] {\scriptsize $\frac 12$} -- (1.2, 0) node {\scriptsize $\times$} node[below] {\scriptsize $\frac 1d$};
\draw (0.9, 0.2) node{\scriptsize $\mathcal{S}$};

\draw (0,3.5) node {\scriptsize $\times$} node[left] {\scriptsize $\frac{1}{{\lambda_1}}$} -- (4.5, 0) node {\scriptsize $\times$} node[below] {\scriptsize $\frac{1+d(m-1)}{d}$} ;
\draw (4, 1) node {\scriptsize $\mathcal{S}_{m,1}^{(q_1, q_2)}$} ;

\draw (0,2.9) node {\scriptsize $\times$} node[left] {\scriptsize $\frac{1+q_{m,d}}{2+q_{m,d}}$}  -- (2.8, 0) node {\scriptsize $\times$} node[below] {\scriptsize $\frac{1+q_{m,d}}{d}$};
\draw (2.4, 1) node{\scriptsize $\mathcal{S}_{m,q}^{(q_1, q_2)}$};

\draw[thick, dotted] (0,2.5) node {\scriptsize $\times$}   -- (2.1, 0) node {\scriptsize $\times$} ;

\draw[very thin] (1.6, 0) node{\scriptsize $*$} node[below]{\scriptsize $\frac{1}{q_1^L}$ } -- (1.6, 1.7);

\draw[very thin] (0, 0.6) node{\scriptsize $*$} node[left]{\scriptsize $\frac{1}{q_2^L}$} -- (1.6, 0.6) node{\scriptsize $*$};

\draw[very thin] (0, 1.25) node{\scriptsize $*$} node[left]{\scriptsize $\frac{1}{q_2}$} -- (1.6, 1.25) node{\scriptsize $*$};

\draw (7, 4) node[right]{\scriptsize $\mathcal{S} = \mathcal{S}_{1,q}^{(q_1,q_2)}$ or $\mathcal{S}_{m, \infty}^{(q_1, q_2)}$ };
\draw (7, 3.5) node[right]{\scriptsize $\mathcal{S}_{m,q^L}^{(q_1,q_2)}$ (dotted line) for $1 < q < q^L < \infty$ };
%\draw (3,3.5) node[right]{ \scriptsize  for $1 < q < q^L < \infty$};

\end{tikzpicture}
\end{center}

\begin{center}
\begin{tikzpicture}[domain=0:16]

\draw (0, -1) node[right] {\scriptsize Figure 2.  $\tilde{\mathcal{S}}_{m,q}^{(q_1, q_2)}$ in Definition~\ref{D:Serrin}. };

\fill[fill= lgray]
(1.1, 0) -- (4,0) -- (4, 3.5) --(1.1, 3.5);

\draw[->] (0,0) -- (6,0) node[right] {\scriptsize $\frac{1}{\tilde{q}_1}$};
\draw[->] (0, 0) -- (0, 5) node[left] {\scriptsize $\frac{1}{\tilde{q}_2}$};

\draw (0,4) node{\scriptsize $+$} node[left] {\scriptsize $1$}
-- (2.2, 0) node{\scriptsize $\bullet$} node[below] {\scriptsize $\frac 2d$} ;
\draw (1.4, 1) node {\scriptsize $\tilde{\mathcal{S}}$};

\draw (0, 4) -- ( 5.5 , 0) node{\scriptsize $\times$} node[below] {\scriptsize $\frac{2+d(m-1)}{d}$};
\draw (5, 1) node {\scriptsize $\tilde{\mathcal{S}}_{m,1}^{(\tilde{q}_1, \tilde{q}_2)}$};

\draw (0,4) -- (4.5, 0) node{\scriptsize $\times$};
\draw (0, 4) -- (3, 0) node{\scriptsize $\bullet$} node[below] {\scriptsize $\frac{2+q_{m,d}}{d}$};
\draw (2.8, 1) node {\scriptsize $\tilde{\mathcal{S}}_{m,q}^{(\tilde{q}_1, \tilde{q}_2)}$};

\draw (7, 4) node[right] {\scriptsize $\tilde{\mathcal{S}} = \tilde{\mathcal{S}}_{1,q}^{(\tilde{q}_1, \tilde{q}_2)}$ or $\tilde{\mathcal{S}}_{m,\infty}^{(\tilde{q}_1, \tilde{q}_2)}$};

\draw[very thin, color=gray]
(1.1, 0) node{\scriptsize $*$} node[below] {\scriptsize $\frac 1d$}
-- (1.1, 3.7) node[above] { \scriptsize{$\tilde{q}_1 = d$}};

\draw[thick] (1.1, 2) circle(0.05) ;
\draw[thick] (1.1, 2.5) circle(0.05) ;
\draw[thick] (1.1, 3) circle(0.05) ;
\draw[thick] (1.1, 3.2) circle(0.05) ;

\draw[thin, dotted] (0, 2) node{\scriptsize $+$} node[left] {\scriptsize $\frac 12$} -- (1.1, 2);
\draw[thin, dotted] (0, 2.5) node{\scriptsize $*$}  -- (1.1, 2.5);

\draw (0, 2.6) node[left] {\scriptsize $\frac{1+q_{m,d}}{2+q_{m,d}}$};
\draw[thin, dotted] (0, 3.2) node{\scriptsize $*$} node[left] {\scriptsize $\frac{1+d(m-1)}{2+d(m-1)}$}  -- (1.1, 3.2);

\draw[very thin, color=gray] (4, 0) node{\scriptsize $*$} node[below] {\scriptsize $1$} -- (4, 3.7) node[above] {\scriptsize{$\tilde{q}_1 = 1$}};
\draw[thick] (4, 0.45) circle(0.05) ;
\draw[thick] (4, 1.1) circle(0.05) ;

\draw[thin, dotted] (0, 0.45) node{\scriptsize $*$} node[left] {\scriptsize $\frac{2-d+q_{m,d}}{2+q_{m,d}}$} -- (4, 0.45);
\draw[thin, dotted] (0, 1.1) node{\scriptsize $*$} node[left] {\scriptsize $\frac{2+d(m-2)}{2+d(m-1)}$} -- (4, 1.1);

\end{tikzpicture}
\end{center}

Our main objective is to establish the existence of $L^q$-weak solutions of \eqref{E:Main} subject to $\rho_0\in L^q(\R^d)$ for $1\le q<\infty$.
More specifically, we provide sufficient conditions of drift term $V$ or $\nabla V$ to ensure the existence of solutions for \eqref{E:Main} with  $L^q$-initial data (the collection of $V$ or $\nabla V$ is defined in \eqref{subSerrin} or \eqref{subSerrin-grad}).
An interesting feature is that such conditions of $V$ or $\nabla V$ depend on $m$ and $q$, which is one of major differences compared to linear equation with drift term, since \eqref{linear-Serrin} of the linear case is independent of $q$ (compare to \cite{KK-arxiv}).

As far as authors understand, weak solutions of PME with general drift terms have been barely studied, although weak solutions of homogeneous PME or coupled system with PME type with gradient flow structures have been extensively studied so far. Since the PME with drift terms under our considerations, in general, does  not have either
maximum principle or gradient flow structure, existence of weak solutions does not seem to be established by the existing methods of proofs based on such arguments.
In fact, our main tool in constructing weak solutions is, so called splitting method, which solves the PME with drift term by solving homogeneous PME and transport equations
via flow maps separately, which enables us to obtain aprroximated solutions and convergence to weak solutions. We are positive that there are probably other ways of proving existence of weak soltuions, but they must be  independent of maximum principle and gradient flow structure, and thus we think that our proof is a theoretical way of new approach. We expect that our method of the proof and consequential results could be applicable to broader classes of equations.

We make several comments, for simplicity, we deal with only the case $d\ge 3$ in the tables below (two dimensional case is also treated in similar ways but there are little differences such as range of parameters' values, compared to higher dimensions).

\begin{table}[hbt!]\label{Table1}
\begin{center}
\caption{\footnotesize Guide of existence results of $L^1$-weak solution in Wasserstein space, $q=1$ and $d\geq 3$.}
\smallskip

{\scriptsize
\begin{tabular}{| c  | c  || c | c |}\hline
\rule[-8pt]{0pt}{22pt}
 \textbf{Range of $m$} & \textbf{Intial data}   &  \textbf{Conditions on $V$ in critial class} & \textbf{References}   \\ \hline \hline

\rule[-8pt]{0pt}{22pt}
\multirow{4}{*}{$1<m\leq 2$}
& \multirow{3}{*}{$1<p\leq {\lambda_1} := 1+\frac{1}{d(m-1)+1}$}
&  $V\in \mathfrak{S}_{m,1}^{(q_1, q_2)}$ for $2\leq q_2 \leq \frac{m}{m-1}$
& Theorem~\ref{Theorem-1}, Figure 3 (dark) \\
%\rule[-8pt]{0pt}{10pt}
%${\lambda_1} := 1+\frac{1}{d(m-1)+1}$ & &  &   Figure 3-(i) (dark) \\
\cline{3-4}

\rule[-8pt]{0pt}{22pt}
& \multirow{3}{*}{$\rho_0 \in \mathcal{P}_p (\mathbb{R}^d)$, $\int_{\mathbb{R}^d}\rho_0 \log \rho_0 \, dx < \infty$ }
&  $\nabla \cdot V = 0$ \& $V\in \mathfrak{S}_{m,1}^{(q_1, q_2)}$ for ${\lambda_1}\leq q_2 \leq \frac{{\lambda_1} m}{m-1}$
& Theorem~\ref{T:log-div-free}, Figure 3 (light) \\
%\rule[-8pt]{0pt}{10pt}
%$\rho_0 \in \mathcal{P}_p (\mathbb{R}^d)$ \& & &  $V\in \mathcal{S}_{m,1}^{(q_1, q_2)}$ for ${\lambda_1}\leq q_2 \leq \frac{{\lambda_1} m}{m-1}$ & Figure 3 (light) \\
\cline{3-4}

\rule[-8pt]{0pt}{22pt}
&
&  $V \in \tilde{\mathfrak{S}}_{m,1}^{(\tilde{q}_1, \tilde{q}_2)}$ for ${\lambda_1} < \tilde{q}_2 \leq \frac{m}{m-1}$
& Theorem~\ref{T:log-tilde},  Figure 9-$(i)$ (line $\overline{\textit{\textbf{DE}}}$\,) \\
%\rule[-8pt]{0pt}{10pt}
%& & &  Figure 8 (line $\overline{\textit{\textbf{DE}}}$\,) \\
\hline
\end{tabular}
}
\end{center}
\end{table}

Firstly, we distinguish the case $q=1$ and $q>1$. For the case $q=1$ (in Table~\ref{Table1}), in general, we need to restrict $m\in (1, 2]$, since a priori entropy estimate is not clear for the case $m>2$, unless $V$ is divergence-free or $\nabla V$ belongs sub-scaling classes. Compactness arguments also give the same restriction on $m$, even for the cases of $\nabla \cdot V = 0$ or $ V \in \tilde{\mathfrak{S}}_{m,1}^{(\tilde{q}_1, \tilde{q}_2)}$.

In addition, initial data requires more than just $L^1$, and more precisely it is necessary that $\int_{\R^d} \rho_0 \log \rho_0 dx<\infty$ due to entropy estimates. In this case, we suppose that $V\in \mathfrak{S}_{m,1}^{(q_1, q_2)}$ with $2\le q_2\le \frac{m}{m-1}$. Furthermore, if $\rho_0 \in \calP_p(\bbr^d)$ with $1<p\le {\lambda_1}:=1+\frac{1}{d(m-1)+1}$, we can obtain the $p$-th moment estimates and the \emph{speed}  (see Remark~\ref{R:speed}) estimates, which assure the existence of weak solutions in Wasserstein spaces
 (see Theorem~\ref{Theorem-1}).

Similar analysis can be performed for the case
$V \in \tilde{\mathfrak{S}}_{m,1}^{(\tilde{q}_1, \tilde{q}_2)}$  (refer Theorem~\ref{T:log-tilde}) or $V$ is divergence-free (see Theorem~\ref{T:log-div-free}).
In divergence-free case, we emphasize that the range of subscaling class becomes wider, i.e.  $V\in \mathfrak{S}_{m,1}^{(q_1, q_2)}$ with ${\lambda_1}\le q_2\le \frac{{\lambda_1} m}{m-1}$, mainly because  the drift term does not show up in the energy estimate via integration by parts. We refer Section~\ref{Exist-weak} for the details of proofs.

%%%%%%
\begin{table}[hbt!] %\label{Table2}
\begin{center}

\caption{\footnotesize Guide of existence results of $L^q$-weak solutions, $m>1$ and $d\geq 3$.}
\smallskip

{\scriptsize
\begin{tabular}{| c  | c || c | c |}\hline

\rule[-8pt]{0pt}{22pt}
\textbf{Range of $q$} &\textbf{Intial data} &  \textbf{Conditions on $V$ in critical class} & \textbf{References}   \\ \hline \hline

%% L^q data
\rule[-8pt]{0pt}{22pt}
$q>1$ and $q\geq m-1$
& \multirow{5}{*}{$\rho_0 \in L^q (\bbr^d) \cap \calP(\bbr^d)$}
& $V\in \mathfrak{S}_{m,q}^{(q_1, q_2)}$ for $2\leq q_2 \leq \frac{q+m-1}{m-1}$
& Theorem~\ref{Theorem-2-a}, Figure 4, 5 \\

\cline{1-1}
\cline{3-4}

\rule[-8pt]{0pt}{22pt}
\multirow{3}{*}{ $q > \max\{1, \frac m2 \}$}
&
& $\nabla \cdot V = 0$ \&  $V\in \mathfrak{S}_{m,q}^{(q_1, q_2)}$ holding \eqref{T4:V-energy}
& Theorem~\ref{Theorem-4a}, Figure 6-$(i,ii)$\\
%\rule[-8pt]{0pt}{10pt}
%& & $\nabla \cdot V = 0$ \& $V\in \mathcal{S}_{m,q}^{(q_1, q_2)}$ for $q_1, q_2 \geq 1$ &  Theorem~\ref{Theorem-4a}, Figure 6-(i, ii)  \\
\cline{3-4}

\rule[-8pt]{0pt}{22pt}
& & $V \in \tilde{\mathfrak{S}}_{m,1}^{(\tilde{q}_1, \tilde{q}_2)}$ for $\frac{2+q_{m,d}}{1+q_{m,d}} < \tilde{q}_2 \leq \frac{q+m-1}{m-1}$  & Theorem~\ref{Theorem-5a},  Figure 9-$(i,ii)$ \\
\hline
\end{tabular}
}
\end{center}
\end{table}

On the other hand (see Table~2), in case $q>1$, when usual $L^q$-energy estimates are performed, it is necessary to assmue that  $\rho_0\in L^q(\R^d)$.
Here we need to consider $1<m\le 2$ and $m>2$ separately.
In case that $m>2$, an additional restriction, $q\ge m-1$ is necessary for the existence of weak solutions (see Theorem \ref{Theorem-2-a}).
Indeed, when $L^q$-energy estimates and compactness arguments are conducted, we should control the quantity, $\rho$ to the power $q-m+1$, which requres nonnegative sign that is equivalent to $q\ge m-1$.
This restriction might be technical but we do not know how to handle
the case of $1<q<m-1$, and thus we leave such case as an open question.

In case $m>2$, the restriction on $q\geq m-1$ can be relaxed to $q> \frac m2$ if either $V$ is divergence-free or $\nabla V$ is under consideration (refer Theorem~\ref{Theorem-4a} and \ref{Theorem-5a}). This is because the restriction $q > \frac m2$ for $m>2$ is required for compactness arguments while a priori $L^q$-energy estimates are obtained for all $m,q>1$.
% The relaxation is due to the compactness argument (refer Proposition~\ref{Proposition : AL-2})
%If either $V$ is divergence-free or $\nabla V$ is under consideration, the $L^q$-energy estimates are obtained for all $m,q>1$. However $q\geq \frac m2$ comes from compactness arguments (refer Proposition~\ref{Proposition : AL-2}).

\begin{table}[hbt!] %\label{Table3}
\begin{center}
\caption{\footnotesize Guide of existence results of $L^q$-weak solutions in Wasserstein spaces, $m>1$ and $d\geq 3$.}
\smallskip

{\scriptsize
\begin{tabular}{| c  | c  || c | c |}\hline

\rule[-8pt]{0pt}{22pt}
\textbf{Range of $ q$} & \textbf{Intial data}   &  \textbf{Conditions on $V$ in critical class} & \textbf{References}   \\ \hline \hline

\rule[-8pt]{0pt}{22pt}
 $q>1$ and $q\geq m-1$
 &
 & $V\in \mathfrak{S}_{m,q}^{(q_1, q_2)}$ holding \eqref{T2:V}
 & Theorem~\ref{Theorem-2-b} $(i)$ , Figure 4-($i$-$iii$), 5-$(i,ii)$ \\
\cline{1-1}
\cline{3-4}

\rule[-8pt]{0pt}{22pt}
\multirow{3}{*}{ $q>1$ and $q\geq \frac m2$}
&  \multirow{2}{*}{$\rho_0 \in \mathcal{P}_p(\mathbb{R}^d)\cap L^{q}(\mathbb{R}^d)$}
& $\nabla \cdot V = 0$ \& $V\in \mathfrak{S}_{m,q}^{(q_1, q_2)}$ holding \eqref{T4:V-divfree}
& Theorem~\ref{Theorem-4} $(i)$, Figure 7-($i$-$iii$), 8-$(i,ii)$ \\
\cline{3-4}

\rule[-8pt]{0pt}{22pt}
& \multirow{2}{*}{where $1<p\leq {\lambda_q}$}
& $V \in \tilde{\mathfrak{S}}_{m,1}^{(\tilde{q}_1, \tilde{q}_2)}$ holding \eqref{T5:V-tilde}
& Theorem~\ref{Theorem-5} $(i)$, Figure 9-$(i,ii)$ \\
\cline{1-1}
\cline{3-4} 

\rule[-8pt]{0pt}{22pt}
\multirow{3}{*}{\textit{Embedding results}}
& \multirow{1}{*}{${\lambda_q} :=\min\{2, 1+\frac{d(q-1)+q}{d(m-1)+q}\}$}
& \multicolumn{2}{c|}{ Theorem~\ref{Theorem-2-b} $(ii)$ , Figure 4-$(e)$, 5-$(e)$  }\\
\rule[-8pt]{0pt}{10pt}
& & \multicolumn{2}{c|}{ Theorem~\ref{Theorem-4} $(ii)$, Figure 7-$(e)$, 8-$(e)$ }\\
\rule[-8pt]{0pt}{10pt}
& & \multicolumn{2}{c|}{Theorem~\ref{Theorem-5} $(ii)$, Figure 9-$(e)$ }\\
\hline

\end{tabular}
}
\end{center}
\end{table}

The same restrictions on $m$ and $q$ are required, when we look for weak solutions in Wasserstein space (see Table 3). In such case, the subscaling class of $V$ or $\nabla V$ becomes narrower because
the $p$-th moment and the speed estimates are shown to be valid for more restrictive values of parameters
(see Theorem \ref{Theorem-2-b} $(i)$, Theorem \ref{Theorem-4} $(i)$ and Theorem \ref{Theorem-5} $(i)$ for the details).

Therefore, it could happen that
there is a weak solution satisfying only energy estimates but we do not know if it belongs to Wasserstein spaces, because subscaling class of $V$ corresponding to energy space turns out to be wider than that of $V$ relevant to Wasserstein spaces.
Nevertheless, using temporal embedding, some cases guarantee that weak solutions with energy estimates belong to  Wasserstein spaces as well (see  Theorem \ref{Theorem-2-b}  $(ii)$, Theorem \ref{Theorem-4} $(ii)$ and Theorem \ref{Theorem-5} $(ii)$ for the details).
%In such cases, the moment and the speed estimates are, however, weaker than what are expected, since they are caused by embedding in time.

As a main tool, we adopt the splitting method to establish approximated solutions, which converges to a $L^q$-weak solution (see section \ref{splitting method} and section \ref{Exist-weak}). %We remark that
In the process of constructing the approximated solutions, %it is used that if initial data is sufficiently regular,
we use that  H\"{o}lder continuity of PME in the absense of drift term is preserved uniformly up to initial time, which seems to be of independt interest (see Appendix~\ref{Appendix:Holder}).

On the other hand, uniqueness for solutions of \eqref{E:Main} does not seem to be obvious, since the maximum principle, in general, is not true due to nonlinear diffusion with drift term. Nevertheless, we obtain uniqueness result of $L^m$-weak solutions under additional conditions for gradient of drift and initial data, exploiting mass transportation theory (see Theorem \ref{Corollary : Uniqueness} and Corollary \ref{Corollary : Uniqueness-2}, and refer Section~\ref{PME-KS-eq} for proofs).

As an application, we consider a repulsive Keller-Segel system of porous medium type.
With the aid of existence results, if the power $m$ of degenerate diffusion is beyond a certain positive number, we show that
solutions become bounded, which turns out to be an improvement, compared to known results (see Theorem~\ref{thm-jengchik50} and refer Section~\ref{PME-KS-eq} for its proof). Existence results developed in this paper are expected to be applicable to other situations related to equations of PME type as well.
%{\color{magenta} 
%We remark that there are huge literatures regarding blow-up of solutions for classical Keller-Segel equations of linear parabolic or PME type. 
%Blow-up results references \cites{BiaLiu, HorWin, Win}}

Our paper is organized as follows:
In Section~\ref{S: Main}, we state all main results and several remarks and figures are provided to help readers understand.
Some preliminaries are prepared in Section~\ref{S:Preliminaries}. Section~\ref{S:a priori} is devoted to giving a priori estimates for regular solutions (see Definition~\ref{D:regular-sol}).
In Section~\ref{splitting method}, the existence of regular solutions of PME is established by the splitting method.
Section~\ref{Exist-weak} is prepared for proofs of existence results stated in Section~\ref{S: Main}.
Finally, we provide the proof of uniqueness results and we treat a repulsive Keller-Segel system of porous medium type in Section~\ref{PME-KS-eq}.
We present the proof of uniform H\"{o}lder continuity of homogeneous PME up to initial time in Appendix~\ref{Appendix:Holder}.
In Appendix~\ref{Appendix:fig}, supplemetary figures are additionally added.

%------------------------------------------------------------------------------------------

%------------------------------------------------------------------------------------------

\section{Main results}\label{S: Main}

In this section, we describe our main results and make relevant remarks for each of them. First, existence of weak solutions for \eqref{E:Main} are categorized according to hypotheses of given initial data and drifts.
Uniqueness of weak solutions are discussed as well. Finally, as an application, we study a  repulsive type Keller-Segel equations to improve previously known results by taking advantage of develped main results.

%In this section, we introduce main theorems organized into four categories: first the existence results, second the uniqueness results, third the uniform H\"{o}lder continuity, and the last an appplication. First, the existence results are composed with four smaller sections depending on the initial data and assumptions on $V$. The second is about uniqueness results in Wasserstein space under suitable vector field $V$. The third is prepared for uniform H\"{o}lder continuity of homogeneous PME up to $t=0$. The fourth is to deliver how to apply earlier existence results to a certain repulsive Keller-Segel model of PME type.

Before stating main theorems, for convenience, we introduce some notations and make a few remarks.
\begin{itemize}
\item Let us denote by $\mathcal{P}(\mathbb{R}^d)$ the set of all Borel
probability measures on $\mathbb{R}^d$.
Furthermore, denote $\mathcal{P}_p(\mathbb{R}^d)$ as the space $\mathcal{P}(\mathbb{R}^d)$ with a finite $p$-th moment, that is $\int_{\mathbb{R}^d} \langle x \rangle^{p} \,d\mu < \infty$ if
$ \mu \in \mathcal{P}_p(\mathbb{R}^d)$.
We refer Section~\ref{SS:Wasserstein} for definitions and related properties of the Wasserstein distance denoted by $W_p$, the Wasserstein space, and AC (absolutely continuous) curves.

  \item The $p$-th power is given as
%\footnote{KK: $1\le p<\infty$ may be right, instead $1< p\le 2$}
\begin{equation*}%\label{p-moment}
\langle x \rangle^{p} = \left(1+ |x|^2\right)^{\frac p2}, \ \quad \text{ for } \ 1 \leq p < \infty \ \text{ and } \ x\in\bbr^d, \ d\geq 2.
\end{equation*}

\item  Let $A\subset \mathbb{R}$ be a measurable set. We denote by $\delta_A$ the indicator function of $A$. 

  \item For any $m>1$ and $q\geq 1$, let us define the following constant
\begin{equation}\label{lambda_q}
 \lambda_q := \min \left\{2, 1+\frac{d(q-1)+q}{d(m-1)+q} \right\},
\end{equation}
that can be rewritten as ${\lambda_q} = \left\{1+\frac{d(q-1)+q}{d(m-1)+q} \right\}\cdot \delta_{\{1 \leq q \leq m\}} + 2 \cdot \delta_{\{q > m\}}$.
The constant ${\lambda_q} \in (1, 2]$ in \eqref{lambda_q}  naturally comes from evaluating the $p$-th moment and the speed estimates which are essential to play in Wasserstein spaces (refer Section~\ref{SS:speed}). A straightforward observation gives that $ 1 + \frac{d(q-1)+q}{d(m-1)+q}=2$ if $q=m$.
%So, ${\lambda_q} = 1 + \frac{d(q-1)+q}{d(m-1)+q} \in (1,2)$ for $1 \leq  q < m$, and ${\lambda_q} =2$ if $q \geq m$.
Also, for any $q\geq 1$, note that ${\lambda_q} =2$ if $m=1$.

\item The letters $c$ and $C$ are used for generic constants.  Also, the letter $\theta$ is a generic constant  which varies with arguments of interpolation.
Throughout the paper, we omit the dependence on $m$, $q$, $d$, $T$, and $\|\rho_0\|_{L^{1}(\bbr^d)}$, because we regard $m>1$, $q \geq 1$, $d\geq 2$, $T>0$ are given constants and $\|\rho_0\|_{L^{1}(\bbr^d)} = 1$
in either $\mathcal{P}(\mathbb{R}^d)$ or $\mathcal{P}_p(\mathbb{R}^d)$.

\item In figures, the notation $\cal{R}(\textbf{Ab\ldots Yz})$ is used to indicate a polygon with vertices $\textbf{A}, \textbf{b}, \ldots, \textbf{Y}, \textbf{z}$. Also the notation $\overline{\textbf{Ab\ldots Yz}}$ is used for the piecewise line segments connecting $\textbf{A}$, $\textbf{b}$, $\ldots$, $\textbf{Y}$, $\textbf{z}$.
\end{itemize}

Here we introduce the notion of weak solutions of \eqref{E:Main}.
\begin{defn}\label{D:weak-sol}
%Let $q\in [1, \infty)$ and $r\in [1, q]$.
%Let $\rho$ be a nonnegative measurable function and $V$ be a measurable vector field satisfying
%\[
%\rho \in L^{\infty}\left(0, T; \left(L^{1}\cap L^r\right)(\mathbb{R}^d)\right)\cap L^{m} (Q_T), \quad \nabla \rho^{\frac{m+r-1}{2}} \in L^{2}(Q_T) \quad   \text{and} \quad \rho V \in L^{1}(Q_T).
%\]
%Then we call $\rho$ is a nonnegative \underline{$L^q$-weak solution} of \eqref{E:Main} with $\rho_0 \in L^{1}(\mathbb{R}^d)\cap L^{q}(\mathbb{R}^d)$,
%if the identity holds
%\begin{equation}\label{KK-May7-40}
%\iint_{Q_{T}} \left\{ \rho \varphi_t + \rho^m \Delta\varphi + \rho V \cdot \nabla \varphi \right\} \,dx dt = -\int_{\mathbb{R}^d} \rho_{0} (\cdot) \varphi(\cdot, 0) \,dx,
%\end{equation}
%for all functions $\varphi \in {\mathcal{C}}_c^\infty (\bbr^d \times [0,T))$.
Let $q\in [1, \infty)$ and $V$ be a measurable vector field.
We say that a nonnegative measurable function $\rho$ is
a \textbf{$L^q$-weak solution} of \eqref{E:Main} with $\rho_0 \in L^{1}(\mathbb{R}^d)\cap L^{q}(\mathbb{R}^d)$ if the followings are satisfied:
\begin{itemize}
\item[(i)] It holds that
\[
\rho \in L^{\infty}\left(0, T; \left(L^{1}\cap L^q\right)(\mathbb{R}^d)\right)\cap L^{m} (Q_T), \quad \nabla \rho^{\frac{m+q-1}{2}} \in L^{2}(Q_T), \quad   \text{and} \quad \rho V \in L^{1}(Q_T).
\]
\item[(ii)] For any function $\varphi \in {\mathcal{C}}_c^\infty (\bbr^d \times [0,T))$ it holds that
\begin{equation}\label{KK-May7-40}
\iint_{Q_{T}} \left\{ \rho \varphi_t + \rho^m \Delta\varphi + \rho V \cdot \nabla \varphi \right\} \,dx dt = -\int_{\mathbb{R}^d} \rho_{0} (\cdot) \varphi(\cdot, 0) \,dx.
\end{equation}
\end{itemize}
\end{defn}

We remind the property of \emph{mass conservation} for nonnegative solutions of \eqref{E:Main}, i.e.
$\|\rho(\cdot, t)\|_{L^{1} (\mathbb{R}^d)} = \|\rho_0\|_{L^{1}(\mathbb{R}^d)}$ for a.e. $t\in [0, T]$. This can be derived from the weak formulation (see e.g. \cite[Section~9.5.1]{Vaz06}).

\begin{remark}
Without loss of generality, we assume that $\rho_0 \in \mathcal{P} (\mathbb{R}^d)$, that is, $\|\rho_0\|_{L^1 (\bbr^d)} = 1$. We note that the mass conservation property in time implies that $\|\rho(\cdot, t)\|_{L^{1} (\mathbb{R}^d)} = 1$ a.e.  $t\in [0, T]$.

If $\|\rho_0\|_{L^1(\bbr^d)} = c > 0$, then we replace $\rho$ by $\tilde{\rho} = \rho/c$. Thus we have $\tilde{\rho}_0 \in \mathcal{P}(\bbr^d)$ and
\[
\partial_{t} \rho + \Delta \rho^m = \nabla \cdot (V \rho)  \quad \Longleftrightarrow \quad \partial_{t} \tilde{\rho} + c^{m-1}\Delta (\tilde{\rho})^m = \nabla \cdot (V \tilde{\rho}).
\]
In fact, the diffusive coefficient $c^{m-1}$ of the equation for $\tilde{\rho}$ does not play important role in our analysis, and thus, for simplicity, we assume that $c=1$.
\end{remark}

\subsection{Existence results}\label{SS:Existence}

In this subsection, we state all existence results and related remarks and figures are presented.

\subsubsection{Existence for case: $ \int_{\bbr^d} \rho_0 \log \rho_0 \,dx < \infty$.}\label{SS:m less 2}
In this subsection, we assume that the initial data satisfy
\begin{equation}\label{T1:rho0-log}
   \rho_0 \in \mathcal{P}_p(\bbr^d) \quad \text{and} \quad \int_{\bbr^d} \rho_0 \log \rho_0 \,dx < \infty,
  \end{equation}
  where $p\in (1, 2]$ will be specified later.
We remark that it is not enough to assume merely $\rho_0 \in L^1 (\mathbb{R}^d)$ for showing the existence of $L^1$-weak solutions. Instead, with slightly better initial data satisfying \eqref{T1:rho0-log}, we compute a priori estimates in Proposition~\ref{P:log-energy} and Proposition~\ref{P:log-energy-speed}, which enable us to obtain
existence of weak solutions via suitable approximations.
Here we consider separately three different cases depending on the condition of $V$.
Firstly, Theorem~\ref{Theorem-1} is when $V$ belongs to a certain subscaling classses.
Moreover, Theorem~\ref{T:log-div-free} deals with the case that $V$ is a divergence-free vector and Theorem~\ref{T:log-tilde} is for the case $\nabla V$ belongs a sub-scaling classes, respectively.
The restriction of $1<m\leq 2$ appears in general because of compactness arguments (see Remark~\ref{remark:theorem1} $(ii)$).

For the first case, our result reads as follows:
\begin{theorem}\label{Theorem-1}
 Let $1<m\leq 2$ and  $p>1 $. Suppose that
\eqref{T1:rho0-log} and
  \begin{equation}\label{T1:V-log}
V\in\mathfrak{S}_{m,1}^{(q_1, q_2)} \quad \text{for} \quad
\begin{cases}
 2 \leq q_2 \leq \frac{m}{m-1},  & \text{ if } d> 2, \vspace{1 mm}\\
 2 \leq q_2 < \frac{m}{m-1}, & \text{ if } d= 2.    \end{cases}
\end{equation}

\begin{itemize}
\item [(i)] For  $1 < p \leq  {\lambda_1} := 1+ \frac{1}{d(m-1)+1}$, then there exists
  a nonnegative $L^1$-weak solution of \eqref{E:Main} in Definition~\ref{D:weak-sol} such that  $\rho \in AC(0,T; \mathcal{P}_p(\mathbb{R}^d))$ with $\rho(\cdot, 0)=\rho_0$.
Furthermore, $\rho$ satisfies
\begin{equation}\label{T1:apriori-log}
\sup_{0\leq t \leq T} \int_{\mathbb{R}^d\times \{t\}} \rho\left(
\abs{\log \rho} + \langle x \rangle^p \right) \,dx
+ \iint_{Q_T} \{ \abs{\nabla \rho^{\frac m2}}^2 + \left(\abs{\frac{\nabla
\rho^m}{\rho}}^p+|V|^{p}\right ) \rho  \}\,dx\,dt
 \leq  C,
\end{equation}
and
\begin{equation}\label{KK-May7-60}
 W_p(\rho(t),\rho(s))\leq C (t-s)^{\frac{p-1}{p}},\qquad
\forall ~~0\leq s\leq t\leq T,
\end{equation}
where the constant $C= C ( \|V\|_{\mathfrak{S}_{m,1}^{(q_1,q_2)}},\,
\int_{\mathbb{R}^d} \left(\rho_0 \log \rho_0 + \rho_0 \langle x
\rangle^p \right) \,dx )$.

\item [(ii)] For  $ p > {\lambda_1} $, the same conclusions hold as in $(i)$ except that $p$ is replaced by ${\lambda_1}$.
That is, there exists
  a nonnegative $L^1$-weak solution of \eqref{E:Main} in Definition~\ref{D:weak-sol} such that  $\rho \in AC(0,T; \mathcal{P}_{{\lambda_1}}(\mathbb{R}^d))$ with $\rho(\cdot, 0)=\rho_0$.
Furthermore, $\rho$ satisfies
\begin{equation}\label{T1:apriori-log-1}
\sup_{0\leq t \leq T} \int_{\mathbb{R}^d\times \{t\}} \rho\left(
\abs{\log \rho} + \langle x \rangle^{{\lambda_1}} \right) \,dx
+ \iint_{Q_T} \{ \abs{\nabla \rho^{\frac m2}}^2 + \left(\abs{\frac{\nabla
\rho^m}{\rho}}^{{\lambda_1}}+|V|^{{\lambda_1}}\right ) \rho \} \,dx\,dt
 \leq  C,
\end{equation}
and
\begin{equation}\label{KK-May7-60-1}
 W_{{\lambda_1}}(\rho(t),\rho(s))\leq C (t-s)^{\frac{{\lambda_1}-1}{{\lambda_1}}},\qquad
\forall ~~0\leq s\leq t\leq T,
\end{equation}
where the constant $C= C ( \|V\|_{\mathfrak{S}_{m,1}^{(q_1,q_2)}},\,
\int_{\mathbb{R}^d} \left(\rho_0 \log \rho_0 + \rho_0 \langle x
\rangle^{p} \right) \,dx )$.
\end{itemize}
\end{theorem}

In the next theorem, $V$ is assumed to be divergence-free.
The scaling invariant class of $V$ is larger than that of Theorem~\ref{Theorem-1} (See Remark~\ref{remark:theorem1} $(iii)$ and Fig. 3 for a more definte comparison).

\begin{theorem}\label{T:log-div-free}
Let  $1<m\leq 2$ and $p>1$. Assume \eqref{T1:rho0-log} and the divergence-free $V$ (i.e., $\nabla \cdot V = 0$) satisfying
\begin{equation}\label{T1:V-log-divfree}
V\in \mathfrak{S}_{m,1}^{(q_1, q_2)} \quad \text{for} \quad
\begin{cases}
{\lambda_1} \leq q_2 \leq \frac{{\lambda_1} m}{m-1}, & \text{ if }  d>2, \vspace{1 mm}\\
{\lambda_1} \leq q_2 < \frac{{\lambda_1} m}{m-1}, & \text{ if }   d=2,
\end{cases}
\end{equation}
 where ${\lambda_1}$ is the number defined in Theorem \ref{Theorem-1}.
Then the same conclusions hold as in Theorem~\ref{Theorem-1}.
\end{theorem}

The last theorem is concerned about the case that $\nabla V$ belongs to a sub-scaling class (see Remark~\ref{remark:theorem1} $(iv)$).
\begin{theorem}\label{T:log-tilde}
Let $1<m\leq 2$ and let $p>1$. Suppose that \eqref{T1:rho0-log} and
\begin{equation}\label{T:V-tilde-log-energy}
V \in \tilde{\mathfrak{S}}_{m,1}^{(\tilde{q}_1, \tilde{q}_2)} \
\text{ for } \
\begin{cases}
\frac{2+d(m-1)}{1+d(m-1)} < \tilde{q}_2 \leq \frac{m}{m-1}, & \text{ if } d>2,
\vspace{1 mm}\\
\frac{2+d(m-1)}{1+d(m-1)} < \tilde{q}_2 < \frac{m}{m-1}, & \text{ if
} d=2.
\end{cases}
\end{equation}
Then the same conclusions hold as in Theorem~\ref{Theorem-1}
with
$C= C ( \|V\|_{\tilde{\mathfrak{S}}_{m,1}^{(\tilde{q}_1, \tilde{q}_2)}},\,  \int_{\mathbb{R}^d} \left(\rho_0 \log \rho_0 + \rho_0 \langle x \rangle^p \right) \,dx )$.
\end{theorem}

Let us provide some remarks regarding three theorems in this subsection.
\begin{remark}\label{remark:theorem1}
\begin{itemize}
\item[(i)] Results in Theorem~\ref{Theorem-1} are also true for linear case $(m=1)$ for $1\le p\le 2$ (see \cite{KK-arxiv}).  We do not know whether or not the moment estimate \eqref{T1:apriori-log} can be extended to $p>{\lambda_1}$, and thus we leave it as an open question.

\item[(ii)]
There is restriction on $m$ that $1 < m \leq 2$ in Theorem~\ref{T:log-div-free} and \ref{T:log-tilde} for cases $\nabla \cdot V =0$ or $V\in \tilde{\mathfrak{S}}_{m,1}^{(\tilde{q}_1, \tilde{q}_2)}$, though a priori estimates are obtained for $m>1$ in Proposition~\ref{P:log-energy} (ii), (iii). Such a restriction is necessary because the compactness argument in Proposition~\ref{Proposition : AL-1} only works for $1<m\leq 2$.
%On the other hand, in Theorem~\ref{Theorem-1}, the range of $1<m\leq 2$ is necessary for both a priori estiamte Proposition~\ref{P:log-energy}$(i)$ and the compactness argument.
    Therefore, if one can resolve the compactness argument in cases $\nabla \cdot V =0$ or $V\in \tilde{\mathfrak{S}}_{m,1}^{(\tilde{q}_1, \tilde{q}_2)}$ for $m>2$, then the results in Theorem~\ref{T:log-div-free} and \ref{T:log-tilde} become valid for any $m>1$.

\item[(iii)] In the following figure, we illustrate the range of $(q_1, q_2)$ satisfying \eqref{T1:V-log} and \eqref{T1:V-log-divfree} in particular for scaling invariant class $\mathcal{S}_{m,1}^{(q_1, q_2)}$ as dark-shaded region $\mathcal{R}\left(\textbf{ABC}\right)$ and as lightly-shaded region $\mathcal{R}\left(\textbf{ABDE}\right)$, respectively. When $d=2$, the line $\overline{\textbf{BC}}$ and $\overline{\textbf{BD}}$ are excluded from $\mathcal{R}\left(\textbf{ABC}\right)$ and $\mathcal{R}\left(\textbf{ABDE}\right)$, respectively.

\begin{center}
\begin{tikzpicture}[domain=0:16]

\draw (0, -1) node[right] { \scriptsize Figure 3. Theorem~\ref{Theorem-1} \& \ref{T:log-tilde} for $1<m\leq 2$, $d>2$.};

\fill[fill= lgray]
(0, 2) -- (1.2,0) -- (3, 1.4) -- (0, 3.2) -- (0, 2);

\fill[fill= gray]
(0, 2) -- (1.2,0) -- (2, 2);

\draw[->] (0,0) node[left] {\scriptsize $0$} -- (6,0) node[right] {\scriptsize $\frac{1}{q_1}$};
\draw[->] (0,0) -- (0,4.5) node[left] { \scriptsize $\frac{1}{q_2}$};

\draw (0,4) node{\scriptsize $+$} node[left] {\scriptsize $1$} ;
\draw (4,0) node{\scriptsize $+$} node[below] {\scriptsize $1$} ;

\draw[very thin]
(0, 2) node{\scriptsize $\bullet$} node[left] {\scriptsize \textbf{A}} -- (2, 2) ;

\draw[very thin]
(2, 0) node{\scriptsize $+$} node[below]{\scriptsize $\frac 12$}   -- (2, 2);

\draw[very thin]
(0, 0) -- (2, 2) node{\scriptsize $\bullet$} node[above] {\scriptsize \textbf{C}};

\draw[very thin]
(1.2, 0) node{\scriptsize $\bullet$} -- (2, 2);

\draw[very thin]
(1.2, 0) -- (3, 1.4) node{\scriptsize $\bullet$} node[right] {\scriptsize \textbf{D}};

\draw (0,3.2) node{\scriptsize $\bullet$} node[left] {\scriptsize $\frac{1+d}{2+d}$} node[right] {\scriptsize \textbf{E}}
-- (5.4, 0) node {\scriptsize $\times$} node[below] {\scriptsize $\frac{1+d}{d}$} ;
\draw (4.8, 0.8) node {\scriptsize $\mathcal{S}_{2,1}^{(q_1, q_2)}$} ;

\draw (0,2)  -- (1.2, 0)  node[below] {\scriptsize \textbf{B}};
%\draw (0.8, 1) node{\scriptsize $\mathcal{S}_{1,1}^{(q_1, q_2)}$};

%\draw (3.5,5) node{ \scriptsize $\mathcal{S} = \mathcal{S}_{1,q}^{(q_1,q_2)}$ or $\mathcal{S}_{m, \infty}^{(q_1, q_2)}$ };

\draw (0,2.7) node {\scriptsize $\bullet$} node[left] {\scriptsize $\frac{1}{{\lambda_1}}$} node[right] {\scriptsize \textbf{f}}  -- (3.1, 0) node {\scriptsize $\times$} node[below] {\scriptsize $\frac{1+d(m-1)}{d}$};
\draw (3.3, 0.5) node{\scriptsize $\mathcal{S}_{m,1}^{(q_1, q_2)}$};

\draw[thin, dotted] (0.8, 0) node{\scriptsize$*$}--(0.8, 2) node{\scriptsize $\bullet$} node[right] {\scriptsize \textbf{g}};
\draw[thin, dotted] (1.7, 0) node{\scriptsize$*$} -- (1.7, 1.22) node{\scriptsize $\bullet$}node[right] {\scriptsize \textbf{h}} ;
\draw[thin, dotted] (0, 1.22) node{\scriptsize$*$} -- (1.7, 1.22);
\draw[thin, dotted] (2.2, 0) node{\scriptsize$*$}--(2.2, 0.8) node{\scriptsize $\bullet$} node[right] {\scriptsize \textbf{i}} ;
\draw[thin, dotted] (0, 0.8) node{\scriptsize$*$}--(2.2, 0.8);

\draw[thin, dotted] (0, 1.4) node{\scriptsize$*$}  -- (3, 1.4) ;
\draw[thin, dotted] (3, 0) node{\scriptsize$*$}  -- (3, 1.4) ;

\draw (7, 4) node[right] {\scriptsize $\overline{\textbf{AB}} = \mathcal{S}_{1,1}^{(q_1, q_2)}$, $\textbf{A} = (0, \frac 12)$, $\textbf{B} = (\frac 1d, 0)$};

\draw (7, 3.5) node[right] {\scriptsize $\mathcal{R}\left(\textbf{ABC}\right)$: scaling invarant case of \eqref{T1:V-log} };
\draw (7, 3) node[right] {\scriptsize $\mathcal{R}\left(\textbf{ABDE}\right)$: scaling invariant case of \eqref{T1:V-log-divfree} };

\draw (7, 2) node[right] {\scriptsize  $\textbf{C} = (\frac 12, \frac 12)$, $\textbf{D} = (\frac{1+d}{2d}, \frac{1+d}{2(2+d)})$, $\textbf{E} = (0, \frac{1+d}{2+q})$ };
\draw (7, 1.5) node[right] {\scriptsize $\textbf{f, g, h,i}$: on $\mathcal{S}_{m,1}^{(q_1,q_2)}$ for $1 < m < 2$};
\draw (7, 1) node[right] {\scriptsize $\textbf{g} = (\frac{m-1}{2}, \frac 12)$, $\textbf{h} = (\frac{(2-m)+d(m-1)}{md}, \frac{m-1}{m})$ };
\draw (7, 0.5) node[right] {\scriptsize $\textbf{f} = (0, \frac{1}{{\lambda_1}})$, $\textbf{i} = (\frac{1+d(m-1)}{md}, \frac{m-1}{{\lambda_1} m})$ };
\end{tikzpicture}
\end{center}

\item[(iv)] The range of $(\tilde{q}_1, \tilde{q}_2)$ for $V$ satisfying \eqref{T:V-tilde-log-energy} in $\tilde{\mathcal{S}}_{m,1}^{(\tilde{q}_1, \tilde{q}_2)}$ is illustrated as the line $\overline{\textbf{DE}}$ on Figure 9-$(i)$ in Section~\ref{SS:tilde Serrin} (including $\textbf{D}$ but excluding $\textbf{E}$ if $d>2$ and excluding both if $d=2$).
\end{itemize}
\end{remark}

\subsubsection{Existence for case: $\rho_0 \in L^{q}(\mathbb{R}^d), \ q>1$}\label{SS:m greater 2}

We first state results regarding the existence of a $L^q$-weak solution satisfying an energy inequality.
In case that $1<m\le 2$, there is no restriction on $q$ that $q>1$.
It is, however, not clear, for the case $m>2$ and, in fact, according to our analysis, it is requred that $q\ge m-1$. We also construct a $L^q$-weak solution satisfying moments and speed estimates as well as an energy inequality, and the restriction $q\ge m-1$ is necessary  as similarly, if $m>2$ (Theorem~\ref{Theorem-2-b}).

Now we start with $L^q$-weak solutions satisfying energy inequality.
\begin{theorem}\label{Theorem-2-a}
Let $m>1$. Also let $q>1$ and $q\geq m-1$.
Assume that $\rho_0 \in  \mathcal{P}(\bbr^d) \cap L^{q}(\bbr^d)$ and
\begin{equation}\label{T2a:V}
V \in \mathfrak{S}_{m,q}^{(q_1,q_2)} \ \text{ for } \
    \begin{cases}
         2 \leq q_2 \leq  \frac{q+m-1}{m-1},  & \text{ if } d > 2 \vspace{1 mm}\\
         2 \leq q_2 <  \frac{q+m-1}{m-1},  & \text{ if } d = 2.
    \end{cases}
\end{equation}
Then, there exists a nonnegative $L^q$-weak solution of \eqref{E:Main} in Definition~\ref{D:weak-sol}
that holds
 \begin{equation}\label{T:Energy-Lq}
 \esssup_{0 \leq t \leq T} \int_{\mathbb{R}^d} \rho^q (\cdot, t) \,dx
 + \iint_{Q_T} \abs{\nabla \rho^{\frac{q+m-1}{2}}}^2 \,dx\,dt
\leq C,
\end{equation}
with $C = C ( \|V\|_{\mathfrak{S}_{m,q}^{(q_1, q_2)}}, \|\rho_{0}\|_{L^{q} (\bbr^d)})$.
\end{theorem}

We make a few remarks regarding Theorem~\ref{Theorem-2-a}.
\begin{remark}\label{R:Theorem-2-a}
\begin{enumerate}
\item[(i)] Figure 4 (for cases: $1<m\leq 2$, $q>1$) is to visualize the pairs of $(q_1, q_2)$ satisfying \eqref{T2a:V} in $\mathcal{S}_{m,q}^{(q_1, q_2)}$ for $d>2$ as the shaded region $\mathcal{R}(\textbf{ABCD})$. If $d=2$, then the line $\overline{\textbf{BC}}$ does not included to the shaded region. The points $\textbf{e}\to \textbf{A}$ and $\textbf{f} \to \textbf{B}$ in both cases either $m\to 1$ or $q\to \infty$. Moreover, we observe that $\textbf{C}\to \textbf{B}$ and $\textbf{D}\to \textbf{A}$ as $m\to 1$. Also, when $m=2$, then $\textbf{C}=\textbf{D}=(1/2, 1/2)$.
\item[(ii)] Figure 5 (for cases: $m>2$, $q \geq m-1$) is illustrating the pairs of $(q_1, q_2)$ satisfying \eqref{T2a:V} in $\mathcal{S}_{m,q}^{(q_1, q_2)}$ as the shaded region $\mathcal{R}(\textbf{ABC})$ for $d>2$. If $d=2$, then the line $\overline{\textbf{BC}}$ does not included to the shaded region. The points $\textbf{e}\to \textbf{A}$ and $\textbf{f} \to \textbf{B}$ as $q\to \infty$.
\end{enumerate}

\begin{center}
\begin{tikzpicture}[domain=0:16]

%\draw[very thin, color=gray](0, 0) grid (6, 5);

\draw (0, -1) node[right] { \scriptsize Figure 4: Theorem~\ref{Theorem-2-a} for $1<m\leq 2$, $q>1$. };

\fill[fill= lgray]
(0,2) -- (1.2, 0) -- (2.35, 1.1) -- (1.45, 2);

%\fill[fill= gray]
%(0.4, 1.3)--(0.8, 0.63) -- (1.85,0.63) -- (2.35, 1.1) -- (1.45, 2)--(0.4, 2);

\draw[->] (0,0) node[left] {\scriptsize $0$}
-- (5,0) node[right] {\scriptsize $\frac{1}{q_1}$};
\draw[->] (0,0) -- (0,4.5) node[left] { \scriptsize $\frac{1}{q_2}$};

\draw (0,4) node{\scriptsize $+$} node[left]{\scriptsize $1$} ;
\draw (4,0) node{\scriptsize $+$} ;

\draw[very thin]
(0, 2) -- (1.45, 2) ;
\draw[thick](1.45, 2) circle(0.05) node[right] {\scriptsize \textbf{D}};

\draw (0, 2) -- (1.45, 2);

\draw (1.2, 0) -- (2.35, 1.1);

\draw[thick] (2.35, 1.1) circle(0.05) node[right] {\scriptsize \textbf{C}};

% S
\draw[thick] (0,2) circle(0.05) node[left] {\scriptsize \textbf{A}};
\draw[thick] (1.2,0) circle(0.05) node[below] {\scriptsize \textbf{B}};
\draw[thick, dotted] (0,2) -- (1.2, 0);
\draw (0.7, 0.5) node{\scriptsize $\mathcal{S}_{m,\infty}^{(q_1, q_2)}$};

%\draw[thick] (0,2) circle(0.05);

% S_1
\draw[thick, dotted]
(0,3.4) node {\scriptsize $\times$} node[left] {\scriptsize $\frac{1}{{\lambda_1}}$}
-- (3.5, 0) node {\scriptsize $\times$} node[below] {\scriptsize $\frac{1+d(m-1)}{d}$};
\draw (3.4, 0.8) node{\scriptsize $\mathcal{S}_{m,1}^{(q_1, q_2)}$};

%S_q
\draw (0,2.5) node {\scriptsize $\times$} node[left] {\scriptsize $\frac{1+q_{m,d}}{2+q_{m,d}}$}
    -- (2.3, 0) node {\scriptsize $\times$} node[below] {\scriptsize $\frac{1+q_{m,d}}{d}$};
\draw (1.4, 1.5) node{\scriptsize $\mathcal{S}_{m,q}^{(q_1, q_2)}$};

\draw (1.79,0.56) node{\scriptsize $\bullet$}node[above]{\scriptsize \textbf{f}};

\draw (0.47,2) node{\scriptsize $\bullet$} node[above]{\scriptsize \textbf{e}};

\draw (7,4) node[right]{\scriptsize $\overline{\textbf{AB}} = \mathcal{S}_{m,\infty}^{(q_1, q_2)}, \, \textbf{A} = (0, \frac 12), \, \textbf{B} = (\frac 1d, 0)$} ;
\draw (7, 3.5)node[right]{\scriptsize $\mathcal{R}(\textbf{ABCD})$ : scaling invariant class of \eqref{T2a:V} } ;
\draw (7,3) node[right]{\scriptsize $\textbf{C} = (\frac{(2-m)+d(m-1)}{md}, \frac{m-1}{m})$, $\textbf{D} = (\frac{m-1}{2}, \frac 12)$} ;
\draw (7,2.5) node[right]{\scriptsize $\textbf{e, f}$: on $\mathcal{S}_{m,q}^{(q_1, q_2)}$ for $1< q < \infty$ } ;
\draw (7,2) node[right]{\scriptsize  coordinates are given in Fig. 5 } ;

\end{tikzpicture}
\end{center}

% Figure 5
\begin{center}
\begin{tikzpicture}[domain=0:16]

\draw (0, -1) node[right] { \scriptsize Figure 5: Theorem~\ref{Theorem-2-a} for $m> 2$, $q\geq m-1$. };

\fill[fill= lgray]
(0, 2) -- (1.2,0) -- (2,2);

\draw[->] (0,0) node[left] {\scriptsize $0$}
-- (5.5,0) node[right] {\scriptsize $\frac{1}{q_1}$};
\draw[->] (0,0) -- (0,4.5) node[left] { \scriptsize $\frac{1}{q_2}$};

\draw(0,4) node{\scriptsize $+$} node[left] {\scriptsize $1$} ;
\draw (4,0) node{\scriptsize $+$}  ;

\draw[thick] (0,2) circle(0.05) node[left] {\scriptsize \textbf{A}};
\draw[thick] (1.2,0) circle(0.05) node[below] {\scriptsize \textbf{B}};
\draw[thick, dotted] (0,2) -- (1.2, 0);

\draw (0, 2) -- (2, 2) ;

\draw(2, 2) node{\scriptsize $\bullet$} node[above] {\scriptsize \textbf{C}};

\draw (1.2, 0)  -- (2, 2);
\draw (0.7, 0.5) node{\scriptsize $\mathcal{S}_{m,\infty}^{(q_1, q_2)}$};

%\draw[very thin] (0, 2)-- (1.65, 1.13);

%\draw (7,3.2) node {\scriptsize $\times$} node[left] {\scriptsize $\frac{1+d}{2+d}$}
%-- (11.4, 0) node {\scriptsize $\times$} node[below] {\scriptsize $ \frac{1+d}{d}$} ;
\draw (0,3.5) node {\scriptsize $\times$} node[left] {\scriptsize $\frac{1+d}{2+d}$}
-- (4.6, 0) node {\scriptsize $\times$} node[below] {\scriptsize $\frac{1+d}{d}$} ;

\draw (4.3, 0.8) node {\scriptsize $\mathcal{S}_{m,m-1}^{(q_1, q_2)}$} ;

%S_m
\draw (0,2.5) node {\scriptsize $\times$} node[left] {\scriptsize $\frac{1+q_{m,d}}{2+q_{m,d}}$}
    -- (3, 0) node {\scriptsize $\times$} node[below] {\scriptsize $\frac{1+q_{m,d}}{d}$};
\draw (2.7, 0.8) node{\scriptsize $\mathcal{S}_{m,q}^{(q_1, q_2)}$};

\draw (1.64,1.15) node{\scriptsize $\bullet$}node[right]{\scriptsize \textbf{f}};

\draw (0.6,2) node{\scriptsize $\bullet$} node[above]{\scriptsize \textbf{e}};

\draw (7,4) node[right]{\scriptsize $\overline{\textbf{AB}} = \mathcal{S}_{m,\infty}^{(q_1, q_2)}$, $\textbf{C} = (\frac 12, \frac 12)$, $\textbf{A,B}$: in Fig. 4} ;
\draw (7,3.5) node[right]{\scriptsize $\mathcal{R}(\textbf{ABC})$: scaling invariant class of \eqref{T2a:V}} ;

\draw (7,3) node[right]{\scriptsize $\textbf{e, f}$: on $S_{m,q}^{(q_1, q_2)}$ for $m-1 < q < \infty$} ;
\draw (7,2.5) node[right]{\scriptsize $\textbf{e} = (\frac{m-1}{2q}, \frac 12)$, $\textbf{f} = (\frac{(q-m+1)+d(m-1)}{d(q+m-1)}, \frac{m-1}{q+m-1})$} ;

\end{tikzpicture}
\end{center}
\end{remark}

In the next theorem, we express the existence of a $L^q$-weak solution as an absolutely continuous curve in Wasserstein space for $\rho_0 \in \mathcal{P}_p(\bbr^d)$. More precisely, $p$-th moment and speed estimates are additionally added, compared to weak solutions in Theorem~\ref{Theorem-2-a}.
Due to constraint of handling speed and $p$-th moment estimates for $p\le 2$, the condition on $V$ is subject to restrain when $q>m$ that \eqref{T2:V} gives smaller region compared to \eqref{T2a:V} (cf. Figure 4 Vs. 4-$(i)$, Figure 5 Vs. 5-$(i)$). We refer Section~\ref{SS:Combination} for more details.
Meanwhile, because the temporal domain is bounded, we are able to extend conditions on $V$ (meaningful when $q>m$) by applying embedding arguments in Proposition~\ref{P:Energy-embedding}.

\begin{theorem}\label{Theorem-2-b}
Let $m>1$ and $p>1$. Also let $q>1$ and $q\geq m-1$.
Suppose that $\rho_0 \in
\mathcal{P}_p(\bbr^d) \cap L^{q} (\bbr^d)$.
\begin{itemize}
  \item [(i)]   Assume that
\begin{equation}\label{T2:V}
V \in \mathfrak{S}_{m,q}^{(q_1,q_2)} \ \text{ for } \
    \begin{cases}
        q_1 \leq \frac{2m}{m-1}, \  2 \leq q_2 \leq  \frac{q+m-1}{m-1} \,\delta_{\{1 < q \leq m\}} + \frac{2m-1}{m-1}\,\delta_{ \{q > m \} },  & \text{ if } d > 2, \vspace{1 mm}\\
        q_1 \leq \frac{2m}{m-1}, \ 2 \leq q_2 <  \frac{q+m-1}{m-1}\,\delta_{\{1 < q \leq m\}} + \frac{2m-1}{m-1} \,\delta_{ \{q > m \} },  & \text{ if } d = 2.
    \end{cases}
\end{equation}

\begin{itemize}
\item[(a)] Then, for $1 < p \leq {\lambda_q} := \min \left\{ 2, 1+\frac{d(q-1)+q}{d(m-1)+q} \right\}$, there exists a nonnegative $L^q$-weak solution of \eqref{E:Main} in Definition~\ref{D:weak-sol} such that  $\rho \in AC(0,T; \mathcal{P}_p(\mathbb{R}^d))$ with $\rho(\cdot, 0)=\rho_0$.
In addition, $\rho$ satisfies \eqref{KK-May7-60} and
\begin{equation}\label{T:Energy-Speed-Lq}
 \sup_{0 \leq t \leq T} \int_{\mathbb{R}^d} \left(  \rho^q +\rho  \langle x \rangle^p \right) \,dx
 + \iint_{Q_T} \{ \left| \nabla \rho^{\frac{q+m-1}{2}}\right |^2 + \left(\abs{\frac{\nabla \rho^m}{\rho}}^p+|V|^{p}\right ) \rho \} \,dx\,dt \leq C,
\end{equation}
where $C = C (\|V\|_{\mathfrak{S}_{m,q}^{(q_1, q_2)}},
\int_{\mathbb{R}^d} \left(\rho_0^q + \rho_0 \langle x
\rangle^p\right)\,dx )$.

 \item[(b)] For $p>{\lambda_q}$, the same conclusions hold as in (a) except that $p$ is replaced by ${\lambda_q}$.
\end{itemize}

\item[(ii)] (Embedding) Let $q>m$. Assume that
\begin{equation}\label{T2:V-embedding}
V \in \mathfrak{S}_{m,q}^{(q_1,q_2)} \ \text{ for } \
    \begin{cases}
        q_1 \leq \frac{2m}{m-1}, \ \frac{2m-1}{m-1} < q_2 \leq  \frac{q+m-1}{m-1},  & \text{ if } d > 2, \vspace{1 mm}\\
        q_1 \leq \frac{2m}{m-1}, \ \frac{2m-1}{m-1} \leq q_2 <  \frac{q+m-1}{m-1},  & \text{ if } d = 2.
    \end{cases}
\end{equation}
For $(q_1, q_2)$ in \eqref{T2:V-embedding}, there exists a constant
$q_{2}^{\ast} = q_{2}^{\ast}(q_1) \in [2, q_2)$ such that $ V \in
\mathfrak{S}_{m, q^\ast}^{(q_1, q_2^{\ast})}$ satisfying
\eqref{T2:V}.
Furthermore,  the same conclusions hold as in $(i)$ except that ${\lambda_q}$ is replaced by ${\lambda_{q^\ast}}$.
\end{itemize}
\end{theorem}

Here are a few remarks about Theorem~\ref{Theorem-2-b}.
\begin{remark}\label{R:Theorem2-b}
\begin{itemize}
\item[(i)] Results in Theorem \ref{Theorem-2-b} are also true for the heat equation ($m=1$) with drift term for $1 < p\le 2$ (see \cite{KK-arxiv}). We do not know whether or not the moment estimate \eqref{T:Energy-Speed-Lq} can be extended to $p>{\lambda_q}$, and thus we leave those as open questions.

\item[(ii)] When $1 < q \leq m$, the condition \eqref{T2:V} is the same as \eqref{T2a:V}. When $q > m$, we use speed and $p$-th moment estimates for $q=m$, therefore the condition \eqref{V-p-moment} for $q=m$ is the same as \eqref{T2:V} for $q>m$ (refer Proposition~\ref{P:Energy-speed} $(i)$).

\item[(iii)] Let $1<m\leq 2$ and $q>1$.  The following figure shows the pairs $(q_1, q_2)$ satisfying \eqref{T2:V} and \eqref{T2:V-embedding} in $\mathcal{S}_{m,q}^{(q_1, q_2)}$. There are Figure 4-(ii) and 4-(iii) in Appendix~\ref{Appendix:fig} for cases $(ii)$ $\max\{2, \frac{2m}{(2m-1)(m-1)}\} < d \leq \frac{2m}{m-1}$ and $(iii)$ $d > \frac{2m}{m-1}$. When $d=2$, $\overline{\textbf{BCD}}$ is excluded from $\mathcal{R}(\textbf{ABCDEF})$.
\end{itemize}

\begin{center}
\begin{tikzpicture}[domain=0:16]

%\draw[very thin, color=gray](0, 0) grid (6, 5);

\draw (0, -1) node[right] { \scriptsize Figure 4-(i). Theorem~\ref{Theorem-2-b} for $1<m\leq 2$, $q>1$, $2 < d \leq \max\{2, \frac{2m}{(2m-1)(m-1)}\}$.};

\fill[fill= lgray]
(0.8, 0.63) -- (1.2,0) -- (1.85, 0.63);

\fill[fill= gray]
(0.4, 1.3)--(0.8, 0.63) -- (1.85,0.63) -- (2.35, 1.1) -- (1.45, 2)--(0.4, 2);

\draw[->] (0,0) node[left] {\scriptsize $0$}
-- (5,0) node[right] {\scriptsize $\frac{1}{q_1}$};
\draw[->] (0,0) -- (0,4.5) node[left] { \scriptsize $\frac{1}{q_2}$};

\draw (0,4) node{\scriptsize $+$} node[left]{\scriptsize $1$} ;
\draw (4,0) node{\scriptsize $+$} ;

\draw[very thin]
(0, 2) -- (1.45, 2) ;
\draw[thick](1.45, 2) circle(0.05) node[right] {\scriptsize \textbf{E}};

\draw (0, 2) -- (1.45, 2);

\draw (1.2, 0) -- (2.35, 1.1);

\draw[thick] (2.35, 1.1) circle(0.05) node[right] {\scriptsize \textbf{D}};

\draw[thick] (0.8, 0.63) circle(0.05) node[left] {\scriptsize \textbf{B}};
\draw[very thin]
 (0.8, 0.63)
-- (1.85, 0.63) node{\scriptsize $\bullet$} node[right]{\scriptsize \textbf{C}} ;

\draw[thick] (0.4,1.3) circle(0.05) node[below] {\scriptsize \textbf{A}};
\draw[very thin]
(0.4, 1.3)
 --(0.4, 2) node{\scriptsize $\bullet$} node[above] {\scriptsize \textbf{F}};

% S
\draw[thick] (0,2) node{\scriptsize $\times$} node[left] {\scriptsize \textbf{a}};
\draw[thick] (1.2,0) circle(0.05) node[below] {\scriptsize \textbf{b}};
\draw[thick, dotted] (0,2) -- (1.2, 0);
% S_1
\draw[thick, dotted]
(0,3.4) node {\scriptsize $\times$} node[left] {\scriptsize $\frac{1}{{\lambda_1}}$}
-- (3.5, 0) node {\scriptsize $\times$} node[below] {\scriptsize $\frac{1+d(m-1)}{d}$};
\draw (3.9, 0.3) node{\scriptsize $\mathcal{S}_{m,1}^{(q_1, q_2)}$};

% S_m
\draw (0,2.4) node {\scriptsize $\times$} node[left] {\scriptsize $\frac{m+d(m-1)}{2m+d(m-1)}$}
    -- (2.5, 0) node {\scriptsize $\times$} node[below] {\scriptsize \textbf{c}};
\draw (2.7, 0.3) node{\scriptsize $\mathcal{S}_{m,m}^{(q_1, q_2)}$};

% dotted line
\draw[thin, dotted] (0.4, 0) node{\scriptsize $*$} node[below] {\scriptsize $\frac{m-1}{2m}$} --(0.4, 2);

\draw[thin, dotted] (0, 1.3) node{\scriptsize $*$}  -- (0.4, 1.3);
\draw[thin, dotted] (0, 0.63) node{\scriptsize $*$} node[left]{\scriptsize $\frac{m-1}{2m-1}$}  -- (0.8, 0.63);
\draw[thin, dotted] (0.8, 0) node{\scriptsize $*$}  -- (0.8, 0.63);

\draw[thin, dotted] (1.85, 0) node{\scriptsize $*$}  -- (1.85, 0.63);

\draw[thin, dotted] (2.35, 0) node{\scriptsize $*$}  -- (2.35, 1.1);
\draw[thin, dotted] (0, 1.1) node{\scriptsize $*$}  -- (2.35, 1.1);

\draw[thin, dotted] (1.45, 0) node{\scriptsize $*$}  -- (1.45, 2);
%\draw[thin, dotted] (0, 2) node{\scriptsize $*$}  -- (1.45, 2);

\draw (7, 4) node[right] {\scriptsize $\overline{\textbf{ab}} = \mathcal{S}_{m, \infty}^{(q_1, q_2)}$, $\textbf{a} = (0, \frac 12)$, $\textbf{b} = (\frac 1d, 0)$ };
%\draw (7, 4) node[right] {\scriptsize $\textbf{a} = (0, \frac 12)$, $\textbf{b} = (\frac 1d, 0)$ };

\draw (7, 3.5) node[right] {\scriptsize $\mathcal{R} (\textbf{ABCDEF})$: scaling invariant class of \eqref{T2:V}. };
\draw (7, 3) node[right] {\scriptsize $\mathcal{R} (\textbf{bCB})$: scaling invariant class of \eqref{T2:V-embedding}. };

\draw (7, 2.5) node[right] {\scriptsize $\textbf{A} = (\frac{m-1}{2m}, \frac{d+m(2-d)}{4m})$, $\textbf{B} = (\frac{1}{d(2m-1)}, \frac{m-1}{2m-1})$ };
\draw (7, 2) node[right] {\scriptsize $\textbf{C} = (\frac{1+d(m-1)}{d(2m-1)}, \frac{m-1}{2m-1})$, $\textbf{c}= (\frac{m+d(m-1)}{md},0)$ };
\draw (7, 1.5) node[right] {\scriptsize $\textbf{D} = (\frac{(2-m)+d(m-1)}{md}, \frac{m-1}{m})$};
\draw (7, 1) node[right] {\scriptsize $\textbf{E} = (\frac{m-1}{2}, \frac 12)$, $\textbf{F} = (\frac{m-1}{2m}, \frac 12)$};

\end{tikzpicture}
\end{center}
\begin{itemize}
\item[(iv)] Figure 5-(i) is illustrating the pairs $(q_1, q_2)$ satisfying \eqref{T2:V} and \eqref{T2:V-embedding} in $\mathcal{S}_{m,q}^{(q_1, q_2)}$ when $m>2$ and $q\geq m-1$. There is Figure 5-(ii) in Appendix~\ref{Appendix:fig} for case $d > \frac{2m}{m-1}$. When $d=2$, $\overline{\textbf{GCD}}$ is excluded from $\mathcal{R}(\textbf{GCDE})$.
\end{itemize}

% Figure 7
\begin{center}
\begin{tikzpicture}[domain=0:16]

\draw (0, -1) node[right] { \scriptsize Figure 5-(i). Theorem~\ref{Theorem-2-b} for $m>2$, $q\geq m-1$, $ 2 < d \leq \frac{2m}{m-1}$.};

\fill[fill= lgray]
(0.7, 1.15)-- (0.7, 0.6) -- (1, 0) -- (1.58,1.15);

\fill[fill= gray]
(0.7, 2) -- (0.7,1.15) -- (1.58, 1.15) --(2,2);

\draw[->] (0,0) node[left] {\scriptsize $0$} -- (5,0) node[right] {\scriptsize $\frac{1}{q_1}$};
\draw[->] (0,0) -- (0,4.5) node[left] { \scriptsize $\frac{1}{q_2}$};

\draw (0,4) node{\scriptsize $+$} node[left] {\scriptsize $1$} ;
\draw (4,0) node{\scriptsize $+$};

\draw
(0, 2) node[left] {\scriptsize \textbf{a}} -- (2, 2) ;

\draw[thick, dotted] (0,2) node{\scriptsize $\times$} node[left] {\scriptsize $\textbf{a}$}
 -- (1, 0)  node[below] {\scriptsize $\textbf{b}$};
%\draw[thick] (0,2) circle(0.05);
\draw[thick] (1,0) circle(0.05);

\draw(2, 2) node{\scriptsize $\bullet$} node[above] {\scriptsize \textbf{D}};

\draw(0.7,2) node{\scriptsize $\bullet$} node[above]{\scriptsize \textbf{E}}
--(0.7, 0.6) node[left]{\scriptsize \textbf{A}} ;
\draw[thick] (0.7,0.6) circle(0.05);

\draw (0.43, 1.15) node[left]{\scriptsize \textbf{B}}
--(1.58, 1.15)node{\scriptsize $\bullet$}node[right]{\scriptsize \textbf{C}};
\draw[thick] (0.43, 1.15) circle(0.05);

\draw(0.7, 1.15) node{\scriptsize $\bullet$} node[above]{\scriptsize \textbf{G}};

\draw (1, 0)  -- (2, 2);

%\draw[very thin] (0, 2)-- (1.65, 1.13);

\draw (0,3.5) node {\scriptsize $\times$} node[left] {\scriptsize $\frac{1+d}{2+d}$}
-- (4.6, 0) node {\scriptsize $\times$} node[below] {\scriptsize $\frac{1+d}{d}$} ;
\draw (4.5, 1) node {\scriptsize $\mathcal{S}_{m,m-1}^{(q_1, q_2)}$} ;

\draw (0,2.7) node {\scriptsize $\times$} node[left] {\scriptsize $\frac{m+d(m-1)}{md}$}
 -- (2.8, 0) node {\scriptsize $\times$} node[below] {\scriptsize $\frac{m+d(m-1)}{2m+d(m-1)}$};
\draw (3.2, 0.5) node{\scriptsize $\mathcal{S}_{m,m}^{(q_1, q_2)}$};

\draw (7, 4) node[right] {\scriptsize $\overline{\textbf{ab}} = \mathcal{S}_{m, \infty}^{(q_1, q_2)}$, $\textbf{a} = (0, \frac 12)$, $\textbf{b} = (\frac 1d, 0)$ };
%\draw (6, 4.5) node[right] {\scriptsize $\overline{\textbf{ab}}, \textbf{A}-\textbf{E}$: same as in Fig. 4-(b-i)};
\draw (7, 3.5) node[right] {\scriptsize $\mathcal{R}(\textbf{GCDE})$: scaling invariant class of \eqref{T2:V} };
\draw (7, 3) node[right] {\scriptsize $\mathcal{R}(\textbf{GAbC}))$: scaling invariant class of \eqref{T2:V-embedding} };

\draw (7, 2.5) node[right] {\scriptsize $\textbf{A} = (\frac{m-1}{2m}, \frac{d+m(2-d)}{4m})$, $\textbf{B} = (\frac{1}{d(2m-1)}, \frac{m-1}{2m-1})$};
%\draw (6, 1.5) node[right] {\scriptsize  $\textbf{B} = (\frac{1}{d(2m-1)}, \frac{m-1}{2m-1})$ };
\draw (7, 2) node[right] {\scriptsize $\textbf{C} = (\frac{1+d(m-1)}{d(2m-1)}, \frac{m-1}{2m-1})$, $\textbf{D} = (\frac 12, \frac 12)$ };
\draw (7, 1.5) node[right] {\scriptsize  $\textbf{E} = (\frac{m-1}{2m}, \frac 12)$, $\textbf{G}= (\frac{m-1}{2m}, \frac{m-1}{2m-1}) $};
%\draw (6, 0.5) node[right] {\scriptsize $\textbf{G}= (\frac{m-1}{2m}, \frac{m-1}{2m-1}) $ };
\end{tikzpicture}
\end{center}
\end{remark}

%%%%%%%%%%%%%%%%%%%%%%%%%%%%%%%%

\subsubsection{Existence for case: $\nabla\cdot V=0$ and $\rho_0 \in L^{q}(\mathbb{R}^d), \ q>1$.}\label{SS:divergence-free}

Here we assume that $V$ is a divergence-free vector field. Compared to Section~\ref{SS:m greater 2}, theorems in this subsection  work for less restricted range for $q$ and less restrictions on $V$ in sub-scaling class.

Let us define
 \[
 q_2^M := \left\{
 \begin{array}{cl}
 \left(\frac{d(q+m-1)+q}{d(q+m-1)+2q} - \frac{q}{q+m-1}\right)^{-1}, & \text{if } 1 < q < q^\ast, \vspace{2 mm} \\
 \infty, & \text{if } q \geq q^{\ast},
 \end{array}\right.
 \]
 where
 \begin{equation}\label{q_ast}
 q^\ast := \frac{2d(m-1)}{\sqrt{d^2 + 6d +1} - (d+1)}.
 \end{equation}
In this case, the existence result is stated for $L^q$-weak solutions with an energy inequality as follows:

\begin{theorem}\label{Theorem-4a}
Let $m>1$. Also let $q > \max\{ 1, \frac m2 \}$. Suppose that $\rho_0 \in  \mathcal{P} (\bbr^d) \cap L^{q} (\bbr^d)$.
Furthermore, assume that $V$ is divergence-free (that is, $\nabla \cdot V =0$) and
 \begin{equation}\label{T4:V-energy}
V\in \mathfrak{S}_{m,q}^{(q_1,q_2)} \quad \text{for} \quad
\begin{cases}
 \frac{2+q_{m,d}}{1+q_{m,d}}\leq q_2 \leq q_2^M,
 & \text{if } d>2, \vspace{2 mm} \\
\frac{2+q_{m,d}}{1+q_{m,d}}\leq q_2 < q_2^M \cdot \delta_{\{1 < q \leq q^\ast\}}, & \text{if } d=2, \vspace{2 mm} \\
\frac{2+q_{m,d}}{1+q_{m,d}}\leq q_2 \leq  q_2^M \cdot \delta_{\{ q > q^\ast\}}, & \text{if } d=2.
\end{cases}
\end{equation}
Then, there exists
  a nonnegative $L^q$-weak solution of \eqref{E:Main} in Definition~\ref{D:weak-sol} that holds
\begin{equation}\label{T:Energy-divfree}
\begin{aligned}
 \esssup_{0 \leq t \leq T} &\int_{\mathbb{R}^d} \rho^{q} (\cdot, t) dx
 +\iint_{Q_T}  \abs{\nabla \rho^{\frac{q+m-1}{2}}}^2 \,dx\,dt\leq C,
\end{aligned}\end{equation}
where $C = C(\|\rho_0\|_{L^{q}(\bbr^d)})$.
\end{theorem}

We make a few remarks regarding the above theorem.
\begin{remark}\label{R:Theorem-4a}
\begin{itemize}
\item[(i)] The restriction on \eqref{T4:V-energy} is directly from Lemma~\ref{L-compact} and Proposition~\ref{Proposition : AL-2} $(c)$ for carrying compactness arguments which is essential to show the convergence of approximated solutions, although a priori estimate is calculated for any $m, q>1$ in Proposition~\ref{P:Lq-energy} (iii). Moreover, we note that \eqref{T4:V-energy} is weaker than \eqref{T2a:V} or \eqref{T4:V-divfree}.

\item[(ii)] The range of $q > \max\{ 1, \frac m2 \}$ is sorted into two cases: either $1<m\leq 2$ and $q>1$, or $m>2$ and $q > \frac{m}{2}$.
 However, the compactness arguments in Proposition~\ref{Proposition : AL-2} $(ii)$ works for $q > \frac m2$ (which is strictly less than $m-1$ for $m>2$). In fact, when $m>2$, the range $q \geq m-1$ in Proposition~\ref{Proposition : AL-2} $(i)$ is weakened to $q > \frac m2$ thanks to $\nabla \cdot V =0$.

\item[(iii)] Here we plot the region of pairs $(q_1, q_2)$ satisfying \eqref{T4:V-energy} in $\mathcal{S}_{m,q}^{(q_1, q_2)}$ for $1<m \leq 2$, $q > 1$, $d> 2$.
Note that $\textbf{E}, \frac{1+q_{m,d}}{2+q_{m,d}} ,\textbf{f} \to \textbf{A}$ and $\textbf{C}, \frac{1+q_{m,d}}{d} \to \textbf{B}$ as either $m\to 1$ or $q\to \infty$. When $d=2$, the line $\overline{\textbf{CD}}$ is excluded from the region $\mathcal{R}(\textbf{ABCDE})$.
There is Figure 6-$(ii)$ for $m>2$ and $q>\frac{m}{2}$ in Appendix~\ref{Appendix:fig}.
\end{itemize}

% Figure 8
\begin{center}
\begin{tikzpicture}[domain=0:16]

\draw (-0.5, -1) node[right] { \scriptsize Figure 6-(i). Theorem~\ref{Theorem-4a} for $ 1 < m \leq 2$, $q>1$. };

\fill[fill= lgray]
(0, 2) -- (1.1,0) -- (2.5, 0) -- (3.52, 0.95) -- (0, 3.6) -- (0, 2);

\draw[->] (0,0) node[left] {\scriptsize $0$}
-- (5.5,0) node[right] {\scriptsize $\frac{1}{q_1}$};
\draw[->] (0,0) -- (0,4.5) node[left] { \scriptsize $\frac{1}{q_2}$};

\draw (0,4) node{\scriptsize $+$} node[left]{\scriptsize $1$} ;
%\draw (4,0) node{\scriptsize $+$} node[below] {\scriptsize $1$} ;

\draw[thick] (0, 2) circle(0.05) node[left] {\scriptsize \textbf{A}};
\draw (0.9, 0.5) node{\scriptsize $\mathcal{S}_{m,\infty}^{(q_1, q_2)}$};

% S
\draw[thick, dotted] (0,2) -- (1.1, 0)  node[below] {\scriptsize $\textbf{B}$};
\draw[thick] (1.1, 0) circle(0.05) ;

% S_1
\draw [thick, dotted]
(0,3.6)  node[left] {\scriptsize $\frac{1}{{\lambda_1}}$} node[right] {\scriptsize \textbf{E}}
-- (4.8, 0) node{\scriptsize $\times$} node[below]{\scriptsize $\frac{1+d(m-1)}{d}$} ;
\draw[thick] (0, 3.6) circle(0.05) ;
%\draw[thick] (4.5, 0) circle(0.05) ;
\draw (1.5, 3) node{\scriptsize $\mathcal{S}_{m,1}^{(q_1, q_2)}$};

% S_q^\ast
\draw (0,2.6) node {\scriptsize $\bullet$}  node[left]{\scriptsize \textbf{f}}
    -- (2.5, 0) node {\scriptsize $\bullet$}  node[below] {\scriptsize \textbf{C}};
\draw (2.2, 0.8) node{\scriptsize $\mathcal{S}_{m,q^\ast}^{(q_1, q_2)}$};

% S_q
\draw (0,3) node {\scriptsize $\bullet$} node[left] {\scriptsize $\frac{1+q_{m,d}}{2 + q_{m,d}}$ }
    -- (3.6, 0) node{\scriptsize $\times$} node[below] {\scriptsize $\frac{1+q_{m,d}}{d}$} ;
\draw (2, 1.8) node{\scriptsize $\mathcal{S}_{m,q}^{(q_1, q_2)}$};
\draw (3,0.47) node{\scriptsize $\bullet$} node[right]{\scriptsize \textbf{g}};

%\draw (4, 0) -- (4, 1) node[above]{\scriptsize $q_1 =1$};
%\draw[thick] (4, 0.4) circle(0.05) node[right]{\scriptsize $F$} ;

% cutting line
\draw[thick] (3.52, 0.95) circle(0.05) node[right]{\scriptsize \textbf{D}} ;
\draw (2.5, 0) -- (3.52,0.95);

% Note
\draw (7, 4) node[right] {\scriptsize $\overline{\textbf{AB}} = \mathcal{S}_{m, \infty}^{(q_1, q_2)}$, $\textbf{A} = (0, \frac 12)$, $\textbf{B} = (\frac 1d, 0)$ };

\draw (7, 3.5) node[right] {\scriptsize $\mathcal{R} (\textbf{ABCDE})$: scaling invariant class of \eqref{T4:V-energy}};

\draw (7, 3) node[right] {\scriptsize $\overline{\textbf{fC}} = \mathcal{S}_{m, q^\ast}^{(q_1, q_2)}$, $\textbf{f} = (0, \frac{1+q^\ast_{m,d}}{2+q^\ast_{m,d}})$, $\textbf{C} = (\frac{1+q^\ast_{m,d}}{d}, 0)$ };
\draw (7, 2.5) node[right] {\scriptsize $q^\ast$ is defined in \eqref{q_ast} and $q^\ast_{m,d}=\frac{d(m-1)}{q^\ast}$};
\draw (7, 2) node[right] {\scriptsize $\textbf{g} = (\frac{d(q+m-1)+q}{d(q+m-1)+2q} - \frac{q(d-2)}{d(q+m-1)}, \frac{d(q+m-1)+q}{d(q+m-1)+2q} - \frac{q}{q+m-1})$};
\draw (7, 1.5) node[right]{\scriptsize $\textbf{g}$ is on $\mathcal{S}_{m, q}^{(q_1, q_2)}$ for $1 <q < q^\ast$};
\draw (7, 1) node[right] {\scriptsize $\textbf{D}= (\frac{md+1}{md+2} - \frac{d-2}{md}, \frac{md+1}{md+2} - \frac{1}{m})$ on $\mathcal{S}_{m, 1}^{(q_1, q_2)}$};
\end{tikzpicture}
\end{center}
\end{remark}

Next, we show the existence of a $L^q$-weak solution in Wasserstein space in which speed and moment estimates play importantly. The second part of the following theorem is obtained by embedding in time variable. Unlike result in Theorem~\ref{Theorem-2-b} $(ii)$, one can make observation that the embedding argument is meaningful for all $q>1$.

  \begin{theorem}\label{Theorem-4}
Let $m>1$ and $p>1$. Also let $q>1$ and $q \geq \frac m2$. Suppose that $\rho_0 \in  \mathcal{P}_p(\bbr^d) \cap L^{q} (\bbr^d)$ and  $\nabla \cdot V =0$.
 \begin{itemize}
 \item[(i)] Suppose that
\begin{equation}\label{T4:V-divfree}
V\in \mathfrak{S}_{m,q}^{(q_1, q_2)}  \ \text{ for } \
\begin{cases}
{\lambda_q} \leq q_2 \leq \frac{{\lambda_q} (q+m-1)}{q+m-2}, & \text{ if }  d>2, \ 1<q\leq m, \vspace{1 mm}\\
q_1 \leq \frac{2m}{m-1},\  2 \leq q_2 \leq \frac{2m-1}{m-1}, & \text{ if }  d>2, \ q> m, \vspace{1 mm}\\
{\lambda_q} \leq q_2 < \frac{{\lambda_q} (q+m-1)}{q+m-2}, & \text{ if }  d=2, \ 1<q\leq m, \vspace{1 mm}\\
q_1 \leq \frac{2m}{m-1},\  2 \leq q_2 < \frac{2m-1}{m-1}, & \text{
if }  d=2, \ q> m.
\end{cases}
\end{equation}
 \begin{itemize}
 \item[(a)] Then, for $1 < p \leq {\lambda_q} := \min\left\{ 2, 1 + \frac{d(q-1)+q}{d(m-1)+q}\right\}$, there exists a nonnegative $L^q$-weak solution of \eqref{E:Main} in Definition~\ref{D:weak-sol} such that
 $\rho \in AC(0,T; \mathcal{P}_p(\mathbb{R}^d))$ with $\rho(\cdot, 0)=\rho_0$.
Furthermore, $\rho$
satisfies \eqref{T:Energy-divfree}, \eqref{T:Energy-Speed-Lq}, and \eqref{KK-May7-60} with
$C=C(\left\| V\right\|_{\mathfrak{S}_{m,q}^{(q_1, q_2)}},
\int_{\mathbb{R}^d} \left(\rho_0^q + \rho_0 \langle x
\rangle^p\right)\,dx )$.

\item[(b)] For $ p > {\lambda_q} $, the same conclusions hold as in (a) except that $p$ is replace by ${\lambda_q}$.
\end{itemize}
\item[(ii)] (Embedding) Suppose that
\begin{equation}\label{T4:V-divfree-embedding}
V\in \mathfrak{S}_{m,q}^{(q_1, q_2)} \ \text{ for } \
 \begin{cases}
\frac{2+q_{m,d}}{1+q_{m,d}} < q_2 < \frac{{\lambda_1} m}{m-1}, & \text{ if } 1<q \leq \frac{md}{d-1}, \vspace{1 mm}\\
\frac{2+q_{m,d}}{1+q_{m,d}} < q_2 \leq \infty, & \text{ if } q >
\frac{md}{d-1}.
\end{cases}
\end{equation}
For given $(q_1, q_2)$ in \eqref{T4:V-divfree-embedding}, there
exists constants $q_{2}^{\ast} = q_{2}^{\ast}(q_1) \in [2, q_2]$
such that $ V \in \mathfrak{S}_{m, q^\ast}^{(q_1, q_2^{\ast})}$
satisfying \eqref{T4:V-divfree}.
 Furthermore, the same conclusion
holds as in $(i)$ except that ${\lambda_q}$ is replaced by ${\lambda_{q^\ast}}$.
\end{itemize}
\end{theorem}

Now we explain more on the condition \eqref{T4:V-divfree} in Theorem~\ref{Theorem-4}.
\begin{remark}\label{R:Theorem-4}
\begin{itemize}
\item[(i)] The condition \eqref{T4:V-divfree} is distinguished for $1<q \leq m$ and $q>m$ because the structure of ${\lambda_q}$ in \eqref{lambda_q} is fixed as $2$ if $q>m$ (it is hard to obtain speed and moment estimates for $p>2$).
Details of getting a priori estimate under \eqref{T4:V-divfree} are in Proposition~\ref{P:Energy-speed} $(iii)$.
%Moreover, since \eqref{V-p-moment} is more restricted compared to \eqref{T4:V-energy}, compactness arguments are vaild under \eqref{T4:V-divfree}.
The limiting case $q=\frac m2$ is allowed because $\rho_0 \in \mathcal{P}_{p}(\bbr^d)$ and $V$ satisfying \eqref{T4:V-divfree} by Remark~\ref{R:AL-3} $(ii)$.

\item[(ii)] In the following figure, we plot pairs $(q_1, q_2)$ satisfying \eqref{T4:V-divfree} and \eqref{T4:V-divfree-embedding} in $\mathcal{S}_{m,q}^{(q_1,q_2)}$ for $1<m\leq 2$ and $q>1$. The irregular shape is mainly because of \eqref{V-p-moment} for estimating speed and $p$-th moment estimates which separates two cases either $1 < q \leq m$ or $q>m$. Moreover, Figure 7-(i) is when $2<d \leq \max\{2, \frac{2m}{(2m-1)(m-1)}\}$. For cases $\max\{2, \frac{2m}{(2m-1)(m-1)}\} < d \leq \frac{2m}{m-1}$ and $d > \frac{2m}{m-1}$, there are Figure 7-$(ii)$, $(iii)$ in Appendix~\ref{Appendix:fig}. Additionally, Figure 8-$(i)$, $(ii)$ in Appendix~\ref{Appendix:fig} are when $m>2$ and $q\geq \frac m2$. When $d=2$, $\overline{\textbf{BCD}}$ is excluded from $\mathcal{R}(\textbf{ABCDEF})$.
\end{itemize}

% Figure 9
\begin{center}
\begin{tikzpicture}[domain=0:16]

\draw (0, -1) node[right] { \scriptsize Figure 7-(i). Theorem~\ref{Theorem-4} for $1<m\leq 2$, $q>1$, $2 < d \leq \max\{2, \frac{2m}{(2m-1)(m-1)}\}$.};

%\fill[fill= lgray]
%(0, 2) -- (1.1,0) -- (3.7, 0) -- (0, 3.6) -- (0, 2);

\fill[fill= lgray]
(0, 2) -- (1.1,0) -- (2.3, 0) -- (2.3, 1.35) --(0,3.6)--(0, 2);

\fill[fill= gray]
(0.5, 1.25) -- (0.6,0.85) -- (1.75, 0.85) -- (2.3, 1.35) -- (0, 3.6)--(0.5,2);

%\fill[fill= lgray]
%(0, 2) -- (1.46,0.64) -- (1.85, 1.6) -- (1.45, 2);

\draw[->] (0,0) node[left] {\scriptsize $0$}
-- (5,0) node[right] {\scriptsize $\frac{1}{q_1}$};
\draw[->] (0,0) -- (0,4.5) node[left] { \scriptsize $\frac{1}{q_2}$};

\draw (0,4) node{\scriptsize $+$} node[left]{\scriptsize $1$} ;
\draw (4,0) node{\scriptsize $+$} ;

\draw[dotted](0, 2)  node[left] {\scriptsize \textbf{a}}--(0.4, 2);
\draw[thick] (0,2) circle(0.05) ;

%\draw[very thin]
%(0, 0) -- (1.8, 1.8);

\draw (1.75, 0.85) -- (2.3, 1.35) node[right] {\scriptsize \textbf{D}};
\draw[thick] (2.3, 1.35) circle(0.05) ;

\draw
(0.5, 2) -- (0, 3.6) ;

%\draw
%(0, 2)  -- (1.42, 0.85) ;

\draw (0.6,0.85)  node[left]{ \scriptsize $\textbf{B}$}
--(1.75, 0.85)node{\scriptsize $\bullet$}node[right]{ \scriptsize $\textbf{C}$};
\draw[thick] (0.6, 0.85) circle(0.05);

% S
\draw[thick, dotted] (0,2) -- (1.1, 0)  node[below] {\scriptsize $\textbf{b}$};
\draw[thick] (1.1, 0) circle(0.05) ;

% S_1
\draw [thick, dotted]
(0,3.6)  node[left] {\scriptsize $\frac{1}{{\lambda_1}}$} node[right] {\scriptsize \textbf{E}}
-- (3.7, 0)  node[below]{\scriptsize \textbf{g}};
\draw[thick] (0, 3.6) circle(0.05) ;
\draw[thick] (3.7, 0) circle(0.05) ;
\draw (4, 0.5) node{\scriptsize $\mathcal{S}_{m,1}^{(q_1, q_2)}$};

% S_m
\draw (0,2.45) node {\scriptsize $\times$} node[left] {\scriptsize $\frac{m+d(m-1)}{2m+d(m-1)}$}
    -- (2.7, 0) node {\scriptsize $\times$} node[below] {\scriptsize \textbf{i}};
\draw (2.3, 0.5) node{\scriptsize $\mathcal{S}_{m,m}^{(q_1, q_2)}$};

\draw (0.5, 1.1) node[left]{\scriptsize \textbf{A}} -- (0.5, 2);
\draw[thick] (0.5, 1.1) circle(0.05);

% dotted line
\draw[thin, dotted] (0.5, 0) node{\scriptsize $*$}
--(0.5, 2) node{\scriptsize $\bullet$} node[right] {\scriptsize \textbf{F}};

\draw[thin, dotted] (0.5, 0) node{\scriptsize $*$}  -- (0.5, 1.1);
\draw[thin, dotted] (0, 1.1) node{\scriptsize $*$}  -- (0.5, 1.1);

\draw[thin, dotted] (0, 0.85) node{\scriptsize $*$}  -- (1.75, 0.85);
\draw[thin, dotted] (0.6, 0) node{\scriptsize $*$}  -- (0.6, 0.85);

\draw[thin, dotted] (1.75, 0) node{\scriptsize $*$}--(1.75, 0.85);

\draw[thin, dotted] (2.3, 0)  node[below]{\scriptsize \textbf{h}}  -- (2.3, 1.35) ;
\draw[thick] (2.3, 0) circle(0.05) ;

\draw[thin, dotted] (0, 1.35) node{\scriptsize $*$}  -- (2.3, 1.35) ;

\draw (7, 4) node[right] {\scriptsize $\overline{\textbf{ab}} = \mathcal{S}_{m, \infty}^{(q_1, q_2)}$, $\textbf{a} = (0, \frac 12)$, $\textbf{b} = (\frac 1d, 0)$ };
%\draw (7, 4) node[right] {\scriptsize $\textbf{a} = (0, \frac 12)$, $\textbf{b} = (\frac 1d, 0)$ };
%\draw (7, 3.5) node[right] {\scriptsize  $\textbf{g} = (\frac{1}{d+2}, \frac{1}{d+2})$ };

%\draw (7, 3.5) node[right] {\scriptsize $\mathcal{R} (\textbf{abgE})$ where \eqref{T:Energy-divfree} is obtained.} ;
\draw (7, 3.5) node[right] {\scriptsize $\mathcal{R} (\textbf{ABCDEF})$: scaling invariant class of \eqref{T4:V-divfree}. };
\draw (7, 3) node[right] {\scriptsize $\mathcal{R} (\textbf{abhDE})$: scaling invariant class of \eqref{T4:V-divfree-embedding}. };

\draw (7, 2.5) node[right] {\scriptsize $\textbf{A} = (\frac{m-1}{2m}, \frac{d+m(2-d)}{4m})$, $\textbf{B} = (\frac{1}{d(2m-1)}, \frac{m-1}{2m-1})$ };
\draw (7, 2) node[right] {\scriptsize $\textbf{C} = (\frac{1+d(m-1)}{d(2m-1)}, \frac{m-1}{2m-1})$, $\textbf{D} = (\frac{1+d(m-1)}{md}, \frac{m-1}{{\lambda_1} m})$   };
\draw (7, 1.5) node[right] {\scriptsize  $\textbf{E} = (0, \frac{1}{{\lambda_1}})$, $\textbf{F} = (\frac{m-1}{2m}, \frac 12)$ };
\draw (7, 1) node[right] {\scriptsize $\textbf{g} = (\frac{1+d(m-1)}{d}, 0)$};
\draw (7, 0.5) node[right] {\scriptsize  $\textbf{h} = (\frac{1+d(m-1)}{md}, 0)$, $\textbf{i} = (\frac{m + d(m-1)}{md}, 0)$  };

\end{tikzpicture}
\end{center}

\end{remark}

%%%
\subsubsection{Existence for case: $\nabla V$ in sub-scaling classes }\label{SS:tilde Serrin}

Now we assume that $\nabla V$ is satisfying sub-scaling classes, that is, $V \in \tilde{\mathfrak{S}}_{m,q}^{(\tilde{q}_1, \tilde{q}_2)}$.
We then see that there are corresponding existence results and collect them in this subsection when $\rho_0 \in L^{q}(\mathbb{R}^d)$.

The first result is concerned about $L^q$-weak solutions with an energy inequality.
\begin{theorem}\label{Theorem-5a}
Let $m>1$. Also let $q > \max\{ 1,\frac m2 \}$. Suppose that $\rho_0 \in  \mathcal{P} (\bbr^d) \cap L^{q} (\bbr^d)$ and
 \begin{equation}\label{T5a:V-tilde}
V\in \tilde{\mathfrak{S}}_{m,q}^{(\tilde{q}_1, \tilde{q}_2)}  \ \text{ for } \
\begin{cases}
 \frac{2+q_{m,d}}{1+q_{m,d}} < \tilde{q}_2 \leq \frac{q+m-1}{m-1},   & \text{ if } d> 2, \vspace{1 mm}\\\
 \frac{2+q_{m,2}}{1+q_{m,2}} < \tilde{q}_2 < \frac{q+m-1}{m-1},   & \text{ if } d= 2.
 \end{cases}
\end{equation}
Then, there exists a nonnegative $L^q$-weak solution of \eqref{E:Main} in Definition~\ref{D:weak-sol}
satisfying \eqref{T:Energy-Lq} with
$C = C(\left\| V\right\|_{\tilde{\mathfrak{S}}_{m,q}^{(\tilde{q}_1, \tilde{q}_2)}}, \|\rho_0\|_{L^{q}(\mathbb{R}^d)} )$.
\end{theorem}

The second result is the existence of a $L^q$-weak solution as an absolutely continuous curve in Wasserstein space.
\begin{theorem}\label{Theorem-5}
Let $m>1$ and $p>1$. Also let $q>1$ and $q \geq \frac m2$. Suppose that $\rho_0 \in  \mathcal{P}_p(\bbr^d)\cap L^{q} (\bbr^d)$.
  \begin{itemize}
 \item[(i)] Assume that
 \begin{equation}\label{T5:V-tilde}
V\in \tilde{\mathfrak{S}}_{m,q}^{(\tilde{q}_1, \tilde{q}_2)}  \
\text{ for } \
\begin{cases}
 \frac{2+q_{m,d}}{1+q_{m,d}} < \tilde{q}_2 \leq \frac{q+m-1}{m-1} \,\delta_{\{1 < q \leq m\}} + \frac{2m-1}{m-1}\,\delta_{ \{q > m \} },   & \text{ if } d> 2, \vspace{1 mm}\\
 \frac{2+q_{m,2}}{1+q_{m,2}} < \tilde{q}_2 < \frac{q+m-1}{m-1} \,\delta_{\{1 < q \leq m\}} + \frac{2m-1}{m-1}\,\delta_{ \{q > m \} },   & \text{ if } d= 2.
 \end{cases}
\end{equation}

\begin{itemize}
\item[(a)] Then, for $1 < p \leq {\lambda_q} := \min\left\{ 2, 1 + \frac{d(q-1)+q}{d(m-1)+q}\right\}$,
there exists a nonnegative $L^q$-weak solution of \eqref{E:Main} in Definition~\ref{D:weak-sol} such that  $\rho \in AC(0,T; \mathcal{P}_p(\mathbb{R}^d))$ with $\rho(\cdot, 0)=\rho_0$.
Furthermore, $\rho$ satisfies \eqref{T:Energy-Lq} with $C = C (\left\|
V\right\|_{\tilde{\mathfrak{S}}_{m,q}^{(\tilde{q}_1, \tilde{q}_2)}},
\|\rho_0\|_{L^q(\bbr^d)} )$. Also,
\eqref{T:Energy-Speed-Lq} and \eqref{KK-May7-60} hold with $C=C
(\left\| V\right\|_{\tilde{\mathfrak{S}}_{m,q}^{(\tilde{q}_1,
\tilde{q}_2)}}, \int_{\mathbb{R}^d} \left(\rho_0^q + \rho_0 \langle
x \rangle^p\right)\,dx )$.

\item[(b)] For $p> {\lambda_q}$, the same conclusions hold as in (a) except that $p$ is replaced by ${\lambda_q}$.
\end{itemize}

\item[(ii)] (Embedding) Let $q > m$.
Assume that
\begin{equation}\label{T5:V-tilde-embedding}
V\in \tilde{\mathfrak{S}}_{m,q}^{(\tilde{q}_1, \tilde{q}_2)}  \
\text{ for } \
\begin{cases}
\frac{2m-1}{m-1} < \tilde{q}_2 \leq \frac{q+m-1}{m-1},   & \text{ if } d> 2, \vspace{1 mm}\\
\frac{2m-1}{m-1} \leq \tilde{q}_2 < \frac{q+m-1}{m-1},  & \text{ if
} d= 2 .
\end{cases}
\end{equation}
For given $(\tilde{q}_1, \tilde{q}_2)$ in
\eqref{T5:V-tilde-embedding}, there exists a constant
$\tilde{q}_{2}^{\ast} = \tilde{q}_{2}^{\ast}(\tilde{q}_1) \in
(\frac{2+q^{\ast}_{m,d}}{1+q^{\ast}_{m,d}}, \tilde{q}_2]$ such that
$ V \in \tilde{\mathcal{S}}_{m, q^\ast}^{(\tilde{q}_1,
\tilde{q}_2^{\ast})}$ satisfying \eqref{T5:V-tilde}.
Furthermore,
the same conclusion holds as in Theorem~\ref{Theorem-5} $(i)$ except
that ${\lambda_q}$ is replaced by $ {\lambda_{q^\ast}}$.
\end{itemize}
\end{theorem}

Here we make a few remarks of Theorem~\ref{Theorem-5a} and Theorem~\ref{Theorem-5}.
\begin{remark}\label{R:Theoerm-5}
\begin{itemize}
\item[(i)]  When $V\in \tilde{\mathcal{S}}_{m,q}^{(\tilde{q}_1, \tilde{q}_2)}$, a priori estimate is calculated for any $m, q>1$ in Proposition~\ref{P:Lq-energy} (ii), but the restriction $q > \frac m2$ is required because of the compactness arguments in Proposition~\ref{Proposition : AL-2} $(ii)$. In Theorem~\ref{Theorem-5a}, we allow the limiting case $q = \frac m2$ if $\rho_0 \in \mathcal{P}_{p}(\bbr^d)$ with $C$ depending also on $\int_{\bbr^d} \rho_0 \langle x \rangle^p \,dx$ under the same condition on $V$ satisfying \eqref{T5a:V-tilde}.

\item[(ii)] In Figure 9-$(i)$, we plot the pair $(\tilde{q}_1, \tilde{q}_2)$ satisfying $\tilde{S}_{m,q}^{(\tilde{q}_1, \tilde{q}_2)}$ for $1<m\leq 2$, $q > 1$. In particular, the region $\mathcal{R}\left( \textbf{ABCDE}\right)$ (dark shaded) to visualize \eqref{T5:V-tilde}, and the region $\mathcal{R}\left( \textbf{bCB}\right)$ (lightly shaded) is for \eqref{T5:V-tilde-embedding}. The restriction $\tilde{q}_1 \in (1,d)$ is because of the relation of $\mathcal{S}_{m,q}^{(q_1,q_2)}$ and $\tilde{\mathcal{S}}_{m,q}^{(\tilde{q}_1,\tilde{q}_2)}$ in Remark~\ref{R:S-tildeS} (ii).
There is Figure 9-$(ii)$ in Appendix~\ref{Appendix:fig} for $m>2$ and $q\geq \frac m2$. When $d=2$, the line $\overline{\textbf{BCD}}$ is excluded from $\mathcal{R}(\textbf{ABCDE})$.
\end{itemize}

\begin{center}
\begin{tikzpicture}[domain=0:16]

\draw (0, -1) node[right] {\scriptsize Firgure 9-(i). Theorem~\ref{Theorem-5a} and \ref{Theorem-5} for $1<m\leq 2$, $q>1$.};

\fill[fill= lgray]
(1.27, 0.8) -- (1.6,0) -- (2.4, 0.8) ;

\fill[fill= gray]
(0.8, 2) -- (1.27, 0.8)-- (2.4,0.8) -- (3.2, 1.68) --(0.8, 3.43);

\draw[->] (0,0) -- (6,0) node[right] {\scriptsize $\frac{1}{\tilde{q}_1}$};
\draw[->] (0, 0) -- (0, 5) node[left] {\scriptsize $\frac{1}{\tilde{q}_2}$};

\draw[thick, dotted] (0, 4)node{\scriptsize $\times$} node[left]{\scriptsize $\textbf{a}$} -- ( 5.5 , 0) node{\scriptsize $\times$} node[below] {\scriptsize $\frac{2+d(m-1)}{d}$};
%\draw[thick] (0, 4) circle(0.05) node[left]{\scriptsize $\textbf{a}$};

\draw (5, 1) node {\scriptsize $\tilde{\mathcal{S}}_{m,1}^{(\tilde{q}_1, \tilde{q}_2)}$};

\draw (0, 4) -- (3, 0) node{\scriptsize $\times$} node[below] {\scriptsize $\frac{2m+ d(m-1)}{md}$};
\draw (3.3, 0.5) node {\scriptsize $\tilde{\mathcal{S}}_{m,m}^{(\tilde{q}_1, \tilde{q}_2)}$};

\draw[thick, dotted] (0,4)
-- (1.6, 0)  node[below] {\scriptsize $\textbf{b}$} ;
\draw[thick] (1.6, 0) circle(0.05);

\draw[very thin]
(0.8, 0) node{\scriptsize$*$} node[below] {\scriptsize $\frac 1d$}
-- (0.8, 4) node[above]{\scriptsize $\tilde{q}_1 = d$} ;
\draw[thick] (0.8, 2) circle(0.05) node[left]{\scriptsize \textbf{A}};
\draw[thick] (0.8, 2.94) circle(0.05) node[left]{\scriptsize \textbf{f}} ;
\draw[thick] (0.8, 3.43) circle(0.05) node[right]{\scriptsize \textbf{E}} ;

\draw[thin, dotted] (0,0) -- (2.3,2.3);
\draw (1.6, 0) -- (3.2,1.68) ;
\draw[thick] (3.2, 1.68) circle(0.05)  node[right]{\scriptsize \textbf{D}};

\draw (1.27, 0.8) node[left] {\scriptsize \textbf{B}} -- (2.4,0.8) node{\scriptsize $\bullet$} node[right]{\scriptsize \textbf{C}};
\draw[thick] (1.27, 0.8) circle(0.05);

\draw[thin, dotted] (0, 2) node{\scriptsize $+$} node[left] {\scriptsize $\frac 12$}
-- (0.8, 2);
\draw[thin, dotted] (0, 2.94) node{\scriptsize$*$} -- (0.8, 2.94);
\draw[thin, dotted] (2.4, 0) node{\scriptsize$*$} -- (2.4, 0.8);

\draw[thin, dotted] (0, 3.43) node{\scriptsize$*$} node[left] {\scriptsize $\frac{1+d(m-1)}{2+d(m-1)}$}  -- (0.8, 3.43);

\draw[very thin] (4, 0) node{\scriptsize$+$} node[below] {\scriptsize $1$}
-- (4, 2) node[above] {\scriptsize{$\tilde{q}_1 = 1$}};
\draw (4, 1.1) node{\scriptsize $*$} node[right] {\scriptsize \textbf{g}} ;

\draw[thin, dotted] (0, 0.8) node{\scriptsize$*$}  -- (2.4, 0.8);
\draw[thin, dotted] (0, 1.68) node{\scriptsize$*$}  -- (3.2, 1.68);
\draw[thin, dotted] (3.2, 0) node{\scriptsize$*$}  -- (3.2, 1.68);

\draw[thin, dotted] (0, 1.1) node{\scriptsize$*$} node[left] {\scriptsize $\frac{2+d(m-2)}{2+d(m-1)}$} -- (4, 1.1);

\draw (7, 4.5) node[right] {\scriptsize $\overline{\textbf{ab}} = \tilde{\mathcal{S}}_{m, \infty}^{(\tilde{q}_1, \tilde{q}_2)}$, $\textbf{a} = (0, 1)$, $\textbf{b} = (\frac 2d, 0)$};
\draw (7, 4) node[right] {\scriptsize $\mathcal{R} (\textbf{AbDE})$: scaling invariant class of \eqref{T5a:V-tilde}. };
\draw (7, 3.5) node[right] {\scriptsize $\mathcal{R} (\textbf{ABCDE})$: scaling invariant class of \eqref{T5:V-tilde}. };
\draw(7,3) node[right] {\scriptsize $\mathcal{R} (\textbf{bCB})$:  scaling invariant class of \eqref{T5:V-tilde-embedding}. };

\draw (7, 2.5) node[right] {\scriptsize $\textbf{A} = (\frac 1d, \frac 12)$, $\textbf{B} = (\frac{2m}{d(2m-1)}, \frac{m-1}{2m-1})$  };
\draw (7, 2) node[right] {\scriptsize $\textbf{C} = (\frac{2m+d(m-1)}{d(2m-1)}, \frac{m-1}{2m-1})$  };
\draw (7, 1.5) node[right] {\scriptsize $\textbf{D} = (\frac{2+d(m-1)}{md}, \frac{m-1}{m})$, $\textbf{E} = (\frac 1d, \frac{1}{{\lambda_1}})$   };
\draw (7, 1) node[right] {\scriptsize $\textbf{f} = (\frac 1d, \frac{m+d(m-1)}{2m+d(m-1)})$, $\textbf{g}= (1, \frac{2+d(m-2)}{2+d(m-1)})$  };

\end{tikzpicture}
\end{center}

\begin{enumerate}
\item[(iii)] By camparing Figures 4-(i) and 9-(i), we may capture similiarities of $\mathcal{S}_{m,q}^{(q_1, q_2)}$ and $\tilde{\mathcal{S}}_{m,q}^{(\tilde{q}_1, \tilde{q}_2)}$. The points \textbf{E} and \textbf{D} in Figure 9-(i) correspond to $(0, \frac{1}{{\lambda_1}})$ and  \textbf{C} in Figure 4-(i), respectively. Moreover, \textbf{f}, \textbf{C}, and \textbf{A} in Figure 9-(i) match to  $(0, \frac{m+d(m-1)}{2m+d(m-1)})$, \textbf{C}, and \textbf{a} in Fig. 4-(i), respectively.
\end{enumerate}

\end{remark}

%%%%%%%%%%%%%%%%%%%%%%%%%%
%%%%%%%%%%%%%%%%%%%%%%%%%%
\subsection{Uniqueness}\label{sub-unique}

The uniqueness of $L^q$-weak solution of \eqref{E:Main} is not obvious because, in general, the maximal principle fails for nonlinear diffusive equations with drifts. However, with additional assumptions on $\nabla V$ and restriction on $q$, we are able to obtain uniqueness results.
Here, we state the uniqueness of $L^m$-weak solution of \eqref{E:Main}.
\begin{theorem}\label{Corollary : Uniqueness}
Let $\rho_0 \in L^m(\bbr^d) \cap \mathcal{P}_2(\bbr^d)$. Assume $V$ satisfy
\begin{equation}\label{V-p-moment : uniqueness}
V\in \mathcal{S}_{m,m}^{(q_1, q_2)} \ \text{ for } \
\begin{cases}
 2 \leq q_2 \leq \frac{2m-1}{m-1}, & \text{ if } \ d>2,  \vspace{1 mm}\\
 2 \leq q_2 < \frac{2m-1}{m-1}, & \text{ if } \ d=2.
 \end{cases}
\end{equation}
%which is the condition \eqref{V-p-moment} with $q=m$ and ${\lambda_q}=2$.
Assume further $\nabla V \in L_{x,t}^{\infty, 1} (Q_T)$. Suppose
that $\rho$ is a $L^m$-weak solution of \eqref{E:Main} such that
$\rho :[0,T]\mapsto \mathcal{P}(\bbr^d)$ is narrowly
continuous. Then it is unique.
%Let $\rho$ be a solution stated in one of Theorem~\ref{Theorem-2-b}, \ref{Theorem-4}, and \ref{Theorem-5}
%where $\rho_{0} \in \calP_{p}(\bbr^d) \cap L^{q}(\bbr^d)$ for $p \geq 2$ and $q\geq m$. {\color{red} Then, $\rho$ is the unique $L^q$-weak solution
%such that $\rho : [0,T] \mapsto \mathcal{P}(\bbr^d)$ is narrowly continuous.}
\end{theorem}

Next corollary is an immediate consequence of Theorem \ref{Corollary : Uniqueness}, which reads as follows:
\begin{cor}\label{Corollary : Uniqueness-2}
Assume $\rho_0 \in L^q(\bbr^d) \cap \mathcal{P}_p(\bbr^d)$ for $p \geq 2$ and $q \geq m$.
Let $V$ satisfy the hypothesis in one of Theorems~\ref{Theorem-2-b}, \ref{Theorem-4} and \ref{Theorem-5}, and let $\nabla V \in L_{x,t}^{\infty, 1} (Q_T)$.
Suppose that $\rho$ is narrowly continuous $L^m$-weak solution of \eqref{E:Main}. Then it is unique.
%Let $\rho$ be a solution stated in one of Theorem~\ref{Theorem-2-b}, \ref{Theorem-4}, and \ref{Theorem-5}
%where $\rho_{0} \in \calP_{p}(\bbr^d) \cap L^{q}(\bbr^d)$ for $p \geq 2$ and $q\geq m$. {\color{red} Then, $\rho$ is the unique $L^q$-weak solution
%such that $\rho : [0,T] \mapsto \mathcal{P}(\bbr^d)$ is narrowly continuous.}
\end{cor}

\begin{remark}\label{Remark : uniqueness}
\begin{itemize}
\item[(i)]
%We mean by narrowly continuous the continuity in the weak* topology of $\left ( \mathcal{C}_b(\bbr^d)\right)'$ (refer \eqref{D:narrowly convergent} in Section~\ref{SS:Wasserstein}).
The notion of \textit{narrowly continuous} is the continuity in the weak* topology of $\left ( \mathcal{C}_b(\bbr^d)\right)'$ (refer \eqref{D:narrowly convergent} in Section~\ref{SS:Wasserstein}).

\item[(ii)] Assume $\rho_0 \in L^q(\bbr^d) \cap \mathcal{P}_p(\bbr^d)$ for $p \geq 2$ and $q \geq m$.
In Theorems~\ref{Theorem-2-b}, \ref{Theorem-4} and \ref{Theorem-5}, for each $1< \tilde{p} \leq p$, we show the existence of $L^q$-weak solution $\rho_{\tilde{p}}$ satisfying
$\rho_{\tilde{p}} \in AC(0,T ;\mathcal{P}_{\tilde{p}}(\bbr^d))$. Especially, this implies that $\rho_{\tilde{p}} : [0,T] \mapsto \mathcal{P}(\bbr^d)$ is narrowly continuous.
We note that $\rho_{\tilde{p}}$ is also $L^m$-weak solution.
%(refer proposition \ref{P:Lq-energy-r})}.
Hence, if we further assume
$\nabla V \in L_{x,t}^{\infty, 1} (Q_T)$, then all these solutions $\rho_{\tilde{p}}$ are identical.

\item[(iii)] We do not know the uniqueness of $L^m$-weak solutions in the absence of the hypothesis $\nabla V \in L_{x,t}^{\infty, 1} (Q_T)$. Furthermore, even with $\nabla V \in L_{x,t}^{\infty, 1} (Q_T)$,
uniqueness is not clear if $q<m$ or $p<2$. Therefore, those questions remain open.
\end{itemize}
\end{remark}

%\begin{cor}\label{Corollary : Uniqueness}
%In Theorem \ref{Theorem-2-b}, Theorem
%\ref{Theorem-4} and Theorem \ref{Theorem-5}, we impose one more
%assumption $\nabla V \in L_{x,t}^{\infty,1}(\bar{Q}_T)$. If
%$\rho_0 \in L^q(\bbr^d) \cap \mathcal{P}_p(\bbr^d)$ for $q \geq m$
%and $p \geq 2$ then our solution in each Theorem is the unique $q-$
%weak solution of
%\begin{equation}\label{eq 10 : uniqueness}
%\partial_t \rho + \nabla \cdot (-\nabla \rho^m +V \rho)=0, \qquad \rho(\cdot, 0)=\rho_0
%\end{equation}
%such that
%\begin{equation}\label{eq 11 : uniqueness}
%\esssup_{0\leq t \leq T} \int_{\mathbb{R}^d\times \{t\}}   \rho^q
%\,dx + \iint_{Q_T} \left |\nabla \rho^{\frac{q+m-1}{2}}
%\right |^2 \,dx\,dt < \infty.
%\end{equation}
%\end{cor}

\subsection{An application}

In this subsection, we provide an application to which our results can be properly referred.
We consider an repulsion type of Keller-Segel equations, which is
of the form
\begin{equation}\label{milchim-10}
\rho_t-\Delta \rho^m=\nabla \cdot \bke{\rho \nabla c},\qquad c_t-\Delta c=\rho,\quad \text{in } \ Q_T :=\bbr^d\times (0, T],
\end{equation}
where $d\ge 2$ and $0<T<\infty$.
If \eqref{milchim-10} is considered in a bounded domain $\Omega \subset \bbr^d$,
no flux boundary condition is typically  assgined, that is
\begin{equation}\label{milchim-11}
\frac{\partial\rho}{\partial\nu }=\frac{\partial c}{\partial\nu}=0\qquad \mbox{ on }\,\,\partial\Omega.
\end{equation}
One can compare to the porous medium type of  classical Keller-Segel equations, which is a coupled system of the form
\begin{equation}\label{milchim-15}
\rho_t-\Delta \rho^m=-\nabla \cdot \bke{\rho \nabla c},\qquad c_t-\Delta c=\rho,\quad \text{in } \ Q_T.
\end{equation}

Contrast to \eqref{milchim-10}, it shows aggregation effect due to the opposite sign of reaction-diffusion term for the equation $\rho$, and it turns out that solutions globally exist if
$m > \frac{2(d-1)}{d}$ or blow-ups can occur in a finite time if $1<m \leq \frac{2(d-1)}{d}$ (see e.g. \cite{HashiraSachikoYokota}, \cite{Sug} and reference therein). We remark that, in general, the linear case ($m=1$) with large data yields blow-up in a finite time except for one dimensional case (see e.g. \cites{JL92, Herrero-Velazquez, HorWin, Win, Win12, Win13}  and reference therein). 
%Some remarks are prepared for the system \eqref{milchim-15}.

 In the case that $m=\frac{2(d-1)}{d}$ and  equation of $c$ in \eqref{milchim-15} is of elliptic type, asymptotic behaviors of solutions were classified in \cite{BlaCarLau} by using the gradient flow structure of the equation for each case of subcritical, critical or supercritical mass.
The free energy functional was analyzed in \cite{CarHitVolYao} to show that global minimizers are radial and compactly supported, and, in particular, the minimizer is unique in two dimensions and solutions of PME Keller-Segel equations asymptotically converges to the minimizer (see also the survey paper \cite{CarCraYao}).
Meanwhile, the existence of global $L^q$-weak solutions, $q>1$, was established in \cite{BiaLiu}, provided that $L^p$-norm, $p=\frac{d(2-m)}{2}$, of initial data is sufficiently small.

Since the equation \eqref{E:Main} %under our consideration
is quite general,  our results do not seem to be directly applicable to those mentioned above for specific system \eqref{milchim-15}.
Our analysis, however, provide a heuristic evidence why $p=\frac{d(2-m)}{2}$ appears in \cite{BiaLiu}.
Suppose that a priori $L^q$-energy estimate is available. Then, $\rho\in L^{\frac{d(m+q-1)}{d-2}, m+q-1}_{x,t}$, so is $\Delta c$ due to $-\Delta c=\rho$. Applying the subscaling invariant condition for $\nabla V$ (in this case $\nabla V=\Delta c$), we can conclude that $L^q$-weak solutions exist, provided that $q$ satisfies
\[
\frac{d+q_{m,d}}{m+q-1}\le 2+q_{m,d}\quad\Longrightarrow\quad q\ge \frac{d(2-m)}{2}.
\]
On the other hand, %as mentioned earlier, 
it was also shown in \cite{BiaLiu} that,
if no smallness restriction is assumed, blow-up in a finite time may occur. More precisely, $ \lim_{t \to T^\ast}\norm{\nabla c(t)}_{L^{2}_{x}} = \lim_{t\to T^\ast} \norm{\rho(t)}_{L^{r}_{x}} = \infty$ for any $r>p$,  where $T^\ast$ is a finite blow-up time. Our results can be also used to see that blow-up properties are relevant to each other. Indeed, suppose  $\norm{\nabla c(t)}_{L^{2}_{x}}$ is finite until the blow-up time. Then, via the subscaling invariant condition, we can have global $L^q$-weak solution wtih $q$ satisfying
$q\le \frac{2d(m-1)}{d-2}$.
If $m\ge \frac{2d}{d+2}$, then there, however, exists $q\in (\frac{d(2-m)}{2}, \frac{2d(m-1)}{d-2}]$ so that $L^q$-weak solution exists globally in time, which is contrary to the fact that $L^r$-norm of $\rho$ for $r>p$ becomes infinite at time $T^\ast$. Therefore, $\norm{\nabla c(t)}_{L^{2}_{x}}$ must be unbounded if $\lim_{t\to T^\ast} \|\rho(t)\|_{L^{r}_{x}} = \infty$ for $r>p$.
Our results do not give any proof of the results in \cite{BiaLiu}, but it could be used to conduct heuristic observations.

Now we return to the system \eqref{milchim-10} of repulsive type.
Due to repusive effect of reaction diffusion term in \eqref{milchim-10}, it is, however, unknown whether or not blow-up can occur in a finite time.
More precisely, the repulsive effect deduces via cancelation property the following estimate:
\begin{equation}\label{KS-moment5}
\frac{d}{dt}\bke{\int_{\bbr^d} \rho\log\rho+\frac{1}{2}\int_{\bbr^d}\abs{\nabla c}^2}+\frac{4}{m^2}\int_{\bbr^d}\abs{\nabla\rho^{\frac{m}{2}}}^2+\int_{\bbr^d}\abs{\Delta c}^2 =0,
\end{equation}
which is not available to the case of  classical Keller-Segel equations \eqref{milchim-15}.

It is not difficult to see that the a priori estimate \eqref{KS-moment5} yields existence of entropy weak solutions. However, it is not simple to show whether or not such weak solutions are regular.
Similar phenomena can be found in the Nernst-Planck-Navier-Stokes system, which has a type of repulsive structure as in \eqref{milchim-10}. We remark that \cite{ConIgn} established global well-posedness of regular solutions in two
dimensions and \cite{LiuWan} proved the existence of weak solutions for
two-ionic species by using free energy dissipation via gradient
flow structure in three dimensions.
In this directions, lots of improvements have recently been accomplished and we do not make a list of all relavent results here.

Our aim in this subsection is to improve the result of \cite{Freitag} with the aid of our main results.
As far as authors know, the best known results in regularity theory for \eqref{milchim-10}-\eqref{milchim-11} is that  if $m>1+\frac{(d-1)(d-2)}{d^2}$, $d\ge 2$, then solutions become bounded globally in time (see \cite{Freitag}). We remark that note that $1+\frac{(d-1)(d-2)}{d^2}<\frac{2(d-1)}{d}$.
%Freitag showed in \cite{Freitag} that if $m>1+\frac{(d-1)(d-2)}{d^2}$, $d>2$, then solutions become bounded.
Our main objective is to show that solutions for \eqref{milchim-10} become globally bounded for the case $ m> \max \{ \frac{2d-3}{d}, \, \frac{2d}{d+2} \}$.
We note, due to previous result in \cite{Freitag}, that it suffices to consider the case $ \max \{ \frac{2d-3}{d}, \, \frac{2d}{d+2} \} < m\le 1+\frac{(d-1)(d-2)}{d^2}$ for $d>2$.
%Our result holds also  for the case of bounded domain, since the control of $\int_{\bbr^d} \abs{\rho\log\rho}$ is simpler, and thus we skip its details for such case.
With the aid of our main results, we obtain the following.
The proof will be given in Section~\ref{PME-KS-eq}.

\begin{theorem}\label{thm-jengchik50}
Let $m> \max \{ \frac{2d-3}{d}, \, \frac{2d}{d+2} \}$, $d>2$, and $p\in (1, \frac{md-d+2}{d-1} ]$. Suppose that $\rho_0\in {\mathcal P}_p(\R^d)\cap L^{\infty}(\R^d)$ and $c_0\in L^1 (\bbr^d)\cap W^{2,\infty}(\R^d)$. Then, a pair of solution $(\rho, c)$ for the equations \eqref{milchim-10} becomes bounded globally in time.
\end{theorem}

\begin{remark}
\begin{itemize}
\item[(i)] We remark that $m > \frac{2d-3}{d}$ is equivalent to $\frac{md-d+2}{d-1} > 1$. Moreover, we note that the condition $\max \{ \frac{2d-3}{d}, \, \frac{2d}{d+2} \}$ can be rewritten as $\frac{2d}{d+2} \cdot \delta_{\{d \leq 6\}} + \frac{2d-3}{d}\cdot \delta_{\{ d >6\}}$.

\item[(ii)] If $\rho_0\in {\mathcal P}_p(\R^d)\cap L^{\infty}(\R^d)$ is assumed to be H\"{o}lder continuous, then  so is $\rho$ for $t \in  (0, T]$. Since its proof is rather straightforward once solutions become bounded (see e.g. \cite{CHKK17}), we omit the details. The uniqueness of solutions, however, remains open, since the results in Section~\ref{sub-unique} do not seem to be applicable to the system \eqref{milchim-10}.
\end{itemize}
\end{remark}

%---------------------------------------

\section{Preliminaries}\label{S:Preliminaries}

In this section, we introduce preliminaries that are used throughout the paper.

For a function $f:\Omega \times [0,T] \to \mathbb{R}$, $\Omega \subset \mathbb{R}^d, d \geq 2$ and constants $q_1, q_2 > 1$, we define
\[
\|f\|_{L^{q_1, q_2}_{x, t}} : = \left(\int_{0}^{T}\left[\int_{\Omega} \abs{f(x,t)}^{q_1} \,dx\right]^{q_2 / q_1} \,dt\right)^{\frac{1}{q_2}}.
\]
For simplicity, let $\|f\|_{L^{q}_{x,t}} = \|f\|_{L^{q, q}_{x,t}}$ for some $q>1$. Also we denote for $\Omega \subseteq \bbr^d$ and $\Omega_T \subseteq Q_T$ that
\begin{equation*}
\begin{gathered}
\|f(\cdot,t)\|_{\calC^{\alpha}(\Omega)} := \sup_{x,y \in \Omega, \, x \neq y} \frac{|f(x, t) - f(y,t)|}{|x-y|^\alpha}, \\
\|f\|_{\calC^{\alpha}(\Omega_T)} := \sup_{(x, t), (y,s) \in \Omega_T, \, (x,t) \neq (y,s)} \frac{|f(x, t) - f(y,s)|}{|x-y|^\alpha + |t-s|^{\alpha/2}}.
\end{gathered}
\end{equation*}

\subsection{Technical lemmas}
We introduce well known Sobolev type inequalities. Throughout this subsection, let $\Omega \subseteq \bbr^d$.
\begin{lemma}\label{T:Sobolev} \cite[Theorem~7.10]{GT}
Let $\rho \in W_{o}^{1,p}(\Omega)$ and assume $ 1 \leq p < d$. There exists a constant $c=c(d,p)$ such that
\[
\|\rho\|_{L_{x}^{\frac{dp}{d-p}}(\Omega)} \leq c \|\nabla \rho\|_{L^{p}_{x}(\Omega)}.
\]
\end{lemma}

\begin{lemma}\label{T:pSobolev} \cite[Proposition~I.3.1]{DB93}
  Let $\rho \in L^{\infty}(0, T; L^{q}(\Omega)) \cap L^{p}(0, T; W_{o}^{1,p}(\Omega))$ for some $1 \leq p < d$, and $0 < q < \frac{dp}{d-p}$. Then there exists a constant $c=c(d,p,q)$ such that
  \[
  \iint_{\Omega_{T}} |\rho(x,t)|^{\frac{p(d+q)}{d}}\,dx\,dt \leq c \left(\sup_{0\leq t \leq T}\int_{\Omega} |\rho (x,t)|^{q}\,dx\right)^{\frac{p}{d}} \iint_{\Omega_{T}} |\nabla \rho(x,t)|^{p}\,dx\,dt.
  \]
\end{lemma}

The following lemma is the Garilardo-Niremberg multiplicative embedding inequality.
\begin{lemma}\label{T:GN} \cite[Theorem~I.2.1]{DB93}
Let $\rho \in W_{o}^{1,p}(\Omega)$ for $p \geq 1$. For every fixed number $s\geq 1$, there exists a constant $c=c(d,p,s)$ such that
\[
\|\rho\|_{L_{x}^{q}(\Omega)} \leq c \|\nabla \rho\|_{L^{p}_{x}(\Omega)}^{\alpha} \|\rho\|_{L^{s}_{x}(\Omega)}^{1-\alpha},
\]
where $\alpha \in [0,1]$, $p, q \geq 1$, are linked by
$\alpha = \left(\frac 1s - \frac 1q \right) \left( \frac 1d - \frac 1p + \frac 1s\right)^{-1}$,
and their admissible range is:
\begin{equation*}
\begin{cases}
s \leq q \leq \infty, \ 0 \leq \alpha \leq \frac{p}{p+s(p-1)}, & \text{ if } \ d=1, \\
s \leq q \leq \frac{dp}{d-p}, \ 0 \leq \alpha \leq 1, &\text{ if } \ 1\leq p < d, \ s \leq \frac{dp}{d-p}, \\
\frac{dp}{d-p} \leq q \leq s, \ 0 \leq \alpha \leq 1, &\text{ if } \ 1 \leq p < d, \ s \geq \frac{dp}{d-p}, \\
s \leq q < \infty, \ 0\leq \alpha < \frac{dp}{dp + s(p-d)}, &\text{ if } \ p \geq d > 1.
\end{cases}
\end{equation*}

\end{lemma}

Now we derive the following lemma which is useful to obtain a priori estimates.
 \begin{lemma}\label{P:L_r1r2}
 Let $m >1$ and $q \geq 1$. Suppose that
 \begin{equation}\label{rho-space}
 \rho \in L^{\infty}(0, T; L^{q}(\Omega)) \quad \text{and} \quad
 \rho^{\frac{q+m-1}{2}} \in L^{2}(0, T; W^{1,2}_0(\Omega)).
 \end{equation}
  Then $\rho\in L^{r_1, r_2}_{x, t} (\Omega \times [0, T])$ such that
 \begin{equation}\label{q-r1r2}
 \begin{gathered}
\frac{d}{r_1} + \frac{2+q_{m,d}}{r_2} =
\frac{d}{q}, \quad \text{for} \quad q_{m,d} : = \frac{d(m-1)}{q}, \\
 \text{for} \ \
 \begin{cases}
  q \leq r_1 \leq \frac{d(q+m-1)}{d-2},  \quad q+m-1 \leq r_2 \leq \infty, & \text{ if } \ d > 2, \\
  q \leq r_1 < \infty, \quad q+m-1 < r_2 \leq \infty, &\text{ if }\ d = 2, \\
   q \leq r_1 \leq \infty, \quad q+m-1 \leq r_2 \leq \infty, &\text{ if } \ d= 1.
  \end{cases}
\end{gathered}
\end{equation}
Moreover, there exists a constant $c=c(d)$ such that, for $(d-2)_{+} = \max \{0, d-2\}$,
 \begin{equation}\label{norm-q-r1r2}
\|\rho\|_{L^{r_1, r_2}_{x, t}} \leq c\left(\sup_{0 \leq t \leq T} \int_{\mathbb{R}^d} \rho^{q}(\cdot, t) \,dx\right)^{\frac{1}{r_1} \left[1-\frac{(r_1 - q) (d-2)_{+}}{d(m-1)+2q}\right]} \,
\left\|\nabla \rho^{\frac{q+m-1}{2}} \right\|_{L^{2}_{x,t}}^{\frac{2}{r_2}}.
\end{equation}
 \end{lemma}

\begin{proof}
First, let $d>2$. By applying Lemma~\ref{T:Sobolev} with \eqref{rho-space}, we let
\begin{equation}\label{pme-interpolate-20}
q \leq r_1 \leq \frac{d(q+m-1)}{d-2}, \quad  \frac{1}{r_1}=\frac{\theta}{q}+\frac{(1-\theta)(d-2)}{(q+m-1)d} \quad \Longrightarrow \quad \theta=\frac{q}{r_1}\bke{1-\frac{(r_1-q)(d-2)}{2q+ d(m-1)}}.
\end{equation}
%which gives
%\begin{equation}\label{pme-interpolate-20}
%\theta=\frac{q}{r_1}\bke{1-\frac{(r_1-q)(d-2)}{2q+ d(m-1)}}.
%\end{equation}
Then we carry the following computations
\begin{equation}\label{pme-interpolate-10}
\norm{\rho}_{L^{r_1}_x}\le \norm{\rho}^{\theta}_{L^q_x}\norm{\rho}^{1-\theta}_{L^{\frac{(q+m-1)d}{d-2}}_x}\le \norm{\rho}^{\theta}_{L^q_x}\norm{\rho^{\frac{q+m-1}{2}}}^{\frac{2(1-\theta)}{q+m-1}}_{L^{\frac{2d}{d-2}}_x}\leq c \norm{\rho}^{\theta}_{L^q_x}\norm{\nabla \rho^{\frac{q+m-1}{2}}}^{\frac{2(1-\theta)}{q+m-1}}_{L^{2}_x}.
\end{equation}
By taking $L^{r_2}_{t}$, we obtain
\[
\|\rho\|_{L^{r_1,r_2}_{x,t}} \leq c \norm{\rho}^{\theta}_{L^{q, \infty}_{x, t}}\norm{\nabla \rho^{\frac{q+m-1}{2}}}^{\frac{2}{r_2}}_{L^{2}_{x,t}},
\]
provided
\begin{equation*}
\frac{(1-\theta) r_2}{q+m-1} = 1 \quad \iff \quad \frac{d}{r_1} + \frac{2+q_{m,d}}{r_2} = \frac{d}{q}.
\end{equation*}

Now let $1 \leq d \leq 2$, then Lemma~\ref{T:GN} with $v=\rho^{\frac{q+m-1}{2}}$ and $s=\frac{2q}{q+m-1}$ gives that
\[
\left(\int_{\mathbb{R}^d} \rho^{\gamma \frac{q+m-1}{2}} \,dx\right)^{\frac{1}{\gamma}}
\leq c \left(\int_{\mathbb{R}^d} \left| \nabla \rho^{\frac{q+m-1}{2}} \right|^2  \,dx\right)^{\frac{\alpha}{2}} \left(\int_{\mathbb{R}^d} \rho^{q}  \,dx\right)^{\frac{1-\alpha}{s}},
\]
for $\alpha = \frac{\gamma - s}{\gamma}$ and $\frac{2q}{q+m-1} \leq \gamma < \infty$. Then we set
\[
r_1 = q \theta + \gamma \frac{q+m-1}{2} (1-\theta), \quad \text{for} \quad \theta = 1-\frac{2(r_1 - q)}{\gamma (q+m-1) - q} \in [0, 1].
\]
Therefore, the combination of H\"{o}lder inequality for $\theta + (1-\theta) =1$ yield
\[
\int_{\mathbb{R}^d} \rho^{r_1}\,dx
\leq c \left(\int_{\mathbb{R}^d} \rho^q \,dx\right)^{\theta+ \frac{\theta \gamma (1-\alpha)}{s}}
\left(\int_{\mathbb{R}^d} \abs{\nabla \rho^{\frac{q+m-1}{2}}}^2\,dx\right)^{(1-\theta)\frac{\alpha \gamma}{2}}.
\]
Then by taking power by $\frac{r_2}{r_1}$ for $r_2 > 1$, we set $(1-\alpha)\frac{\alpha \gamma r_2}{2 r_1} = 1 $ which gives \eqref{q-r1r2} and complete the proof.

\end{proof}

Now we introduce the Aubin-Lions lemma.
\begin{lemma}\label{AL}\cite[Proposition~III.1.3]{Show97}, \cite{Sim87}
Let $X_0$, $X$ and $X_1$ be Banach spaces with $X_0 \subset X \subset X_1$. Suppose that $X_0$ and $X_1$ are reflexive, and $X_0 \hookrightarrow X $ is compact, and $X \hookrightarrow X_1 $ is continuous. For $1 \leq p, q \leq \infty$, let us define
$W = \left\{ u \in L^{p}\left(0, T; X_0\right),\ u_{t} \in L^{q}\left(0, T; X_1\right) \right\}$.
If $p < \infty$, then the inclusion $W \hookrightarrow L^{p}(0,T; X)$ is compact. If $p=\infty$ and $q>1$, then the embedding of $W$ into $\calC(0,T;X)$ is compact.
\end{lemma}

%\section{Properties of Porous Medium Equations}
%
%\subsection{Equation and its approximation}

%In this paper, we concern $\rho(x,t)$ a nonnegative solution of the porous medium equation with drifts given in
%\begin{equation}\label{PME}
%  \rho_t = \Delta \rho^m + \nabla \cdot \left( \rho V\right) \quad \text{ in } \ \mathbb{R}^d \times [0, \infty) \quad \text{ with } \ m \geq 1.
%\end{equation}
%The drift term $V(x,t) : \mathbb{R}^d \times [0, \infty) \to \mathbb{R}^d \times [0, \infty) $ is assumed in a proper $L_{x}^{p}L_{t}^{q}$ type space.

%Because of the degeneracy of the diffusion, we provide an approximated equation to deliver a priori estimates. For some large $r > 0$, we consider $u_{\epsilon, r}$ which solves the following problem:
%\begin{equation}\label{PMEapp}
%  \begin{cases}
%    \partial_t \rho_{\epsilon} = \dv \left( (\rho_{\epsilon} + \epsilon)^{m-1}  \nabla \rho_{\epsilon} \right) + \nabla \cdot (\rho_{\epsilon} V_{\epsilon}) , & \mbox{in } \mathbb{R}^{d} \times [0, T] \\
%    \rho_{\epsilon} (x,0) = \rho_{0}(x), & \mbox{on } \mathbb{R}^d.
%  \end{cases}
%\end{equation}
%Because \eqref{PMEapp} carries the uniform ellipticity, there exists a classical solution based on classical theories in PDEs. Before to establish H\"{o}lder regularity, we need to observe the boundedness of a solution. Then we apply H\"{o}lder regularity theory under appropriate conditions in the drift term.

%-----------------------------------------

\subsection{Wasserstein space}\label{SS:Wasserstein}
In this subsection, we introduce the Wasserstein space and its properties. For more detail, we refer \cites{ags:book, V}.

\begin{defn}
Let $\mu$ be a probability measure on $\mathbb{R}^d$. Suppose there is a measurable map $\mathcal{T}: \mathbb{R}^d\mapsto \mathbb{R}^d$,
then $\mathcal{T}$ induces a probability measure $ \nu$ on $\mathbb{R}^d$
which is defined as
$$\int_{\mathbb{R}^d} \varphi(y)\,d\nu(y) = \int_{\mathbb{R}^d} \varphi(\mathcal{T}(x))\,d\mu(x), \qquad \forall ~ \varphi\in C(\mathbb{R}^d).$$
We denote $\nu:=\mathcal{T}_\#\mu$ and say that $\nu$ is the push-forward of $\mu$ by $\mathcal{T}$.
\end{defn}

%\begin{defn}
Let us denote by $\mathcal{P}_p(\mathbb{R}^d)$ the set of all Borel
probability measures on $\mathbb{R}^d$ with a finite $p$-th moment. For $\mu,\nu\in\mathcal{P}_p(\mathbb{R}^d)$, we consider
\begin{equation}\label{Wasserstein dist}
W_p(\mu,\nu):=\left(\inf_{\gamma\in\Gamma(\mu,\nu)}\int_{\mathbb{R}^d\times
\mathbb{R}^d}|x-y|^p\, d\gamma(x,y)\right)^{\frac{1}{p}},
\end{equation}
where $\Gamma(\mu,\nu)$ denotes the set of all Borel probability
measures on $\mathbb{R}^d\times \mathbb{R}^d$ which has $\mu$ and
$\nu$ as marginals;
$$\gamma(A\times \mathbb{R}^d)= \mu(A) \quad \text{and} \quad \gamma(\mathbb{R}^d\times A)= \nu(A) $$
for every Borel set $A\subset \mathbb{R}^d.$
Equation (\ref{Wasserstein dist}) defines a distance on
$\mathcal{P}_p(\mathbb{R}^d)$ which is called the {\it Wasserstein distance}.
 Equipped with the Wasserstein distance,  $\mathcal{P}_p(\mathbb{R}^d)$ is called the {\it Wasserstein space}.
 It is known that the infimum in the right hand side of Equation
(\ref{Wasserstein dist}) always achieved.

We will denote by $\Gamma_o(\mu,\nu)$ the set of all $\gamma$ which minimize the expression.
If $\mu$ is
absolutely continuous with respect to the Lebesgue measure then there exists a map $\mathcal{T} : \bbr^d \rightarrow \bbr^d$ such that
$\gamma:=(Id\times \mathcal{T})_\# \mu$ is the unique element of $\Gamma_o(\mu,\nu)$, that is, $\Gamma_o(\mu,\nu)=\{(Id\times  \mathcal{T})_\# \mu \}$.

We say that a sequence $\{\mu_n\} \subset \mathcal{P}(\mathbb{R}^d)$ is narrowly convergent to $ \mu \in \mathcal{P}(\mathbb{R}^d)$ as
$n \rightarrow \infty$ if
\begin{equation}\label{D:narrowly convergent}
\lim_{n\rightarrow \infty} \int_{\bbr^d} \varphi(x) \,d\mu_n(x) =\int_{\bbr^d} \varphi(x) \,d\mu(x)
\end{equation}
for every function $ \varphi \in \calC_b (\bbr^d)$, the space of continuous and bounded real functions defined on $\bbr^d$. Then we recall
 \begin{equation*}
\lim_{n \rightarrow \infty} W_p(\mu_n, \mu)=0 \quad \Longleftrightarrow \quad
\begin{cases}
\mu_n ~\mbox{narrowly converges to } \mu ~\mbox{ in }  \mathcal{P}(\mathbb{R}^d)\\
\lim_{n\rightarrow \infty} \int_{\bbr^d} |x|^p\, d\mu_n(x) = \int_{\bbr^d} |x|^p\,d \mu(x).
\end{cases}
\end{equation*}

%\end{defn}
%The Wasserstein space $\mathcal{P}_2(\mathbb{R}^d)$ is a compact metric space since $\mathbb{R}^d$ is compact.

%\begin{lemma}[\cite{ags:book}, Proposition 7.1.5]
%Let $\{\mu_n\}_n\subset \mathcal{P}_2(\mathbb{R}^d)$ be a sequence.
%Then, $\mu_n$ converges to $\mu\in \mathcal{P}(\mathbb{R}^d)$ in
%the topology induced by the Wasserstein distance is characterized
%as
%%equivalent to that $\mu_n$ converges to $\mu$ weakly. That is, for any $\varphi\in C(\mathbb{R}^d)$
%\begin{equation}
%\lim_{n\rightarrow \infty}W_2(\mu_n,\mu) =0 \Longleftrightarrow
%\left \{
%\begin{aligned}
%&\lim_{n\rightarrow \infty} \int_{\mathbb{R}^d}\varphi(x)
%d\mu_n(x)
%= \int_{\mathbb{R}^d}\varphi(x) d\mu(x) ~~ \forall~\varphi\in C_b(\mathbb{R}^d),\\
%&\lim_{n\rightarrow \infty} \int_{\mathbb{R}^d}|x|^2 d\mu_n(x)
%=\int_{\mathbb{R}^d}|x|^2 d\mu(x) .
%\end{aligned}
%\right .
%\end{equation}
%
%\end{lemma}
%
%
Following \cites{AG, ags:book}, we give a notion of a differential and a definition of convex functions on $\mathcal{P}_p(\bbr^d)$.
\begin{defn}\label{Convexity}
Let $\phi:\mathcal{P}_p(\mathbb{R}^d)\mapsto (-\infty,\infty]$.
We say that $\phi$ is geodesically convex in $\mathcal{P}_p(\mathbb{R}^d)$ if for every
couple $\mu_1,\mu_2\in \mathcal{P}_p(\mathbb{R}^d)$ there exists an
optimal plan $\gamma\in\Gamma_o(\mu_1,\mu_2)$ such that
$$\phi(\mu_t^{1\rightarrow 2})\leq (1-t)\phi(\mu_1)+t\phi(\mu_2),\qquad \forall ~ t\in[0,1], $$
where $\mu_t^{1\rightarrow 2}$ is a constant speed geodesic
between $\mu^1$ and $\mu^2$ defined as
$$\mu_t^{1\rightarrow 2} := ((1-t)\pi^1 + t\pi^2)_\#\gamma .$$
Here, $\pi^1: \mathbb{R}^d\times \mathbb{R}^d \mapsto
\mathbb{R}^d$ and $\pi^2: \mathbb{R}^d\times \mathbb{R}^d \mapsto
\mathbb{R}^d$ are the first and second projections of $
\mathbb{R}^d\times \mathbb{R}^d$ onto $ \mathbb{R}^d$ defined by
$$\pi^1(x,y)=x,\quad  \pi^2(x,y)=y,\qquad \forall ~ x,y \in \mathbb{R}^d.$$
\end{defn}

\begin{defn} \label{def:Fdifferentiable}
Let $\phi: \mathcal{P}_p(\bbr^d) \rightarrow (-\infty,\infty]$ be a proper function on $\mathcal{P}_p(\bbr^d)$, {\it i.e.} the effective domain of $\phi$ defined by
$D(\phi):=\{\mu \in \mathcal{P}_p(\bbr^d) : \phi(\mu) < \infty \}$ is not empty. We say that
$\xi \in L^{\frac{p}{p-1}}(\mu)$ belongs to the \textit{subdifferential} $\partial_{-} \phi(\mu)$ if
$$ \phi(\nu) \geq \phi(\mu) + \sup_{\gamma \in \Gamma_o(\mu,\nu)}\int_{\bbr^d\times \bbr^d} \langle\xi(x),y-x\rangle d\gamma(x,y) + o(W_p(\mu,\nu)) ,$$
as $\nu \rightarrow \mu  $ in $\mathcal{P}_p(\bbr^d)$. %We denote the domain of subdifferential by $D(\partial_-H):=\{\mu: \partial_-H(\mu)\neq \emptyset   \}$.\\
% If $-\xi \in \partial_{-} (-H)(\mu)$ then
%we say that $\xi$ belongs to the \textit{superdifferential}
%$\partial^+ H(\mu)$.
 \end{defn}

 \begin{lemma} \cite[Proposition~4.2]{AG}
Let $\phi: \mathcal{P}_p(\bbr^d) \rightarrow (-\infty,\infty]$ be lower semicontinuous and geodesically convex and let $\xi \in \partial_{-} \phi(\mu)$. Then
we have
$$ \phi(\nu) \geq \phi(\mu) + \sup_{\gamma \in \Gamma_o(\mu,\nu)}\int_{\bbr^d\times \bbr^d} \langle\xi(x),y-x\rangle d\gamma(x,y),  \qquad \forall ~ \nu \in \mathcal{P}_p(\bbr^d).$$

 \end{lemma}

Now, we introduce the notion of absolutely continuous curve and its relation with the continuity equation.
\begin{defn}
Let $\sigma:[0,T]\mapsto \mathcal{P}_p(\mathbb{R}^d)$ be a curve.
We say that $\sigma$ is absolutely continuous and denote it by $\sigma \in AC(0, T;\mathcal{P}_p(\mathbb{R}^d))$, if there exists $l\in
L^1([0,T])$ such that
\begin{equation}\label{AC-curve}
W_p(\sigma(s),\sigma(t))\leq \int_s^t l(r)dr,\qquad \forall ~ 0\leq s\leq t\leq T.
\end{equation}
If $\sigma \in AC(0,T;\mathcal{P}_p(\mathbb{R}^d))$, then the limit
$$|\sigma'|(t):=\lim_{s\rightarrow t}\frac{W_p(\sigma(s),\sigma(t))}{|s-t|} ,$$
exists for $L^1$-a.e $t\in[0,T]$. Moreover, the function $|\sigma'|$ belongs to $L^1(0,T)$ and satisfies
\begin{equation*}
|\sigma'|(t)\leq l(t) \qquad \mbox{for} ~L^1\mbox{-a.e.}~t\in [0,T],
\end{equation*}
for any $l$ satisfying \eqref{AC-curve}.
We call $|\sigma'|$ by the metric derivative of $\sigma$.
\end{defn}

\begin{lemma}\label{representation of AC curves} \cite[Theorem~8.3.1]{ags:book}
If $\sigma\in AC(0,T;\mathcal{P}_p(\mathbb{R}^d))$ then there exists a Borel vector field $ w: \bbr^d\times(0,T)\mapsto \bbr^d$ such that
$$w(t) \in L^p(\sigma(t)),\quad  \|w(t) \|_{L^p(\sigma(t))} \leq |\sigma '|(t), \quad \mbox{   for} ~~L^1\text{-a.e}  ~~ t\in [0,T]  , $$
and the continuity equation
$$\partial_t\sigma +\nabla\cdot( w \, \sigma)=0,  $$
holds in the sense of distribution.
Conversely, if a narrowly continuous curve $\sigma : [0,T] \mapsto \mathcal{P}_p(\mathbb{R}^d)$ satisfies the continuity equation
for some Borel vector field $w$ with $\| w(t)\|_{L^p(\sigma(t))}\in L^1(0,T)$, then $\sigma: [0,T]\mapsto \mathcal{P}_p(\mathbb{R}^d)$
is absolutely continuous and
$|\sigma'|(t)\leq \| w(t)\|_{L^p(\sigma(t))}$ for $L^1$-a.e $t\in [0,T]$.
\end{lemma}

\begin{remark}\label{R:speed}
We call $\| w(t)\|_{L^p(\sigma(t))}$ an auxiliary speed of curve $\sigma(t)$ in Wasserstein space $\mathcal{P}_p (\mathbb{R}^d)$. We use `speed' instead of `auxiliary speed' unless any confusion is to be expected.
\end{remark}

%{\it Notation} : In Lemma \ref{representation of AC
%curves}, we use notation $v_t:=v(\cdot,t)$ and
%$\sigma_t:=\sigma(t)$. Throughout this paper, we keep this
%convention, unless any confusion is to be expected, and a usual
%notation $\partial_t$ is adopted for temporal derivative, i.e.
%%More precisely, let $f: (a,b)\times \bbr^k\mapsto \bbr^m$ be given.
%%Then we use notation
%$f_t:=f(\cdot,t)$ and $\partial_t f:= \frac{\partial f}{\partial t}$.
%%for any differentiable function $f: (a,b)\times \bbr^k\mapsto
%%\bbr^m$.
\begin{lemma}\label{Lemma : Arzela-Ascoli} \cite[Proposition 3.3.1]{ags:book}
Let $K  \subset \mathcal{P}_p(\bbr^d)$ be a sequentially compact set w.r.t the narrow topology.
Let $\sigma_n : [0, T] \rightarrow \mathcal{P}_p(\mathbb{R}^d)$ be curves such that
\begin{equation*} %\label{equi-continuity}
\begin{aligned}
\sigma_n(t) \in K, \quad \forall ~ n &\in \mathbb{N}, ~ t \in [0,T],\\
W_p(\sigma_n(s),\sigma_n(t))&\leq \omega(s,t), \qquad \forall ~
0\leq s\leq t\leq T, ~ n \in \mathbb{N},
\end{aligned}
\end{equation*}
for a continuous function $\omega : [0, T] \times [0, T] \rightarrow [0, \infty)$ such that
$$\omega(t,t)=0, ~~ \forall ~ t \in [0,T].$$
Then there exists a subsequnece $\sigma_{n_k}$ and a limit curve $\sigma : [0,T] \rightarrow {P}_p(\mathbb{R}^d)$ such that
\begin{equation*}\label{eq1 : Lemma : Arzela-Ascoli}
\sigma_{n_k}(t) \mbox{ narrowly converges to} ~~\sigma(t), \qquad
\text{for all} ~~t\in[0, T].
\end{equation*}
%for some $\sigma\in AC(a,b;\mathcal{P}_p(\mathbb{R}^d))$
%satisfying
%\begin{equation}\label{eq2 : Lemma : Arzela-Ascoli}
%W_2(\sigma(s),\sigma(t))\leq \int_s^t m(r)dr\qquad \forall ~ a\leq
%s\leq t \leq b.
%\end{equation}
\end{lemma}

 \subsection{Flows on $\mathcal{P}_2(\bbr^d)$ generated by vector fields}
%In this subsection, we study a flow on $\mathcal{P}_2(\bbr^d)$ which is generated by the vector field $v:= u+ G $.
For a given $T>0$, let $V\in L^1(0,T; W^{1,\infty}(\bbr^d;\bbr^d))$. For any $s, ~t\in [0,T]$, let $\psi:[0,T]\times [0,T]\times\mathbb{R}^d\mapsto
\mathbb{R}^d$ be the flow map of the vector field $V$.
 More precisely, $\psi$ solves the following ODE
 \begin{equation}\label{ODE}
\begin{cases}
\frac{d}{dt}\psi(t;s,x)= V(\psi(t;s,x),t),  &\text{for } s, t\in[0,T] \vspace{1 mm}\\
\psi(s;s,x)=x,  &\text{for }  x\in\bbr^d.
\end{cases}
\end{equation}
Using the flow map $\psi$, we define a flow $\Psi: [0,T]\times [0,T]\times \mathcal{P}_2(\bbr^d)\mapsto \mathcal{P}_2(\bbr^d)$ through the push forward operation as follows
\begin{equation}\label{Flow on Wasserstein}
\Psi(t;s,\mu):={\psi}(t;s,\cdot)_\# \mu, \qquad \forall ~ \mu \in \mathcal{P}_2(\bbr^d).
\end{equation}

%\subsection{Flow maps on $\bbr^d$ corresponding to vector fields}

%For a given $T>0$, let $u\in L^1(0,T; W^{1,\infty}(\bbr^d;\bbr^d))$ be a divergence-free vector field
%and let $G \in L^1(0,T; W^{1,\infty}(\bbr^d;\bbr^d))$. Hence, the vector field $v:= u + G$ also belongs to $ L^1(0,T; W^{1,\infty}(\bbr^d;\bbr^d))$.
%
% For any $s\in [0,T]$, let $\psi_s:[s,T]\times\mathbb{R}^d\mapsto \mathbb{R}^d$ be the flow map of the vector field $v$ starting at time $s$.
% More precisely,
% \begin{equation}\label{ODE}
%\left\{
%\begin{array}{l}
%\frac{d}{dt}\psi_s(x,t)= v(t,\psi_s(x,t))  \qquad t\in[s,T]\\
%\psi_s(s,x)=x  \qquad \qquad  x\in\bbr^d.
%\end{array}
%\right .
%\end{equation}

In this subsection, we remind two basic results on the flow map $\psi$.

\begin{lemma}\label{Lemma : estimation 1: ODE} \cite[Lemma~3.1]{KK-SIMA}
Let $s\in [0,T]$ and $\psi$ be defined as in \eqref{ODE}. Then, for any $t\in [s,T]$, the map $\psi(t;s,\cdot):\mathbb{R}^d\mapsto \mathbb{R}^d$ satisfies
%Let $x,y\in \mathbb{R}^d$ and $t\in[s,T]$. Then, we have the following inequality
\begin{equation*} %\label{estimation 1: ODE}
e^{-\int_s^t \mbox{Lip}(V(\tau)) \,d\tau}|x-y|\leq |\psi(t;s,x)-\psi(t;s,y)|\leq e^{\int_s^t \mbox{Lip}(V(\tau)) \,d\tau}|x-y| ,
\qquad \forall ~ x,y \in\mathbb{R}^d,
\end{equation*}
where
\[\mbox{Lip}(V(\tau)) := \sup_{x,y \in \bbr^d , \, x \neq y} \frac{|V(x, \tau) - V(y,\tau)|}{|x-y|}.\]
\end{lemma}

\begin{lemma}\label{Lemma : Lipschitz of Jacobian}
Let $s\in [0,T]$ and $\psi$ be defined as in \eqref{ODE}. For any $t\in [s,T]$, let $J_{s,t}$ be the Jacobian corresponding to the map
 $\psi(t;s,\cdot):\mathbb{R}^d\mapsto \mathbb{R}^d$ . That is,
\begin{equation}\label{Jacobian}
\int_{\mathbb{R}^d} \zeta(y) dy:= \int_{\mathbb{R}^d} \zeta(\psi(t;s,x))
J_{s,t}(x)\,dx,\qquad \forall ~\zeta\in C(\mathbb{R}^d) .
\end{equation}
Then, the Jacobian  $ J_{s,t}$ is given as
\begin{equation}\label{Jacobian - formular}
J_{s,t}(x)=e^{\int_s^t \nabla \cdot V (\psi(\tau;s,x),\tau) \,d\tau}, \qquad \forall ~ x\in \bbr^d.
\end{equation}
\end{lemma}

\begin{proof}
Let $x\in \mathbb{R}^d$ be given. Then, we have
\begin{equation*} %\label{eq6 : Lipschitz of Jacobian}
\begin{aligned}
\frac{d}{dt} J_{s,t}(x)\big |_{t=\tau}& =\nabla \cdot V(\psi(\tau;s,x), \tau)J_{s,\tau}(x).% \\
% &=\nabla \cdot G(\tau, \psi_s(\tau,x))J_{s,\tau}(x) \qquad \tau\in[s,T],%\\
%\frac{d}{dt} J_{s,t}(y)\big |_{t=\tau}=\nabla \cdot v(\tau,\psi_s(\tau,y))J_{s,\tau}(y).
\end{aligned}
\end{equation*}
%where we use the fact that $u$ is divergence-free.
Since $J_{s,s}\equiv 1$, we have
\begin{equation}\label{eq7 : Lipschitz of Jacobian}
\log J_{s,t}(x) = \int_s^t \nabla \cdot V( \psi(\tau;s,x),\tau)
\,d\tau.
\end{equation}
This completes the proof.
\end{proof}

\begin{remark}
We note that $\psi(t;s,\cdot) \circ \psi(s;t,\cdot) = Id $, that is,
$$ \psi(t;s, \psi(s; t, x)) = x,  \quad \forall ~ x\in \bbr^d. $$
Exploiting this, we have
\begin{equation*}
J_{s,t}(\psi(s;t,x )) = \frac{1}{J_{t,s}(x)}.
\end{equation*}
\end{remark}

 \begin{lemma}\label{Corollary-4:Lipschitz}\cite[Lemma~3.3]{KK-SIMA}
Let $s\in [0,T]$ and $\Psi$ be  in \eqref{Flow on Wasserstein}. For any $t\in [s,T]$ and $\mu,\nu\in \mathcal{P}_p(\mathbb{R}^d)$,
we have
\begin{equation}\label{eq1:Corollary-4:Lipschitz}
e^{-\int_s^t \mbox{Lip}(V(\tau)) \,d\tau}W_p(\mu,\nu)\leq W_p(\Psi(t;s,\mu),\Psi(t;s,\nu))\leq e^{\int_s^t \mbox{Lip}(V(\tau)) \,d\tau}W_p(\mu,\nu).
\end{equation}
\end{lemma}

%\begin{proof}
%Refer Lemma 3.3 in \cite{KK-SIMA}.
%For given $\mu,\nu \in \mathcal{P}_2(\bbr^d)$, let $\gamma\in\Gamma_o(\mu,\nu).$ We note
%$(\psi_s(t,\cdot)\times \psi_s(t,\cdot))_\#\gamma\in\Gamma (\Psi_s(t,\mu),\Psi_s(t,\nu))$, and hence we have
%\begin{equation}
%\begin{aligned}
%W_2^2 (\Psi_s(t,\mu),\Psi_s(t,\nu))  &\leq \int_{\mathbb{R}^d\times\mathbb{R}^d} |x-y|^2
%d (\psi_s(t,\cdot)\times \psi_s(t,\cdot))_\#\gamma(x,y)\\
%&= \int_{\mathbb{R}^d\times\mathbb{R}^d} |\psi_s(x,t)-\psi_s(t,y)|^2 d\gamma(x,y)\\
%&\leq \int_{\mathbb{R}^d\times\mathbb{R}^d} \left(e^{\int_s^t \mbox{Lip}(v_\tau) \,d\tau}|x-y|\right)^2 d\gamma(x,y)\\
%&= \left(e^{\int_s^t \mbox{Lip}(v_\tau) \,d\tau}\right)^2 W_2^2(\mu,\nu),
%\end{aligned}
%\end{equation}
%which gives the second inequality of \eqref{eq1:Corollary-4:Lipschitz}.
%
%Similarly, by changing the role of $\mu$ and $\Psi_s(t,\mu)\big(=\psi_s(t,\cdot)_\# \mu \big)$, and using the flow map $\tilde{\psi}_s$ defined in \eqref{ODE-inverse}, we get
%the first inequality of \eqref{eq1:Corollary-4:Lipschitz} which concludes the proof.
%\end{proof}

\begin{lemma}\label{Lemma : density relation on the flow}\cite[Lemma~3.4]{KK-SIMA}
Let $\psi$ and $\Psi$ be defined as in \eqref{ODE} and \eqref{Flow on Wasserstein}, respectively. If $\mu \in \mathcal{P}_p(\bbr^d)$ then
$\Psi(t;s,\mu)\in \mathcal{P}_p(\bbr^d)$. Furthermore, suppose $\mu=\varrho \,dx$ and $\Psi(t;s,\mu)=\rho \,dx$ then
\begin{equation}\label{Density relation}
\rho(\psi(t;s,x))J_{s,t}(x)=\varrho(x), \qquad a.e \quad x\in \bbr^d,
\end{equation}
where $J_{s,t}$ is the Jacobian of the map $\psi(t;s,\cdot)$ as in \eqref{Jacobian}.
We also have
\begin{equation}\label{Entropy relation} \int_{\bbr^d} \rho \log
\rho \,dx = \int_{\bbr^d}\varrho \log \varrho \,dx -
\int_{\bbr^d}\varrho \log J_{s,t}  \,dx.
\end{equation}
Moreover, if $\varrho \in L^{q}(\bbr^d)$ for $q\in[1,\infty]$, then
${\rho}\in L^{q}(\bbr^d)$ and we have
\begin{equation}\label{L^p relation}
\|{\rho}\|_{L^{q}(\bbr^d)} \leq
\|\varrho\|_{L^{q}(\bbr^d)}e^{\frac{q-1}{q}\int_s^t\|\nabla\cdot
V\|_{L^\infty_x} \,d\tau},
\end{equation}
where $\frac{q-1}{q}=1$ if $q=\infty$.
\end{lemma}
%\begin{proof} \footnote{Removing the proof ????}
%Refer Lemma 3.4 in \cite{KK-SIMA}.
%\end{proof}

%\begin{lemma}\label{Lemma : Holder regularity on the flow}
%Suppose $V \in L^1(0,T;  \calC^{1, \alpha}(\bbr^d)) $ and let $\psi$ and $\Psi$ be defined as in \eqref{ODE} and
%\eqref{Flow on Wasserstein}, respectively.
%
%If $\varrho\in \mathcal{P}_2(\bbr^d) \cap \calC^\alpha (\bbr^d)$ is Holder
%continuous then $\rho:=\Psi(t;s,\varrho)$ is also H\"{o}lder continuous. More precisely, we have, for any $x,~y\in \bbr^d$
%\begin{equation}\label{eq 1 : Holder}
%|\rho(x)-\rho(y)| \leq C |x-y|^\alpha,
%\end{equation}
%here $C$ depends on $\big(\| \varrho\|_{\calC^\alpha (\bbr^d)} , ~\int_s^t \|\nabla V \|_{\calC^\alpha (\bbr^d)} \,d\tau\big).$\\
%%
%%(ii) If $\varrho \in  L^\infty \cap \mathcal{P}(\mathbb{R}^d)$ and suppose there exists $M>0$ such that
%%$$|\varrho(x)-\varrho(y)| < |x-y|^\alpha, \qquad \forall~|x-y| \leq M. $$
%%Let $0\leq s<t \leq T$ be such that
%%$$\int_s^t \|v\|_\infty \,d\tau \leq \frac{M}{4} .$$
%%Then, we have
%% \begin{equation}\label{eq 1 : Remark}
%%|\rho(x)-\rho(y)| \leq C |x-y|^\alpha, \qquad \forall~|x-y| \leq \frac{M}{2},
%% \end{equation}
%%here $C$ only depends on $\big(\|\varrho \|_\infty, ~\int_s^t \|\nabla v(\tau) \|_\infty \,d\tau,  ~\int_s^t \|\nabla v(\tau) \|_\alpha \,d\tau\big).$
%\end{lemma}

\begin{lemma}\label{Lemma : Holder regularity on the flow}
Let
$\psi$ and $\Psi$ be defined as in \eqref{ODE} and \eqref{Flow on
Wasserstein}, respectively. Suppose that $V \in  L^1(0,T;W^{1,\infty}(\bbr^d))$.
\begin{itemize}
\item[(i)] If $V \in L^1(0,T;  \calC^{1, \alpha}(\bbr^d)) $ and $\varrho\in
\mathcal{P}_2(\bbr^d) \cap \calC^\alpha (\bbr^d)$ for some $\alpha \in (0,1)$,
then $\rho:=\Psi(t;s,\varrho)$ is also H\"{o}lder continuous. More
precisely, we have $\| \rho\|_{\calC^\alpha (\bbr^d)} < C $ where $C=C(\| \varrho\|_{\calC^\alpha (\bbr^d)},\, \int_s^t \|\nabla V \|_{\calC^\alpha (\bbr^d)} d\tau )$.
%for any $x,~y\in \bbr^d$
%\footnote{It is now defined in very first part of preliminaries.}
%\begin{equation*}\label{eq 1 : Holder}
%|\rho(x)-\rho(y)| \leq C |x-y|^\alpha,
%\end{equation*}
%where $C = C (\| \varrho\|_{\calC^\alpha (\bbr^d)} , ~\int_s^t \|\nabla V \|_{\calC^\alpha (\bbr^d)} \,d\tau ).$
\item[(ii)] If $V \in L^1(0,T;  W^{2, \infty}(\bbr^d)) $, then, for any  $a>0,~ q\geq 1$, we have
\begin{equation}\label{eq5 : Sobolev}
\begin{aligned}
 \|\nabla \rho^a\|_{L^q (\bbr^d)}&\leq e^{(a+2)\int_s^t \|\nabla V\|_{L^\infty_x}\, d\tau}\left \{ \|\nabla \varrho^a\|_{L^q (\bbr^d)} + \| \varrho^{a}\|_{L^p_x}
 \times \left ( a\int_s^t \| \nabla^2 V\|_{L^\infty_x} \,d\tau\right ) \right \}.
\end{aligned}
\end{equation}
\end{itemize}
\end{lemma}

\begin{proof}
$(i)$ From \eqref{Density relation}, we have
\begin{equation}\label{eq 2 : Holder}
\rho(x) = \frac{{\varrho}(\psi(s;t,x))}{J_{s,t}(\psi(s;t, x))}.
\end{equation}
Let $f := \varrho(\psi(s;t,x))$ and $ g:=\frac{1}{J_{s,t}(\psi(s;t,
x))}$, then
%\begin{equation*}\label{eq 3 : Holder}
$\| \rho\|_{\calC^\alpha(\bbr^d)} \leq \|f \|_{L^\infty_x} \| g \|_{\mathcal{C}^\alpha(\bbr^d)}
+ \| f \|_{\calC^\alpha (\bbr^d)} \| g\|_{L^\infty_x}.$
%\end{equation*}
We note that
%\begin{equation*}\label{eq 4 : Holder}
$\| f\|_{L^\infty_x} =   \| \varrho\|_{L^\infty_x}$
%\end{equation*}
and from \eqref{Jacobian - formular}
\begin{equation*}%\label{eq 5 : Holder}
\| g\|_{L^\infty_x} \leq e^{\int_s^t \|\nabla \cdot V \|_{L^\infty_x}
\,d\tau}.
\end{equation*}
Next, we estimate $\|f\|_{\calC^\alpha (\bbr^d)}$ that
\begin{equation*}%\label{eq 5 : Holder}
%\begin{aligned}
|\varrho(\psi(s;t,x)) - \varrho(\psi(s;t,y))| \leq \|\varrho\|_{\calC^\alpha} |\psi(s; t,x) -\psi(s;t,y) |^\alpha
\leq \|\varrho\|_{\calC^\alpha} e^{\alpha \int_s^t \|\nabla \cdot V
\|_{L^\infty_x} \,d\tau} |x-y|^\alpha.
%\end{aligned}
\end{equation*}
Hence, it gives
\begin{equation*}%\label{eq 6 : Holder}
\|f\|_{\calC^\alpha} \leq \|\varrho\|_{\calC^\alpha} e^{\alpha \int_s^t \|\nabla
\cdot V \|_{L^\infty_x} \,d\tau}.
\end{equation*}
Next, we note that
\begin{equation*}%\label{eq 7 : Holder}
\begin{aligned}
g(x)= e^{-\int_s^t \nabla \cdot V(\psi(\tau;s,\psi(s; t,x)
),\tau)\,d\tau} =e^{-\int_s^t \nabla \cdot V(\psi(\tau;
t,x),\tau)\,d\tau},
\end{aligned}
\end{equation*}
which gives us
\begin{equation}\label{eq 8 : Holder}
\begin{aligned}
|g(x)-g(y)|&= \left |e^{-\int_s^t \nabla \cdot V( \psi(\tau; t,x),\tau)\,d\tau}-e^{-\int_s^t \nabla \cdot V(\psi(\tau;t,y),\tau)\,d\tau}\right |\\
&\leq e^{\int_s^t \|\nabla \cdot V\|_{L^\infty_x} \,d\tau}\int_s^t |\nabla \cdot V( \psi(\tau; t,x),\tau) - \nabla \cdot V(\psi(\tau;t,y),\tau) |\,d\tau\\
&\leq e^{\int_s^t \|\nabla \cdot V\|_{L^\infty_x} \,d\tau}\int_s^t \|\nabla \cdot V(\tau)\|_{\calC^\alpha}
|\psi(\tau; t,x) - \psi(\tau;t,y) |^\alpha \,d\tau \\
&\leq e^{\int_s^t \|\nabla \cdot V\|_{L^\infty_x} \,d\tau}\int_s^t
\|\nabla \cdot V(\tau)\|_{\calC^\alpha} \,d\tau ~ e^{\alpha \int_s^t
Lip(V(\theta)) d\theta} |x-y|^\alpha.
\end{aligned}
\end{equation}
We combine \eqref{eq 2 : Holder}-\eqref{eq 8 : Holder} to get, for
any $x,~y\in \bbr^d$,
\begin{equation*}%\label{eq 9 : Holder}
\begin{aligned}
|\rho(x)-\rho(y)| &\leq e^{(\alpha+1)\int_s^t \|\nabla V
\|_{L^\infty_x} \,d\tau} \left( \|\varrho\|_{L^\infty_x} \int_s^t \|\nabla
\cdot V(\tau)\|_{\calC^\alpha} \,d\tau
 + \|\varrho\|_{\calC^\alpha} \right) |x-y|^\alpha.
\end{aligned}
\end{equation*}
$(ii)$  First of all, we exploit
$\psi(\tau;s,\psi(s;t,x))=\psi(\tau;t,x)$ in \eqref{eq 2 : Holder}
and get
\begin{equation*}
\begin{aligned}
\rho(x) =\varrho(\psi(s;t,x))~e^{-\int_s^t \nabla \cdot
V(\psi(\tau;t,x),\tau)\,d\tau}.
\end{aligned}
\end{equation*}

Hence, for $a>1$, we have
\begin{equation*}
\begin{aligned}
\rho^a(x) =\varrho^a(\psi(s;t,x))~e^{-a\int_s^t \nabla \cdot
V(\psi(\tau;t,x),\tau)\,d\tau},
\end{aligned}
\end{equation*}
and
\begin{equation}\label{eq1 : Sobolev}
\begin{aligned}
\nabla \rho^a(x)&=\nabla \varrho^a(\psi(s;t,x))~\nabla_x \psi(s;t,x)~e^{-a\int_s^t \nabla \cdot V(\psi(\tau;t,x),\tau)\,d\tau}\\
&\quad + \varrho^a(\psi(s;t,x))~e^{-a\int_s^t\nabla\cdot
V(\psi(\tau;t,x)\,d\tau)}\left(-a \int_s^t \nabla^2
V(\psi(\tau;t,x)) \nabla_x\psi(\tau;t,x)\,d\tau \right).
\end{aligned}
\end{equation}
From Lemma \ref{Lemma : estimation 1: ODE}, we note
\begin{equation}\label{eq2 : Sobolev}
\left |\nabla_x \psi(s;t,x) \right | \leq e^{\int_s^t\|\nabla V\|_{L^\infty_x} \, d\tau}.
\end{equation}
We combine \eqref{eq1 : Sobolev} and \eqref{eq2 : Sobolev} to get
\begin{equation}\label{eq3 : Sobolev}
\begin{aligned}
\left |\nabla \rho^a(x) \right |&\leq \left \{ \left |\nabla
\varrho^a(\psi(s;t,x)) \right | +\varrho^a(\psi(s;t,x))\times \left
( a\int_s^t \| \nabla^2 V\|_{L^\infty_x} \,d\tau\right ) \right
\}e^{(a+1)\int_s^t \|\nabla V\|_{L^\infty_x} \, d\tau}.
\end{aligned}
\end{equation}
Next, we note that for any nonnegative function $f$
\begin{equation}\label{eq4 : Sobolev}
%\begin{aligned}
\int_{\bbr^d} f(\psi(s;t,x)) dx = \int_{\bbr^d} f(y) J_{t,s}(y) \,dy
 \leq e^{\int_s^t\|\nabla V\|_{L^\infty_x} \,d\tau} \int_{\bbr^d} f(y)
\,dy.
%\end{aligned}
\end{equation}
Combining \eqref{eq3 : Sobolev} and \eqref{eq4 : Sobolev}, we have \eqref{eq5 : Sobolev}.
%for $q\geq1$
%\begin{equation*}
% \|\nabla \rho^a\|_{L^q (\bbr^d)} \leq e^{(a+2)\int_s^t \|\nabla V\|_{L^\infty_x}\, d\tau}\left \{ \|\nabla \varrho^a\|_{L^q (\bbr^d)} + \| \varrho^{a}\|_{L^q (\bbr^d)}
% \times \left ( a\int_s^t \| \nabla^2 V\|_{L^\infty_x} \,d\tau\right ) \right \},
%\end{equation*}
%which gives \eqref{eq5 : Sobolev}.
%(ii) We define $f$ and $g$ as above. Then, we have
% \begin{equation}
%\rho(x)-\rho(y) = f(x)(g(x)-g(y)) + (f(x)-f(y))g(y).
% \end{equation}
% We note that, if $|x-y|\leq \frac{M}{2}$ then
% \begin{equation}
% \begin{aligned}
%| \psi(s;x,t)-\psi(s;t,y)| &\leq |x-y| + 2\int_s^t \|v(\tau) \|_\infty \,d\tau\\
%& \leq M.
%\end{aligned}
% \end{equation}
% Hence, we have
%\begin{equation}
% \begin{aligned}
%|f(x)-f(y)|&= |\varrho(\psi(s;x,t)) - \varrho(\psi(s;t,y))| \\
%&<  |\psi(s;x,t) -\psi(s;t,y) |^\alpha\\
%&\leq  e^{\alpha \int_s^t \|\nabla \cdot v_\tau \|_\infty \,d\tau} |x-y|^\alpha
%\end{aligned}
% \end{equation}
% and, similar to \eqref{eq 8 : Holder}
% \begin{equation}\label{eq 2 : Remark}
%\begin{aligned}
%|g(x)-g(y)|\leq e^{\int_s^t \|\nabla \cdot v(\tau)\|_\infty \,d\tau}\int_s^t \|\nabla \cdot v(\tau)\|_\alpha \,d\tau e^{\alpha \int_s^t Lip(v_\theta) d\theta} |x-y|^\alpha
%\end{aligned}
%\end{equation}
%We combine \eqref{eq 1 : Remark}-\eqref{eq 2 : Remark}, and get
%\begin{equation}
%\begin{aligned}
%|\rho(x)-\rho(y)| %&\leq  f(x)(g(x)-g(y)) + (f(x)-f(y))g(x)\\
%&\leq e^{(\alpha+1)\int_s^t \|\nabla v(\tau) \|_\infty \,d\tau} \left( \| \varrho\|_\infty \int_s^t \|\nabla \cdot v(\tau)\|_\alpha \,d\tau
% + 1 \right) |x-y|^\alpha
% \end{aligned}
% \end{equation}
\end{proof}

\section{A priori estimates}\label{S:a priori}

Here we provide a priori estimates of a regular solution of \eqref{E:Main} given as following definition.
\begin{defn}\label{D:regular-sol}
%Let  $q\in [1, \infty)$ and $r\in [1, q]$.
%A nonnegative measurable function $\rho: Q_T \mapsto \mathbb{R}$ satisfying
%\[
%\rho \in L^{1}(Q_{T}) \cap L^{\infty}(0,T;\mathcal{C}^{\alpha}(\mathbb{R}^d)), \quad \nabla \rho^{\frac{m+r-1}{2}} \in L^{2} (Q_T), \quad  \text{and} \quad  V \in L^{1}_{\loc}(Q_{T}),
%\]
%is \underline{a regular solution} of \eqref{E:Main} with compactly supported initial data $\rho_0 \in L^1 (\mathbb{R}^d) \cap \calC^{\alpha}(\bbr^d)$, if the identity holds
%\[
%\iint_{Q_{T}} \left\{  \rho \varphi_t - \nabla\rho^m \cdot \nabla \varphi + \rho V\cdot \nabla \varphi \right\} \,dx\,dt = \int_{\mathbb{R}^d} \rho_{0} \varphi(\cdot, 0) \,dx,
%\]
%for all functions $\varphi \in \mathcal{C}_{c}^{\infty}(\mathbb{R}^d \times [0, T))$.
Let $q\in [1, \infty)$ and $V$ be a measurable vector field.
We say that a nonnegative measurable function $\rho$ is
a \textbf{regular solution} of \eqref{E:Main} with $\rho_0 \in L^{1}(\mathbb{R}^d)\cap \calC^{\alpha}(\mathbb{R}^d)$ for some $\alpha \in (0,1)$ if the followings are satisfied:
\begin{itemize}
\item[(i)] It holds that
\[
\rho \in L^{1}(Q_{T}) \cap L^{\infty}(0,T;\mathcal{C}^{\alpha}(\mathbb{R}^d)), \quad \nabla \rho^{\frac{m+q-1}{2}} \in L^{2} (Q_T), \quad  \text{and} \quad  V \in L^{1}_{\loc}(Q_{T}).
\]
\item[(ii)] For any $\varphi \in {\mathcal{C}}_c^\infty (\bbr^d \times [0,T))$, it holds that
\[
\iint_{Q_{T}} \left\{  \rho \varphi_t - \nabla\rho^m \cdot \nabla \varphi + \rho V\cdot \nabla \varphi \right\} \,dx\,dt = \int_{\mathbb{R}^d} \rho_{0} \varphi(\cdot, 0) \,dx.
\]
\end{itemize}
\end{defn}

% mass conservation property
By following similar arguments as in \cite[Proposition~9.15]{Vaz07}, for $\rho_0 \in \mathcal{P} (\mathbb{R}^d)$, we have that a regular solution $\rho$ of \eqref{E:Main} carries the mass conservation property in time; that is, $\|\rho(\cdot, t)\|_{L^{1} (\mathbb{R}^d)} = 1$ a.e.  $t\in [0, T]$.
Throughout this section, positive constants $m$, $d$, $p$, $q$, and $T$ are given, hence we omit the dependence of $c$ and $C$, unless we need to specify the dependencies clearly.

\subsection{ Supplementary estimates of speed and $p$-th moment}\label{SS:speed}

Note that \eqref{E:Main} can be written as
\begin{equation}\label{E:Main:w}
\partial_t \rho + \nabla \cdot \left( w \rho\right) = 0, \quad \text{where} \quad w:= - \frac{\nabla \rho^m}{\rho} + V .
\end{equation}
Due to Lemma~\ref{representation of AC curves}, \eqref{E:Main:w} can be seen as a curve in $\calP_{p}(\bbr^d)$ whose speed at time $t$ is limited by $\|w(\cdot, t)\|_{L^p (\rho(t))}$. In the following lemma, we deliver the estimate of  $\|w\|_{L^p(0,T;L^p(\rho(t)))}$.

\begin{lemma}\label{P:W-p}
Let $m>1$, $1 \leq q \leq m$, and $1<p\le {\lambda_q} := 1 + \frac{d(q-1) + q}{d(m-1)+q}$. Suppose that $\rho: Q_T \mapsto \bbr$ is a nonnegative measurable function satisfying
 \begin{equation*}%\label{rho-space-speed}
 \rho \in L^{\infty}(0, T; L^{q}(\mathbb{R}^d)) \quad \text{and} \quad
 \rho^{\frac{q+m-1}{2}} \in L^{2}(0, T; W^{1,2}(\mathbb{R}^d)).
 \end{equation*}
 Furthermore, assume that
\begin{equation}\label{V-p-moment}
V\in \mathcal{S}_{m,q}^{(q_1, q_2)} \ \text{ for } \
\begin{cases}
 {\lambda_q} \leq q_2 \leq \frac{{\lambda_q} (q+m-1)}{q+m-2}, & \text{ if } \ d>2,  \vspace{1 mm}\\
 {\lambda_q} \leq q_2 < \frac{{\lambda_q} (q+m-1)}{q+m-2}, & \text{ if } \ d=2.
 \end{cases}
\end{equation}
 Then, for any $\varepsilon >0$, there exists constant $c=c(\varepsilon)$ such that
\begin{equation}\label{W-p}
\iint_{Q_T} \abs{\frac{\nabla \rho^m}{\rho}}^p \rho \,dx\,dt
\leq \varepsilon \iint_{Q_T} \abs{ \nabla \rho^{\frac{q+m-1}{2}}}^2 \,dx\,dt
  +cT\left( \sup_{0\leq t \leq T}\int_{\mathbb{R}^d} \rho^q \,dx \right)^{\frac{2\theta}{d(1-\theta)}},
\end{equation}
\begin{equation}\label{V-p}
\iint_{Q_T} |V|^{p}\rho \,dx\,dt
\leq \varepsilon \iint_{Q_T} \abs{ \nabla \rho^{\frac{q+m-1}{2}}}^2 \,dx\,dt
 + c T^{\beta_1 \frac{{\lambda_q} - p}{{\lambda_q}}} \|\rho\|_{L^{q,\infty}_{x,t}}^{\beta\beta_1 \frac{q p (q_1 - {\lambda_q})}{q_1 {\lambda_q}}} \|V\|_{\mathcal{S}_{m,q}^{(q_1,q_2)}}^{p \beta_1},
\end{equation}
where $\theta=\frac{dp(m-q)}{(2-p)[(d+2)q+(m-2)d]}$, $\beta = 1 - \frac{(d-2)(q_1(1-q)+q {\lambda_q} )}{(q_1 - {\lambda_q})(d(m-1)+2q)}$, and $ \beta_1 = \frac{q_2 {\lambda_q} }{q_2 ({\lambda_q} - p) + p {\lambda_q}}$.
\end{lemma}

\begin{proof}
For simplicity, let $\lambda := \lambda_q$.
First we prove \eqref{W-p}.  If $1 \leq q < m$, then $1 < p \leq {\lambda} < 2$ and we observe that
\[ 
\abs{\frac{\nabla \rho^m}{\rho}}^p \rho = \left(\frac{2m}{q+m-1}\right)^p \rho^{\frac{2-p + p(m-q)}{2}} \abs{\nabla \rho^{\frac{q+m-1}{2}}}^p.
\]
By applying Young's inequality with $\frac{p}{2} + \frac{2-p}{2} = 1$, we have the following, for any $\varepsilon > 0$,
\begin{equation}\label{W-p-01}
\iint_{Q_T}\abs{\frac{\nabla \rho^m}{\rho}}^p \rho \,dx\,dt
\leq
\frac{\varepsilon}{2}\iint_{Q_T} \abs{\nabla \rho^{\frac{q+m-1}{2}}}^2 \,dx\,dt
+ c (\varepsilon) \iint_{Q_T}  \rho^{\frac{2-p + p(m-q)}{2-p}} \,dx\,dt.
\end{equation}

By Lemma~\ref{T:pSobolev} for $d>2$ (or Lemma~\ref{T:GN}, for $d=2$), it holds that $\rho \in L_{x,t}^{(q+m-1)+\frac{2q}{d}}$.
%it yields
%\begin{equation}\label{pSobolev2}
%\iint_{Q_T} \rho^{(q+m-1) + \frac{2q}{d}} \,dx\,dt
%\leq c \left[\sup_{0\leq t \leq T} \int_{\mathbb{R}^d \times \{t\}} \rho^q \,dx \right]^{\frac{2}{d}}
%\int_{0}^{T}\int_{\mathb%b{R}^d} \abs{\nabla\rho^{\frac{q+m-1}{2}}}^2\,dx\,dt.
%\end{equation}
Therefore, we set
\[
1\leq \frac{2-p+p(m-q)}{2-p} \leq q+m-1+\frac{2q}{d} \ \iff \ 1 < p \leq 1 + \frac{d(q-1)+q}{d(m-1) + q}.
\]
To handle the second term on the RHS of \eqref{W-p-01}, let us set
\[
\frac{2-p + p(m-q)}{2-p} = (1-\theta) + \left(q+m-1+\frac{2q}{d}\right)\theta, \quad \text{for} \quad  \theta \in (0,1).
\]
Then the H\"{o}lder inequality and Lemma~\ref{T:pSobolev} yield
\begin{equation}\label{W-p-02}
\begin{aligned}
\iint_{Q_T}  \rho^{\frac{2-p + p(m-q)}{2-p}} \,dx\,dt
&\leq \|\rho\|_{L_{x,t}^{1}}^{1-\theta} \|\rho\|_{L^{q,\infty}_{x,t}}^{\frac{2\theta q}{d}}
\left\|\nabla \rho^{\frac{q+m-1}{2}}\right\|_{L^{2}_{x,t}}^{2\theta} \\
&\leq \frac{\varepsilon}{2}\left\|\nabla \rho^{\frac{q+m-1}{2}}\right\|_{L^{2}_{x,t}}^{2} + cT\|\rho_0\|_{L^{1}(\mathbb{R}^d)} \|\rho\|_{L^{q,\infty}_{x,t}}^{\frac{2\theta q}{d(1-\theta)}},
\end{aligned}
\end{equation}
by taking the Young's inequality with $(1-\theta)+\theta =1$ and applying the mass conservation property. Hence the combination of \eqref{W-p-01} and \eqref{W-p-02} yield \eqref{W-p} for $1<q<m$.

If $q = m$, then we observe that ${\lambda}_q=2$ and
\[
\abs{\frac{\nabla \rho^m}{\rho}}^2 \rho = \left(\frac{2m}{2m-1}\right)^2 \abs{\nabla \rho^{\frac{2m-1}{2}}}^2.
\]
For any $1 < p <2$, Young's inequality with $\frac{p}{2} + \frac{2-p}{2}=1$ gives
\[
\abs{\frac{\nabla \rho^m}{\rho}}^p \rho \leq \varepsilon \abs{\frac{\nabla \rho^m}{\rho}}^2 \rho + c(\varepsilon) \rho,
\]
that completes the proof of \eqref{W-p} by taking integration on $Q_T$.

Now we prove \eqref{V-p}. First, by taking the H\"{o}lder inequality with $\frac{{\lambda}}{q_1}+\frac{q_1 - {\lambda}}{q_1} =1$ for $q_1 > {\lambda}$, we have
\begin{equation}\label{V-p_q}
 \int_{\mathbb{R}^d} |V|^{{\lambda}}\rho \,dx \leq \|V\|_{L^{q_1}_{x}}^{{\lambda}}\|\rho\|_{L^{r_1}_{x}}
 \quad \text{where} \quad r_1 := \frac{q_1}{q_1 - {\lambda}}.
\end{equation}
Let us set
\begin{equation}\label{speed00}
q \leq r_1 \leq \frac{d(q+m-1)}{d-2} \quad \text{ and} \quad \theta:= \frac{q}{r_1} \left(1-\frac{(r_1 - q)(d-2)}{2q + d(m-1)}\right).
\end{equation}
Then by  \eqref{pme-interpolate-20} and \eqref{pme-interpolate-10}  in Lemma~\ref{P:L_r1r2}, we yield, from \eqref{V-p_q}, that
\begin{equation}\label{speed00.5}
\begin{aligned}
 \int_{\mathbb{R}^d} |V|^{{\lambda}}\rho \,dx
 &\leq \|V\|_{L^{q_1}_{x}}^{{\lambda}}\|\rho\|_{L^{q}_{x}}^{\theta} \left\|\nabla \rho^{\frac{q+m-1}{2}}\right\|_{L^{2}_{x}}^{\frac{2(1-\theta)}{q+m-1}}\\
 &\leq \varepsilon \left\|\nabla \rho^{\frac{q+m-1}{2}}\right\|_{L^{2}_{x}}^{2}
 + c(\varepsilon) \|\rho\|_{L^{q}_{x}}^{\frac{\theta(q+m-1)}{q+m-1-(1-\theta)}} \|V\|_{L^{q_1}_{x}}^{\frac{{\lambda} (q+m-1)}{q+m-1 - (1-\theta)}},
\end{aligned}
\end{equation}
by taking Young's inequality with $\frac{1-\theta}{q+m-1} + \frac{q+m-1-(1-\theta)}{q+m-1} = 1$. Then now let us set
\begin{equation}\label{speed01}
q_2 :=  \frac{{\lambda} (q+m-1)}{q+m-1 - (1-\theta)} = \frac{q_1 {\lambda} (2q + d(m-1))}{q_1 (2q + d(m+q-2)) - d q {\lambda}}.
\end{equation}
Then the direct computation gives that the pairs $(q_1, q_2)$ satisfy $\mathcal{S}_{m,q}^{(q_1,q_2)}$. Moreover, from \eqref{speed00} and \eqref{speed01}, we deduce that
\begin{equation*}
\frac{d {\lambda} (q+m-1)}{2+d(q+m-2)} \leq q_1 \leq \frac{q {\lambda}}{q-1}     \quad \Longrightarrow \quad
{\lambda} \leq q_2 \leq \frac{{\lambda} (q+m-1)}{q+m-2}.
\end{equation*}
Moreover, by taking integration in terms of $t$ to \eqref{speed00.5}, it gives
\begin{equation}\label{speed02}
\iint_{Q_T} |V|^{{\lambda}}\rho \,dx\,dt \leq  \varepsilon \left\|\nabla \rho^{\frac{q+m-1}{2}}\right\|_{L^{2}_{x,t}}^{2}
 + c(\varepsilon) \|\rho\|_{L^{q,\infty}_{x,t}}^{\frac{\theta q_2}{{\lambda}}} \|V\|_{\mathcal{S}_{m,q}^{(q_1,q_2)}}.
\end{equation}

For any $1 < p < {\lambda}$, we take the H\"{o}lder inequality with $\frac{p}{{\lambda}} + \frac{{\lambda} - p}{{\lambda}} = 1$ and apply the mass conservation property
\begin{equation}\label{speed03}
%\begin{aligned}
\iint_{Q_T} |V|^{p}\rho \,dx\,dt
\leq
\left(\iint_{Q_T} |V|^{{\lambda}}\rho \,dx\,dt\right)^{\frac{p}{{\lambda}}}
\left(\iint_{Q_T} \rho \,dx\,dt\right)^{\frac{{\lambda} - p}{{\lambda}}}
\leq T^{\frac{{\lambda} - p}{{\lambda}}} \left(\iint_{Q_T} |V|^{{\lambda}}\rho \,dx\,dt\right)^{\frac{p}{{\lambda}}}.
%\end{aligned}
\end{equation}
Then by taking Young's inequality to the combined inequality of \eqref{speed03} and \eqref{speed02}, we obtain \eqref{V-p}.
\end{proof}

Recalling $p$-th power $\langle x \rangle^{p} = \left(1+ |x|^2\right)^{\frac p2}$ for $1< p\leq 2$, we note the following properties:
\begin{equation}\label{p-moment-grad}
  \nabla \langle x \rangle^{p} = p \, x \left(1+ |x|^2\right)^{\frac{p-2}{2}}
  \quad \text{so} \quad
  \abs{\nabla \langle x \rangle^{p}} <  p \langle x \rangle^{p-1},
\end{equation}
and
\begin{equation}\label{p-moment-2nd-grad}
\Delta \langle x \rangle^{p}
= p \frac{1 + (p-1)|x|^2}{\left(1+ |x|^2\right)^{2 - p/2}}
< p \frac{1 + |x|^2}{\left(1+ |x|^2\right)^{2 - p/2}}= p \left(1+ |x|^2\right)^{\frac{p-2}{2}} \leq p.
\end{equation}
In the following lemma, we compute $p$-th moment estimate assuming that $V$ satisfies \eqref{V-p-moment}.

\begin{lemma}\label{L:p-moment-estimate} ($p$-th moment estimate)
Let $m>1$, $1 \leq q \leq m$, and $1<p\le {\lambda_q} := 1 + \frac{d(q-1) + q}{d(m-1)+q}$. Suppose that  $\rho : [0,T] \mapsto \mathcal{P}(\bbr^d)$ is narrowly continuous, and is a regular solution with sufficiently fast decay at infinity in spatial variables of \eqref{E:Main}. If $V \in \calC^{\infty} (\bar{Q}_{T})$ satisfies \eqref{V-p-moment}, then it holds
\begin{equation}\label{p-moment}
\begin{aligned}
\sup_{0\leq t \leq T}&\int_{\mathbb{R}^d} \, \rho (\cdot, t) \langle x \rangle^{p} \,dx
\leq \int_{\mathbb{R}^d} \, \rho_0 \langle x \rangle^{p} \,dx+\iint_{Q_T} \rho\langle x \rangle^{p} \,dx\,dt \\
& + \varepsilon \iint_{Q_T} \abs{\nabla \rho^{\frac{q+m-1}{2}}}^2 \,dx\,dt
+ c_1 T
+ c_2 T^{\beta_1 \frac{{\lambda_q} - p}{{\lambda_q}}} \|\rho\|_{L^{q,\infty}_{x,t}}^{\beta \beta_1 \frac{q p(q_1 - {\lambda_q})}{q_1 {\lambda_q}}} \|V\|_{\mathcal{S}_{m,q}}^{p \beta_1},
\end{aligned}
\end{equation}
where $c_1 = c_1 (\varepsilon, p)$, $c_2 = c_2 (\varepsilon, m, d, p, q)$, $\beta = 1 - \frac{(d-2)(q_1(1-q)+q {\lambda_q} )}{(q_1 - {\lambda_q})(d(m-1)+2q)}$, and $\beta_1 = \frac{q_2 {\lambda_q}}{q_2 ({\lambda_q} - p) + p {\lambda_q}}$.
\end{lemma}

\begin{proof} For simplicity, let $\lambda := \lambda_q$.
By testing $\langle x \rangle^{p}$ for $p > 1$ to \eqref{E:Main} and taking the integration by parts, we obtain
\begin{equation*}
\frac{d}{dt} \int_{\mathbb{R}^d} \rho \langle x\rangle^p \, dx
= \int_{\mathbb{R}^d} \left[\Delta \rho^{m} - \nabla \left(\rho V\right)\right] \langle x\rangle^p \,dx
= \int_{\mathbb{R}^d} \left[\rho^{m}\Delta \langle x\rangle^p + \rho V \cdot \nabla \langle x\rangle^p\right] \,dx.
\end{equation*}
By taking the integration in $ t \in [0, T]$, we have
\begin{equation}\label{J-0}
%\begin{aligned}
\sup_{0\leq t \leq T}\int_{\mathbb{R}^d}\rho \langle x \rangle^{p} \,dx
\leq \int_{\mathbb{R}^d} \rho_0 \langle x\rangle^p\,dx
+ \underbrace{\iint_{Q_T} \left |\rho^{m}\Delta \langle x\rangle^p \right |\,dx\,dt}_{:=J_1}
+\underbrace{\iint_{Q_T}\left |\rho V \cdot \nabla \langle x\rangle^p \right | \,dx\,dt }_{:= J_2}.
% &=  \int_{\mathbb{R}^d} \rho_0 \langle x \rangle^{p} \,dx + J_1 + J_2.
% \end{aligned}
\end{equation}

Let us set
\[
m = (1-\theta) + \frac{d(q+m-1)}{d-2}\theta \quad \iff \quad \theta = \frac{(m-1)(d-2)}{d(q+m-1)-(d-2)} \in (0,1).
\]
Then by \eqref{p-moment-2nd-grad}, H\"{o}lder inequality, and Lemma~\ref{T:Sobolev}, we have
\begin{equation}\label{J-1}
\begin{aligned}
J_1 &\leq p \iint_{Q_T} \rho^m \,dx\,dt
\leq p\int_{0}^{T} \left(\int_{\mathbb{R}^d} \rho \,dx \right)^{1-\theta} \left(\int_{\mathbb{R}^d} \rho^{\frac{d(q+m-1)}{d-2}}\,dx\right)^{\theta} \,dt \\
&\leq p\int_{0}^{T} \left(\int_{\mathbb{R}^d} \rho \,dx \right)^{1-\theta} \left(\int_{\mathbb{R}^d} \abs{ \nabla \rho^{\frac{q+m-1}{2}}}^2\,dx\right)^{\frac{\theta d}{d-2}} \,dt \\
&\leq \frac{\varepsilon}{2} \iint_{Q_T} \abs{ \nabla \rho^{\frac{q+m-1}{2}}}^2 \,dx\,dt + c(p, \varepsilon) T \|\rho_0\|_{L^{1}(\mathbb{R}^d)}^{\frac{d(q-1)+2m}{d(q-1)+2}},
\end{aligned}
\end{equation}
by mass conservation property and Young's inequality for $\frac{\theta d}{d-2} + \frac{d(1-\theta)-2}{d-2}=1$.

Now let us handle $J_2$. First, consider the case $p= {\lambda}$. Then by \eqref{p-moment-grad}, we take H\"{o}lder inequality to get
\begin{equation*}%\label{p-power01}
    \iint_{Q_T} \rho |V| \langle x\rangle^{{\lambda} -1} \,dx\,dt
    \leq \iint_{Q_T} \rho \langle x \rangle^{{\lambda}}\,dx\,dt+
     \iint_{Q_T} \rho |V|^{{\lambda}} \,dx\,dt.
\end{equation*}
Then we apply \eqref{speed02} to handle the last term on RHS which gives, for $\beta = 1-\frac{(r_1-q)(d-2)}{d(m-1)+2q}$,
\begin{equation*}%\label{p-power03}
%\begin{aligned}
\iint_{Q_T} \rho |V| \langle x\rangle^{{\lambda} -1} \,dx\,dt
\leq  \iint_{Q_T} \rho \langle x \rangle^{{\lambda}}\,dx\,dt  + \|V\|_{\mathcal{S}_{m,q}^{(q_1,q_2)}}^{{\lambda}}  \|\rho\|_{L^{q,\infty}_{x,t}}^{\frac{q\beta}{r_1}} \,
\left\|\nabla \rho^{\frac{q+m-1}{2}} \right\|_{L^{2}_{x,t}}^{\frac{2}{r_2}}.
%\end{aligned}
\end{equation*}

On the other hand, we consider the case $1< p < {\lambda}$ in $J_2$ from \eqref{J-0}. Then we apply \eqref{p-moment-grad} and Young's inequality with $\frac{p-1}{p} + \frac{1}{p}=1$ to obtain
\begin{equation*}
\begin{aligned}
J_2 &\leq p \iint_{Q_T} \rho |V| \langle x\rangle^{p -1} \,dx\,dt
\leq  \iint_{Q_T} \rho \langle x \rangle^{p}\,dx\,dt+
    c(p) \iint_{Q_T} \rho |V|^{p} \,dx\,dt \\
 &\leq \iint_{Q_T} \rho \langle x \rangle^{p}\,dx\,dt+
    c(p) \left(\iint_{Q_T} \rho |V|^{{\lambda}} \,dx\,dt\right)^{\frac{p}{{\lambda}}}
    \left(\iint_{Q_T} \rho \,dx\,dt\right)^{\frac{{\lambda} - p}{{\lambda}}}\\
    &\leq \iint_{Q_T} \rho \langle x \rangle^{p}\,dx\,dt
    + c(p)T^{\frac{{\lambda} -p}{{\lambda}}}\|\rho_0\|_{L^{1}(\mathbb{R}^d)}^{\frac{{\lambda} - p}{{\lambda}}} \left(\iint_{Q_T} \rho |V|^{{\lambda}} \,dx\,dt\right)^{\frac{p}{{\lambda}}},
\end{aligned}
\end{equation*}
by taking another H\"{o}lder inequality with $\frac{p}{{\lambda}} + \frac{{\lambda} - p}{{\lambda}} =1$, and applying the mass conservation property.
By \eqref{speed02} and Young's inequality with $\frac{p}{{\lambda} r_2} + \frac{{\lambda} r_2 - p}{{\lambda} r_2}=1$ , we have the following, for any $\varepsilon >0$,
\begin{equation}\label{J-2}
J_2 \leq \iint_{Q_T} \rho \langle x \rangle^{p}\,dx\,dt
          + \frac{\varepsilon}{2} \left\|\nabla \rho^{\frac{q+m-1}{2}} \right\|_{L^{2}_{x,t}}^{2}
      +c(\varepsilon, p) T^{\frac{r_2({\lambda} -p)}{{\lambda} r_2 - p}} \|\rho\|_{L^{q,\infty}_{x,t}}^{\frac{ q p \beta r_2}{r_1({\lambda} r_2 - p)}}
    \|V\|_{\mathcal{S}_{m,q}^{(q_1,q_2)}}^{\frac{p{\lambda} r_2}{{\lambda} r_2 - p}}.
\end{equation}
The combination of the above with \eqref{J-0} and \eqref{J-1} completes the proof.
\end{proof}

\begin{remark}\label{Remark : moment estimate}
A similar estimate as in above lemma holds for $L^q$-weak solutions that are narrowly continuous in time variable. For the sketch of proof, let $\eta : \bbr^d \rightarrow [0, 1]$ be a smooth cut-off function satisfying
\begin{equation*} %\label{eta}
\eta(x) = \
\begin{cases}
1, & \text{ if } \ |x| \leq 1, \\
 0, & \text{ if } \ |x| \geq 2.
 \end{cases}
\end{equation*}
Let $\rho$ be a $L^q$-weak solution of \eqref{E:Main} such that $\rho:[0, T]\mapsto \calP(\bbr^d)$ is narrowly continuous. For $\delta>0$, we define $\eta_\delta(x):=\eta(\delta x)$. As in Lemma \ref{L:p-moment-estimate}, by testing $\langle x \rangle^{p} \eta_\delta$  to \eqref{E:Main}, we obtain
\begin{equation}\label{J-00}
\begin{aligned}
\sup_{0\leq t \leq T}\int_{\mathbb{R}^d}\rho \left( \langle x \rangle^{p}\eta_\delta(x)\right) \,dx
&\leq \int_{\mathbb{R}^d} \rho_0 \left( \langle x \rangle^{p}\eta_\delta(x)\right)\,dx \\
& \quad + \underbrace{\iint_{Q_T} \left |\rho^{m}\Delta \left( \langle x \rangle^{p}\eta_\delta(x)\right) \right |\,dx\,dt}_{:=J_3}
+\underbrace{\iint_{Q_T} \left |\rho V \cdot \nabla \left( \langle x \rangle^{p}\eta_\delta(x)\right)\right | \,dx\,dt .}_{:= J_4}
% &=  \int_{\mathbb{R}^d} \rho_0 \langle x \rangle^{p} \,dx + J_1 + J_2.
\end{aligned}
\end{equation}
We note that $\| \nabla \langle x \rangle^{p} \nabla\eta_\delta \|_{L_x^\infty}, ~\| \langle x \rangle^{p} \Delta \eta_\delta \|_{L_x^\infty} \lesssim  \delta^{2-p}$. Hence
\begin{equation}\label{J-3}
\begin{aligned}
J_3 &\lesssim \iint_{Q_T} \left| \rho^{m}\Delta \langle x \rangle^{p}\right |\eta_\delta(x) \,dx\,dt +  \delta^{2-p}\iint_{\{\frac{1}{\delta} \leq |x| \leq \frac{2}{\delta} \}\times [0,T]} \rho^{m} \,dx\,dt\\
& \rightarrow \iint_{Q_T}  \left |\rho^{m}\Delta \langle x \rangle^{p} \right | \,dx\,dt \quad \mbox{as} ~~ \delta \rightarrow 0.
% p \iint_{Q_T} \rho^m \,dx\,dt
%\leq p\int_{0}^{T} \left(\int_{\mathbb{R}^d} \rho \,dx \right)^{1-\theta} \left(\int_{\mathbb{R}^d} \rho^{\frac{d(q+m-1)}{d-2}}\,dx\right)^{\theta} \,dt \\
%&\leq p\int_{0}^{T} \left(\int_{\mathbb{R}^d} \rho \,dx \right)^{1-\theta} \left(\int_{\mathbb{R}^d} \abs{ \nabla \rho^{\frac{q+m-1}{2}}}^2\,dx\right)^{\frac{\theta d}{d-2}} \,dt \\
%&\leq \frac{\varepsilon}{2} \iint_{Q_T} \abs{ \nabla \rho^{\frac{q+m-1}{2}}}^2 \,dx\,dt + c(p, \varepsilon) T \|\rho_0\|_{L^{1}(\mathbb{R}^d)}^{\frac{d(q-1)+2m}{d(q-1)+2}}
\end{aligned}
\end{equation}
Similarly, we note $|  \langle x \rangle^{p} \nabla\eta_\delta | \lesssim  |\langle x \rangle^{p-1}|$ and have
\begin{equation}\label{J-4}
\begin{aligned}
J_4 &\lesssim \iint_{Q_T}\left| \rho V \cdot \nabla  \langle x
\rangle^{p} \right | \eta_\delta(x) \,dx\,dt
%+ C\iint_{\{\frac{1}{\delta} \leq |x| \leq \frac{2}{\delta} \}\times [0,T]}\rho |V|  \langle x \rangle^{p-1} \,dx\,dt\\
+ \iint_{Q_T}\rho |V|  \langle x \rangle^{p-1} \,dx\,dt\\
&\rightarrow \iint_{Q_T}\left| \rho V \cdot \nabla  \langle x
\rangle^{p} \right |  \,dx\,dt  + \iint_{Q_T}\rho |V|  \langle x
\rangle^{p-1} \,dx\,dt, \quad \mbox{as} ~~ \delta \rightarrow 0.
\end{aligned}
\end{equation}
Finally, we take $\delta \rightarrow 0$ in \eqref{J-00}. Then, we combine \eqref{J-1}--\eqref{J-2}, \eqref{J-3} and \eqref{J-4} to have
\begin{equation}\label{p-moment : q-weak}
\begin{aligned}
\sup_{0\leq t \leq T}&\int_{\mathbb{R}^d} \, \rho (\cdot, t) \langle
x \rangle^{p} \,dx
\leq \int_{\mathbb{R}^d} \, \rho_0 \langle x \rangle^{p} \,dx+ 2\iint_{Q_T} \rho\langle x \rangle^{p} \,dx\,dt \\
& + \varepsilon \iint_{Q_T} \abs{\nabla \rho^{\frac{q+m-1}{2}}}^2
\,dx\,dt + c_1 T + c_2 T^{\beta_1 \frac{{\lambda_q} - p}{{\lambda_q}}}
\|\rho\|_{L^{q,\infty}_{x,t}}^{\beta \beta_1 \frac{q p(q_1 -
{\lambda_q})}{q_1 {\lambda_q}}} \|V\|_{\mathcal{S}_{m,q}}^{p \beta_1},
\end{aligned}
\end{equation}
where $c_1, \, c_2, \, \beta$ and $\beta_1$ are same as in Lemma \ref{L:p-moment-estimate}.
%(may refer Appendix for detailed proof)
\end{remark}

\subsection{Energy estimates}
This subsection is composed with energy estimates for two different cases: first when $\int_{\mathbb{R}^d}\rho_0 \log \rho_0 \, dx < \infty$ and second when  $\|\rho_0\|_{L^q (\bbr^d)} < \infty$ for $q>1$.

\subsubsection{Energy estimates for case: $\int_{\mathbb{R}^d}\rho_0 \log \rho_0 \, dx < \infty$. }\label{SS:energy-log}

Under additional assumption $\rho_0 \log \rho_0 \in L^{1} (\mathbb{R}^d)$, not merely $\rho_0 \in L^{1} (\mathbb{R}^d)$, we obtain a priori estimate categorized three cases depending on the range of $m$ and the conditions on $V$.

\begin{proposition}\label{P:log-energy}
Let $m>1$, $1< p \leq {\lambda_1}:= 1+ \frac{1}{d(m-1) + 1}$. Assume that $\rho_{0}\in \mathcal{P}_p (\mathbb{R}^d) \cap \calC^{\alpha}(\mathbb{R}^d)$ holds
\begin{equation}\label{log-initial}
\int_{\mathbb{R}^d} \left(\rho_0\log \rho_0 + \rho_0 \langle x \rangle^p  \right) \,dx < \infty.
\end{equation}
Suppose that $\rho$ is a regular solution of \eqref{E:Main} with sufficiently fast decay at infinity in spatial variables.
\begin{itemize}
\item[(i)] Let $1< m \leq 2$. Assume that
  \begin{equation}\label{V-log-energy}
V\in\mathcal{S}_{m,1}^{(q_1, q_2)} \cap \calC^{\infty} (\bar{Q}_{T}) \quad \text{for}\quad
\begin{cases}
2 \leq q_2 \leq \frac{m}{m-1},  & \text{ if } d> 2, \vspace{1 mm}\\
2\leq q_2 < \frac{m}{m-1}, & \text{ if } d= 2.
 \end{cases}
\end{equation}
Then the following estimate holds
\begin{equation}\label{log-energy}
\sup_{0\leq t \leq T} \int_{\mathbb{R}^d\times \{t\}} \rho \left( \abs{\log \rho} + \langle x \rangle^p \right) \,dx
 + \frac{2}{m}\iint_{Q_T} \left| \nabla \rho^{\frac{m}{2}} \right|^2 \,dx\,dt \leq C,
\end{equation}
where $C=C (\|V\|_{\mathcal{S}_{m,1}^{(q_1,q_2)}}, \, \int_{\mathbb{R}^d} \left(\rho_0\log \rho_0 + \rho_0 \langle x \rangle^p  \right) \,dx )$.

\item[(ii)]  Assume that
\begin{equation}\label{V-tilde-log-energy}
V \in \tilde{\mathcal{S}}_{m,1}^{(\tilde{q}_1, \tilde{q}_2)} \cap \calC^{\infty} (\bar{Q}_{T}) \quad \text{for}\quad
\begin{cases}
\frac{2+d(m-1)}{1+d(m-1)} < \tilde{q}_2 \leq \frac{m}{m-1}, & \text{ if } d>2, \vspace{1 mm}\\
\frac{2+d(m-1)}{1+d(m-1)} < \tilde{q}_2 < \frac{m}{m-1}, & \text{ if } d=2.
\end{cases}
\end{equation}
Then the estimate \eqref{log-energy} holds where
$C= C (\|V\|_{\tilde{\mathcal{S}}_{m,1}^{(\tilde{q}_1,\tilde{q}_2)}}, \, \int_{\mathbb{R}^d} \left(\rho_0\log \rho_0 + \rho_0 \langle x \rangle^p  \right) \,dx ).$

\item[(iii)] Let $\nabla\cdot V=0$. Assume that
\begin{equation}\label{V-log-divfree}
V\in \mathcal{S}_{m,1}^{(q_1, q_2)}\cap \calC^{\infty} (\bar{Q}_{T}) \quad  \text{for} \quad
\begin{cases}
{\lambda_1} \leq q_2 \leq \frac{{\lambda_1} m}{m-1}, & \text{ if }  d>2, \vspace{1 mm}\\
{\lambda_1} \leq q_2 < \frac{{\lambda_1} m}{m-1}, & \text{ if }  d=2.
\end{cases}
\end{equation}
Then the estimate \eqref{log-energy} holds where
$C= C (\|V\|_{\mathcal{S}_{m,1}^{(q_1,q_2)}}, \, \int_{\mathbb{R}^d} \left(\rho_0\log \rho_0 + \rho_0 \langle x \rangle^p  \right) \,dx ).$

\end{itemize}
\end{proposition}

\begin{proof}
First, we recall the $p$-entropy inequality (for example, see \cite{DL89, CKL14})
\begin{equation}\label{entropy-log}
\int_{\mathbb{R}^d} \rho \abs{\log \rho} \,dx
\leq \int_{\mathbb{R}^d} \rho \log \rho \,dx + \int_{\mathbb{R}^d} \rho \langle x\rangle^p \,dx + c(d),
\end{equation}
which gives
\begin{equation}\label{entropy}
  0 \leq \int_{\mathbb{R}^d} \rho \log \rho \,dx - \int_{\mathbb{R}^d} \rho \left(\log\rho\right)_{-} \,dx \leq \int_{\mathbb{R}^d} \rho \log \rho \,dx + d\int_{\mathbb{R}^d} \rho \langle x\rangle^p \,dx + c(d).
\end{equation}

By testing $\log \rho$ to \eqref{E:Main}, we obtain
\begin{equation}\label{Elog-0}
\frac{d}{dt} \int_{\mathbb{R}^d} \rho \log \rho \,dx +  \frac{4}{m}\int_{\mathbb{R}^d} \abs{ \nabla \rho^{\frac{m}{2}}}^2\,dx =  \int_{\mathbb{R}^d} V \cdot \nabla \rho \,dx =: \mathcal{J}.
\end{equation}

$(i)$ For $1<m < 2$, we write RHS of \eqref{Elog-0} as
\begin{equation*} %\label{log-V-1}
    \mathcal{J} = \int_{\mathbb{R}^d} V \cdot \nabla \rho \,dx \leq  \frac{2}{m}\int_{\mathbb{R}^d} \abs{V} \rho^{1-\frac{m}{2}} \abs{ \nabla \rho^{\frac{m}{2}} } \,dx.
\end{equation*}
By H\"{o}lder inequality with $\frac{1}{q_1} + \frac{1}{2} + \frac{q_1 -2}{2q_1} = 1$ for $q_1 > 2$, we have
\begin{equation}\label{Elog-1}
     \mathcal{J}
  \leq  c \|V\|_{L^{q_1}_{x}} \left\|\nabla \rho^{\frac{m}{2}}\right\|_{L^{2}_{x}} \left\|\rho^{\frac{2-m}{2}} \right\|_{L^{\frac{2q_1}{q_1-2}}_{x}}
  \le c \|V\|_{L^{q_1}_{x}} \left\|\nabla \rho^{\frac{m}{2}}\right\|_{L^{2}_{x}} \left\|\rho \right\|^{\frac{2-m}{2}}_{L^{\frac{(2-m)q_1}{q_1-2}}_{x}}.
\end{equation}
When $m=2$, the above inequality corresponds to
$ \mathcal{J} \leq c\|V\|_{L_{x}^{2}} \left\|\nabla \rho^{\frac{m}{2}}\right\|_{L^{2}_{x}}$.

For convinence, we denote
\begin{equation*}
r_1 := \frac{(2-m)q_1}{q_1 -2}, \quad \text{and}\quad \theta := \frac{1}{r_1}\left( 1 - \frac{(r_1 - 1)(d-2)}{2+d(m-1)}\right).
\end{equation*}
Then by \eqref{q-r1r2} in Lemma~\ref{P:L_r1r2}, we find that
\begin{equation}\label{Elog-1.5}
\begin{cases}
 2<\frac{md}{1+(m-1)(d-1)} \leq q_1 \leq \frac{2}{m-1}, & \text{ if } d>2, \\
  2<  q_1 \leq \frac{2}{m-1}, & \text{ if } d=2.
\end{cases}
\end{equation}
Moreover, the combination of  \eqref{pme-interpolate-20} and \eqref{pme-interpolate-10}  in Lemma~\ref{P:L_r1r2} and \eqref{Elog-1} implies that
\begin{equation}\label{Elog-3}
\begin{aligned}
\mathcal{J} &\leq
c \|V\|_{L^{q_1}_{x}} \left\|\nabla \rho^{\frac{m}{2}}\right\|_{L^{2}_{x}} \left\|\rho \right\|^{\frac{2-m}{2}}_{L^{\frac{(2-m)q_1}{q_1-2}}_{x}}
\le c \|V\|_{L^{q_1}_{x}} \norm{\rho}^{\frac{\theta(2-m)}{2}}_{L^1_x}\norm{\nabla \rho^{\frac{m}{2}}}^{1+\frac{(1-\theta)(2-m)}{m}}_{L^{2}_x}\\
&\le \varepsilon \left\|\nabla \rho^{\frac{m}{2}}\right\|^2_{L^{2}_{x}}
+ c(\varepsilon)\|V\|^{\frac{2m}{m-(1-\theta)(2-m)}}_{L^{q_1}_{x}},
\end{aligned}
\end{equation}
by taking Young's inequality. Since $V$ is in $\mathcal{S}_{m,1}^{(q_1, q_2)}$, it follows that
\begin{equation*}
    \frac{2m}{m-(1-\theta)(2-m)} =\frac{q_1(2+d(m-1))}{q_1(1+d(m-1))-d}= q_2.
\end{equation*}
Moreover, \eqref{Elog-1.5} corresponds to
 \eqref{V-log-energy}.

%Here we note that
%\begin{equation}\label{Elog-2}
%     \frac{md}{1+(m-1)(d-1)} \leq q_1 \leq \frac{2}{m-1} \quad \Longrightarrow \quad 2 \leq q_2 \leq \frac{m}{m-1}.
%\end{equation}
%When $d=2$, we obtain \eqref{V-log-energy} by applying Lemma~\ref{P:L_r1r2} in case $d=2$.
%If $d=2$, then from \eqref{Elog-1.5}, we have following pairs
%\begin{equation}
%     \frac{md}{1+(m-1)(d-1)} < q_1 \leq \frac{2}{m-1} \quad \Longrightarrow \quad 2 \leq q_2 < \frac{m}{m-1}.
%\end{equation}
%When $m=2$, then H\"{o}lder inequality yields
%\begin{equation}\label{Elog-3}
%\mathcal{J} \leq c\|V\|_{L_{x}^{2}} \left\|\nabla \rho^{\frac{m}{2}}\right\|_{L^{2}_{x}}.
%\end{equation}
%So $q_1=q_2 =2$, and therefore we extend \eqref{Elog-2} for $1 < m \leq 2$ to obtain
%\begin{equation}
%\mathcal{J} \leq \varepsilon \left\|\nabla \rho^{\frac{m}{2}}\right\|^2_{L^{2}_{x}}
%+ c(\varepsilon)\|V\|^{q_2}_{L^{q_1}_{x}}
%\end{equation}

We combine \eqref{Elog-0} and \eqref{Elog-3}, and then take integration in terms of $t$ with $\varepsilon = \frac{2}{m}$ to have
\begin{equation}\label{Elog-4}
\sup_{0\leq t \leq T} \int_{\mathbb{R}^d} \rho \log \rho \,dx +  \frac{2}{m}\iint_{Q_T} \abs{ \nabla \rho^{\frac{m}{2}}}^2\,dx\,dt
\leq  \int_{\mathbb{R}^d} \rho_0 \log \rho_0 \,dx + c\|V\|^{q_2}_{\mathcal{S}_{m,1}^{(q_1, q_2)}}.
\end{equation}
To \eqref{Elog-4}, we apply \eqref{entropy} and add $p$-th moment estimate \eqref{p-moment} with $\varepsilon=\frac{1}{m}$ to obtain
\begin{equation}\label{Elog-5}
\begin{aligned}
&\sup_{0\leq t \leq T} \int_{\mathbb{R}^d} \left(\rho \abs{\log \rho} + \rho \langle x\rangle^p \right)\,dx +  \frac{1}{m}\iint_{Q_T} \abs{ \nabla \rho^{\frac{m}{2}}}^2\,dx\,dt  \\
&\leq  \int_{\mathbb{R}^d} \left(\rho_0 \log \rho_0 + \rho_0 \langle x\rangle^p\right) \,dx
+ \iint_{Q_T} \rho \langle x\rangle^p \,dx\,dt
+ C (\|V\|_{\mathcal{S}_{m,1}^{(q_1, q_2)}}),
\end{aligned}
\end{equation}
where $(q_1, q_2)$ given in \eqref{V-log-energy} that satisfies \eqref{V-p-moment} too.
To handle the integral of $p$-th moment on RHS of \eqref{Elog-5}, we apply the Gr\"{o}nwall's inequality that leads to
\begin{equation*}%\label{Elog-6}
\begin{aligned}
&\sup_{0\leq t \leq T} \int_{\mathbb{R}^d} \left(\rho \abs{\log \rho} + \rho \langle x\rangle^p \right)\,dx +  \frac{1}{m}\iint_{Q_T} \abs{ \nabla \rho^{\frac{m}{2}}}^2\,dx\,dt  \\
&\leq  c\left[\int_{\mathbb{R}^d} \left(\rho_0 \log \rho_0 + \rho_0 \langle x\rangle^p\right) \,dx
+ C (\|V\|_{\mathcal{S}_{m,1}^{(q_1, q_2)}} )\right].
\end{aligned}
\end{equation*}
Therefore, we complete to obtain \eqref{log-energy} under \eqref{V-log-energy}.
\smallskip

$(ii)$ First we take the integration by parts and  apply H\"{o}lder inequality for $\frac{1}{\tilde{q}_1} + \frac{\tilde{q}_1 -1}{\tilde{q}_1} = 1$ to have
\begin{equation}\label{log-V-2}
    \mathcal{J}  =  \int_{\mathbb{R}^d} V \cdot \nabla \rho \,dx =  \int_{\mathbb{R}^d} \rho \nabla \cdot V  \,dx \leq \|\nabla V\|_{L_{x}^{\tilde{q}_1}} \|\rho\|_{L_{x}^{\frac{\tilde{q}_1}{\tilde{q}_1 -1}}}.
\end{equation}
For convenience, we denote
\begin{equation*}
\tilde{r}_1 := \frac{\tilde{q}_1}{\tilde{q}_1 - 1} \quad \text{and} \quad
\theta := \frac{1}{\tilde{r}_1}  \left( 1 - \frac{(\tilde{r}_1 - 1)(d-2)}{2+d(m-1)} \right).
\end{equation*}
Then Lemma~\ref{P:L_r1r2} with $\tilde{r}_1$ in the place for $r_1$ implies that
\begin{equation}\label{Elog-5.5}
 \begin{cases}
  1 < \frac{md}{2+d(m-1)}\leq \tilde{q}_1 \leq \infty, & \text{ if } d>2, \\
  1< \tilde{q}_1 < \infty, & \text{ if } d=2.
  \end{cases}
\end{equation}
Then the combination of  \eqref{pme-interpolate-20} and \eqref{pme-interpolate-10}  in Lemma~\ref{P:L_r1r2} and \eqref{log-V-2} gives
\begin{equation}\label{Elog-6}
\mathcal{J} \leq c\|\nabla V\|_{L_{x}^{\tilde{q}_1}} \left\|\nabla \rho^{\frac{m}{2}} \right\|_{L_{x}^{2}}^{\frac{m}{m-(1-\theta)}}
\leq \varepsilon    \left\|\nabla \rho^{\frac{m}{2}} \right\|_{L_{x}^{2}}^{2} + c(\varepsilon) \|\nabla V\|_{L_{x}^{\tilde{q}_1}}^{\frac{m}{m-(1-\theta)}},
\end{equation}
by taking Young's inequality. Because of $V \in \tilde{\mathcal{S}}_{m,1}^{(\tilde{q}_1, \tilde{q}_2)}$, it gives
\begin{equation*}
\frac{m}{m-(1-\theta)}=  \frac{\tilde{q}_1 (2+d(m-1))}{\tilde{q}_1 (2+d(m-1)) - 2d} = \tilde{q}_2 .
\end{equation*}
Moreover, \eqref{Elog-5.5} corresponds to \eqref{V-tilde-log-energy}.
% Here we note that
%\begin{equation}\label{Elog-7}
%\frac{md}{2+d(m-1)} \leq \tilde{q}_1 \leq \infty \quad \Longrightarrow
%\quad    1 \leq \tilde{q}_2 \leq \frac{m}{m-1}.
%\end{equation}
%When $d=2$, from \eqref{Elog-5.5}, we have
%\begin{equation}\label{Elog-7}
%\frac{md}{2+d(m-1)} < \tilde{q}_1 \leq \infty \quad \Longrightarrow
%\quad    1 \leq \tilde{q}_2 <\frac{m}{m-1}.
%\end{equation}

Now we combine \eqref{Elog-0} and \eqref{Elog-6}, and then take integration in terms of $t$ with $\varepsilon = \frac{1}{m}$ to have
\begin{equation}\label{Elog-8}
\sup_{0\leq t \leq T} \int_{\mathbb{R}^d} \rho \log \rho \,dx +  \frac{3}{m}\iint_{Q_T} \abs{ \nabla \rho^{\frac{m}{2}}}^2\,dx\,dt
\leq  \int_{\mathbb{R}^d} \rho_0 \log \rho_0 \,dx + c\|V\|^{\tilde{q}_2}_{\tilde{\mathcal{S}}_{m,1}^{(\tilde{q}_1, \tilde{q}_2)}}.
\end{equation}
To \eqref{Elog-8}, we apply \eqref{entropy} and add $p$-th moment estimate \eqref{p-moment} with $\varepsilon = \frac{1}{m}$ to obtain
\begin{equation*} %\label{Elog-9}
\begin{aligned}
&\sup_{0\leq t \leq T} \int_{\mathbb{R}^d} \left(\rho \abs{\log \rho} + \rho \langle x\rangle^p \right)\,dx +  \frac{2}{m}\iint_{Q_T} \abs{ \nabla \rho^{\frac{m}{2}}}^2\,dx\,dt  \\
&\leq  \int_{\mathbb{R}^d} \left(\rho_0 \log \rho_0 + \rho_0 \langle x\rangle^p\right) \,dx
+ \iint_{Q_T} \rho \langle x\rangle^p \,dx\,dt
+ C (\|V\|_{\mathcal{S}_{m,1}^{(q_1, q_2)}}, \|V\|_{\tilde{\mathcal{S}}_{m,1}^{(\tilde{q}_1, \tilde{q}_2)}} ),
\end{aligned}
\end{equation*}
under \eqref{V-tilde-log-energy} and $V\in \mathcal{S}_{m,1}^{(q_1,q_2)}$ holding \eqref{V-p-moment} for $q=1$.
%{\color{violet} where $(q_1, q_2)$ satisfies \eqref{V-p-moment} (with $q=1$)} and $(\tilde{q}_1, \tilde{q}_2)$ is given in \eqref{Elog-7}.
By Remark~\ref{R:S-tildeS} $(ii)$, we see that \eqref{V-tilde-log-energy} yields
\begin{equation*} %\label{Elog-10}
\begin{aligned}
&\sup_{0\leq t \leq T} \int_{\mathbb{R}^d} \left(\rho \abs{\log \rho} + \rho \langle x\rangle^p \right)\,dx +  \frac{2}{m}\iint_{Q_T} \abs{ \nabla \rho^{\frac{m}{2}}}^2\,dx\,dt  \\
&\leq  \int_{\mathbb{R}^d} \left(\rho_0 \log \rho_0 + \rho_0 \langle x\rangle^p\right) \,dx
+ \iint_{Q_T} \rho \langle x\rangle^p \,dx\,dt
+ C (\|V\|_{\tilde{\mathcal{S}}_{m,1}^{(\tilde{q}_1, \tilde{q}_2)}}).
\end{aligned}
\end{equation*}
Then we complete the proof by taking Gr\"{o}nwall's inequality to handle the integral of $p$-th moment as in \eqref{Elog-5}.
%Here we note that \eqref{V-tilde-log-energy} is the intersected range of \eqref{V-tilde-log-energy} and \eqref{V-p-moment} after applying Remark~\ref{R:S-tildeS} (ii).
\smallskip

$(iii)$ Because of \eqref{log-V-2} and $\nabla \cdot V =0$, we observe that RHS of \eqref{Elog-0} vanishes. Therefore, we obtain the following by taking integral to \eqref{Elog-0} with repect to $t$:
\begin{equation}\label{Elog-11}
\sup_{0\leq t \leq T} \int_{\mathbb{R}^d} \rho \log \rho \,dx +  \frac{4}{m}\iint_{Q_T} \abs{ \nabla \rho^{\frac{m}{2}}}^2\,dx\,dt
\leq  \int_{\mathbb{R}^d} \rho_0 \log \rho_0 \,dx.
\end{equation}
For $V$ satisfying \eqref{V-p-moment} which is \eqref{V-log-divfree} when $q=1$, we add \eqref{p-moment} to \eqref{Elog-11}. Then as above, we apply \eqref{entropy} and Gr\"{o}nwall's inequality to complete the proof. We note that \eqref{V-log-divfree} comes directly from \eqref{V-p-moment}.
\end{proof}

\subsubsection{Energy estimates for case: $\rho_0 \in L^{q}(\mathbb{R}^d), \ q >1$.}\label{SS:energy-Lq}

%In the following proposition, we deliver a priori estimates when the initial data $\rho_0$ is given in $L^{q}_{x}(\mathbb{R}^d)$ for $q>1$.
In the following proposition, there are three cases for $m, q > 1$. First case is when $q-m+1 \geq 0$ holds. The second case cares all $m, q >1$ by prescribing conditions on $\nabla V$, and the last case is when $\nabla \cdot V = 0$.

%% Energy estimate

\begin{proposition}\label{P:Lq-energy}
Let $m, q> 1$ and $\rho_{0}\in  \mathcal{P}(\mathbb{R}^d)\cap L^{q}(\mathbb{R}^d) \cap \calC^{\alpha}(\mathbb{R}^d)$. Suppose that $\rho$ is a regular solution of \eqref{E:Main} with sufficiently fast decay at infinity in spatial variables.
%satisfying
%\begin{equation}\label{Lq-initial}
%\int_{\mathbb{R}^d} \left( \rho_0 + \rho_{0}^{q}\right) \,dx < \infty.
%\end{equation}
\begin{itemize}
\item[(i)] Let $q \geq m-1$. Assume that
\begin{equation}\label{V-Lq-energy}
V \in \mathcal{S}_{m,q}^{(q_1,q_2)}\cap \calC^{\infty} (\bar{Q}_{T}) \quad  \text{for} \quad
\begin{cases}
2 \leq q_2 \leq \frac{q+m-1}{m-1},  & \text{ if } d> 2, \vspace{1 mm}\\
2\leq q_2 < \frac{q+m-1}{m-1},  & \text{ if } d= 2.
\end{cases}
\end{equation}
Then the following estimate holds:
\begin{equation}\label{Lq-energy}
 \sup_{0 \leq t \leq T}  \int_{\mathbb{R}^d} \rho^{q} (\cdot, t)\,dx
 + \frac{2 q m (q-1)}{(q+m-1)^2} \iint_{Q_T} \left|\nabla \rho^{\frac{q+m-1}{2}}\right|^2 \,dx\,dt \le C ,
\end{equation}
where  $C = C (\|V\|_{\mathcal{S}_{m,q}^{(q_1, q_2)}}, \, \|\rho_{0}\|_{L^{q} (\mathbb{R}^d)})$.

\item[(ii)]  Assume that
 \begin{equation}\label{V-tilde-Lq}
V\in \tilde{\mathcal{S}}_{m,q}^{(\tilde{q}_1, \tilde{q}_2)} \cap \calC^{\infty} (\bar{Q}_{T}) \quad  \text{for} \quad
\begin{cases}
1 \leq \tilde{q}_2 \leq \frac{q+m-1}{m-1},   & \text{ if } d> 2, \vspace{1 mm}\\
1\leq \tilde{q}_2 < \frac{q+m-1}{m-1},  & \text{ if } d= 2 .
\end{cases}
\end{equation}
Then \eqref{Lq-energy} holds with
$C = C ( \|V\|_{\tilde{\mathcal{S}}_{m,q}^{(\tilde{q}_1, \tilde{q}_2)}}, \, \|\rho_{0}\|_{L^{q} (\mathbb{R}^d)})$.

\item[(iii)]  Assume that $V\in\calC^{\infty} (\bar{Q}_{T})$ and $\nabla\cdot V=0$. Then the following estimate holds
\begin{equation}\label{Lq-energy-divfree}
  \sup_{0 \leq t \leq T} \int_{\mathbb{R}^d} \rho^{q} (\cdot, t)\,dx + \frac{q m (q-1)}{(q+m-1)^2} \iint_{Q_T} \left|\nabla \rho^{\frac{q+m-1}{2}}\right|^2 \,dx\,dt
  \leq \int_{\mathbb{R}^d} \rho_{0}^{q}\,dx.
\end{equation}
\end{itemize}

\end{proposition}

\begin{proof}
Let $q > 1$. By testing $\varphi = q \rho^{q-1}$ to \eqref{E:Main}, we have
\begin{equation}\label{apriori2-eqn}
  \frac{d}{dt} \int_{\mathbb{R}^d} \rho^{q}\,dx
+ \frac{4m q (q-1)}{(q+m-1)^2} \int_{\mathbb{R}^d} \abs{\nabla\rho^{\frac{q+m-1}{2}}}^2\,dx
 =   q \int_{\mathbb{R}^d} V \rho \nabla \rho^{q-1}\,dx =: \mathcal{J}.
\end{equation}
\smallskip

$(i)$ In case $q \geq m-1$, we observe that
\begin{equation}\label{energy-V-1}
\mathcal{J} = q \int_{\mathbb{R}^d} V \rho \nabla \rho^{q-1}\,dx
= \frac{2(q-1)}{q+m-1}\int_{\mathbb{R}^d} V \rho^{\frac{q-m+1}{2}} \nabla \rho^{\frac{q+m-1}{2}}\,dx.
\end{equation}

First, let $q > m-1$. Then by H\"{o}lder inequality with $\frac{1}{q_1} + \frac{1}{2} + \frac{q_1 -2 }{2q_1} = 1$ for $q_1 > 2$, we have
\begin{equation*} %\label{energy00}
\mathcal{J} \leq c  \|V\|_{L^{q_1}_{x}}
  \left\|\nabla \rho^{\frac{m+q-1}{2}}\right\|_{L^{2}_{x}}
 \|\rho\|_{L_{x}^{\frac{q_1 (q-m+1)}{q_1 -2}}}^{\frac{q-m+1}{2}}.
\end{equation*}
Furthermore, when $q=m-1$, we have
\begin{equation*}
\mathcal{J}= c  \int_{\mathbb{R}^d} V \nabla \rho^{\frac{q+m-1}{2}}\,dx
\leq c \|V\|_{L^{2}_{x}}
  \left\|\nabla \rho^{\frac{m+q-1}{2}}\right\|_{L^{2}_{x}}
  \leq \varepsilon  \left\|\nabla \rho^{\frac{m+q-1}{2}}\right\|_{L^{2}_{x}}^{2}
  + c(\varepsilon ) \|V\|_{L^{2}_{x}}^{2}.
\end{equation*}

For simplicity, let
\begin{equation*}
    r_1 := \frac{q_1 (q-m+1)}{q_1 -2}, \quad \text{and} \quad \theta := \frac{q}{r_1} \left( 1 - \frac{(r_1 - q)(d-2)}{2q + d(m-1)}\right).
\end{equation*}
Then by \eqref{q-r1r2} in Lemma~\ref{P:L_r1r2}, we see that
\begin{equation}\label{energy00.5}
\begin{cases}
     2<\frac{d(q+m-1)}{q+(d-1)(m-1)}\leq q_1 \leq \frac{2q}{m-1}, & \text{ if } d>2, \vspace{1 mm}\\
     2< q_1 \leq \frac{2q}{m-1}, & \text{ if } d = 2.
\end{cases}
\end{equation}
Then by  \eqref{pme-interpolate-20} and \eqref{pme-interpolate-10}  in Lemma~\ref{P:L_r1r2}, we deduce
\begin{equation}\label{energy01}
%\begin{aligned}
\mathcal{J}
\leq c     \|V\|_{L^{q_1}_{x}}
  \left\|\nabla \rho^{\frac{m+q-1}{2}}\right\|_{L^{2}_{x}}^{1+\frac{(1-\theta)(q-m+1)}{q+m-1}}
 \|\rho\|_{L_{x}^{q}}^{\frac{\theta(q-m+1)}{2}}
 \leq \varepsilon \left\|\nabla \rho^{\frac{m+q-1}{2}}\right\|_{L^{2}_{x}}^{2}
 +c(\varepsilon) \|V\|_{L^{q_1}_{x}}^{q_2}\|\rho\|_{L_{x}^{q}}^{q_2 \frac{\theta(q-m+1)}{2}},
%\end{aligned}
\end{equation}
by taking Young's inequality. Due to $V\in \mathcal{S}_{m,q}^{(q_1,q_2)}$, it gives
\begin{equation*}
 \frac{2(q+m-1)}{(q+m-1)-(1-\theta)(q-m+1)} = \frac{q_1 (2+q_{m,d})}{q_1 (1+q_{m,d}) - d} =  q_2.
\end{equation*}
The conditions \eqref{energy00.5} corrresponds to \eqref{V-Lq-energy}.

%Moreover, this gives
%\begin{equation}\label{energy02}
%    \frac{d(q+m-1)}{q+(d-1)(m-1)}\leq q_1 \leq \frac{2q}{m-1} \quad \Longrightarrow \quad 2 \leq q_2 \leq \frac{q+m-1}{m-1}.
%\end{equation}
%{\color{violet} Rewrite!!}
%which means $q_1=q_2=2$. Therefore, \eqref{energy02} works for all $q \geq m-1$.
%When $d=2$, then \eqref{energy00.5} and \eqref{energy02} provide that
%\begin{equation}
%    \frac{d(q+m-1)}{q+(d-1)(m-1)}< q_1 \leq \frac{2q}{m-1} \quad \Longrightarrow \quad 2 \leq q_2 < \frac{q+m-1}{m-1}.
%\end{equation}

Now the combination of \eqref{apriori2-eqn} and \eqref{energy01} with $\varepsilon = \frac{4mq(q-1)}{(q+m-1)^2}$ yields that
\begin{equation*} %\label{energy03}
    \frac{d}{dt} \int_{\mathbb{R}^d} \rho^{q}\,dx \leq c \|V\|_{L^{q_1}_{x}}^{q_2}\|\rho\|_{L_{x}^{q}}^{q_2 \frac{\theta(q-m+1)}{2}} \quad \Longrightarrow \quad \frac{d}{dt} \left(\int_{\mathbb{R}^d} \rho^q \,dx \right)^{1-\sigma}
\leq c \|V\|^{q_2}_{L^{q_1}_{x}},
\end{equation*}
because $\sigma:= \frac{\theta q_2 (q-m+1)}{2 q}=\frac{ \theta q_2 (q_1 - 2)}{2 q q_1} < 1$. % we rewrite \eqref{energy03} in the following form
%\begin{equation*}
%\frac{d}{dt} \left(\int_{\mathbb{R}^d} \rho^q \,dx \right)^{1-\sigma}
%\leq c \|V\|^{q_2}_{L^{q_1}_{x}}.
%\end{equation*}
By taking the integration in terms of $t$, and taking power by $\frac{1}{1-\sigma}$, we obtain the following :
\begin{equation}\label{energy04}
\sup_{0\leq t\leq T} \int_{\mathbb{R}^d} \rho^q (t) \,dx \leq c \int_{\mathbb{R}^d} \rho_0^q \,dx + c \|V\|_{\mathcal{S}_{m,q}^{(q_1,q_2)}}^{c}.
\end{equation}
We combine \eqref{apriori2-eqn} and \eqref{energy01} with $\varepsilon = \frac{2mq(q-1)}{(q+m-q)^2}$ and take the integration in terms of $t$, which yields
\begin{equation}\label{energy05}
\begin{aligned}
\sup_{0\leq t\leq T} \int_{\mathbb{R}^d} \rho^{q}\,dx
&+ \frac{2m q (q-1)}{(q+m-1)^2} \iint_{Q_T} \abs{\nabla\rho^{\frac{q+m-1}{2}}}^2\,dx\,dt \\
&\leq \int_{\mathbb{R}^d} \rho_{0}^{q}\,dx
+ c \|V\|_{\mathcal{S}_{m,q}^{(q_1, q_2)}}\sup_{0\leq t \leq T}\|\rho (\cdot, t)\|_{L_{x}^{q}}^{q_2 \frac{\theta(q-m+1)}{2}}.
\end{aligned}
\end{equation}
Therefore, we obtain \eqref{Lq-energy} by substituting \eqref{energy04} into \eqref{energy05}.
%Moreover, we can repeat above for any $\hat{q} \in (1, q)$ because \eqref{Lq-initial} implies that $\rho_0 \in L^{\hat{q}}_x$.
\smallskip

$(ii)$ Here, we take the integration by parts that gives:
\begin{equation}\label{energy-V-2}
\mathcal{J} =  q \int_{\mathbb{R}^d} V \rho \nabla \rho^{q-1}\,dx
=  (q-1) \int_{\mathbb{R}^d} V \nabla \rho^{q}\,dx = -(q-1) \int_{\mathbb{R}^d} \nabla V  \rho^{q}\,dx.
\end{equation}
By taking H\"{o}lder inequality with $\frac{1}{\tilde{q}_1} + \frac{\tilde{q}_1 -1}{\tilde{q}_1} = 1$ for $\tilde{q}_1 > 1$, we have
 \begin{equation*}
 \mathcal{J} \leq c \|\nabla V\|_{L_{x}^{\tilde{q}_1}} \|\rho\|_{L^{\frac{q \tilde{q}_1}{\tilde{q}_1 -1}}_{x}}^{q}.
 \end{equation*}
 For simplicity, let
 \begin{equation*}
 \tilde{r}_1 := \frac{q\tilde{q}_1}{\tilde{q}_1 - 1} \quad \text{ and } \quad \theta := \frac{q}{\tilde{r}_1} \left( 1 - \frac{(\tilde{r}_1 - q)(d-2)}{2q+d(m-1)} \right).
 \end{equation*}
Then by \eqref{q-r1r2} in Lemma~\ref{P:L_r1r2} for $\tilde{r}_1$ instead of $r_1$, it gives
 \begin{equation}\label{energy06}
 \begin{cases}
  1<\frac{d+q_{m,d}}{2+q_{m,d}}\leq \tilde{q}_1 \leq \infty, & \text{ if } d>2, \vspace{1 mm}\\
  1< \tilde{q}_1 \leq \infty, & \text{ if } d= 2.
  \end{cases}
 \end{equation}

 Then by  \eqref{pme-interpolate-20} and \eqref{pme-interpolate-10}  in Lemma~\ref{P:L_r1r2}, we deduce
 \begin{equation}\label{energy07}
 %\begin{aligned}
 \mathcal{J} \leq c \|\nabla V\|_{L_{x}^{\tilde{q}_1}} \|\rho\|_{L^{q}_{x}}^{q\theta}
 \left\|\nabla \rho^{\frac{q+m-1}{2}}\right\|_{L^{2}_{x}}^{\frac{2q(1-\theta)}{q+m-1}}
 \leq \varepsilon \left\|\nabla \rho^{\frac{q+m-1}{2}}\right\|_{L^{2}_{x}}^{2}
 + c(\varepsilon) \|\nabla V\|_{L_{x}^{\tilde{q}_1}}^{\tilde{q}_2} \|\rho\|_{L^{q}_{x}}^{q\theta \tilde{q}_2},
 %\end{aligned}
 \end{equation}
 by taking Young's inequality. Because of $V \in \tilde{\mathcal{S}}_{m,q}^{(\tilde{q}_1, \tilde{q}_2)}$, it gives
 \begin{equation*}
  \frac{q+m-1}{(q+m-1)-(1-\theta)q}  = \frac{\tilde{q}_1 (2+q_{m,d})}{\tilde{q}_1 (2+q_{m,d}) - d} =  \tilde{q}_2.
 \end{equation*}
 Also, we observe that \eqref{energy06} corresponds to \eqref{V-tilde-Lq}.
%it gives
%\begin{equation}\label{energy08}
%\frac{d+q_{m,d}}{2+q_{m,d}}\leq \tilde{q}_1 \leq \infty \quad \Longrightarrow \quad
%1\leq \tilde{q}_2 \leq \frac{q+m-1}{m-1}.
%\end{equation}
%If $d=2$, then from \eqref{energy06}, we have following pairs:
%\begin{equation}
%\frac{d+q_{m,d}}{2+q_{m,d}}< \tilde{q}_1 \leq \infty \quad \Longrightarrow \quad
%1\leq \tilde{q}_2 < \frac{q+m-1}{m-1}.
%\end{equation}

Now we combine \eqref{apriori2-eqn} and \eqref{energy07} with $\varepsilon = \frac{4mq(q-1)}{(q-m+1)^2}$ which yields
\begin{equation}\label{energy08.5}
    \sup_{0 \leq t \leq T} \int_{\mathbb{R}^d} \rho^q (\cdot, t)\,dx
    \leq c \|\nabla V\|_{L_{x}^{\tilde{q}_1}}^{\tilde{q}_2} \|\rho\|_{L^{q}_{x}}^{q\theta \tilde{q}_2}
    \quad \Longrightarrow \quad \frac{d}{dt} \left(\int_{\mathbb{R}^d} \rho^q \,dx \right)^{1-\sigma}
\leq c \|\nabla V\|_{L_{x}^{\tilde{q}_1}}^{\tilde{q}_2},
\end{equation}
because $\sigma:= \theta \tilde{q}_2 < 1$.
%we rewrite \eqref{energy03} in the following form
%\begin{equation*}
%\frac{d}{dt} \left(\int_{\mathbb{R}^d} \rho^q \,dx \right)^{1-\sigma}
%\leq c \|\nabla V\|_{L_{x}^{\tilde{q}_1}}^{\tilde{q}_2}.
%\end{equation*}
By taking the integration in terms of $t$, and taking power by $\frac{1}{1-\sigma}$, we obtain the following :
\begin{equation}\label{energy09}
\sup_{0\leq t\leq T} \int_{\mathbb{R}^d} \rho^q (t) \,dx \leq c \int_{\mathbb{R}^d} \rho_0^q \,dx + c \|V\|_{L_{x}^{\tilde{\mathcal{S}}_{m,q}}}^{(\tilde{q}_1, \tilde{q}_2)}.
\end{equation}
To obtain \eqref{Lq-energy}, we add \eqref{apriori2-eqn} and \eqref{energy07} with $\varepsilon = \frac{2mq(q-1)}{(q-m+1)^2}$. We complete the proof by taking integration in terms of $t$ and applying \eqref{energy08.5}.

$(iii)$ It is trivial that $\mathcal{J}$ is vanishing when \eqref{energy-V-2} is combined with $\nabla\cdot V=0$. Therefore, \eqref{Lq-energy-divfree} follows directly by taking integration of \eqref{apriori2-eqn} in terms of $t$ and applying \eqref{energy09}.
\end{proof}

Under the same hyphthesis in Proposition~\ref{P:Lq-energy}, we prove the following proposition.
\begin{proposition}\label{P:Lq-energy-r}
Let $m, q> 1$ and $\rho_{0}\in  \mathcal{P} (\mathbb{R}^d)\cap L^{q}(\mathbb{R}^d) \cap \calC^{\alpha}(\mathbb{R}^d)$. Suppose that $\rho$ is a regular solution of \eqref{E:Main} with sufficiently fast decay at infinity in spatial variables.
\begin{itemize}
\item[(i)] Let $q \geq m-1$ and assume that $V$ satisfies \eqref{V-Lq-energy}.
Then, for any $r>1$ and $ m-1 \leq r \leq q$, it holds
\begin{equation}\label{Lq-energy-r}
 \sup_{0 \leq t \leq T}  \int_{\mathbb{R}^d} \rho^{r} (\cdot, t)\,dx
 + \iint_{Q_T} \left|\nabla \rho^{\frac{r+m-1}{2}}\right|^2 \,dx\,dt \le C ,
\end{equation}
where  $C = C (\|V\|_{\mathcal{S}_{m,q}^{(q_1, q_2)}}, \, \|\rho_{0}\|_{L^{q} (\mathbb{R}^d)})$.
\item[(ii)]  If $V$ satisfies \eqref{V-tilde-Lq}, then the estimate \eqref{Lq-energy-r} holds with
$C = C ( \|V\|_{\tilde{\mathcal{S}}_{m,q}^{(\tilde{q}_1, \tilde{q}_2)}},\,  \|\rho_{0}\|_{L^{q} (\mathbb{R}^d)})$ for any $1 < r \leq q$.
\item[(iii)] If $V\in \calC^{\infty}(\bar{Q}_{T})$ and $\nabla\cdot V=0$, then the estimate \eqref{Lq-energy-r} holds with
$C = C (\|\rho_{0}\|_{L^{q} (\mathbb{R}^d)})$ for any $1<r\leq q$.
\end{itemize}
\end{proposition}

\begin{remark}
Assuming further that $\rho_0 \in  \mathcal{P}_p (\mathbb{R}^d)$ for $p>1$ and $V$ satisfies \eqref{V-p-moment}, then the estimate \eqref{Lq-energy-r} holds for $(i)$ $\max\{1, m-1\} \leq r \leq q$, and  $(ii)$, $(iii)$ $1\leq r \leq q$ with a constant $C$ depending also on
$\int_{\bbr^d} \rho_0 \langle x\rangle^p \,dx \cdot \delta_{\{r=1\}}$.
\end{remark}

\begin{proof} We note a couple of remarks.
Firstly, note that  $\|\rho_0\|_{L^{r}(\bbr^d)} \leq C(\|\rho_0\|_{L^1 \cap L^{q}(\bbr^d)})$ for any $r\in (1, q)$ by the interpolation inequality because $\rho_0 \in L^1 (\bbr^d) \cap L^q (\bbr^d)$.  Secondly, $V\in \mathcal{S}_{m,q}^{(q_1, q_2)}$ implies that $V \in \mathcal{S}_{m,r}^{(q_1, q_2^{r})}$ for any $r \in [1, q)$ by \eqref{q-qL} in Remark~\ref{R:S-tildeS} $(iii)$.

Then the estimate \eqref{Lq-energy-r} follows by similar arguments as in Proposition~\ref{P:Lq-energy} for $r$ in the place of $q$. Then $C$ depends on $\|V\|_{\mathcal{S}_{m,r}^{(q_1, q_2^{r})}}$ and $\|\rho_0\|_{L^r(\bbr^d)}$ which are replaced by $\|V\|_{\mathcal{S}_{m,q}^{(q_1, q_2)}}$ and $\|\rho_0\|_{L^q(\bbr^d)}$ because of earlier remarks.

Especially, in case $(i)$, we carry similar analysis of \eqref{energy-V-1} with $r$ instead of $q$ if $r \geq m-1$. In case $(ii)$, note that $V\in \tilde{\mathcal{S}}_{m,q}^{(\tilde{q}_1, \tilde{q}_2)}$ implies the membership in $\tilde{\mathcal{S}}_{m,r}^{(\tilde{q}_1, \tilde{q}_2^{r})}$ for any $r \in [1, q]$ by Remark~\ref{R:S-tildeS}.
In case $(iii)$, the estimate \eqref{Lq-energy-r} holds for any $r\in (1, q]$ with $C = C (\|\rho_{0}\|_{L^{q} (\mathbb{R}^d)})$ using the divergence-free condition.
This completes the proof.
\end{proof}

\subsection{Estimates of speed and $p$-th moment}
\label{SS:Combination}

Here we obtain speed and $p$-th moment estimates by combining the auxiliary and energy estimates in earlier subsections.
First result is when $\int_{\bbr^d} (\rho_0 \log \rho_0 + \rho_0 \langle x\rangle^p ) \,dx < \infty$ for $1< p\leq {\lambda_1}$.
\begin{proposition}\label{P:log-energy-speed}
Let  $m>1$ and $1< p \leq {\lambda_1}:= 1+ \frac{1}{d(m-1) + 1}$. Under the same hypothesis in Proposition~\ref{P:log-energy}, the followings hold true.
\begin{itemize}
\item[(i)] Let $1 < m\le 2$.
Assume that $V$ satisfies \eqref{V-log-energy}.
Then the following estimate holds that
\begin{equation}\label{log-energy-speed}
\sup_{0\leq t \leq T} \int_{\mathbb{R}^d\times \{t\}} \rho \left( \abs{\log \rho} + \langle x \rangle^p \right) \,dx
 + \iint_{Q_T}\left(\abs{\frac{\nabla \rho^m}{\rho}}^p+|V|^{p}\right ) \rho \,dx\,dt   \le C,
\end{equation}
where $C=C (\|V\|_{\mathcal{S}_{m,1}^{(q_1,q_2)}}, \, \int_{\mathbb{R}^d} \left(\rho_0\log \rho_0 + \rho_0 \langle x \rangle^p  \right) \,dx )$.

\item[(ii)]  Assume that $V$ satisfies \eqref{V-tilde-log-energy}.
Then the estimate \eqref{log-energy-speed} holds with
$$C= C (\|V\|_{\tilde{\mathcal{S}}_{m,1}^{(\tilde{q}_1,\tilde{q}_2)}},\, \int_{\mathbb{R}^d} \left(\rho_0\log \rho_0 + \rho_0 \langle x \rangle^p  \right) \,dx ).$$

\item[(iii)] Let $\nabla\cdot V=0$. Assume that $V$ satisfies \eqref{V-log-divfree}.
Then the estimate \eqref{log-energy-speed} holds with
$$C= C (\|V\|_{\mathcal{S}_{m,1}^{(q_1,q_2)}},\, \int_{\mathbb{R}^d} \left(\rho_0\log \rho_0 + \rho_0 \langle x \rangle^p  \right) \,dx).$$
\end{itemize}
\end{proposition}

\begin{proof}
    To obtain \eqref{log-energy-speed}, for $q=1$, we add \eqref{W-p} and \eqref{V-p} to \eqref{log-energy} with $2\epsilon = \frac{1}{m}$.
\end{proof}

%%%%%%%%%%%%

Now when $\rho_0 \in L^q (\bbr^d)$, $q>1$, we compute the following estimates assuming additionally $\int_{\mathbb{R}^d} \rho_0 \langle x \rangle^p \,dx < \infty$.
\begin{proposition}\label{P:Energy-speed}
Let $m, q> 1$ and $1< p \leq {\lambda_q}:= \min \left\{ 2, 1 + \frac{d(q-1)+q}{d(m-1)+q} \right\}$.  Suppose that $\rho_0 \in L^{q}(\bbr^d) \cap \mathcal{P}_{p}(\bbr^d)$ and the same hypothesis in Proposition~\ref{P:Lq-energy} hold.
%\begin{equation}\label{P:Energy-speed-initial}
%\int_{\mathbb{R}^d} \left(\rho_{0}^{q} + \rho_0 \langle x \rangle^p \right) \,dx < \infty.
%\end{equation}

\begin{itemize}
\item[(i)] Let $q\geq m-1$.  Assume that
\begin{equation}\label{V-Lq-energy-speed}
V \in \mathcal{S}_{m,q}^{(q_1,q_2)}\cap \calC^\infty(\bar{Q}_T)\quad \text{for}\quad
    \begin{cases}
        q_1 \leq \frac{2m}{m-1}, \  2 \leq q_2 \leq  \frac{q+m-1}{m-1} \,\delta_{\{1 < q \leq m\}} + \frac{2m-1}{m-1}\,\delta_{ \{q > m \} },  & \text{if } d > 2, \vspace{1 mm}\\
        q_1 \leq \frac{2m}{m-1}, \ 2 \leq q_2 <  \frac{q+m-1}{m-1}\,\delta_{\{1 < q \leq m\}} + \frac{2m-1}{m-1} \,\delta_{ \{q > m \} },  & \text{if } d = 2.
    \end{cases}
\end{equation}
Then the following estimate holds
\begin{equation}\label{Lq-energy-speed}
 \sup_{0 \leq t \leq T} \int_{\mathbb{R}^d} \left(\rho^{q} + \rho\langle x \rangle^p\right) \,dx
 + \iint_{Q_T} \left(\abs{\frac{\nabla \rho^m}{\rho}}^p+|V|^{p}\right ) \rho  \,dx\,dt \le C,
\end{equation}
where $C=C (\|V\|_{\mathcal{S}_{m,q}^{(q_1, q_2)}}, \, \int_{\mathbb{R}^d} \left( \rho_0^q + \rho_0 \langle x \rangle^p\right)\,dx )$.

\item[(ii)] Assume that
 \begin{equation}\label{V-tilde-Lq-energy-speed}
V\in \tilde{\mathcal{S}}_{m,q}^{(\tilde{q}_1, \tilde{q}_2)}\cap \calC^\infty(\bar{Q}_T) \quad \text{for}\quad
\begin{cases}
\frac{2+q_{m,d}}{1+q_{m,d}} < \tilde{q}_2 \leq \frac{q+m-1}{m-1} \,\delta_{\{1 < q \leq m\}} + \frac{2m-1}{m-1}\,\delta_{ \{q >  m \} },  & \text{ if } d> 2, \vspace{1 mm}\\
\frac{2+q_{m,2}}{1+q_{m,2}} < \tilde{q}_2 < \frac{q+m-1}{m-1} \,\delta_{\{1 < q \leq m\}} + \frac{2m-1}{m-1}\,\delta_{ \{q > m \} },  & \text{ if } d= 2.
\end{cases}
\end{equation}
Then \eqref{Lq-energy-speed} holds with
$ C =C  ( \|V\|_{\tilde{\mathcal{S}}_{m,q}^{(\tilde{q}_1, \tilde{q}_2)}}, \, \int_{\mathbb{R}^d} \left( \rho_0^q + \rho_0 \langle x \rangle^p\right)\,dx ).$

\item[(iii)]  Let $\nabla\cdot V=0$. Assume that
\begin{equation}\label{V-Lq-divfree}
V\in \mathcal{S}_{m,q}^{(q_1, q_2)}\cap \calC^\infty(\bar{Q}_T) \quad  \text{for} \quad
\begin{cases}
{\lambda_q} \leq q_2 \leq \frac{{\lambda_q} (q+m-1)}{q+m-2}, & \text{ if }  d>2, \ 1<q\leq m, \vspace{1 mm}\\
q_1 \leq \frac{2m}{m-1},\  2 \leq q_2 \leq \frac{2m-1}{m-1}, & \text{ if }  d>2, \ q> m, \vspace{1 mm}\\
{\lambda_q} \leq q_2 < \frac{{\lambda_q} (q+m-1)}{q+m-2}, & \text{ if }  d=2, \ 1<q\leq m, \vspace{1 mm}\\
q_1 \leq \frac{2m}{m-1},\  2 \leq q_2 < \frac{2m-1}{m-1}, & \text{ if }  d=2, \ q> m.
\end{cases}
\end{equation}
Then \eqref{Lq-energy-speed} holds with
$ C=C (\|V\|_{\mathcal{S}_{m,q}^{(q_1, q_2)}}, \, \int_{\mathbb{R}^d} \left( \rho_0^q + \rho_0 \langle x \rangle^p\right)\,dx ). $

\end{itemize}

\end{proposition}

\begin{proof} Here we carry computations for each cases.

$(i)$ First, let $ m-1 \leq q \leq m$. We add \eqref{Lq-energy} and \eqref{p-moment} with $\varepsilon = \frac{mq(q-1)}{(m+q-1)^2}$:
\begin{equation*} %\label{apriori2-04}
\begin{aligned}
&\sup_{0\leq t \leq T} \int_{\mathbb{R}^d} \left(\rho^{q} + \rho \langle x \rangle^p\right)\,dx
 + \frac{m q (q-1)}{(m+q-1)^2} \iint_{Q_T} \abs{\nabla\rho^{\frac{m+q-1}{2}}}^2\,dx\,dt \\
&\leq \int_{\mathbb{R}^d} \left(\rho_{0}^{q} + \rho_{0} \langle x \rangle^p\right)\,dx
 + \iint_{Q_T} \rho \langle x \rangle^p \,dx\,dt
 + C (\|V\|_{S_{m,q}^{(q_1,q_2)}}, \|\rho_0\|_{L^{q}(\mathbb{R}^d)}).
\end{aligned}
\end{equation*}
To handle the $p$-th moment term on RHS, we apply Gr\"{o}nwall's inequality as in the proof of Proposition~\ref{P:log-energy}:
\begin{equation}\label{apriori2-05}
\sup_{0\leq t \leq T} \int_{\mathbb{R}^d} \left(\rho^{q} + \rho \langle x \rangle^p\right)\,dx
 + \frac{m q (q-1)}{(m+q-1)^2} \iint_{Q_T} \abs{\nabla\rho^{\frac{m+q-1}{2}}}^2\,dx\,dt \leq C,
\end{equation}
where $C=C(\|V\|_{S_{m,q}^{(q_1,q_2)}}, \int_{\mathbb{R}^d} \left(\rho_{0}^{q} + \rho_{0} \langle x \rangle^p\right)\,dx)$.
Then we obtain \eqref{Lq-energy-speed} by adding \eqref{W-p} and \eqref{V-p} with $2\varepsilon = \frac{mq(q-1)}{(m+q-1)^2}$ to \eqref{apriori2-05}. We note that \eqref{V-Lq-energy-speed} is the intersected region of \eqref{V-p-moment} and \eqref{V-Lq-energy}.

Now let $q > m$, so ${\lambda_q} =2$. Then $V$ satisfying \eqref{V-Lq-energy-speed} also holds \eqref{V-Lq-energy}. Therefore, we have an estimate as in \eqref{Lq-energy}:
\begin{equation}\label{energy10}
\sup_{0\leq t \leq T} \int_{\mathbb{R}^d} \rho^q \,dx + \iint_{Q_T} \abs{\nabla \rho^{\frac{q+m-1}{2}}}^2 \,dx\,dt \leq C (\|V\|_{\mathcal{S}_{m,q}^{(q_1, q_{2})}}, \|\rho_0\|_{L^{q}(\mathbb{R}^d)}).
\end{equation}
Furthermore, it is true that $\rho_0 \in L^m (\mathbb{R}^d)$ (in fact, $\|\rho_0\|_{L^{m}(\mathbb{R}^d)}\leq c(\|\rho_0\|_{L^{1}(\mathbb{R}^d)}, \|\rho_0\|_{L^{q}(\mathbb{R}^d)})$) because $\rho_0 \in L^q(\bbr^d)\cap \mathcal{P}_p (\bbr^d)$ and assumption $q>m$. Then the corresponding energy estimate under $V$ satisfying \eqref{V-Lq-energy-speed} is
\begin{equation}\label{energy12}
\sup_{0\leq t \leq T} \int_{\mathbb{R}^d} \rho^m \,dx + \iint_{Q_T} \abs{\nabla \rho^{\frac{2m-1}{2}}}^2 \,dx\,dt \leq C (\|V\|_{\mathcal{S}_{m,m}^{(q_1, q_{2})}}, \|\rho_0\|_{L^{m}(\mathbb{R}^d)}),
\end{equation}
in which $q$ in \eqref{energy10} is replaced by $m$ under \eqref{V-Lq-energy-speed}. Then for $q_1$ holding \eqref{V-Lq-energy-speed}, the combination of \eqref{energy10} and \eqref{energy12} provides
\begin{equation*} %\label{energy15}
    \sup_{0\leq t \leq T} \int_{\mathbb{R}^d} \rho^q \,dx  + \iint_{Q_T} \abs{\nabla \rho^{\frac{2m-1}{2}}}^2 \,dx\,dt \leq C (\|V\|_{\mathcal{S}_{m,q}^{(q_1, q_{2})}}, \|\rho_0\|_{L^{q}(\mathbb{R}^d)}),
\end{equation*}
because of \eqref{q-qL} in Remark~\ref{R:S-tildeS} (iii).

Now we recall the speed and $p$-th moment estimates \eqref{P:W-p}, \eqref{V-p-moment}, and \eqref{p-moment} at $q=m$ that, for $1 < p \leq 2$,
\[
\iint_{Q_T} \left(\abs{\frac{\nabla \rho^m}{\rho}}^p + |V|^p \right)\rho \,dx\,dt
\leq \varepsilon \iint_{Q_T} \abs{ \nabla \rho^{\frac{2m-1}{2}}}^2 \,dx\,dt
  + C(\|V\|_{\mathcal{S}_{m,m}^{(q_1, \bar{q}_{2})}}, \|\rho_0\|_{L^m (\mathbb{R}^d)}),
\]
and
\[
\begin{aligned}
\sup_{0\leq t \leq T}\int_{\mathbb{R}^d} \, \rho  \langle x \rangle^{p} \,dx
&\leq \int_{\mathbb{R}^d} \, \rho_0 \langle x \rangle^{p} \,dx
 +\iint_{Q_T} \rho\langle x \rangle^{p} \,dx\,dt \\
&\quad+ \varepsilon \iint_{Q_T} \abs{\nabla \rho^{\frac{2m-1}{2}}}^2 \,dx\,dt
 + C(\|V\|_{\mathcal{S}_{m,m}^{(q_1, q_{2})}},\|\rho_0\|_{L^m(\mathbb{R}^d)}),
\end{aligned}
\]
for any $\varepsilon > 0$ where $V$ satisfies \eqref{V-Lq-energy-speed} which also holds \eqref{V-p-moment}. Then the combination with \eqref{energy10} using \eqref{q-qL} in Remark~\ref{R:S-tildeS}$(iii)$implies the follwoing estimate
\begin{equation*}
\begin{aligned}
\sup_{0\leq t \leq T} &\int_{\mathbb{R}^d} \left(\rho^q + \rho  \langle x \rangle^{p}\right)\,dx +  \iint_{Q_T} \left(\abs{\frac{\nabla \rho^m}{\rho}}^p + |V|^p \right)\rho \,dx\,dt \\
& \leq \int_{\mathbb{R}^d} \, \rho_0 \langle x \rangle^{p} \,dx
 +\iint_{Q_T} \rho\langle x \rangle^{p} \,dx\,dt
 + C (\|V\|_{\mathcal{S}_{m,q}^{(q_1, q_{2})}}, \|\rho_0\|_{L^{q}(\mathbb{R}^d)}).
\end{aligned}
\end{equation*}
By the Gr\"{o}nwall's inequality, we complete the proof to obtain \eqref{Lq-energy-speed} under \eqref{V-Lq-energy-speed}.

$(ii)$  By following similar computations as in above case, we can obtain also \eqref{Lq-energy-speed}. The condition on $V$ in \eqref{V-tilde-Lq-energy-speed} follows as the intersected region of \eqref{V-p-moment}. \eqref{V-tilde-Lq}, and \eqref{tilde-q2}. Because ${\lambda_q}$ behaves differently depending on whether $1<q\leq m$ or $q>m$, \eqref{V-p-moment} and \eqref{V-tilde-Lq} determines the upper bound of $\tilde{q}_2$ in \eqref{V-Lq-energy-speed}.
By considering \eqref{tilde-q2} which affects the lower bound of $q_2$ in \eqref{V-Lq-energy-speed}, hence $C$ depends on $\|V\|_{\tilde{\mathcal{S}}_{m,q}^{(\tilde{q}_1, \tilde{q}_2)}}$.

$(iii)$ To yield \eqref{Lq-energy-speed}, we add \eqref{p-moment}, \eqref{W-p}, and \eqref{V-p} to \eqref{Lq-energy-divfree} by following above computations. Because \eqref{Lq-energy-divfree} is independent of $V$, only \eqref{V-p-moment} is required in this case, that is, \eqref{V-Lq-divfree}.
\end{proof}

%% Energy estimate

Now let us switch a gear to obtain speed and moment estimates based on embedding arguments in time direction assuming $\rho_0 \in L^{q}(\bbr^d) \cap \mathcal{P}_{p}(\bbr^d)$. The embedding argument is possible beacuse of the finite time interval.

\begin{proposition}\label{P:Energy-embedding} (Embedding)
Assume the same hypothesis as in Proposition~\ref{P:Energy-speed}.
\begin{itemize}
\item[(i)]  Let $q > m$. Assume that
\begin{equation}\label{V-Lq-embedding}
V \in \mathcal{S}_{m,q}^{(q_1,q_2)}\cap \calC^\infty(\bar{Q}_T) \quad  \text{for} \quad
\begin{cases}
\frac{2m-1}{m-1} < q_2 \leq \frac{q+m-1}{m-1},  & \text{ if } d> 2, \vspace{1 mm}\\
\frac{2m-1}{m-1} \leq q_2 < \frac{q+m-1}{m-1},  & \text{ if } d= 2.
\end{cases}
\end{equation}
For given $(q_1, q_2)$ in \eqref{V-Lq-embedding}, there exists a constant $q_{2}^{\ast} = q_{2}^{\ast}(q_1) \in [2, q_2)$ such that $ V \in \mathcal{S}_{m, q^\ast}^{(q_1, q_2^{\ast})}$ satisfies \eqref{V-Lq-energy-speed} and the following estimate, for $1< p \leq {\lambda_{q^\ast}}:= \min \left\{ 2, 1 + \frac{d(q^\ast -1)+q^\ast}{d(m-1)+q^\ast} \right\}$,
\begin{equation}\label{Lq-energy-speed-embedding}
 \sup_{0 \leq t \leq T} \int_{\mathbb{R}^d} \left(\rho^{q^\ast} + \rho\langle x \rangle^p\right) \,dx
 + \iint_{Q_T} \left(\abs{\frac{\nabla \rho^m}{\rho}}^p+|V|^{p}\right ) \rho  \,dx\,dt \le C,
\end{equation}
where
$C = C (\|V\|_{\mathcal{S}_{m,q^{\ast}}^{(q_1, q^{\ast}_2)}}, \int_{\mathbb{R}^d}  (\rho_0^{q^\ast} + \rho_0 \langle x \rangle^p )\,dx )$.

\item[(ii)] Let $q> m$. Assume that
 \begin{equation}\label{V-tilde-Lq-embedding}
V\in \tilde{\mathcal{S}}_{m,q}^{(\tilde{q}_1, \tilde{q}_2)}\cap \calC^\infty(\bar{Q}_T) \quad  \text{for} \quad
\begin{cases}
\frac{2m-1}{m-1} < \tilde{q}_2 \leq \frac{q+m-1}{m-1},   & \text{ if } d> 2, \vspace{1 mm}\\
\frac{2m-1}{m-1} \leq \tilde{q}_2 < \frac{q+m-1}{m-1},  & \text{ if } d= 2 .
\end{cases}
\end{equation}
For given $(\tilde{q}_1, \tilde{q}_2)$ in \eqref{V-tilde-Lq-embedding}, there exists a constant $\tilde{q}_{2}^{\ast} = \tilde{q}_{2}^{\ast}(\tilde{q}_1) \in (\frac{2+q^{\ast}_{m,d}}{1+q^{\ast}_{m,d}}, \tilde{q}_2)$ such that $ V \in \tilde{\mathcal{S}}_{m, q^\ast}^{(\tilde{q}_1, \tilde{q}_2^{\ast})}$ satisfies \eqref{V-tilde-Lq-energy-speed} and
the estimate \eqref{Lq-energy-speed-embedding} holds for $1< p \leq {\lambda_{q^\ast}}$ with $C = C (\|V\|_{\tilde{\mathcal{S}}_{m,\tilde{q}^{\ast}}^{(\tilde{q}_1, \tilde{q}^{\ast}_2)}}, \int_{\mathbb{R}^d}  (\rho_0^{q^\ast} + \rho_0 \langle x \rangle^p )\,dx )$.

\item[(iii)] Let $\nabla\cdot V=0$. Assume that
\begin{equation}\label{V-Lq-divfree-embedding}
V\in \mathcal{S}_{m,q}^{(q_1, q_2)}\cap \calC^\infty(\bar{Q}_T) \quad \text{for}\quad
\begin{cases}
\frac{2+q_{m,d}}{1+q_{m,d}} < q_2 \leq \frac{{\lambda_1} m}{m-1}, & \text{ if }  d>2, \vspace{1 mm}\\
\frac{2+q_{m,2}}{1+q_{m,2}} < q_2 < \frac{{\lambda_1} m}{m-1}, & \text{ if }   d=2.
\end{cases}
\end{equation}
For given $(q_1, q_2)$ in \eqref{V-Lq-divfree-embedding}, there exists a constant $q_{2}^{\ast} = q_{2}^{\ast}(q_1) \in [2, q_2]$ such that $ V \in \mathcal{S}_{m, q^\ast}^{(q_1, q_2^{\ast})}$ satisfies \eqref{V-Lq-divfree} and the estimate \eqref{Lq-energy-speed-embedding} holds for $1< p \leq {\lambda_{q^\ast}}$ with
$C = C (\|V\|_{\mathcal{S}_{m,q^{\ast}}^{(q_1, q^{\ast}_2)}}, \, \int_{\mathbb{R}^d}  (\rho_0^{q^\ast} + \rho_0 \langle x \rangle^p )\,dx )$.
\end{itemize}

\end{proposition}

\begin{remark} To capture the idea of searching $q^\ast_2$ and $q^\ast$ in above proposition, see Figure 4-$(e)$, 5-$(e)$, 7-$(e)$, 8-$(e)$, and 9-$(e)$ in Appendix~\ref{Appendix:fig}. For instance, in case $(i)$, for $\frac{d(2m-1)}{1+d(m-1)}<q_1 < d(2m-1)$, one can choose
\begin{equation*} %\label{E01}
q_2^\ast = \frac{2m-1}{m-1} \quad \Longrightarrow \quad q^\ast = \frac{md(m-1)}{d(2m-1)/q_1 -1},
\end{equation*}
where $(q_1, q_2^\ast)$ lies on $\mathcal{S}_{m,q^\ast}^{(q_1, q_2^\ast)}$.
Moreover, in case $(ii)$, for $\frac{d(2m-1)}{2m+d(m-1)}<\tilde{q}_1 < \frac{d(2m-1)}{2m}$, one may choose
\[
\tilde{q}_2^\ast = \frac{2m-1}{m-1} \quad \Longrightarrow \quad q^\ast = \frac{md(m-1)}{d(2m-1)/\tilde{q}_1  - 2m},
\]
so that $(\tilde{q}_1, \tilde{q}^\ast_2)$ lies on $\tilde{\mathcal{S}}_{m,q^\ast}^{(\tilde{q}_1, \tilde{q}^\ast_2)}$.
In case $(iii)$, we note that, for $q_1 < \frac{md}{1+d(m-1)}$ (corresponding to $q_2 > \frac{{\lambda_1} m}{m-1}$), there is no $p$-th moment and speed estimate although an energy inequality in \eqref{Lq-energy-divfree} is possible to obtain.
\end{remark}

\begin{proof}
$(i)$ For any $(q_1, q_2)$ in \eqref{V-Lq-embedding}, one can find $q_{2}^\ast \in [2, q_2)$ satisfying
\begin{equation}\label{V-Lq-energy-speed*}
 \begin{cases}
        2 \leq q_2^\ast \leq  \frac{q^\ast +m-1}{m-1} \,\delta_{\{1 < q^\ast \leq m\}} + \frac{2m-1}{m-1}\,\delta_{ \{q^\ast >  m \} },  & \text{ if } d > 2, \vspace{1 mm}\\
        2 \leq q_2^\ast <  \frac{q^\ast+m-1}{m-1}\,\delta_{\{1 < q^\ast \leq m\}} + \frac{2m-1}{m-1} \,\delta_{ \{q^\ast >  m \} },  & \text{ if } d = 2,
    \end{cases}
\end{equation}
which also determines $q^\ast \in (1, q)$ such that $V \in \mathcal{S}_{m,q^\ast}^{(q_1,q_2^\ast)}$.
Because $q^\ast \in (1, q)$, the interpolation provides that  $ \|\rho_0\|_{L^{q^\ast}(\mathbb{R}^d)} \leq c( \|\rho_{0}\|_{L^{1}(\mathbb{R}^d)}, \|\rho_{0}\|_{L^{q}(\mathbb{R}^d)})$ which means we have \eqref{Lq-energy} for $q^\ast$ in the place of $q$. Also because \eqref{V-Lq-energy-speed*} is same with  \eqref{V-Lq-energy-speed} for $q^\ast$, therefore, we are able to deduce \eqref{Lq-energy-speed-embedding} in terms of $q^\ast$, which completes the proof.

$(ii)$  Similar as above, for any $(\tilde{q}_1, \tilde{q}_2)$ satisfying \eqref{V-tilde-Lq-embedding}, one can find $\tilde{q}_2^\ast < \tilde{q}_2$ and $q^\ast \in (1, q)$ satisfying \eqref{V-tilde-Lq-embedding} where $q^\ast_2$ and $q^\ast$ are replaced for $q_2$ and $q$. Then we have the estimate in Proposition~\ref{P:Energy-speed} $(ii)$ with $q^\ast$ and ${\lambda_{q^\ast}}$ instead of $q$ and ${\lambda_q}$.

$(iii)$ In this case, we first note that \eqref{V-Lq-divfree-embedding} contains the region by \eqref{V-Lq-divfree}. Therefore, if $(q_1, q_2)$ satisfies \eqref{V-Lq-divfree}, then there is nothing to prove but apply results in Proposition~\ref{P:Energy-speed} $(iii)$ directly.
Now let $(q_1, q_2)$ holds \eqref{V-Lq-divfree-embedding} but \eqref{V-Lq-divfree}, then we apply the similar arguments as above. One may choose $q_2^\ast < q_2$ and $q^\ast \in (1, q)$ satisfying \eqref{V-Lq-divfree} with $q^\ast_2$ and $q^\ast$ in the places of $q_2$ and $q$. Then the estimate as in Proposition~\ref{P:Energy-speed} $(iii)$ with $q^\ast$ and ${\lambda_{q^\ast}}$ instead of $q$ and ${\lambda_q}$ completes the proof.
\end{proof}

\subsection{Estimates of temporal and spatial derivatives.}\label{SS:Temporal}

Two main propositions in this subsection, we estimate temporal and spatial derivatives of $\rho$ using the integrability of $\rho$ from energy estimates and PDE \eqref{E:Main}.
Those estimates will be crucially used to show the strong convergence of nonlinear diffusion of approximated solutions.

Before we introduce the first proposition, we specify the integrability of $\rho$ and $V \rho$  based on energy estimate when $q=1$ and parabolic embedding lemma.
\begin{remark}\label{Remark : AL-1}
From \eqref{log-energy} in Proposition~\ref{P:log-energy}, we observe that
 $\rho \in L^{\frac{md +2}{d}}_{x,t}$ by Lemma~\ref{T:pSobolev} for $q=1$, that is,
    \begin{equation}\label{E:rho:1}
        \iint_{Q_{T}} \rho^{\frac{md+2}{d}}  \,dxdt
       \leq
         \sup_{0 \leq t \leq T}  \int_{\mathbb{R}^d} \rho (\cdot, t)\,dx
 +  \iint_{Q_T} \left|\nabla \rho^{\frac{m}{2}}\right|^2 \,dx\,dt \leq C,
    \end{equation}
    where the constant $C$ depends on $\int_{\bbr^d} (\rho_0 \log \rho_0 + \rho_0 \langle x\rangle^p)\,dx$ and either $\|V\|_{\mathcal{S}_{m,1}^{(q_1,q_2)}}$ or $\|V\|_{\tilde{\mathcal{S}}_{m,1}^{(\tilde{q}_1,\tilde{q}_2)}}$.
Then it holds
% The following holds by applying the H\"{o}lder inequalities that
 \begin{equation}\label{E:Vrho:1-1}
\iint_{Q_T} \left|V\rho\right|^{\frac{md+2}{md+1}}\,dx\,dt
\leq \|V\|_{\mathcal{S}_{m,1}^{(q_1,q_2)}}^{\frac{md+2}{md+1}} \|\rho\|_{L^{r_1, r_2}_{x,t}}^{\frac{md+2}{md+1}},
 \end{equation}
for
$\frac{1}{r_1} = \frac{md+1}{md+2} - \frac{1}{q_1}$ and $\frac{1}{r_2} = \frac{md+1}{md+2} - \frac{1}{q_2}$.
For the pair $(q_1, q_2)$ satisfying  one of assumptions on $V$ in Proposition~\ref{P:log-energy}, the pair $(r_1, r_2)$ satisfies \eqref{q-r1r2} in Lemma~\ref{P:L_r1r2} for $q=1$ and the norm $\|\rho\|_{L^{r_1, r_2}_{x,t}}$ is bounded by the same $C$ by \eqref{norm-q-r1r2}.
%the pair $(r_1, r_2)$ satisfying \eqref{q-r1r2} in Lemma~\ref{P:L_r1r2} provides the condition $(q_1, q_2)$ satisfying $\mathcal{S}_{m,1}^{(q_1, q_2)}$.  When $q_1=q_2=\frac{md+2}{d(m-1)+1}$, it corresponds to $r_1=r_2=\frac{md+2}{d}$.
\end{remark}

In the following proposition, we control the temporal and spatial derivatives of $\rho$ for $q=1$.
%Then to obtain property of $\nabla \rho$ in \eqref{grad-rho}, it must hold $2-m \geq 0$ which means $1 < m \leq 2$.

\begin{proposition}\label{Proposition : AL-1}
Let $1 < m \leq 2$. Assume that $(a)$ $V\in\mathcal{S}_{m,1}^{(q_1, q_2)}$ satisfying \eqref{V-log-energy}, or $(b)$ $V\in \tilde{\mathcal{S}}_{m,1}^{(\tilde{q}_1, \tilde{q}_2)}$ satisfying \eqref{V-tilde-log-energy}, or $(c)$ a divergence-free $V$ satisfying \eqref{V-log-divfree}.
Suppose that $\rho$ is a regular solution of \eqref{E:Main} with sufficiently fast decay at infity in spatial variables with $\rho_0 \in \calP_p (\bbr^d)\cap \mathcal{C}^{\alpha}(\bbr^d)$ for $1< p\leq {\lambda_1}$ holding \eqref{log-initial}.
Then it holds that
\begin{equation}\label{rho_t_1}
\norm{\rho_t}_{W_{x}^{-1, \frac{md+2}{md+1}}L_{t}^{\frac{md+2}{md+1}}}
+\norm{\nabla  \rho}_{L_{x,t}^{\frac{md+2}{d+1}}}
< C\left(\|V\|_{\mathcal{S}_{m,1}^{(q_1,q_2)}}, \int_{\mathbb{R}^d} \rho_0 \left(1+ \log \rho_0 + \langle x \rangle^p \right) \,dx\right),
\end{equation}
where $C$ depends on $\int_{\mathbb{R}^d} \rho_0 \left(1+ \log \rho_0 + \langle x \rangle^p \right)\,dx$ and either $\|V\|_{\mathcal{S}_{m,1}^{(q_1,q_2)}}$ or $\|V\|_{\tilde{\mathcal{S}}_{m,1}^{(\tilde{q}_1,\tilde{q}_2)}}$.
\end{proposition}

\begin{remark}\label{R:AL-1}
Moreover, the following is straightforward by applying Lemma~\ref{T:Sobolev} to \eqref{rho_t_1} with $p= \left(\frac{md+2}{d+1}\right)^{\ast}$,
\begin{equation*} %\label{rho_1}
\rho \in L^{p, \frac{md+2}{d+1}}_{x,t} \quad \text{for} \quad   p = \frac{d(md+2)}{d(d-m)+(d-2)}.
\end{equation*}
The estimate  \eqref{rho_t_1} is essential to show the strong convergence of approximated solutions by the compactness arguments via Aubin-Lions Lemma (see Section~\ref{Exist-weak}).
\end{remark}

\begin{proof}
Thanks to \eqref{E:rho:1}, we observe that $\nabla \rho^{\frac{m}{2}} \in L^{2}_{x,t}$ and $\rho^{\frac{m}{2}} \in L^{\frac{2(md+2)}{md}}_{x,t}$.
Hence we have
\begin{equation}\label{Aug13-10}
\rho^{\frac{m}{2}} \nabla \rho^{\frac{m}{2}} \in L^{\frac{md+2}{md+1}}_{x,t} \quad \text{and}\quad V\rho \in L^{\frac{md+2}{md+1}}_{x,t}
\end{equation}
because of $\frac{md+1}{md+2}=\frac{md}{2(md+2)} + \frac{1}{2}$ and \eqref{E:Vrho:1-1}.
 By rewriting PDE \eqref{E:Main} in the form of $
\rho_t = \nabla \cdot \left( 2 \rho^{\frac{m}{2}}\nabla \rho^{\frac{m}{2}} - V \rho\right)$, it follows via \eqref{Aug13-10} that $\rho_t\in {W^{-1, \frac{md+2}{md+1}}L_{t}^{\frac{md+2}{md+1}}}$ and $\rho_t$ satisfies \eqref{rho_t_1}.

 When $m=2$, the it is clear that $\nabla \rho \in L^{2}_{x,t}$.
 Because $\rho^{\frac{2-m}{2}}\in L^{\frac{2(md+2)}{d(2-m)}}_{x,t}$ for $1<m<2$, we observe that
\begin{equation*} %\label{grad-rho}
\nabla \rho = \frac{2}{m}\rho^{\frac{2-m}{2}} \nabla \rho^{\frac{m}{2}} \in L^{\frac{md+2}{d+1}}_{x,t}, \quad \text{for} \quad \frac{d+1}{md+2}=\frac{d(2-m)}{2(md+2)} + \frac{1}{2},
\end{equation*}
and $\nabla \rho$ satisfies  \eqref{rho_t_1} which completes the proof.
\end{proof}

Now with the initial data in $L^{q}_{x}$ for $q >1$, we search further properties of the temporal and spatial derivatives of $\rho^q$ and $\rho^q$ itself. First, we recall a priori estimates in Proposition~\ref{P:Lq-energy} and remark the integrability of $\rho$ and $V\rho^q$.

\begin{remark}\label{Remark : AL-2}
%\begin{itemize}
%\item[(i)]
From the estimate \eqref{Lq-energy} in Proposition~\ref{P:Lq-energy}, it holds that $\rho \in L^{(q+m-1) + \frac{2q}{d}}_{x,t}$ by Lemma~\ref{T:pSobolev}, that is,
 \begin{equation}\label{E:rho:q-1}
      \iint_{Q_T} \rho^{(q+m-1)+\frac{2q}{d}} \,dxdt
        \leq  \sup_{0 \leq t \leq T}  \int_{\mathbb{R}^d} \rho^{q} (\cdot, t)\,dx
 + \iint_{Q_T} \left|\nabla \rho^{\frac{q+m-1}{2}}\right|^2 \,dx\,dt \leq C,
    \end{equation}
where the constant $C$ is depending on $\|\rho_0\|_{L^q (\bbr^d)}$ and either $\|V\|_{\mathcal{S}_{m,q}^{(q_1,q_2)}}$ or $\|V\|_{\tilde{\mathcal{S}}_{m,q}^{(\tilde{q}_1,\tilde{q}_2)}}$.
\end{remark}

We carry the following lemma before Proposition~\ref{Proposition : AL-2}.
\begin{lemma}\label{L-compact}
 Let $m>1$ and $q\geq 1$. Suppose that $V$ satisfies \eqref{T4:V-energy}.
Then, it holds
\[
V\rho^q \in L_{x,t}^{r} \quad \text{for} \quad r = \frac{d(m+q-1)+2q}{d(m+q-1)+q}.
\]
Futhermore, the followin estimate holds
\begin{equation}\label{E:rho:q-2}
\left\|V\rho^q\right\|_{L_{x,t}^{r}}
\leq \|V\|_{\mathcal{S}_{m,q}^{(q_1,q_2)}} \|\rho\|_{L^{r_1, r_2}_{x,t}}^{q}
\end{equation}
where the pair $(r_1, r_2)= (\frac{r q q_1}{q_1 - r}, \frac{rq q_2}{q_2 - r})$ satisfies \eqref{q-r1r2} in Lemma~\ref{P:L_r1r2}.
%for $\frac{1}{\gamma_1} = \frac 1 q \left( \frac{1}{\gamma} - \frac{1}{q_1}\right)$ and $\frac{1}{\gamma_2} = \frac 1 q \left( \frac{1}{\gamma} - \frac{1}{q_2}\right)$, where $q_1, q_2 > \gamma$.
\end{lemma}

\begin{remark}
When $q=1$, we easily observe that the estimate \eqref{E:rho:q-2} coincides with the estimate \eqref{E:Vrho:1-1}. The estimate \eqref{E:rho:q-2} is essential to prove compactness arguments in Proposition~\ref{Proposition : AL-2}.
\end{remark}

\begin{proof}
By taking the H\"{o}lder inequalities, we obtain the estimate \eqref{E:rho:q-2} for $\frac{q}{r_1}= \frac{1}{r} - \frac{1}{q_1}$ and $\frac{q}{r_2}= \frac{1}{r} - \frac{1}{q_2}$. For the pair $(r_1, r_2)$ satisfying \eqref{q-r1r2} and the pair $(q_1, q_2)$ satisfying $\mathcal{S}_{m,q}^{(q_1, q_2)}$, it provides that
\[
\frac{d+2+q_{m,d}}{r} = q \left(\frac{d}{r_1} + \frac{2+q_{m,d}}{r_2}\right) + \left(\frac{d}{q_1} + \frac{2+q_{m,d}}{q_2}\right) = d+1+q_{m,d}
\ \Longrightarrow \ r = \frac{d(m+q-1)+2q}{d(m+q-1)+q}.
\]

Let $d>2$. Then the range of $q \leq r_1 \leq \frac{d(m+q-1)}{d-2}$ in \eqref{q-r1r2} yields the following
\[
\frac 1r - 1\leq \frac{1}{q_1}\leq \frac 1r - \frac{q(d-2)}{d(q+m-1)}
\quad \Longrightarrow \quad 0 \leq \frac{1}{q_1} \leq \frac{1}{q_1^m}:= \frac{2+q_{m,d}}{d+q_{m,d}} - \frac{1}{d+2+q_{m,d}} < 1.
\]
Moreover, the range of $q+m-1\leq r_2 \leq \infty$ yields
\[
\frac 1r - \frac{q}{q+m-1} \leq \frac{1}{q_2} \leq \frac 1r
\quad \Longrightarrow \quad \frac{d(m+q-1)+q}{d(m+q-1)+2q} - \frac{q}{q+m-1}\leq \frac{1}{q_2} \leq \frac{1+q_{m,d}}{2+q_{m,d}},
\]
where the lower bound becomes zero if $q=q^\ast$ defined in \eqref{q_ast} and becomes negative if $q > q^\ast$ (this allows $q_2 = \infty$ for any $q \geq q^\ast$). In case $d=2$, the range of $q+m-1 < r_2 \leq \infty$ in \eqref{q-r1r2} corresponds to
\[
\frac 1r - \frac{q}{q+m-1} < \frac{1}{q_2} \leq \frac 1r
\quad \Longrightarrow \quad \frac{d(m+q-1)+q}{d(m+q-1)+2q} - \frac{q}{q+m-1}< \frac{1}{q_2} \leq \frac{1+q_{m,d}}{2+q_{m,d}},
\]
that allows $q_2 = \infty$ if $q > q^\ast$. Hence we derive restrictions on $V$ in \eqref{T4:V-energy} which completes the proof.
\end{proof}

Now, we estimate temporal and spatial derivatives of $\rho^q$ for $q>1$ when $\rho_0 \in L^{q}(\bbr^d)$.
\begin{proposition}\label{Proposition : AL-2}
Let $m, q >1$. Assume that $(a)$ $V$ satisfying \eqref{V-Lq-energy}, or $(b)$ $V$ satisfying \eqref{V-tilde-Lq}, or $(c)$ a divergence-free $V \in \calC^\infty (\bar{Q}_T)$ satisfying \eqref{T4:V-energy}.
Suppose that $\rho$ is a regular solution of \eqref{E:Main} with sufficiently fast decay at infity in spatial variables with $\rho_{0}\in  \mathcal{P}(\mathbb{R}^d)\cap L^{q}(\mathbb{R}^d) \cap \calC^{\alpha}(\mathbb{R}^d)$.
\begin{itemize}
\item[(i)] Let $q\ge m-1$. Then, it holds that
\begin{equation}\label{rho_t_2}
\norm{\partial_t\rho^q}_{W_{x}^{-1, \frac{d(q+m-1)+2q}{d(q+m-1)+q} }L_{t}^{1}}
+\norm{\nabla  \rho^q}_{L_{x,t}^{\frac{d(q+m-1)+2q}{q(d+1)}}}<C,
\end{equation}
where the constant $C$ depends on $\|\rho_0\|_{L^{q}(\bbr^d)}$ and either $\|V\|_{\mathcal{S}_{m,q}^{(q_1,q_2)}}$ or $\|V\|_{\tilde{\mathcal{S}}_{m,q}^{(\tilde{q}_1,\tilde{q}_2)}}$.

\item[(ii)] In cases $(b)$ or $(c)$ holds, let $m>2$ and $ \frac m2 < q < m-1 $. Then it holds that
\begin{equation}\label{rho_t_3}
\norm{\partial_t\rho^q}_{W_{x}^{-1, \frac{d(q+m-1)+2q}{d(q+m-1)+q} }L_{t}^{1}}
+\norm{\nabla  \rho^q}_{L_{x,t}^{\frac{2d(q+m-1)+4q}{d(3q-r)+2q}}}<C,
\end{equation}
for any $r\in (1, q]$ and where $C$ depends on $\|\rho_0\|_{L^{q}(\bbr^d)}$ and either $\|V\|_{\mathcal{S}_{m,q}^{(q_1,q_2)}}$ or $\|V\|_{\tilde{\mathcal{S}}_{m,q}^{(\tilde{q}_1,\tilde{q}_2)}}$.
\end{itemize}
\end{proposition}

\begin{remark}\label{R:AL-3}
\begin{itemize}
\item[(i)] When $1<m\leq 2$, the above works for all $q>1$. When $m>2$, it works for only $q \geq m-1$, in general. However, in cases $V$ satisfying either $(b)$ or $(c)$, we are able to weaken the restriction, that is, $q > \frac m2$ with the aid of the entropy estimation. We note $q>\frac m2$ is meaningful only for $m>2$.

\item[(ii)] In case $(b)$ or $(c)$ holds, then we are able to extend the range of $r \in [1, q]$ under further assumptions on $\rho_0 \in \mathcal{P}_p(\mathbb{R}^d)$ for $p>1$ and $V$ satisfying \eqref{V-p-moment}. It is because we need moment estimate when $r=1$ (see Lemma~\ref{L:p-moment-estimate} and Proposition~\ref{P:log-energy}). Then the estimate \eqref{rho_t_3} holds for $\frac m 2 \leq q \leq m-1$.

\item[(iii)] By applying Sobolev embedding in Lemma~\ref{T:Sobolev} to \eqref{rho_t_2} and \eqref{rho_t_3}, respectively, the followings hold
\begin{equation}\label{rho_2_2}
\rho^q \in L^{p, \frac{d(q+m-1)+2q}{q(d+1)}}_{x,t} \quad \text{for} \quad p= \frac{d^2(q+m+1)+2dq}{q(d^2-2) - d(m-1)},
\end{equation}
\begin{equation*} %\label{rho_2_3}
\rho^q \in L^{p, \frac{2d(q+m-1)+4q}{d(3q-r)+2q}}_{x,t} \quad \text{for} \quad p= \frac{2d^2(q+m-1)+4dq}{d^2(3q-r)-2d(m-1)-4q}.
\end{equation*}
As we mentioned earlier in Remark~\ref{R:AL-1}, the estimates  \eqref{rho_t_2} and \eqref{rho_t_3} are essential to show the strong convergence of approximated solutions by the compactness arguments via Aubin-Lions Lemma (see Section~\ref{Exist-weak}).
\end{itemize}
\end{remark}

\begin{proof}
$(i)$ Let $q \geq m-1$. Because of $\rho^{\frac{q-m+1}{2}} \in L_{x,t}^{\frac{2(d(q+m-1)+2q)}{d(q-m+1)}}$ and $\nabla \rho^{\frac{q+m-1}{2}} \in L^{2}_{x,t}$ in Remark~\ref{Remark : AL-2}, we obtain
\begin{equation}\label{prop414-5}
\nabla \rho^q = \frac{2q}{q+m-1} \rho^{\frac{q-m+1}{2}}\nabla \rho^{\frac{q+m-1}{2}}\in L_{x,t}^{\frac{d(q+m-1)+2q}{q(d+1)}},
%, \quad \text{for} \quad
%\frac{d(q-m+1)+q}{d(q-m+1)+2q} = \frac{d(q-m+1)}{2d(q-m+1)+4q} + \frac{1}{2}.
\end{equation}
via $\frac{q(d+1)}{d(q+m-1)+2q} = \frac{d(q-m+1)}{2d(q+m-1)+4q} + \frac{1}{2}$ which satifes \eqref{rho_t_2}.

In case $(a)$ holds, we rewrite the PDE \eqref{E:Main} as the following form:
\begin{equation}\label{prop414-10}
\left(\rho^q\right)_{t} = \nabla \cdot \left( c_0 \rho^{\frac{q+m+1}{2}} \nabla \rho^{\frac{q+m-1}{2}} - qV \rho^q\right) - c_1\abs{\nabla \rho^{\frac{q+m-1}{2}}}^2 + (q-1) V \nabla \rho^q,
\end{equation}
where $c_0 = \frac{2mq}{q+m-1}$ and $c_1=\frac{4mq(q-1)}{(q+m-1)^2}$.
To handle the cases of $(b)$ or $(c)$, we alternatively rewrite the equation as
%Now, let $m, q > 1$. Then the PDE \eqref{E:Main} gives
\begin{equation}\label{prop414-20}
\left(\rho^q\right)_{t} = \nabla \cdot \left( c_0 \rho^{\frac{q+m+1}{2}} \nabla \rho^{\frac{q+m-1}{2}} -  V \rho^q\right) - c_1 \abs{\nabla \rho^{\frac{q+m-1}{2}}}^2 - (q-1)\rho^q \,\nabla \cdot V .
\end{equation}
Because $\nabla \rho^{\frac{q+m-1}{2}}\in L_{x,t}^{2}$ and $\rho \in L_{x,t}^{q+m-1+\frac{2q}{d}}$ from \eqref{E:rho:q-1}, we observe via  $\frac{d(q+m-1)+q}{d(q+m-1)+2q} = \frac{d(q+m-1)}{2d(q+m-1)+4q} + \frac{1}{2}$ that
\begin{equation}\label{prop414-30}
\rho^{\frac{q+m-1}{2}} \nabla \rho^{\frac{q+m-1}{2}} \in L_{x,t}^{\frac{d(q+m-1)+2q}{d(q+m-1)+q}}.
\end{equation}
It is clear that $\abs{\nabla \rho^{\frac{q+m-1}{2}}}^2 \in L^{1}_{x,t}$.
Also, Remark~\ref{Remark : AL-2} and similar computations for $V\in \tilde{\calS}_{m,q}^{(\tilde{q}_1, \tilde{q}_2)}$ yield
\begin{equation}\label{prop414-40}
V\rho^q \in L_{x,t}^{\frac{d(m+q-1)+2q}{d(m+q-1)+q}} \quad \text{and} \quad
\rho^q \, \nabla\cdot V \in L^{1}_{x,t}.
\end{equation}
%where we used $V\in \tilde{\mathcal{S}}_{m,q}^{(\tilde{q}_1, \tilde{q}_2)}$.
%When $m, q > 1$, the membership $V\in \tilde{\mathcal{S}}_{m,q}^{(\tilde{q}_1, \tilde{q}_2)}$ provides $\nabla V \rho^q \in L^{1}_{x,t}$.
Furthermore, it holds that $V \nabla \rho^q \in L^{1}_{x,t}$ using $V \in \mathcal{S}_{m,q}^{(q_1, q_2)}$ and \eqref{prop414-5}.
Hence, owing to \eqref{prop414-30} and \eqref{prop414-40}, it follows that
$\partial_t \rho^q\in W_{x}^{-1, \frac{d(q+m-1)+2q}{d(q+m-1)+q} }L_{t}^{1}
$ in \eqref{prop414-10} or \eqref{prop414-20}. Furthermore, $\partial_t \rho^q$ satisfies \eqref{rho_t_2}.

$(ii)$ We follow the above proof to estimate $\partial_t \rho^q$. However, to estimate $\nabla \rho^q$ in \eqref{rho_t_3}, we employ \eqref{Lq-energy-r} for $1<r\leq q$ and observe that
\begin{equation}\label{E1:AL-2-r}
\nabla \rho^q = \frac{2q}{m+r-1} \rho^{q-\frac{m+r-1}{2}} \nabla \rho^{\frac{m+r-1}{2}} \in L_{x,t}^{\frac{2d(q+m-1)+4q}{d(3q-r)+2q}}
\end{equation}
as long as $q \geq \frac{r+m-1}{2} > \frac m2$ for $r\in (1, q]$. Then \eqref{E1:AL-2-r} holds
via $\frac{2qd-d(r+m-1)}{2d(q+m-1)+4q}+ \frac 12 =\frac{d(3q-r)+4q}{2d(q+m-1)+4q} $ because of $\nabla \rho^{\frac{m+r-1}{2}} \in L^{2}_{x,t}$ and $\rho \in L_{x,t}^{q+m-1+\frac{2q}{d}}$ from \eqref{Lq-energy-r}. This completes the proof.
\end{proof}

%----------------
\section{Existence of regular solutions }\label{splitting method}

We recall porous medium equations with divergence form of drifts
\begin{equation}\label{Fokker-Plank 2}
%\begin{aligned}
    \partial_t \rho  = \nabla\cdot (\nabla \rho^m - V \rho)\quad \text{for}\quad \rho(\cdot, 0)=\rho_0.\\
   % &=  \nabla\cdot \Big ( \big (\frac{\nabla \rho^m}{\rho} - V \big) \rho \Big ).
%\end{aligned}
\end{equation}
In this section, we construct a regular solution of \eqref{Fokker-Plank 2} when the initial data $\rho_0$ and the vector field $V$ are smooth enough.
For this, we exploit a splitting method in the Wasserstein space $\mathcal{P}_2(\bbr^d)$ and it turns out that our soultion is in the class of ablsolutely continuous curves in $\mathcal{P}_2(\mathbb{R}^d)$.
For carrying the splitting method, it requires a priori estimates, propagation of compact support, and H\"{o}lder continuity of the following form of homogeneous PME
\begin{equation}\label{H-PME1}
 \partial_t \varrho  = \Delta \varrho^m  \quad \text{ for } \quad \varrho(\cdot, 0)=\varrho_0.
\end{equation}

First, we deliver a priori estimates and compact supports of solutions of \eqref{H-PME1} in the following lemma.
\begin{lemma}\label{P:H-PME:e}
Suppose that $\varrho$ is a nonnegative $L^q$-weak solution of \eqref{H-PME1} in Definition~\ref{D:weak-sol}.
\begin{itemize}
  \item[(i)] Let $\int_{\mathbb{R}^d} \varrho_0 (1+\log \varrho_0 ) \,dx < \infty$. Then there exists a positive constant $c=c(m,d,p)$ such that
      \begin{equation}\label{P:H-PME:e1}
      \int_{\mathbb{R}^d}  \varrho (\cdot, T) \log \varrho(\cdot, T)  \,dx
      + \iint_{Q_T} \abs{\frac{\nabla \varrho^m}{\varrho}}^p \varrho \,dx\,dt
      \leq
      \int_{\mathbb{R}^d}\varrho_0 \log \varrho_0  \,dx +
      c T  \int_{\mathbb{R}^d}\varrho_0   \,dx.
      \end{equation}
      Moreover, there exists $T^\ast = T^\ast (m,d,p)$ such that, for any positive integer $l=1,2,3,\cdots$,
        \begin{equation}\label{C:H-PME:e1}
      \int_{\mathbb{R}^d}  \varrho (\cdot, lT^\ast) \log \varrho(\cdot, l T^\ast)  \,dx
      + \int_{(l-1)T^\ast}^{lT^\ast}\int_{\mathbb{R}^d} \abs{\frac{\nabla \varrho^m}{\varrho}}^p \varrho \,dx\,dt
      \leq
      \int_{\mathbb{R}^d} \left[ \varrho_0 \log \varrho_0  + \varrho_0 \right] \,dx .
      \end{equation}
  \item[(ii)]Let $\varrho_0 \in L^{q}(\bbr^d)$ for $q>1$. Then, we have
  \begin{equation}\label{P:H-PME:e2}
      \int_{\mathbb{R}^d} \varrho^q(\cdot, T) \,dx + \frac{4mq(q-1)}{(m+q-1)^2} \iint_{Q_T} \left| \nabla \varrho^{\frac{m+q-1}{2}} \right|^2  \,dx\,dt = \int_{\mathbb{R}^d} \varrho_0^q \,dx.
          \end{equation}
 \item[(iii)]   If $\varrho_0$ is compactly supported, then $\varrho(t)$ is compactly supported for all $t >0$.
  More precisely, suppose there exists $R_0>0$ such that $\supp(\varrho_0) \subset B_{R_0}(0):=\{x \in \bbr^d : |x| \leq R_0 \}$. Then, for all $t>0$, we have
  \begin{equation}\label{support}
  \supp\left( \varrho(t)\right) \subset B_{R(t)}(0), \quad R(t):=R_0\left(1+\beta t \right), ~~ \beta:=\frac{1}{d(m-1)+2}.
    \end{equation}
\end{itemize}
\end{lemma}
\begin{proof}
We refer \cite[Proposition~5.12]{Vaz07} which gives \eqref{P:H-PME:e2} directly.
Now we prove \eqref{P:H-PME:e1} and \eqref{C:H-PME:e1}.
First, we observe the following by testing $\log \rho$ to \eqref{H-PME} that
% is rewritten as
%  \begin{equation*}
%  \frac{d}{dt} u \log u = \Delta u^m \left(\log u + 1\right).
%  \end{equation*}
%  By taking the integration over $\mathbb{R}^d$ and the integration by parts leads the following:
\begin{equation}\label{H-est-01}
  \frac{d}{dt}\int_{\mathbb{R}^d} \varrho \log \varrho \,dx + \frac{4}{m} \int_{\mathbb{R}^d} \abs{\nabla \varrho^{m/2} }^2 \,dx = 0.
  \end{equation}
To above estimate, we take the integration in $t \in [0, T]$ and the combination of $p$-th moment estimate \eqref{W-p} with $2\varepsilon = 4/m$ yields \eqref{P:H-PME:e1}.

To obtain \eqref{C:H-PME:e1}, we begin from \eqref{P:H-PME:e1} for $T^\ast$ instead of $T$. Then we fix $T^\ast= T^\ast (m,d, p)$ small enough such that $cT^\ast \leq 1$.
For any positive integer $l$, let us set the interval $[(l-1)T^\ast,  l T^\ast]$. In each interval, we apply \eqref{P:H-PME:e1} to obtain
  \begin{equation*}
  \begin{aligned}
      \int_{\mathbb{R}^d}  \varrho (\cdot, lT^\ast) \log \varrho(\cdot, l T^\ast)  \,dx
      &+ \int_{(l-1)T^\ast}^{lT^\ast}\int_{\mathbb{R}^d} \abs{\frac{\nabla \varrho^m}{\varrho}}^p \varrho \,dx\,dt \\
      &\leq
      \int_{\mathbb{R}^d} \left[ \varrho (\cdot, (l-1)T^\ast \log \varrho(\cdot, (l-1)T^\ast)  + \varrho(\cdot, (l-1)T^\ast) \right] \,dx.
  \end{aligned}
  \end{equation*}
  Note that $\frac{d}{dt}\int_{\mathbb{R}^d} \varrho \log \varrho \,dx \leq 0$ that is trivially deduced from \eqref{H-est-01}.
  Then \eqref{C:H-PME:e1} is obtained because of the decreasing property of $\int_{\mathbb{R}^d} \varrho (\cdot, t) \log \varrho(\cdot, t) \,dx$  and the mass conservation property.

To obtain \eqref{support}, we exploit Proposition 9.18 of \cite{Vaz07} and have
 \begin{equation*}
  \supp\left( \varrho(t)\right) \subset B_{R(t)}(0), \quad R(t):=C \left( \left( \frac{R_0}{C}\right)^{\frac{1}{\beta}} +t\right)^\beta, ~~ \beta:=\frac{1}{d(m-1)+2},
    \end{equation*}
for some constant $C>0$ which is independent of $\varrho_0$. Since $0<\beta \leq \frac{1}{2}$, we note that the function $f(x):=x^\beta$ is concave and hence we have $f(x+t) \leq f(x) + f'(x) t$ for any $x,\,t>0$. That is,
\begin{equation}\label{eq 1 : support}
\begin{aligned}
 R(t)&=C \left( \left( \frac{R_0}{C}\right)^{\frac{1}{\beta}} +t\right)^\beta \leq C \left( \frac{R_0}{C} + \beta \left( \frac{R_0}{C}\right)^{\frac{\beta-1}{\beta}} t\right).
 \end{aligned}
    \end{equation}
    Without loss of generality, we may choose $R_0$ big so that $R_0 >C$. Then, from \eqref{eq 1 : support}, we have
    \begin{equation*}
%\begin{aligned}
 R(t) \leq  \left(R_0 + \beta R_0 t\right),
% \end{aligned}
    \end{equation*}
    which completes the proof.
\end{proof}

Second, we deliver H\"{o}lder continuity of a bounded nonnegative solution of \eqref{H-PME1} when the initial data is also H\"{o}lder continuous (say that $\varrho_0 \in \calC^{\alpha_0}$ for some $\alpha_0 \in (0,1)$). It is well known that bounded $L^q$-weak solutions of \eqref{H-PME1} are H\"{o}lder continuous up to the initial boundary with quantitatively determined H\"{o}lder exponents depending on given data (in particular with \eqref{H-PME1}, depending only upon $m$, $d$, and $\alpha_0$). For example, we refer  \cite{Aron70a, Aron70b, CF78, CF79, DF85, DF85A, DB83, DB93, DGV12, Vaz06}.
 Here, we introduce the following theorem that the H\"{o}lder exponent is more clearly specified as the minimum of intial H\"{o}lder exponent $\alpha_0$ and the interior H\"{o}lder exponent $\alpha^\ast = \alpha^\ast (m,d)$ (see Theorem~\ref{T:interior_Holder} for interior H\"{o}lder continuity). This result is probably well-known to experts but we could not find it in the literature. See Appendix~\ref{Appendix:Holder} for more details of proof.
   By the aid of Theorem~\ref{T:Boundary_Holder}, we are able to carry the splitting method because any weak solution of \eqref{H-PME1} keeps the same level of regularity at any time.
\begin{theorem}\label{T:Boundary_Holder}
Let $\varrho$ be a bounded nonnegative weak solution of \eqref{H-PME1} where the nonnegative initial data $\varrho_0$ is bounded and H\"{o}lder continuous with the exponent $\alpha_0 \in (0,1)$. Then there exists $\alpha^\ast = \alpha^\ast (m,d)\in (0,1)$ such that $\varrho$ is uniformly H\"{o}lder continuous on $Q_T$ with the exponent $\tilde{\alpha} = \min\{\alpha_0, \alpha^\ast\}$; that is, $\|\varrho\|_{\mathcal{C}^{\tilde{\alpha}}(Q_T)} \leq \gamma$ where $\gamma = \gamma ( m, d, \|\varrho_0\|_{L^{\infty}(\mathbb{R}^d)}) > 1$.
\end{theorem}

Our main result in this section reads as follow.
\begin{proposition}\label{proposition : regular existence}
 Let  $\tilde{\alpha} \in (0, 1)$ be the constant in Theorem \ref{T:Boundary_Holder}.
 Assume $\alpha \in (0, \tilde{\alpha}]$ and
$$
 {\rho}_0\in\mathcal{P}_2(\mathbb{R}^d)\cap \mathcal{C}^\alpha(\bbr^d) \quad \mbox{and} \quad
V \in L^1(0,T; W^{2,\infty}(\bbr^d))\cap
\calC^\infty (Q_T).
$$
%If ${\rho}_0\in\mathcal{P}_2(\mathbb{R}^d)\cap C^\alpha(\bbr^d) \cap L^\infty(\bbr^d)$,
Then there exists an absolutely continuous curve ${\rho}\in
AC(0,T;\mathcal{P}_2(\mathbb{R}^d))$ which is a solution of \eqref{Fokker-Plank 2}
%\begin{equation*}\label{eq1 : Theorem : Toy Fokker-Plank}
%\left\{\begin{matrix}
%   & \partial_t \rho =   \nabla \cdot( \nabla \rho^m  - V\rho),\\
%   & \rho(0,\cdot)={\rho}_0,
%    \end{matrix}
%    \right.
%\end{equation*}
in the sense of distributions and satisfies the followings:
\begin{itemize}
\item[(i)] For all $0\leq s<t\leq T$, we have
\begin{equation}\label{eq2 : Theorem : Toy Fokker-Plank}
W_2(\rho(s),\rho(t))\leq C\sqrt{t-s } +\int_s^t \|V(\tau)\|_{L^\infty_x}
\,d\tau %\qquad \forall ~~0\leq s< t\leq T.
\end{equation}
where the constant $C=C\left (\|\rho_0\|_{L^m (\bbr^d)}, \,
\|V\|_{L^1(0,T; W^{1,\infty}(\bbr^d))} \right )$.

\item[(ii)] For all $x,y \in \bbr^d$ and $t\in [0, T]$, we have
\begin{equation*}
\left | \rho(x,t)-\rho(y,t)\right | \leq C |x-y|^\alpha, %\quad \forall ~x,y \in \bbr^d, ~ t\in [0,T].
\end{equation*}
where the constant $C=C\left (\|\rho_0\|_{L^\infty (\bbr^d)}, \, \|V\|_{L^1(0,T; C^{1,\alpha}(\bbr^d))} \right )$.

\item[(iii)] $\bullet$ For $q >1$
\begin{equation}\label{KK-Feb19-100}
%\sup_{0\leq t \leq T}\int_{\bbr^d \times \{t\}} \rho^q \,dx +
\int_0^T \|\nabla \rho^{\frac{m+q-1}{2}}\|_{L^2_x}^2 dt \leq  C,
\end{equation}
where the constant $C=C\left (\|\rho_0\|_{L^{m+q-1}(\bbr^d)},~~
\|V\|_{L^1(0,T;W^{2,\infty}(\bbr^d))} \right )$.

$\bullet$ For $q=1$
\begin{equation}\label{KK-June21-10}
%\sup_{0\leq t \leq T}\int_{\bbr^d \times \{t\}} \rho^q \,dx +
\int_0^T \|\nabla \rho^{\frac{m}{2}}\|_{L^2_x}^2 dt \leq  C,
\end{equation}
where the constant $C=C\left (\int_{\bbr^d} \rho_0 \langle x \rangle^2  \,dx, \, \|\rho_0\|_{L^m(\bbr^d)}, \, \|V\|_{L^1(0,T;W^{2,\infty}(\bbr^d))} \right )$.
\begin{comment}
\begin{equation}\label{KK-Feb19-100}
    \begin{cases}
        \int_0^T \|\nabla \rho^{\frac{m+q-1}{2}}\|_{L^2_x}^2 dt \leq  C_2, & \text{if } q> 1, \\
        \int_0^T \|\nabla \rho^{\frac{m}{2}}\|_{L^2_x}^2 dt \leq  C_3, & \text{if } q = 1,
    \end{cases}
    \end{equation}

where the constant $C_2=C_2\left (\|\rho_0\|_{L^{m+q-1}(\bbr^d)},~~
\|V\|_{L^1(0,T;W^{2,\infty}(\bbr^d))} \right )$ and
$C_3=C_3\left (\int_{\bbr^d}| x |^2 \rho_0 \,dx,~~ \|\rho_0\|_{L^m(\bbr^d)},~~\|V\|_{L^1(0,T;W^{2,\infty}(\bbr^d))} \right )$

\begin{itemize}
\item For $q >1$
\begin{equation}\label{KK-Feb19-100}
%\sup_{0\leq t \leq T}\int_{\bbr^d \times \{t\}} \rho^q \,dx +
\int_0^T \|\nabla \rho^{\frac{m+q-1}{2}}\|_{L^2_x}^2 \,dt \leq  C_2,
\end{equation}
where the constant $C_2=C_2\left (\|\rho_0\|_{L^{m+q-1}(\bbr^d)},~~
\|V\|_{L^1(0,T;W^{2,\infty}(\bbr^d))} \right )$.
\item For $q=1$
\begin{equation}\label{KK-June21-10}
%\sup_{0\leq t \leq T}\int_{\bbr^d \times \{t\}} \rho^q \,dx +
\int_0^T \|\nabla \rho^{\frac{m}{2}}\|_{L^2_x}^2 \,dt \leq  C_3,
\end{equation}
where the constant $C_3=C_3\left (\int_{\bbr^d}| x |^2 \rho_0 \,dx,~~ \|\rho_0\|_{L^m(\bbr^d)},~~\|V\|_{L^1(0,T;W^{2,\infty}(\bbr^d))} \right )$.
\end{itemize}
\end{comment}

\item[(iv)] For all $t\in [0, T]$ and $q\in[1, \infty]$, we have
\begin{equation}\label{KK-Nov24-100}
\|\rho(t)\|_{L^q_x} \leq \|\rho_0\|_{L^q(\mathbb{R}^d)}e^{\frac{q-1}{q}\int_0^T
\|\nabla \cdot V \|_{L^\infty_x} \,d\tau} , %\qquad \forall ~~t\in[0,T].
\end{equation}
where $\frac{q-1}{q}=1$ if $q=\infty$.

\item[(v)] If $\rho_0$ is compactly supported, then $\rho$ is also compactly supported in space and time.
More precisely, suppose there exists $R_0>0$ such that $\supp(\varrho_0) \subset B_{R_0}(0):=\{x \in \bbr^d : \abs{x} \leq R_0 \}$. Then, for all $ t \in [0,T]$, we have
  \begin{equation}\label{support-uniform}
 \supp \left( \rho(t)\right) \subset B_{R_T}(0), \quad R_T:=e^{\beta T} \left( R_0 + \int_0^T \|v(t) \|_{L_x^\infty} \,dt \right ), \quad  \beta:=\frac{1}{d(m-1)+2}.
    \end{equation}
\end{itemize}
\end{proposition}

\subsection{Splitting method}
 In this subsection, we introduce the splitting method and construct two sequence of curves in
$\mathcal{P}_2(\mathbb{R}^d)$ which are approximate solutions of \eqref{Fokker-Plank 2} in $Q_{T}$,
\begin{equation*}
\begin{cases}
\partial_t \rho =   \nabla \cdot( \nabla \rho^m -V\rho),\\
\rho(\cdot,0)=\rho_0 \in \mathcal{P}_2(\mathbb{R}^d).
\end{cases}
\end{equation*}
For each $n\in \mathbb{N},$ we define approximated solutions
$\varrho_n,~ \rho_n:[0,T]\mapsto \mathcal{P}_2(\mathbb{R}^d)$ as follows;
\begin{itemize}

\item
 For $t\in [0,\frac{T}{n}]$, we define $\varrho_n:[0,\frac{T}{n}]\mapsto \mathcal{P}_2(\mathbb{R}^d)$ as the solution of
 \begin{equation*} %\label{eq 1 : splitting}
\begin{cases}
   \partial_t \varrho_n =   \Delta (\varrho_n)^m , \\
   \varrho_n(\cdot,0)=\rho_0.
    \end{cases}
\end{equation*} %\label{eq 2 : splitting}
We also define  $\rho_n:[0,\frac{T}{n}]\mapsto \mathcal{P}_2(\mathbb{R}^d)$ as follows
\begin{equation*}
\rho_n(t)= \Psi(t;0, \varrho_n(t)), \quad \forall ~ t \in [0, \frac T n].
\end{equation*}
 Here, we recall \eqref{Flow on Wasserstein} for the definition of $\Psi$. Note that we assume $ V \in L^1(0,T; W^{2,\infty}(\bbr^d))\cap
\calC^\infty (Q_T) $ in Proposition \ref{proposition : regular existence}, and hence this flow map $\Psi$ used for the proof of Proposition
\ref{proposition : regular existence} is well defined.

\item For $t\in (\frac{T}{n} ,\frac{2T}{n}]$,  we define $\varrho_n:(\frac{T}{n} ,\frac{2T}{n}]\mapsto \mathcal{P}_2(\mathbb{R}^d)$ as the solution of
 \begin{equation*} %\label{eq 3 : splitting}
\begin{cases}
   \partial_t \varrho_n =   \Delta (\varrho_n)^m , \\
    \varrho_n(\cdot, \frac Tn)=\rho_n( \frac T n).
\end{cases}
\end{equation*} %\label{eq 4 : splitting}
We also define  $\rho_n:(\frac{T}{n} ,\frac{2T}{n}]\mapsto \mathcal{P}_2(\mathbb{R}^d)$ as follows
\begin{equation*}
\rho_n(t)= \Psi(t; T/n, \varrho_n(t)), \quad \forall ~ t \in (\frac{T}{n} ,\frac{2T}{n}].
\end{equation*}

\item In general, for each $i=1, \ldots, n-1$ and $t\in (\frac{iT}{n} ,\frac{(i+1)T}{n}]$, we define $\varrho_n:(\frac{iT}{n} ,\frac{(i+1)T}{n}]\mapsto \mathcal{P}_2(\mathbb{R}^d)$ as the solution of
 \begin{equation*} %\label{eq 5 : splitting}
\begin{cases}
    \partial_t \varrho_n =   \Delta (\varrho_n)^m,  \\
    \varrho_n(\cdot, \frac{iT}{n})=\rho_n( \frac{iT}{n}).
\end{cases}
\end{equation*} %\label{eq 6 : splitting}
We also define  $\rho_n:(\frac{iT}{n} ,\frac{(i+1)T}{n}]\mapsto \mathcal{P}_2(\mathbb{R}^d)$ as follows
\begin{equation*}
\rho_n(t)= \Psi(t; iT/n, \varrho_n(t)), \quad \forall ~ t \in (\frac{iT}{n} ,\frac{(i+1)T}{n}].
\end{equation*}
\end{itemize}

 Now, we investigate some properties useful for the proof of Proposition \ref{proposition : regular existence}.
%In the following lemma, we investigate various properties of curves $\rho_n$ and $\varrho_n$.
\begin{lemma}\label{Lemma:AC-curve}
Let  $V \in L^1(0,T; W^{1,\infty}(\bbr^d))$ and $\rho_0 \in \mathcal{P}_2(\bbr^d) \cap  \mathcal{C}^\alpha(\bbr^d)$ for some $\alpha \in (0,1)$.
Suppose that $\rho_n,\varrho_n:[0,T]\mapsto \mathcal{P}_2(\mathbb{R}^d)$ are curves defined as above with the initial data $\varrho_n(0),\rho_n(0):=\rho_0 $.
Then, these curves satisfy the following properties;
\begin{itemize}
\item[(i)] For all $t\in [0,T],$ we have
\begin{equation}\label{eq29:Lemma:AC-curve}
\|\varrho_{n}(t)\|_{L^q_x}, \  \|\rho_{n}(t)\|_{L^q_x} \leq \|\rho_0\|_{L^q(\bbr^d)}
e^{\frac{q-1}{q}\int_0^T \|\nabla \cdot V\|_{L^\infty_x} \,d\tau}, \qquad \forall ~  q \geq 1.
\end{equation}

\item[(ii)] For all $s,t \in [0,T],$ we have
\begin{equation}\label{eq20:Lemma:AC-curve}
\begin{aligned}
W_2(\rho_{n}(s), \rho_{n}(t)) &\leq C \sqrt{t-s}+ \int_s^t \|V(\tau)\|_{L^\infty_x} \,d\tau  ,
\end{aligned}
\end{equation}
where $C= C(\int_0^T \|\nabla V\|_{L^\infty_x}  \,d\tau , \|\rho_0\|_{L^m (\bbr^d)})$

\item[(iii)] For all $t \in [0,T],$ we have
\begin{equation}\label{eq21:Lemma:AC-curve}
\begin{aligned}
W_2(\rho_{n}(t), {\varrho}_{n}(t)) \leq  \max_{\{i=0,1,\dots,n\}}\int_{\frac{iT}{n}}^{\frac{(i+1)T}{n}} \|V(\tau)\|_{L^\infty_x} \,d\tau.
\end{aligned}
\end{equation}

\item[(iv)] Suppose there exists $R_0>0$ such that $\supp(\varrho_0) \subset B_{R_0}(0)$. Then, for all $ t \in [0,T]$, we have
  \begin{equation}\label{support-n}
 \supp \left( \rho_n(t)\right) \subset B_{R_{n,T}}(0), \quad R_{n,T}:= R_0\left(1+\beta \frac{T}{n} \right )^n + \left( 1 + \beta \frac{T}{n}\right )^{n-1} \int_0^T \|V(t) \|_{L_x^\infty} \,dt,
    \end{equation}
    where $\beta:=\frac{1}{d(m-1)+2}$.
\end{itemize}
\end{lemma}

\begin{proof}
%For convenience, we use the notation $\|\cdot\|_p := \|\cdot\|_{L^p(\bbr^d)}$ and $\varrho_t:= \varrho(t)$.
\emph{(i)} When $q=1$, \eqref{eq29:Lemma:AC-curve} is true because of the mass conservation property of $\rho$ and $\varrho$. So we only prove for case $q>1$.
Let $t\in \left [0,\frac{T}{n} \right]$ be given. From Lemma \ref{P:H-PME:e} $(ii)$, we have
\begin{equation}\label{eq30:Lemma:AC-curve}
\|\varrho_n(t)\|_{L^q_x} \leq \|\rho_0\|_{L^q(\bbr^d)},   \qquad \forall ~ q > 1  .
\end{equation}
%here, $\varrho_n(t):=\varrho^n(t)$.
We combine \eqref{L^p relation}  and \eqref{eq30:Lemma:AC-curve} to get
\begin{equation}\label{eq31:Lemma:AC-curve}
%\begin{aligned}
\|\rho_n(t)\|_{L^q_x} \leq  \|\varrho_n(t)\|_{L^q_x} ~e^{\frac{q-1}{q}\int_0^t\|\nabla\cdot V\|_{L^\infty_x} \,d\tau}
\leq \|\rho_0\|_{L^q(\mathbb{R}^d)} ~e^{\frac{q-1}{q}\int_0^t\|\nabla\cdot V\|_{L^\infty_x} \,d\tau}.
%\end{aligned}
\end{equation}
Similarly, if $t\in (\frac{iT}{n},\frac{(i+1)T}{n}]$ for $1 \leq i \leq n-1$, then we apply \eqref{eq30:Lemma:AC-curve} and
\eqref{eq31:Lemma:AC-curve} inductively to get
\begin{equation*}
\|\varrho_n(t)\|_{L^q_x} \leq \|\rho_0\|_{L^q(\mathbb{R}^d)} ~ e^{\frac{q-1}{q}\int_0^{\frac{iT}{n}}\|\nabla\cdot V\|_{L^\infty_x} \,d\tau},
\end{equation*}
and
\begin{equation*}
\|\rho_n(t)\|_{L^q_x} \leq \|\rho_0\|_{L^q(\mathbb{R}^d)} ~ e^{\frac{q-1}{q}\int_0^{\frac{(i+1)T}{n}}\|\nabla\cdot V\|_{L^\infty_x} \,d\tau},
\end{equation*}
which prove \eqref{eq29:Lemma:AC-curve}.% for $r=q$. Similarly, we prove the case $r<q$.

\emph{(ii)} Now, we prove $AC$-curve property of $\rho_n$. Let $0\leq s\leq t \leq T$ be given.

$\bullet$ First, we assume $s,t\in[\frac{iT}{n},\frac{(i+1)T}{n}]$ for some $0\leq i\leq n-1$. From the triangle inequality, we have
\begin{equation}\label{eq1:Lemma:AC-curve}
\begin{aligned}
W_2(\rho_n(s), \rho_n(t))&\leq W_2(\rho_n(s), \Psi(s; iT/n, \varrho_n(t))) +  W_2(\Psi(s; iT/n, \varrho_n(t)), \rho_n(t))\\
&= W_2(\Psi(s; iT/n, \varrho_n(s)), \Psi(s; iT/n, \varrho_n(t)))
+  W_2(\Psi(s; iT/n, \varrho_n(t)),\Psi(t; iT/n, \varrho_n(t)))\\
&\leq e^{\int_{iT/n}^s \mbox{Lip}(v_\tau) \,d\tau}W_2(\varrho_n(s),\varrho_n(t))
 +  W_2(\Psi(s; iT/n, \varrho_n(t)),\Psi(t; iT/n, \varrho_n(t))),
\end{aligned}
\end{equation}
where we use \eqref{eq1:Corollary-4:Lipschitz} to get the last inequality. We note that Lemma \ref{representation of AC curves} gives us
\begin{equation}\label{eq2:Lemma:AC-curve}
\begin{aligned}
W_2({\varrho}_n(s), {\varrho}_n(t))
&\leq \int_s^t \Big(\int_{\mathbb{R}^d}\frac{|\nabla \big(\varrho_{n}\big)^m|^2}{\varrho_{n}}~ \,dx\Big)^{\frac{1}{2}} ~\,d\tau.
\end{aligned}
\end{equation}
Next, we estimate the second term as follows
\begin{equation}\label{eq10:Lemma:AC-curve}
\begin{aligned}
  W_2^2(\Psi(s; iT/n, \varrho_n(t)),\Psi(t; iT/n, \varrho_n(t)))
  &= W_2^2(\psi(s; iT/n,\cdot)_\#{\varrho}_n(t), \psi(t; iT/n,\cdot)_\#{\varrho}_n(t)) \\
  &\leq \int_{\mathbb{R}^d} |\psi(s;iT/n,x)-\psi(t;iT/n,x)|^2 \varrho_n(x,t) \,dx\\
  &\leq \int_{\mathbb{R}^d} \left(\int_{s}^t \|V(\tau)\|_{L^\infty_x} \,d\tau\right)^2 \varrho_n(x,t)\,dx\\
  &= \Big (\int_{s}^t \|V(\tau)\|_{L^\infty_x} \,d\tau \Big)^2.
\end{aligned}
\end{equation}
Finally, we combine (\ref{eq1:Lemma:AC-curve}), (\ref{eq2:Lemma:AC-curve})  and \eqref{eq10:Lemma:AC-curve} to get
\begin{equation}\label{eq3:Lemma:AC-curve}
\begin{aligned}
W_2(\rho_n(s), \rho_n(t)) &\leq  e^{\int_{iT/n}^s \mbox{Lip}(V(\tau)) \,d\tau}
\int_s^t \Big(\int_{\mathbb{R}^d}\frac{|\nabla \big(\varrho_n \big)^m|^2}{\varrho_n }~ \,dx\Big)^{\frac{1}{2}} ~\,d\tau
 + \int_s^t\|V(\tau)\|_{L^\infty_x} \,d\tau.
 %&\leq e^{C\int_{iT/n}^s \|A_\tau\|_\infty \,d\tau}|\nabla_-F|({\varrho}^n(s))\int_s^t e^{|\lambda| \tau } \,d\tau+ C\int_s^t\|A_\tau\|_\infty \,d\tau.
\end{aligned}
\end{equation}

$\bullet$ Now, we consider the general case  $s \in [\frac{iT}{n},\frac{(i+1)T}{n}]$ and $t \in [\frac{(i+k)T}{n},\frac{(i+k+1)T}{n}]$.
Triangle inequality gives us
\begin{equation}\label{eq4:Lemma:AC-curve}
\begin{aligned}
W_2(\rho_n(s), \rho_n(t))&\leq W_2 \big(\rho_n(s), \rho_n({(i+1)T/n}) \big)
+ \sum_{j=i+1}^{i+k-1}W_2 \big(\rho_n({jT/n}) , \rho_n({(j+1)T/n}) \big)\\
&\quad + W_2 \big(\rho_n({(i+k)T/n}) , \rho_n(t) \big ).
\end{aligned}
\end{equation}
We apply (\ref{eq3:Lemma:AC-curve}) to each term in (\ref{eq4:Lemma:AC-curve}), and get
\begin{equation}\label{eq19:Lemma:AC-curve}
\begin{aligned}
W_2(\rho_n(s), \rho_n(t)) %&\leq
%e^{\int_{iT/n}^s \mbox{Lip}(v_\tau) \,d\tau} \int_s^{\frac{(i+1)T}{n}}
%\Big(\int_{\mathbb{R}^d}\frac{|\nabla \big(\varrho_n(\tau)\big)^m|^2}{\varrho_n(\tau)}~ \,dx\Big)^{\frac{1}{2}} ~\,d\tau
% + \int_s^{\frac{(i+1)T}{n}}\|v_\tau\|_\infty \,d\tau\\
%& \quad + \sum_{j=i+1}^{i+k-1}\left[e^{\int_{jT/n}^{(j+1)T/n} \mbox{Lip}(v_\tau) \,d\tau} \int_{\frac{jT}{n}}^{\frac{(j+1)T}{n}}
% \Big(\int_{\mathbb{R}^d}\frac{|\nabla \big(\varrho_n(\tau)\big)^m|^2}{\varrho_n(\tau)}~ \,dx\Big)^{\frac{1}{2}}~\,d\tau
% + \int_{\frac{jT}{n}}^{\frac{(j+1)T}{n}}\|v_\tau\|_\infty \,d\tau\right]\\
% &\quad +  \int_{\frac{(i+k)T}{n}}^t \Big(\int_{\mathbb{R}^d}\frac{|\nabla \big(\varrho_n(\tau)\big)^m|^2}{\varrho_n(\tau)}~ \,dx\Big)^{\frac{1}{2}} ~\,d\tau
% + \int_{\frac{(i+k)T}{n}}^t\|v_\tau\|_\infty \,d\tau\\
&\leq   e^{\int_0^T\mbox{Lip}(V (\tau)) \,d\tau}
\int_{s}^t \Big(\int_{\mathbb{R}^d}\frac{|\nabla \big(\varrho_n\big)^m|^2}{\varrho_n}~ \,dx\Big)^{\frac{1}{2}} ~\,d\tau+ \int_s^t \|V (\tau) \|_{L^\infty_x} \,d\tau  .
\end{aligned}
\end{equation}
We note that
\begin{equation}\label{eq22:Lemma:AC-curve}
\int_{s}^t \Big(\int_{\mathbb{R}^d}\frac{|\nabla \big(\varrho_n\big)^m|^2}{\varrho_n}~ \,dx\Big)^{\frac{1}{2}} ~\,d\tau
\leq \sqrt{t-s} \Big(\int_s^t \int_{\mathbb{R}^d}\frac{|\nabla \big(\varrho_n\big)^m|^2}{\varrho_n }~ \,dx ~\,d\tau \Big)^{\frac{1}{2}}.
\end{equation}

Now, we estimate the right hand side of \eqref{eq22:Lemma:AC-curve} as follows
\begin{equation*} %\label{eq23:Lemma:AC-curve}
\begin{aligned}
\int_{s}^t \int_{\mathbb{R}^d}\frac{|\nabla \big(\varrho_n\big)^m|^2}{\varrho_n}\,dx \,d\tau
&=\int_s^{\frac{(i+1)T}{n}} \int_{\mathbb{R}^d}\frac{|\nabla \big(\varrho_n\big)^m|^2}{\varrho_n} \,dx \,d\tau
 + \sum_{j=i+1}^{i+k-1}\int_{\frac{jT}{n}}^{\frac{(j+1)T}{n}}
 \int_{\mathbb{R}^d}\frac{|\nabla \big(\varrho_n\big)^m|^2}{\varrho_n } \,dx \,d\tau\\
 &\quad + \int_{\frac{(i+k)T}{n}}^t \int_{\mathbb{R}^d}\frac{|\nabla \big(\varrho_n \big)^m|^2}{\varrho_n }\,dx \,d\tau\\
&\leq \frac{1}{m-1}\Big(\int_{\mathbb{R}^d} \big[\varrho_n(s)\big]^m \,dx - \int_{\mathbb{R}^d} \big [\varrho_n((i+1)T/n)\big]^m \,dx   \Big)\\
&\quad + \frac{1}{m-1}\sum_{j=i+1}^{i+k-1}\Big( \int_{\mathbb{R}^d} \big [\rho_n(jT/n) \big]^m  \,dx-\int_{\mathbb{R}^d} \big [\varrho_n((j+1)T/n)\big]^m \,dx  \Big)\\
&\quad + \frac{1}{m-1}\Big(  \int_{\mathbb{R}^d} \big[\rho_n((i+k)T/n)\big]^m \,dx-\int_{\mathbb{R}^d} \big[ \varrho_n(t)\big]^m \,dx  \Big),
\end{aligned}
\end{equation*}
where we exploit \eqref{P:H-PME:e2} for the inequality.
Due to \eqref{L^p relation}, we have
\begin{equation}\label{eq24:Lemma:AC-curve}
\begin{aligned}
&(m-1)\int_{s}^t \int_{\mathbb{R}^d}\frac{|\nabla \big(\varrho_n\big)^m|^2}{\varrho_n}~ \,dx \,d\tau\\
%\Big(\int_{\mathbb{R}^d} \varrho^m_n(s) \,dx- \int_{\mathbb{R}^d} \varrho^m_{n, \frac{(i+1)T}{n}} \,dx   \Big)
%+ \sum_{j=i+1}^{i+k-1}\Big( \int_{\mathbb{R}^d} \rho^m_{n, \frac{jT}{n}}  \,dx-\int_{\mathbb{R}^d} \varrho^m_{n,\frac{(j+1)T}{n}} \,dx  \Big)\\
%&+ \Big(  \int_{\mathbb{R}^d} \rho^m_{n, \frac{(i+1)T}{n}} \,dx-\int_{\mathbb{R}^d} \varrho^m_n(t) \,dx \Big)\\
&\leq \Big(\int_{\mathbb{R}^d} \big [\varrho_n(s)\big]^m \,dx- \int_{\mathbb{R}^d} \big [\varrho_n((i+1)T/n)\big ]^m \,dx   \Big)\\
&\quad + \sum_{j=i+1}^{i+k-1}\Big( e^{(m-1)\int_{\frac{jT}{n}}^{\frac{(j+1)T}{n}} \|\nabla \cdot V \|_{L^\infty_x} \,d\tau}
\int_{\mathbb{R}^d} \big[\varrho_n(jT/n) \big]^m  \,dx-\int_{\mathbb{R}^d} \big[\varrho_n((j+1)T/n)\big]^m \,dx  \Big)\\
&\quad + \Big( e^{(m-1)\int_{(i+k)T/n}^{t} \|\nabla \cdot V \|_{L^\infty_x} \,d\tau}
\int_{\mathbb{R}^d} \big[\varrho_n((i+k)T/n)\big]^m \,dx-\int_{\mathbb{R}^d} \big [\varrho_n(t)\big]^m \,dx \Big)\\
&= \Big(\int_{\mathbb{R}^d}  \big [\varrho_n(s)\big]^m  \,dx -\int_{\mathbb{R}^d}  \big [\varrho_n(t)\big]^m \,dx \Big)
 + \sum_{j=i+1}^{i+k-1}\Big(e^{(m-1)\int_{\frac{jT}{n}}^{\frac{(j+1)T}{n}} \|\nabla \cdot V \|_{L^\infty_x} \,d\tau}-1 \Big)
\int_{\mathbb{R}^d} \big[\varrho_n(jT/n)\big]^m  \,dx\\
&\leq \Big(\int_{\mathbb{R}^d}  \big [\varrho_n(s)\big]^m  \,dx -\int_{\mathbb{R}^d}  \big [\varrho_n(t)\big]^m \,dx \Big)\\
&\quad  + e^{(m-1)\int_{0}^{T} \|\nabla \cdot V \|_{L^\infty_x} \,d\tau}\sum_{j=i+1}^{i+k-1}
\Big((m-1)\int_{\frac{jT}{n}}^{\frac{(j+1)T}{n}} \|\nabla \cdot V \|_{L^\infty_x} \,d\tau \Big)\int_{\mathbb{R}^d} \big[\varrho_n(jT/n)\big]^m  \,dx\\
&\leq \left [(m-1) e^{(m-1)\int_{0}^{T} \|\nabla \cdot V \|_{L^\infty_x} \,d\tau} \int_{0}^{T} \|\nabla \cdot V \|_{L^\infty_x} \,d\tau +1\right ]
\sup_{0\leq \tau \leq T} \int_{\mathbb{R}^d} \big[\varrho_n(\tau)\big ]^m  \,dx\\
&\leq \left [(m-1) e^{(m-1)\int_{0}^{T} \|\nabla \cdot V \|_{L^\infty_x} \,d\tau} \int_{0}^{T} \|\nabla \cdot V \|_{L^\infty_x} \,d\tau +1\right ]
e^{(m-1)\int_{0}^{T} \|\nabla \cdot V \|_{L^\infty_x} \,d\tau} \int_{\mathbb{R}^d} {\varrho}_0^m \,dx.
\end{aligned}
\end{equation}
Combining \eqref{eq19:Lemma:AC-curve}-\eqref{eq24:Lemma:AC-curve}, we have
\begin{equation*} %\label{eq25:Lemma:AC-curve}
\begin{aligned}
W_2(\rho_n(s), \rho_n(t)) &\leq C \sqrt{t-s}+ \int_s^t \|V(\tau)\|_{L^\infty_x} \,d\tau,
\end{aligned}
\end{equation*}
where the constant $C= C (\int_0^T \|\nabla V \|_{L^\infty_x} \,d\tau, \|{\rho}_0\|_{L^m(\bbr^d)} )$.  %only depends on
%$$\int_0^T \|\nabla V \|_{L^\infty_x} \,d\tau ~~~~~~~~\mbox{and}~~~~~~~~~~~~~~~~ \|{\rho}_0\|_{L^m(\bbr^d)} .$$
 This concludes the proof of (\ref{eq20:Lemma:AC-curve}).

%%%%%%%%%%%%%%%%%%%%%%%%%%%%%%%%%%%%%%%%%%%%%%%%%%%%%%%%%%%%%%%%%%%%%%%%%%%%%%%%%%%%%%%%%%%%%%%%%%%%%%%%%%%%%%%%%%%%%%%%%%%%%%%%%%%%%%%

\emph{(iii)} Next, we estimate the distance between $\varrho_n$ and $\rho_n$. For any $t\in[0,T]$, there exists $i$ such that
$t\in[\frac{iT}{n},\frac{(i+1)T}{n}]$. Then
\begin{equation*}
\begin{aligned}
 W_2^2({\varrho}_n(t), \rho_n(t))&= W_2^2({\varrho}_n(t), \Psi(t; iT/n,\cdot)_\#{\varrho}_n(t)) \\
 &\leq \int_{\bbr^d} |x-\psi(t; iT/n,x)|^2 \varrho_n(x,t)~\,dx\\
  &\leq \int_{\bbr^d} \Big (\int_{\frac{iT}{n}}^t \|V(\tau)\|_{L^\infty_x} \,d\tau \Big )^2 \varrho_n(x,t)~\,dx\\
  &= \Big (\int_{\frac{iT}{n}}^t \|V(\tau)\|_{L^\infty_x} \,d\tau \Big )^2,
\end{aligned}
\end{equation*}
which gives us (\ref{eq21:Lemma:AC-curve}).

 \emph{(iv)}
Let $t\in \left [0,\frac{T}{n} \right]$ be given. From \eqref{support}, we have
\begin{equation*} %\label{eq 2 : support}
\supp(\varrho_n(t)) \subset B_{R(t)}(0), \quad R(t):=R_0(1+\beta t).
\end{equation*}
%here, $\varrho_n(t):=\varrho^n(t)$.
Next, it is easy to check
\begin{equation*} %\label{eq 3 : support}
\begin{aligned}
\supp(\rho_n(t)) &\subset \supp(\varrho_n(t)) + B_{D(t)}(0), && D(t):=\int_0^t \|V(\tau)\|_{L_x^\infty} \,d\tau \\
&\subset  B_{R(t)}(0),  && R(t):=R_0(1+\beta t) + \int_0^t \|V(\tau)\|_{L_x^\infty} \,d\tau\\
& \subset B_{R_1}(0),  && R_1:=R_0 \left(1+\beta \frac{T}{n}\right) + \int_0^{\frac{T}{n}} \|V(t)\|_{L_x^\infty} \,dt.
\end{aligned}
\end{equation*}
Similarly, for $t\in \left [\frac{T}{n},\frac{2T}{n} \right]$, we have
\begin{equation*} %\label{eq 4 : support}
\supp(\varrho_n(t)) \subset B_{R(t)}(0), \qquad R(t):=R_1 \left(1+\beta \big(t-\frac{T}{n}\big)\right),
\end{equation*}
and
\begin{equation*} %\label{eq 5 : support}
\begin{aligned}
\supp(\rho_n(t)) %&\subset \supp(\varrho_n(t)) + B_{D(t)}(0), \quad D(t):=\int_0^t \|V(\tau)\|_{L_x^\infty} \,d\tau \\
%&\subset  B_{R(t)}(0), \quad R(t):=R_0(1+\beta t) + \int_0^t \|V(\tau)\|_{L_x^\infty} \,d\tau\\
& \subset B_{R_2}(0), \qquad R_2:=R_1 \left(1+\beta \frac{T}{n}\right) + \int_{\frac{T}{n}}^{\frac{2T}{n}} \|V(t)\|_{L_x^\infty} \,dt.
\end{aligned}
\end{equation*}
Note that
\begin{equation*}
\begin{aligned}
R_2 & = R_0 \left(1+\beta \frac{T}{n}\right)^2 + \left(1+\beta \frac{T}{n}\right)\int_{0}^{\frac{T}{n}} \|V(t)\|_{L_x^\infty} \,dt + \int_{\frac{T}{n}}^{\frac{2T}{n}} \|V(t)\|_{L_x^\infty} \,dt \\
& \leq R_0 \left(1+\beta \frac{T}{n}\right)^2 + \left(1+\beta \frac{T}{n}\right)\int_{0}^{\frac{2T}{n}} \|V(t)\|_{L_x^\infty} \,dt .
\end{aligned}
\end{equation*}
By continuing the process, for $t\in \left [\frac{(i-1)T}{n},\frac{iT}{n} \right]$ and $i=1, 2, \dots, n$, we have
\begin{equation*} %\label{eq 6 : support}
\supp(\rho_n(t))  \subset B_{R_i}(0), \quad R_i:=R_0 \left(1+\beta \frac{T}{n}\right)^i + \left(1+\beta \frac{T}{n}\right)^{i-1}\int_{0}^{\frac{iT}{n}} \|V(t)\|_{L_x^\infty} \,dt.
\end{equation*}
Since $R_i$ is increasing w.r.t $i$, for any $t\in [0, T],$ we have $\supp(\rho_n(t)) \subset B_{R_n}(0)$ which gives \eqref{support-n}. This completes the proof.
\end{proof}

%{\color{red}
\begin{lemma}\label{Lemma:H-curve}
Let $V \in L^1(0,T : W^{2,\infty}(\bbr^d)) \cap C^\infty(\bar{Q}_T)$ and $\rho_0 \in \mathcal{P}_2(\mathbb{R}^d)\cap \mathcal{C}^\alpha(\bbr^d)$ for some $\alpha \in (0,1)$.
Suppose that $\rho_n,\varrho_n:[0,T]\mapsto
\mathcal{P}_2(\mathbb{R}^d)$ are curves defined as in Lemma
\ref{Lemma:AC-curve} with the initial data
$\varrho_n(0),\rho_n(0):=\rho_0$. Then, we have
\begin{equation}\label{main : Lemma:H-curve}
\int_0^T \|\nabla (\rho_{n})^{\frac{m+q-1}{2}}\|_{L^2_x}^2 dt\leq C
\left(\|\rho_0\|_{L^{m+q-1}(\bbr^d)}, ~\| V\|_{ L^1(0,T : W^{2,\infty}(\bbr^d))}\right), \quad \forall \ q > 1.
\end{equation}
\end{lemma}
%\footnote{KK:{\color{magenta}  $L^1(0,T : W^{2,\infty}(\bbr^d))$ must be $\norm{V}_{L^1(0,T : W^{2,\infty}(\bbr^d))}$}}
\begin{proof}
For notational convenience, we use $\rho$ and $\varrho$ instead of $\rho_n$ and $\varrho_n$, respectively.
 First, we note from \eqref{eq29:Lemma:AC-curve}
%\begin{equation}\label{eq6 : Lemma : H-curve}
%\begin{aligned}
% \|\varrho_n(t)\|_{L^m_x},~~\|\rho_n(t)\|_{L^m_x} &\leq \|\rho_0\|_{L^m (\bbr^d)}~e^{\frac{m-1}{m} \int_0^T\|\nabla V\|_{L^\infty_x} \,d\tau}, \quad \forall ~ t \in [0, T]\\
%  \|\varrho_n(t)\|_{L^{2m-1}},~~\|\rho_n(t)\|_{L^{2m-1}} &\leq \|\rho_0\|_{L^{2m-1}}~e^{\frac{2m-2}{2m-1} \int_0^T\|\nabla V\|_{L^\infty_x} \,d\tau}.
% \end{aligned}
%\end{equation}
\begin{equation}\label{eq6 : Lemma : H-curve}
\begin{aligned}
 \|\varrho(t)\|_{L^q_x}%,~~ \|\varrho_n(t)\|_{L^{2m-1}}
 \leq \|\rho_0\|_{L^q(\bbr^d)}~e^{ \int_0^T\|\nabla \cdot V\|_{L^\infty_x} \,d\tau}, \quad \forall ~ t \in [0, T].
 \end{aligned}
\end{equation}
Next, \eqref{P:H-PME:e2} gives us
\begin{equation*}
\|\varrho(T/n)\|_{L^q_x}^q + \frac{4q}{m+q-1}\int_0^{\frac{T}{n}} \| \nabla \varrho^{\frac{m+q-1}{2}}\|_{L^2_x}^2 \,dt =\|\rho_0\|_{L^q (\mathbb{R}^d)}^q,
\end{equation*}
and we have from \eqref{eq5 : Sobolev}
\begin{equation}\label{eq14 : Lemma : H-curve}
\begin{aligned}
 \|\nabla \rho^{\frac{m+q-1}{2}}\|_{L^2_x}^2 &\leq e^{(m+q+3) \int_0^{\frac{T}{n}} \|\nabla  V\|_{L^\infty_x}\, d\tau} \\
 & \quad \times \left \{ \|\nabla \varrho^{\frac{m+q-1}{2}}\|_{L^2_x}^2
 + \| \varrho^{\frac{m+q-1}{2}}\|_{L^2_x}^2  \times \left ( \frac{m+q-1}{2}\int_0^{\frac{T}{n}} \| \nabla^2 V\|_{L^\infty_x} \,d\tau\right )^2 \right \}.%\\
% & \leq  e^{2(m+\frac{3}{2}) \int_0^{\frac{T}{n}} \|\nabla V\|_{L^\infty_x}\, d\tau}\left \{ \|\nabla \varrho^{m-\frac{1}{2}}\|_{L^2_x}^2
% + C \right \}
\end{aligned}
\end{equation}
Combining the fact $\|\rho_0\|_{L^1 (\bbr^d)}=1$ and \eqref{eq6 : Lemma :
H-curve} with \eqref{eq14 : Lemma : H-curve}, we have
\begin{equation}\label{eq9 : Lemma : H-curve}
\begin{aligned}
 \|\nabla \rho^{\frac{m+q-1}{2}}\|_{L^2_x}^2
 & \leq  e^{(m+q+3) \int_0^{\frac{T}{n}} \|\nabla V\|_{L^\infty_x}\, d\tau}\left \{ \|\nabla \varrho^{\frac{m+q-1}{2}}\|_{L^2_x}^2  + C \right \},
\end{aligned}
\end{equation}
where the constant $C=C\left (\|\rho_0\|_{L^{m+q-1}(\bbr^d)}, ~ \|V\|_{L^1(0,T ; W^{2,\infty}(\bbr^d))} \right )$. This gives us
\begin{equation}\label{eq10 : Lemma : H-curve}
\begin{aligned}
\int_0^{\frac{T}{n}} \|\nabla \rho^{\frac{m+q-1}{2}}\|_{L^2_x}^2  & \leq
e^{(m+q+3) \int_0^{\frac{T}{n}} \|\nabla V\|_{L^\infty_x}\,
d\tau} \left \{ \int_0^{\frac{T}{n}}\|\nabla
\varrho^{\frac{m+q-1}{2}}\|_{L^2_x}^2  + C\frac{T}{n} \right \}.
\end{aligned}
\end{equation}
Similarly, for all $i=1, 2, \dots, n-1$, we have
\begin{equation}\label{eq7 : Lemma : H-curve}
\begin{aligned}
 \left\|\varrho(\frac{(i+1)}{n}T) \right\|_{L^q_x}^q & + \frac{4q}{m+q-1}\int_{\frac{iT}{n}}^{\frac{(i+1)T}{n}} \| \nabla \varrho^{\frac{m+q-1}{2}}\|_{L^2_x}^2 \,d\tau
= \left\|\rho(\frac{iT}{n}) \right\|_{L^q_x}^q \\
&\leq \left\|\varrho(\frac{(i-1)}{n}T) \right\|_{L^q_x}^q
~e^{(q-1) \int_{\frac{(i-1)T}{n}}^{\frac{iT}{n}}\|\nabla V\|_{L^\infty_x}
\,d\tau},
\end{aligned}
\end{equation}
where we exploit \eqref{L^p relation} in the inequality, and
\begin{equation}\label{eq8 : Lemma : H-curve}
\begin{aligned}
\int_{\frac{iT}{n}}^{\frac{(i+1)}{n}T} \|\nabla
\rho^{\frac{m+q-1}{2}}\|_{L^2_x}^2  & \leq  e^{(m+q+3)
\int_{\frac{iT}{n}}^{\frac{(i+1)T}{n}} \|\nabla V\|_{L^\infty_x}\,
d\tau} \left \{ \int_{\frac{iT}{n}}^{\frac{(i+1)T}{n}}\|\nabla
\varrho^{\frac{m+q-1}{2}}\|_{L^2_x}^2  + C\frac{T}{n} \right \}.
\end{aligned}
\end{equation}
From \eqref{eq9 : Lemma : H-curve} and \eqref{eq7 : Lemma :
H-curve}, we have
\begin{equation}\label{eq11 : Lemma : H-curve}
\begin{aligned}
&\quad \|\varrho(t)\|_{L^q_x}^q + \frac{4q}{m+q-1}\int_{0}^{T} \| \nabla \varrho^{\frac{m+q-1}{2}} \|_{L^2_x}^2 \,d\tau\\
& \leq \|\rho(0)\|_{L^q_x}^q + \sum_{i=1}^{n-1}\|\varrho(iT/n) \|_{L^q_x}^q ~\left (e^{(q-1)\int_{\frac{(i-1)T}{n}}^{\frac{iT}{n}}\|\nabla V\|_{L^\infty_x} \,d\tau} -1 \right )\\
&\leq \|\rho(0)\|_{L^q_x}^q + \sum_{i=1}^{n-1}\|\varrho(iT/n)
\|_{L^q_x}^q ~\left ((q-1) \left[\int_{\frac{(i-1)T}{n}}^{\frac{iT}{n} }\|\nabla V\|_{L^\infty_x} \,d\tau\right]
~e^{(q-1)\int_{\frac{(i-1)T}{n}}^{\frac{iT}{n}}\|\nabla V\|_{L^\infty_x}
\,d\tau} \right ).
\end{aligned}
\end{equation}
From \eqref{eq10 : Lemma : H-curve} and \eqref{eq8 : Lemma :
H-curve}, we have
\begin{equation}\label{eq12 : Lemma : H-curve}
\begin{aligned}
\int_{0}^{T} \|\nabla \rho^{\frac{m+q-1}{2}}\|_{L^2_x}^2  & \leq
e^{(m+q+3) \int_{0}^{T} \|\nabla V\|_{L^\infty_x}\, d\tau}
\left ( \int_{0}^{T}\|\nabla \varrho^{\frac{m+q-1}{2}}\|_{L^2_x}^2  + CT
\right ).
\end{aligned}
\end{equation}
Finally we combine \eqref{eq11 : Lemma : H-curve} and \eqref{eq12 :
Lemma : H-curve} and get
\begin{equation*} %\label{eq13 : Lemma : H-curve}
%\begin{aligned}
\int_{0}^{T} \|\nabla \rho^{\frac{m+q-1}{2}}\|_{L^2_x}^2
\leq  e^{(m+2q+2) \int_{0}^{T} \|\nabla V\|_{L^\infty_x}\,d\tau}
 \times \left [\frac{(m+q-1)}{4q}
\left (\|\rho_0\|_{L^q (\mathbb{R}^d)}^q + \max_{i} \|\varrho(\frac{iT}{n}) \|_{L^q_x}^q\int_{0}^{T}\|\nabla V\|_{L^\infty_x} d\tau \right ) + CT \right ],%\\
%&\leq  e^{(3m+2) \int_{0}^{T} \|\nabla V\|_{L^\infty_x}\, d\tau}
%\left (\frac{\|\rho(0)\|_{L^m_x}^m}{(m-1)} +\|\rho_0 \|_{L^m_x}^m~ e^{\frac{m-1}{m}\int_{0}^{T} \|\nabla V\|_{L^\infty_x}\, d\tau}\int_{0}^{T}\|\nabla V\|_{L^\infty_x} d\tau  + TC \right )
%\end{aligned}
\end{equation*}
which concludes \eqref{main : Lemma:H-curve}.
\end{proof}
%}

\begin{lemma}\label{Lemma:H-curve-1}
Let $V \in L^1(0,T : W^{2,\infty}(\bbr^d)) \cap \mathcal{C}^\infty(\bar{Q}_T)$ and
$\rho_0 \in \mathcal{P}_2(\bbr^d) \cap \mathcal{C}^\alpha(\bbr^d)$ for some $\alpha \in (0,1)$. % satisfying $\int_{\bbr^d}\rho_0 \log \rho_0 \,dx < \infty$.
Suppose that $\rho_n,\varrho_n:[0,T]\mapsto
\mathcal{P}_2(\mathbb{R}^d)$ are curves defined as in Lemma
\ref{Lemma:AC-curve} with the initial data
$\varrho_n(0),\rho_n(0):=\rho_0 $. Then, we have
\begin{equation}\label{main : Lemma:H-curve-1}
\int_0^T \|\nabla (\rho_{n})^{\frac{m}{2}}\|_{L^2_x}^2 dt\leq C,
\end{equation}
where $C= C(\int_{\bbr^d} \rho_0 \langle x\rangle^2  \,dx,\, \|\rho_0\|_{L^m(\bbr^d)},\, \|V\|_{L^1(0,T; W^{2,\infty}(\bbr^d))} ).$

\end{lemma}
%\footnote{KK:{\color{magenta}  $L^1(0,T : W^{2,\infty}(\bbr^d))$ must be $\norm{V}_{L^1(0,T : W^{2,\infty}(\bbr^d))}$}}

\begin{proof}
For notational convenience, we use $\rho$ and $\varrho$ instead of $\rho_n$ and $\varrho_n$, respectively.
 First, we note that \eqref{H-est-01} gives us
%\begin{equation}\label{eq6 : Lemma : H-curve-1}
%\begin{aligned}
% \|\varrho_n(t)\|_{L^m_x},~~\|\rho_n(t)\|_{L^m_x} &\leq \|\rho_0\|_{L^m (\bbr^d)}~e^{\frac{m-1}{m} \int_0^T\|\nabla V\|_{L^\infty_x} \,d\tau}, \quad \forall ~ t \in [0, T]\\
%  \|\varrho_n(t)\|_{L^{2m-1}},~~\|\rho_n(t)\|_{L^{2m-1}} &\leq \|\rho_0\|_{L^{2m-1}}~e^{\frac{2m-2}{2m-1} \int_0^T\|\nabla V\|_{L^\infty_x} \,d\tau}.
% \end{aligned}
%\end{equation}
%\begin{equation}\label{eq6 : Lemma : H-curve-1}
%\begin{aligned}
% \|\varrho(t)\|_{L^q_x}%,~~ \|\varrho_n(t)\|_{L^{2m-1}}
% \leq \|\rho_0\|_{L^q(\bbr^d)}~e^{ \int_0^T\|\nabla V\|_{L^\infty_x} \,d\tau}, \quad \forall ~ t \in [0, T].
% \end{aligned}
%\end{equation}
%Next, \eqref{P:H-PME:e2} gives us
\begin{equation}\label{eq16 : Lemma : H-curve-1}
\int_{\bbr^d} \left (\varrho \log \varrho\right) ( \frac T n ) \,dx  + \frac{4}{m}\int_0^{\frac{T}{n}} \| \nabla \varrho^{\frac{m}{2}}\|_{L^2_x}^2 \,dt
 = \int_{\bbr^d} \rho_0 \log \rho_0 \,dx,
\end{equation}
and we have from \eqref{eq5 : Sobolev}
\begin{equation}\label{eq14 : Lemma : H-curve-1}
\begin{aligned}
 \|\nabla \rho^{\frac{m}{2}}\|_{L^2_x}^2 &\leq e^{(m+4) \int_0^{\frac{T}{n}} \|\nabla V\|_{L^\infty_x}\, d\tau}\left \{ \|\nabla \varrho^{\frac{m}{2}}\|_{L^2_x}^2
 + \| \varrho^{\frac{m}{2}}\|_{L^2_x}^2  \times \left ( \frac{m}{2}\int_0^{\frac{T}{n}} \| \nabla^2 V\|_{L^\infty_x} \,d\tau\right )^2 \right \}.%\\
% & \leq  e^{2(m+\frac{3}{2}) \int_0^{\frac{T}{n}} \|\nabla V\|_{L^\infty_x}\, d\tau}\left \{ \|\nabla \varrho^{m-\frac{1}{2}}\|_{L^2_x}^2
% + C \right \}
\end{aligned}
\end{equation}
Combining the fact $\|\rho_0\|_{L^1 (\bbr^d)}=1$ and \eqref{eq6 : Lemma : H-curve} with \eqref{eq14 : Lemma : H-curve-1}, we have
\begin{equation*} %\label{eq9 : Lemma : H-curve-1}
\begin{aligned}
 \|\nabla \rho^{\frac{m}{2}}\|_{L^2_x}^2  & \leq  e^{(m+4) \int_0^{\frac{T}{n}} \|\nabla V\|_{L^\infty_x}\, d\tau}\left \{ \|\nabla \varrho^{\frac{m}{2}}\|_{L^2_x}^2  + C \right \},
\end{aligned}
\end{equation*}
where the constant $C=C\left (\|\rho_0\|_{L^m (\bbr^d)}, ~ \|V\|_{L^1(0,T ; W^{2,\infty}(\bbr^d))} \right )$. This gives us
\begin{equation*} %\label{eq10 : Lemma : H-curve-1}
\begin{aligned}
\int_0^{\frac{T}{n}} \|\nabla \rho^{\frac{m}{2}}\|_{L^2_x}^2  & \leq
e^{(m+4) \int_0^{\frac{T}{n}} \|\nabla V\|_{L^\infty_x}\,
d\tau} \left \{ \int_0^{\frac{T}{n}}\|\nabla
\varrho^{\frac{m}{2}}\|_{L^2_x}^2  + C\frac{T}{n} \right \}.
\end{aligned}
\end{equation*}
Similarly, for all $i=1, 2, \dots, n-1$, we have
\begin{equation}\label{eq7 : Lemma : H-curve-1}
\begin{aligned}
&\quad \int_{\bbr^d} \left (\varrho \log \varrho\right)(\frac{(i+1)T}{n}) \,dx + \frac{4}{m}\int_{\frac{iT}{n}}^{\frac{(i+1)T}{n}} \| \nabla \varrho^{\frac{m}{2}}\|_{L^2_x}^2 \,dt \\
&=\int_{\bbr^d} \left (\rho \log \rho\right)(\frac{iT}{n}) \,dx
\leq \int_{\bbr^d} \left (\varrho \log \varrho\right)(\frac{iT}{n}) \,dx + \int_{\frac{(i-1)T}{n}}^{\frac{iT}{n}} \|\nabla V\|_{L^\infty_x} dt,
\end{aligned}
\end{equation}
where we exploit \eqref{eq7 : Lipschitz of Jacobian} and \eqref{Entropy relation} in the inequality, and
\begin{equation}\label{eq8 : Lemma : H-curve-1}
\begin{aligned}
\int_{\frac{iT}{n}}^{\frac{(i+1)}{n}T} \|\nabla
\rho^{\frac{m}{2}}\|_{L^2_x}^2  & \leq  e^{(m+4)
\int_{\frac{iT}{n}}^{\frac{(i+1)T}{n}} \|\nabla V\|_{L^\infty_x}\,
d\tau} \left \{ \int_{\frac{iT}{n}}^{\frac{(i+1)T}{n}}\|\nabla
\varrho^{\frac{m}{2}}\|_{L^2_x}^2  + C\frac{T}{n}\right \}.
\end{aligned}
\end{equation}
From \eqref{eq16 : Lemma : H-curve-1} and \eqref{eq7 : Lemma : H-curve-1}, we have
\begin{equation*} %\label{eq15 : Lemma : H-curve-1}
\begin{aligned}
\int_{\bbr^d} \left (\varrho \log \varrho\right)(T) \,dx + \frac{4}{m}\int_{0}^{T} \| \nabla \varrho^{\frac{m}{2}}\|_{L^2_x}^2 \,dt
%=\int_{\bbr^d} \left (\rho \log \rho\right)(\frac{iT}{n}) \,dx\\
\leq \int_{\bbr^d} \rho_0 \log \rho_0  \,dx + \int_{0}^{T} \|\nabla V\|_{L^\infty_x} dt  ,
\end{aligned}
\end{equation*}
and hence
\begin{equation}\label{eq17 : Lemma : H-curve-1}
\begin{aligned}
 \frac{4}{m}\int_{0}^{T} \| \nabla \varrho^{\frac{m}{2}}\|_{L^2_x}^2 \,dt
%=\int_{\bbr^d} \left (\rho \log \rho\right)(\frac{iT}{n}) \,dx\\
\leq \int_{\bbr^d} \rho_0 \log \rho_0  \,dx + \int_{0}^{T} \|\nabla V\|_{L^\infty_x} dt -\int_{\varrho(T) <1} \left (\varrho \log \varrho\right)(T) \,dx.
\end{aligned}
\end{equation}
We recall that there exists a constant $\alpha>0$ (independent of $\varrho$) such that
\begin{equation}\label{eq18 : Lemma : H-curve-1}
-\int_{\varrho(x) <1} \varrho \log \varrho \,dx \leq \int_{\bbr^d} |x|^2 \varrho(x) \,dx + \alpha.
\end{equation}
We plug \eqref{eq18 : Lemma : H-curve-1} into \eqref{eq17 : Lemma : H-curve-1}, and combine with \eqref{eq8 : Lemma : H-curve-1} to get
\begin{equation}\label{eq12 : Lemma : H-curve-1}
%\begin{aligned}
\int_{0}^{T} \|\nabla \rho^{\frac{m}{2}}\|_{L^2_x}^2  \leq
e^{(m+4) \int_{0}^{T} \|\nabla V\|_{L^\infty_x}\, d\tau}
\left ( \int_{0}^{T}\|\nabla \varrho^{\frac{m}{2}}\|_{L^2_x}^2  + CT
\right )
\leq C,
%\end{aligned}
\end{equation}
where $C=C\left(\int_{\bbr^d} \rho_0 \log \rho_0 \,dx, \|\rho_0\|_{L^m(\bbr^d)}, \|V\|_{L^1(0,T; W^{2,\infty}(\bbr^d))}, \int_{\bbr^d} |x|^2 \varrho(x,T) \,dx \right).$
We note
\begin{equation}\label{eq13 : Lemma : H-curve-1}
\begin{aligned}
\int_{\bbr^d} |x|^2 \varrho(x,T) \,dx \leq 2W_2^2( \rho_0, \varrho(T)) +2 \int_{\bbr^d} \rho_0 | x |^2  \,dx,
\end{aligned}
\end{equation}
and, from \eqref{eq20:Lemma:AC-curve} and \eqref{eq21:Lemma:AC-curve}, we have
\begin{equation}\label{eq19 : Lemma : H-curve-1}
W_2(\rho_0, \varrho(T)) \leq C \left(\|\rho_0\|_{L^m(\bbr^d)},\, \|V\|_{L^1(0,T;W^{1, \infty}(\bbr^d))} \right).
\end{equation}
Finally, we combine \eqref{eq12 : Lemma : H-curve-1}, \eqref{eq13 : Lemma : H-curve-1} and \eqref{eq19 : Lemma : H-curve-1} to get
\begin{equation*}
\int_{0}^{T} \|\nabla \rho^{\frac{m}{2}}\|_{L^2_x}^2  \leq C(\int_{\bbr^d} \rho_0 |x|^2   \,dx, \, \|\rho_0\|_{L^m(\bbr^d)}, \, \|V\|_{L^1(0,T; W^{2,\infty}(\bbr^d))}),
\end{equation*}
%Finally we combine \eqref{eq11 : Lemma : H-curve-1} and \eqref{eq12 :
%Lemma : H-curve-1} and get
%\begin{equation}\label{eq13 : Lemma : H-curve-1}
%\begin{aligned}
%\int_{0}^{T} \|\nabla \rho^{\frac{m+q-1}{2}}\|_{L^2_x}^2  & \leq  e^{(m+2q+2) \int_{0}^{T} \|\nabla V\|_{L^\infty_x}\, d\tau} \\
%&\times \left [\frac{(m+q-1)}{4q}
%\left (\|\rho(0)\|_{L^q_x}^q + \max_{i}\|\varrho(\frac{iT}{n}) \|_{L^q(\bbr^d)}^q\int_{0}^{T}\|\nabla V\|_{L^\infty_x} d\tau \right ) + TC \right ]%\\
%%&\leq  e^{(3m+2) \int_{0}^{T} \|\nabla V\|_{L^\infty_x}\, d\tau}
%%\left (\frac{\|\rho(0)\|_{L^m_x}^m}{(m-1)} +\|\rho_0 \|_{L^m_x}^m~ e^{\frac{m-1}{m}\int_{0}^{T} \|\nabla V\|_{L^\infty_x}\, d\tau}\int_{0}^{T}\|\nabla V\|_{L^\infty_x} d\tau  + TC \right )
%\end{aligned}
%\end{equation}
which concludes \eqref{main : Lemma:H-curve-1}.
\end{proof}

Next lemma is similar to Lemma 4.2 \cite{KK-SIMA}. So, we briefly sketch the proof of it (see Lemma 4.2 \cite{KK-SIMA}).
\begin{lemma}\label{Lemma:solving ODE}
Let $\varphi \in \calC_c^\infty({\mathbb{R}^d} \times [0,T])$ and $\rho_n,\varrho_n:[0,T]\mapsto \mathcal{P}_2(\mathbb{R}^d)$ be the curves defined as in Lemma \ref{Lemma:AC-curve}.
For any $s,~t \in [0,T]$, we define
%\begin{equation}\label{eq3:solving lemma}
%\begin{aligned}
%& \int_{\mathbb{R}^d}\varphi(x,t) ~\rho_n(x,t)\,dx - \int_{\mathbb{R}^d}\varphi(x,s) ~\rho_n(x,s)\,dx\\
%&= \int_{s}^{t} \int_{\mathbb{R}^d}  ( \nabla   \varphi \cdot V) ~\rho_n \,dx \,d\tau + \int_{s}^{t}
%\int_{\mathbb{R}^d}  \left( \partial_\tau \varphi ~ \varrho_n + \Delta \varphi ~(\varrho_n)^m\right) \,dx \,d\tau
%+ E_n,
%\end{aligned}
%\end{equation}
\begin{equation*}
\begin{aligned}
E_n &:= \int_{\mathbb{R}^d}\varphi(x,t) ~\rho_n(x,t)\,dx - \int_{\mathbb{R}^d}\varphi(x,s) ~\rho_n(x,s)\,dx\\
  &\quad  - \int_{s}^{t} \int_{\mathbb{R}^d}  ( \nabla   \varphi \cdot V) ~\rho_n \,dx \,d\tau
  - \int_{s}^{t}
\int_{\mathbb{R}^d}  \{ \partial_\tau \varphi ~ \varrho_n + \Delta \varphi ~(\varrho_n)^m \}\,dx \,d\tau.
\end{aligned}
\end{equation*}
Then,
\begin{equation*}
\begin{aligned}
|E_n| \leq  C\frac{T}{n}\|\varphi\|_{\calC^2} e^{C\int_{s}^{t}\|\nabla V\|_{L^\infty_x} \,d\tau}
 \int^{t}_{0} \|V (\tau)\|_{W^{1,\infty}_x} \,d\tau  \longrightarrow 0, \quad as \quad n\rightarrow \infty,
\end{aligned}
\end{equation*}
where $C$ is the same constant in Lemma \ref{Lemma:AC-curve}.
\end{lemma}

\begin{proof}
First, we suppose $[s,t]\subset \big[\frac{kT}{n},\frac{(k+1)T}{n}\big ]$ for some $0\leq k \leq n-1 $. Then, we have
\begin{equation}\label{eq4:solving lemma}
\mathcal{I}:= \int_{\bbr^d}\varphi(x,t) \rho_n(x,t) \,dx - \int_{\bbr^d}\varphi(x,s) \rho_n(x,s)\,dx\\
= I + II + III,
\end{equation}
where
\begin{equation*}%\label{eq4:solving lemma}
\begin{gathered}
I = \int_{\bbr^d}\varphi(x,t) \rho_n(x,t)\,dx - \int_{\bbr^d}\varphi(x,t) \varrho_n(x,t) \,dx, \\
II = \int_{\bbr^d}\varphi(x,t) \varrho_n(x,t) \,dx - \int_{\bbr^d}\varphi(x,s)\varrho_n(x,s) \,dx,\\
III= \int_{\bbr^d}\varphi(x,s) \varrho_n(x,s) \,dx - \int_{\bbr^d}\varphi(x,s)\rho_n(x,s) \,dx.
\end{gathered}
\end{equation*}
Note that
\begin{equation}\label{eq4:II}
II = \int_{s}^{t} \int_{\bbr^d}   \big[\partial_\tau \varphi ~\varrho_n + \Delta \varphi \big(\varrho_n \big)^m \big] \,dx \,d\tau .
\end{equation}
%Now we estimate the terms on the right-hand side of \eqref{eq4:solving lemma}.
Following the same estimation in the proof of Lemma 4.2 \cite{KK-SIMA}, we have
\begin{equation}\label{eq10:solving lemma}
%\begin{aligned}
 %\int_{\bbr^d}\varphi(x,t) \rho_n(x,t) \,dx - \int_{\bbr^d} \varphi(x,t) \varrho_n(x,t) \,dx
 I = \int_{kT/n}^{t} \int_{\bbr^d}   ( \nabla   \varphi \cdot V ) \rho_n \,dx \,d\tau
  + E_{kT/n,t},
 %\end{aligned}
\end{equation}
where
\begin{equation}\label{eq11:solving lemma}
\begin{aligned}
 |E_{kT/n,t} | \leq C\|\varphi\|_{\calC^2} \Big[\Big (t-\frac{kT}{n}\Big ) + \sqrt{t-\frac{kT}{n}}\, \Big]
 \int^{t}_{kT/n} \|V\|_{W^{1,\infty}_x} \,d\tau.
 \end{aligned}
\end{equation}
%\item
Similarly, we have
\begin{equation}\label{eq12:solving lemma}
\begin{aligned}
%\int_{\bbr^d}\varphi(x,s) \rho_n(x,s) \,dx- \int_{\bbr^d}\varphi(x,s) \varrho_n(x,s) \,dx
- III=\int_{kT/n}^{s} \int_{\bbr^d}   ( \nabla   \varphi\cdot V)   \rho_n \,dx \,d\tau  + E_{kT/n,s},
\end{aligned}
\end{equation}
with
\begin{equation}\label{eq13:solving lemma}
\begin{aligned}
 |E_{kT/n,s}| \leq C\|\varphi\|_{\calC^2} \Big[\Big (s-\frac{kT}{n}\Big ) + \sqrt{s-\frac{kT}{n}} \,\Big]
 \int^{s}_{kT/n} \|V\|_{W^{1,\infty}_x} \,d\tau.
 \end{aligned}
\end{equation}
%\item
We %exploit (\ref{eq: Ohta}), and
plug (\ref{eq10:solving lemma}), \eqref{eq4:II} and (\ref{eq12:solving lemma})  into \eqref{eq4:solving lemma} to get
\begin{equation}\label{eq6:solving lemma}
\begin{aligned}
% &\int_{\bbr^d}\varphi(x,t) \rho_n(x,t) \,dx - \int_{\bbr^d}\varphi(x,s) \rho_n(x,s) \,dx\\
%&= \int_{kT/n}^{t} \int_{\bbr^d}   \nabla   \varphi_\tau\cdot v_\tau \rho_n(\tau)\,dx \,d\tau  + E_{kT/n,t}
%  + \int_{s}^{t} \int_{\bbr^d}   [\partial_\tau \varphi_\tau\varrho_n(\tau) + \Delta \varphi_\tau\varrho_n(\tau)^m] \,dx \,d\tau\\
% & \quad -\int_{kT/n}^{s} \int_{\bbr^d}   \nabla   \varphi_\tau\cdot v_\tau \rho_n(\tau)\,dx  \,d\tau + E_{kT/n,s}\\
\mathcal{I} &= \int_{s}^{t} \int_{\bbr^d}   (\nabla   \varphi \cdot V)   \rho_n  \,dx  \,d\tau
  + \int_{s}^{t} \int_{\bbr^d}   \big[\partial_\tau \varphi ~\varrho_n + \Delta \varphi \big(\varrho_n \big)^m \big] \,dx \,d\tau
  +E_{kT/n,t} + E_{kT/n,s}.
\end{aligned}
\end{equation}
%\end{itemize}

Now, we consider the general case : $s\in \big[\frac{iT}{n},\frac{(i+1)T}{n}\big]$ and $t\in \big[\frac{jT}{n},\frac{(j+1)T}{n}\big]$.
For this case, we have
\begin{equation}\label{eq7:solving lemma}
\begin{aligned}
& \int_{\bbr^d}\varphi(x,t) \rho_n(x,t)\,dx - \int_{\bbr^d}\varphi(x,s) \rho_n(x,s)\,dx\\
&= \int_{\bbr^d}\varphi(x,t)  \rho_n(x,t)\,dx - \int_{\bbr^d}\varphi(x,jT/n) \rho_n(x,jT/n)\,dx \\
&\quad + \int_{\bbr^d}\varphi(x,(i+1)T/n) \rho_n(x,(i+1)T/n) \,dx - \int_{\bbr^d}\varphi(x,s) \rho_n(x,s)\,dx\\
&\quad + \sum_{k=i+1}^{j-1}\left(\int_{\bbr^d}\varphi(x,(k+1)T/n) \rho_n(x,(k+1)T/n) \,dx - \int_{\bbr^d}\varphi(x,kT/n) \rho_n(x,kT/n) \,dx\right).
\end{aligned}
\end{equation}
We plug (\ref{eq6:solving lemma}) in each term of (\ref{eq7:solving lemma}), and get
\begin{equation*}
\begin{aligned}
& \int_{\bbr^d}\varphi(x,t) \rho_n(x,t)\,dx- \int_{\bbr^d}\varphi(x,s) \rho_n(x,s)\,dx\\
 &= \int_{s}^{t} \int_{\bbr^d}  ( \nabla   \varphi \cdot V)  \rho_n \,dx  \,d\tau + \int_{s}^{t} \int_{\bbr^d}
 [\partial_\tau \varphi \varrho_n  + \Delta \varphi \big(\varrho_n\big)^m] \,dx \,d\tau\\
 &\quad + E_{jT/n,t} + E_{iT/n,s} + E_{iT/n,(i+1)T/n}+ \sum_{k=i+1}^{j-1} E_{kT/n,(k+1)T/n}.
\end{aligned}
\end{equation*}
By using (\ref{eq11:solving lemma}) and (\ref{eq13:solving lemma}), we estimate the error term as follows,
\begin{equation*}
\begin{aligned}
E_n &:= E_{jT/n,t} + E_{iT/n,s} + E_{iT/n,(i+1)T/n}+ \sum_{k=i+1}^{j-1} E_{kT/n,(k+1)T/n}\\
&\leq C\frac{T}{n}\|\varphi \|_{\calC^2} e^{C\int_{s}^{t}\|\nabla V\|_{L^\infty_x} \,d\tau}
 \int^{t}_{iT/n} \|V\|_{W^{1,\infty}_x} \,d\tau \longrightarrow 0 ,
\end{aligned}
\end{equation*}
as $ n\rightarrow \infty$.
\end{proof}

\begin{lemma}\label{Lemma : equi-continuity}
Let $\tilde{\alpha}\in (0,1)$ be the constant in Theorem \ref{T:Boundary_Holder}. Suppose that $ \alpha \in (0, \tilde{\alpha}]$ and
\begin{equation*}
\rho_0 \in \mathcal{P}_2(\bbr^d) \cap \calC^\alpha (\bbr^d)  \quad \mbox{and} \quad V \in  L^1(0,T ; \calC^{1, \alpha}(\bbr^d)).
\end{equation*}
Let $\rho_n,\varrho_n:[0,T]\mapsto \mathcal{P}_2(\mathbb{R}^d)$ be curves defined as in Lemma \ref{Lemma:AC-curve}.
% the initial data $\varrho_n(0),\rho_n(0):=\rho_0 \in \mathcal{P}_2(\bbr^d)$  and  $V \in L^1(0,T; W^{1,\infty}(\bbr^d) \cap L^1(0,T ; \calC^{1, \alpha}(\bbr^d))$.
Then,we have
 \begin{equation}\label{eq 2 : Holder of splitting}
|\varrho_n(x,t)-\varrho_n(y,t)| + |\rho_n(x,t)-\rho_n(y,t)|\leq C|x-y|^\alpha,  \qquad \forall ~ x, y \in \bbr^d , ~ t \in [0, T],
 \end{equation}
 %where $\beta(\alpha, m, d)$, $M(\|\rho_0\|_\infty, m,d)$  and
  where $C(\|\rho_0\|_{L^\infty(\mathbb{R}^d)},  \int_0^T \|\nabla V\|_{\calC^\alpha (\bbr^d)}  \,d\tau)$.
\end{lemma}

\begin{proof} %We assume $n$ is sufficiently large whose meaning will be clear soon.
 First, we suppose $t \in \left[0, \frac{T}{n}\right]$. %Note that \eqref{eq 2 : Holder-Porous} says $\varrho_n(t) \in C^\alpha(\bbr^d)$.
Due to Theorem \ref{T:Boundary_Holder}, it holds
%$C_1(\|\rho_0 \|_{L^\infty(\bbr^d)})$  such that
\begin{equation}\label{eq 1 : Lemma equi-cont}
|\varrho_n(x,t)-\varrho_n(y,t) |\leq C |x-y|^\alpha,  \quad \forall ~ x, ~y \in \bbr^d,
\end{equation}
where $C= C(\|\rho_0 \|_{L^\infty(\bbr^d)})$.
%Hence,
%\begin{equation}
%\begin{aligned}
%|\varrho_n(t)(x)-\varrho_n(t)(y) |&\leq C_1 |x-y|^\gamma \\
%&=\big(C_1^{\frac{2}{\gamma}}|x-y |\big)^{\frac{\gamma}{2}}|x-y|^{\frac{\gamma}{2}}\\
%&\leq |x-y|^{\frac{\gamma}{2}}
%\end{aligned}
%\end{equation}
%if
%$$ |x-y|\leq \Big(\frac{1}{C_1}\Big)^{\frac{2}{\gamma}}.$$
%Next, we assume $n$ is large enough to have
%$$\max_{1 \leq i \leq n} \left \{ \int_{(i-1)T/n}^{iT/n} \|v_\tau \|_\infty d \tau \right \} \leq \frac{1}{4}\Big(\frac{1}{C_1}\Big)^{\frac{2}{\gamma}}, $$
%which is possible due to the fact $v \in L^1(0, T; W^{1, \infty}(\bbr^d))$.
%
%In particular,  we have
%$$\int_{0}^{T/n} \|v_\tau \|_\infty d \tau  \leq \frac{1}{4}\Big(\frac{1}{C_1}\Big)^{\frac{2}{\gamma}}. $$
Next, due to Lemma \ref{Lemma : Holder regularity on the flow}, we have
\begin{equation}\label{eq 2 : Lemma equi-cont}
|\rho_n(x,t)-\rho_n(y,t) | \leq C |x-y|^{\alpha}, \quad   \forall ~x, ~ y \in \bbr^d,
\end{equation}
where $C= C \left(\text{sup}_{t \in \left[0, \frac{T}{n}\right ]} \|\varrho_n(t)\|_{\calC^\alpha (\bbr^d)}, \,\int_0^{T/n} \|\nabla V \|_{\calC^\alpha (\bbr^d)} \,d\tau \right)$.
%Hence, $C_2$ depends on $\left( \|\rho_0\|_\infty, \int_0^{T/n}\|\nabla V\|_{L^\infty_x} d \tau,
%\int_0^{T/n} \|\nabla V \|_\alpha \,d\tau \right)$.  We define
%\begin{equation}
%M:=\frac{1}{2} \Big(\frac{1}{C_1}\Big)^{\frac{2}{\gamma}}, \quad \beta:=\frac{2}{\gamma}.
%\end{equation}

We combine \eqref{eq 1 : Lemma equi-cont} and \eqref{eq 2 : Lemma equi-cont}, and exploit \eqref{eq29:Lemma:AC-curve} for $q=\infty$ to conclude
 \begin{equation}\label{eq 3 : Lemma equi-cont}
|\varrho_n(x,t)-\varrho_n(y,t)|+ |\rho_n(x,t)-\rho_n(y,t)|\leq C|x-y|^\alpha,  \qquad \forall ~x, ~ y \in \bbr^d ,
 \end{equation}
where $C(\|\rho_0\|_{L^\infty(\mathbb{R}^d)}, ~ \int_0^{T/n} \|\nabla V \|_{\calC^\alpha (\bbr^d)} \,dt)$.

For the general case $t \in \left [\frac{iT}{n}, \frac{(i+1)T}{n}\right ], ~ i=1, 2, \dots, n-1$, we repeat what we did in the first step and get
 \begin{equation}\label{eq 4 : Lemma equi-cont}
|\varrho_n(x,t)-\varrho_n(y,t)| +|\rho_n(x,t)-\rho_n(y,t)|\leq C|x-y|^\alpha, \qquad \forall ~ x, ~ y \in \bbr^d ,
 \end{equation}
where $C(\|\rho(iT/n)\|_{L^\infty_x},  \int_{iT/n}^{(i+1)T/n} \|\nabla V \|_{\calC^\alpha (\bbr^d)} \,d\tau)$. We combine
\eqref{eq 3 : Lemma equi-cont} and \eqref{eq 4 : Lemma equi-cont} with \eqref{eq29:Lemma:AC-curve} for $q=\infty$, and conclude \eqref{eq 2 : Holder of splitting}
which completes the proof.
\end{proof}

\subsection{Proof of Proposition \ref{proposition : regular existence}}

Here, we prove Proposition \ref{proposition : regular existence}.
For each $n\in \mathbb{N},$ from Lemma \ref{Lemma:solving ODE}, there exist curves $\rho_n,\varrho_n:[0,T]\mapsto \mathcal{P}_2(\mathbb{R}^d)$
with $\rho_n(0)=\varrho_n(0)=\rho_0$ satisfying,  for any $\varphi \in \calC_c^\infty(\bbr^d \times [0,T))$ and $s,~t \in [0,T]$,
\begin{equation}\label{eq3 : Theorem : Toy Fokker-Plank}
\begin{aligned}
& \int_{\bbr^d}\varphi(x,t) \rho_n(x,t)\,dx - \int_{\bbr^d}\varphi(x,s) \rho_n(x,s)\,dx\\
&= \int_{s}^{t} \int_{\bbr^d}   (\nabla   \varphi \cdot V) \rho_n \,dx  \,d\tau + \int_{s}^{t} \int_{\bbr^d}
\left\{ \partial_\tau\varphi  \varrho_n+ \Delta \varphi (\varrho_n)^m\right\}\,dx \,d\tau
 + E_n,
\end{aligned}
\end{equation}
with
\begin{equation}\label{eq4 : Theorem : Toy Fokker-Plank}
\begin{aligned}
|E_n| \leq  C\frac{T}{n}\|\varphi\|_{\calC^2} e^{C\int_{0}^{T}\|\nabla V\|_{L^\infty_x} \,d\tau}
 \int^{T}_{0} \|V\|_{W^{1,\infty}_x} \,d\tau \longrightarrow 0, \quad\mbox{    as   }\quad n \rightarrow \infty.
\end{aligned}
\end{equation}
 Recalling Lemma \ref{Lemma:AC-curve}, we know%
% \begin{equation}\label{eq7 : Theorem : Toy Fokker-Plank}
% \|\varrho_n(t)\|_\infty, \, \|\rho_n(t)\|_\infty \leq C\|\rho_0\|_\infty
% \end{equation}
% for a constant $C$ depending on $\|v\|_{L^1(0,T;W^{1,\infty})}$. We also have
\begin{equation}\label{eq5 : Theorem : Toy Fokker-Plank}
W_2(\rho_n(s),\rho_n(t))\leq C\sqrt{t-s} + \int_s^t \|V(\tau)\|_{L^\infty_x}  \,d\tau
\end{equation}
and
\begin{equation}\label{eq6 : Theorem : Toy Fokker-Plank}
\begin{aligned}
W_2(\rho_n(t), {\varrho}_n(t)) \leq  \max_{\{i=0,1,\dots,n-1\}}\int_{\frac{iT}{n}}^{\frac{(i+1)T}{n}} \|V(\tau)\|_{L^\infty_x} \,d\tau,
\end{aligned}
\end{equation}
%and
%\begin{equation}\label{eq7 : Theorem : Toy Fokker-Plank}
%|\nabla_-F|(\rho_n(t)
%)\leq C,
%\end{equation}
where the constant  $C=C\left(  \| \rho_0\|_{L^m(\bbr^d)},  \,  \,\|V\|_{L^1(0,T; W^{1,\infty}(\bbr^d))} \right) $.

Due to the fact $\rho_n(0)=\rho_0, \, \forall ~n \in \mathbb{N}$ and \eqref{eq5 : Theorem : Toy Fokker-Plank}, there exists $M>0$ such that
$$\rho_n (t) \in B_M, \quad \forall ~ t \in [0,T], ~ n \in \mathbb{N},$$
where $B_M:=\{\mu \in \mathcal{P}_2(\bbr^d ) | W_2(\mu, \delta_0) \leq M\}$ is a narrowly compact subset of $\mathcal{P}_2(\bbr^d )$.
Furthermore, the estimate in \eqref{eq5 : Theorem : Toy Fokker-Plank} says that the sequence of curves $\rho_n:[0,T]\mapsto \mathcal{P}_2(\mathbb{R}^d)$ are
equi-continuous. Hence, from Lemma \ref{Lemma : Arzela-Ascoli}, there exist a subsequence (by abusing notation)
$\rho_n$ and a limit curve  $\rho :[0,T]\mapsto \mathcal{P}_2(\mathbb{R}^d)$ such that
\begin{equation*}
\rho_{n}(t) ~~ \mbox{ narrowly converges to} ~~\rho(t), \qquad \text{for all} ~~ t\in[0,T].
\end{equation*}
 Due to \eqref{eq6 : Theorem : Toy Fokker-Plank}, we note that $\varrho_n(t)$ also narrowly converges to $\rho(t), \, \forall ~ t \in [0,T]$.
 This implies
\begin{equation*}
\begin{aligned}
\int_{\bbr^d}\varphi(x,t) \rho_n(x,t)\,dx - \int_{\bbr^d}\varphi(x,s) \rho_n(x,s)\,dx \longrightarrow \int_{\bbr^d}\varphi(x,t) \rho(x,t)\,dx - \int_{\bbr^d}\varphi(x,s) \rho(x,s)\,dx,
\end{aligned}
\end{equation*}
\begin{equation*}
\begin{aligned}
\int_{s}^{t} \int_{\bbr^d}  \left[ (\nabla   \varphi \cdot V) \rho_n + \partial_\tau\varphi  \varrho_n \right] \,dx  \,d\tau
\longrightarrow \int_{s}^{t} \int_{\bbr^d}  \left[ (\nabla   \varphi\cdot V) + \partial_\tau\varphi  \right]\rho \,dx  \,d\tau,
\end{aligned}
\end{equation*}
as $n \rightarrow \infty$. Furthermore,  from Lemma \ref{Lemma : equi-continuity}, $\rho_n$ and $\varrho_n$ are equi-continuous with respect to the space variable.
We combine this equi-continuity with  the uniform bound \eqref{eq29:Lemma:AC-curve} to claim that $\rho_n$ and  $\varrho_n$ actually pointwise converge to $\rho$.
Hence, again from Lemma \ref{Lemma : equi-continuity}, we have
\begin{equation*}
\left |\rho (x,t)-\rho(y,t)\right | \leq C |x-y|^\alpha, \quad \forall ~x,~y \in \bbr^d, \, t \in [0,T].
\end{equation*}
Next, we note
\begin{equation}\label{eq7 : Theorem : Toy Fokker-Plank}
\|\varrho_n(t)\|_{L^q_x},~\|\rho_n(t)\|_{L^q_x} \leq \|\rho_0\|_{L^q(\mathbb{R}^d)}e^{\frac{q-1}{q}\int_0^T \|\nabla \cdot V \|_{L^\infty_x} \,d\tau}, \quad \forall ~t\in[0,T],
%\qquad \forall \,t\in[0,T],
\end{equation}

coming from \eqref{eq29:Lemma:AC-curve} and the pointwise convergence of $(\varrho_n)^m$. Exploiting the Lebesgue dominated convergence theorem, we have
\begin{equation*}
\int_{s}^{t} \int_{\bbr^d}  (\varrho_n )^m \,\Delta \varphi  \,dx \,d\tau \longrightarrow \int_{s}^{t} \int_{\bbr^d}  \rho^m\, \Delta \varphi \,dx \,d\tau.
\end{equation*}
Finally, we exploit the estimation \eqref{eq3 : Theorem : Toy Fokker-Plank} and (\ref{eq4 : Theorem : Toy Fokker-Plank}) and get
\begin{equation*}
%\begin{aligned}
 \int_{\bbr^d}\varphi(x,t) \rho(x,t) \,dx - \int_{\bbr^d}\varphi(x,s) \rho(x,s)\,dx
=\int_{s}^{t} \int_{\bbr^d}  \left [ \left(\partial_\tau\varphi  + V \cdot \nabla \varphi \right){\rho}
+ \rho^m \Delta \varphi   \right ] \,dx\,d\tau.
%\end{aligned}
\end{equation*}
From the lower semi-continuity of Wasserstein distance w.r.t the narrow convergence (refer \cite[Lemma~7.1.4]{ags:book}) and \eqref{eq5 : Theorem : Toy Fokker-Plank}, we have \eqref{eq2 : Theorem : Toy Fokker-Plank}.
%Also the $L^\infty$ bound on $\rho_n$ \eqref{eq7 : Theorem : Toy Fokker-Plank} implies
% $\rho(t)\in \mathcal{P}_2^{ac}(\mathbb{R}^d)$ for all $t\in[0,T]$.
We also see that the estimation \eqref{eq7 : Theorem : Toy Fokker-Plank} implies \eqref{KK-Nov24-100}. Note that \eqref{main : Lemma:H-curve} gives us
\begin{equation}\label{eq8 : Theorem : Toy Fokker-Plank}
\int_0^T \|\nabla \rho_n^{\frac{m+q-1}{2}}\|_{L^2_x}^2 \,dt \leq C,
\end{equation}
where the constant  $C=C\left(\|\rho_0 \|_{L^{m+q-1}(\bbr^d)},\,\|V\|_{L^1(0,T;W^{2,\infty}(\bbr^d)} \right)$. This implies that $\nabla \rho_n^{\frac{m+q-1}{2}}$ (up to a subsequence)
has a weak limit $\eta \in L^2(\bbr^d\times[0,T])$ and \eqref{eq8 : Theorem : Toy Fokker-Plank} gives us
\begin{equation}\label{eq9 : Theorem : Toy Fokker-Plank}
\iint_{Q_T}|\eta(x,t) |^2\,dx \,dt \leq C.
\end{equation}
Since  $\rho_n$ pointwise converges to $\rho$, we conclude $\eta = \nabla \rho^{\frac{m+q-1}{2}}$ and \eqref{KK-Feb19-100} follows from \eqref{eq9 : Theorem : Toy Fokker-Plank}. Actually, we note that the whole sequence $\nabla \rho_n^{\frac{m+q-1}{2}}$ weakly converges to $\nabla \rho^{\frac{m+q-1}{2}}$ in $L^2(\bbr^d\times[0,T])$.
Similarly, \eqref{main : Lemma:H-curve-1} gives us
\begin{equation}\label{eq10 : Theorem : Toy Fokker-Plank}
\int_0^T \|\nabla \rho_n^{\frac{m}{2}}\|_{L^2_x}^2 \,dt \leq C,
\end{equation}
where $C=C\left(\int_{\bbr^d} \rho_0 \langle x\rangle^2  \,dx, \|\rho_0 \|_{L^{m}(\bbr^d)}, \|V\|_{L^1(0,T;W^{2,\infty}(\bbr^d)} \right)$, and
we can get \eqref{KK-June21-10} from \eqref{eq10 : Theorem : Toy Fokker-Plank}.

To obtain \eqref{support-uniform}, we exploit
\begin{equation*}
\begin{aligned}
\supp(\rho(t)) &\subset \liminf_{n \rightarrow \infty} (\supp(\rho_n(t))), && \forall ~ t\in[0,T] \\
& \subset \lim_{n \rightarrow \infty} B_{R_{n,T}}(0), && R_{n,T}:= R_0\left(1+\beta \frac{T}{n} \right )^n + \left( 1 + \beta \frac{T}{n}\right )^{n-1} \int_0^T \|V(t) \|_{L_x^\infty} \,dt\\
& \subset  B_{R_T}(0),  && R_{T}:= R_0e^{\beta T} +e^{\beta T} \int_0^T \|V(t) \|_{L_x^\infty} \,dt,
\end{aligned}
\end{equation*}
where we use \eqref{support-n} in the second line. This completes the proof.
\qed

\section{Existence of weak solutions}\label{Exist-weak}
%\footnote{KK: How about `Existence of weak solutions'?}
In this section, we show the existence of $L^q$-weak solutions in Definition~\ref{D:weak-sol}.% where the main theorems are introduced in Section~\ref{S: Main}.
\subsection{Existence for case: $\int_{\bbr^d} \rho_0 \log \rho_0 \,dx<\infty$}
%{\color{red} In this section, we prove Theorems in subsection 1.2.}
Here we provide proofs of theorems in Section~\ref{SS:m less 2}: Theorem~\ref{Theorem-1}, \ref{T:log-div-free}, and \ref{T:log-tilde}.
\subsubsection{Proof of Theorem \ref{Theorem-1}}
{\it Proof of (i).} Let $1< p \leq {\lambda_1} $. Suppose $\rho_0\in  \mathcal{P}_p(\bbr^d) $ with $\int_{\bbr^d}\rho_0 \log \rho_0 \,dx < \infty$ and
 $V \in  \mathcal{S}_{m,1}^{(q_1,q_2)} $ satisfies \eqref{T1:V-log}.
  %with $\int_{\bbr^d} \rho_0^{m-1} \,dx < \infty$.
By exploiting mollification, we may choose a sequence of vector fields $V_n\in \calC^\infty(Q_T) $
such that
\begin{equation}\label{eq7 : Theorem-1}
\lim_{n \rightarrow \infty}\|V_n - V \|_{\mathcal{S}_{m,1}^{(q_1,q_2)} } =0.
\end{equation}
Using truncation, mollification and normalization, we have a sequence of functions
$\rho_{0,n}\in \calC_c^\infty(\bbr^d)\cap \mathcal{P}(\bbr^d)$ satisfying
\begin{equation}\label{eq9 : Theorem-1}
\lim_{n \rightarrow \infty} W_p(\rho_{0,n} , \rho_0) =0,
\quad \text{and} \quad  \lim_{n \rightarrow \infty} \int_{\bbr^d} \rho_{0,n} \log \rho_{0,n} \,dx = \int_{\bbr^d} \rho_0\log \rho_{0} \,dx.
\end{equation}
We note that $\rho_{0,n}\in \calC_c^\infty(\bbr^d)\cap \mathcal{P}(\bbr^d)$ implies $\rho_{0,n}\in \calC_c^\infty(\bbr^d)\cap \mathcal{P}_2(\bbr^d)$.
Hence, exploiting Proposition \ref{proposition : regular existence}, we have $\rho_n\in AC(0,T; \mathcal{P}_2(\bbr^d))$ which is a regular solution of
\begin{equation}\label{eq5 : Theorem-1}
\begin{cases}
   \partial_t \rho_n =   \nabla \cdot( \nabla (\rho_n)^m  - V_n\rho_n),\\
   \rho_n(0)=\rho_{0,n}.
 \end{cases}
\end{equation}
Thus, $\rho_n$ satisfies
\begin{equation}\label{eq21 : Theorem-1}
\begin{aligned}
\iint_{Q_T}  \left [\partial_t\varphi + (V_n \cdot \nabla \varphi )\right ]\rho_n
 + (\rho_n)^m \Delta \varphi  \,dx\,dt
 =  - \int_{\bbr^d}\varphi(0) \rho_{0,n}\,dx,
\end{aligned}
\end{equation}
for each $\varphi\in \calC_c^\infty( \bbr^d \times
[0,T))$, and $\rho_n$ enjoys all properties in
Proposition \ref{proposition : regular existence}. Especially, we note
\begin{equation}\label{eq - regular}
\sup_{0 \leq t \leq T}\int_{\bbr^d \times \{ t\}}\rho_n \,dx + \iint_{Q_T} \left|\nabla \rho_{n}^{\frac{m}{2}} \right|^2 \,dx\,dt < \infty.
\end{equation}

Now, we prove $\rho_n \in AC(0,T;\mathcal{P}_p(\bbr^d))$. For this, we first note that $\rho_{0,n}\in \calC_c^\infty(\bbr^d)\cap \mathcal{P}(\bbr^d)$ gives us
 $\rho_{0,n}\in \mathcal{P}_p(\bbr^d)$. Next, due to the assumption \eqref{T1:V-log},
  we can exploit Proposition \ref{P:log-energy} $(i)$ and Proposition \ref{P:log-energy-speed} $(i)$ to have
\begin{equation}\label{eq17 : Theorem-1}
\sup_{0 \leq t \leq T}\int_{\bbr^d \times \{ t\}}\rho_n(|\log \rho_n| + \langle x\rangle^p ) \,dx
+ \iint_{Q_T} \{ \left|\nabla \rho_{n}^{\frac{m}{2}} \right|^2 + \left (\left | \frac{\nabla \rho_n^m}{\rho_n}\right |^p +|V_n|^p \right )\rho_n \}\,dx\,dt <C,
\end{equation}
%and
%\begin{equation}\label{eq11 : Theorem-1}
%\sup_{0 \leq t \leq T}\int_{\bbr^d \times \{ t\}}\rho_n(|\log \rho_n| + \langle x\rangle^p ) \,dx
%+\iint_{Q_T} \left (\left | \frac{\nabla \rho_n^m}{\rho_n}\right |^p +|V_n|^p \right )\rho_n\,dx \,dt <C,
%\end{equation}
 where the constant $C=C (\|V\|_{\mathcal{S}_{m, 1}^{(q_1,q_2)}}, \int_{\mathbb{R}^d} \rho_{0,n} \left(\log \rho_{0,n} + \langle x \rangle^p )
\,dx\right)$.
Hence, we have $\rho_n(t)\in \mathcal{P}_p (\bbr^d)$ for all $t\in [0, T]$. Actually, due to \eqref{eq7 : Theorem-1} and \eqref{eq9 : Theorem-1}, we may choose the constant
$C=C(\|V\|_{\mathcal{S}_{m, 1}^{(q_1,q_2)}}, \int_{\mathbb{R}^d} \rho_{0} \left(\log \rho_{0} + \langle x \rangle^p )
\,dx\right)$. Now, we rewrite \eqref{eq5 : Theorem-1} as follows
$$\partial_t \rho_n +  \nabla \cdot (w_n \rho_n )=0, \qquad \mbox{where}\quad
w_n:= -\frac{\nabla (\rho_n)^m}{\rho_n}+ V_n.$$
Then, \eqref{eq17 : Theorem-1} says
\begin{equation*} %\label{eq25 : Theorem-1}
\sup_{0 \leq t \leq T}\int_{\bbr^d \times \{ t\}}\rho_n(|\log \rho_n| + \langle x\rangle^p ) \,dx +\int_0^T \|w_{n}(t)\|_{L^p(\rho_n(t))}^p \,dt <C.
\end{equation*}
We exploit Lemma \ref{representation of AC curves} and \eqref{eq17 : Theorem-1} to have $\rho_n \in AC(0,T;\mathcal{P}_p(\bbr^d))$ and
\begin{equation}\label{eq30 : Theorem-1}
\begin{aligned}
W_p(\rho_n(s),\rho_n(t)) &\leq \int_s^t  \|w_n(\tau)\|_{L^p(\rho_n(\tau))} \,d\tau \leq  C (t-s)^{\frac{p-1}{p}}, \qquad \forall ~ 0\leq s\leq t\leq T.
\end{aligned}
\end{equation}

Next, from  Lemma \ref{Lemma : Arzela-Ascoli} with  \eqref{eq30 : Theorem-1}, we note that there
exists a curve $\rho:[0,T] \mapsto \mathcal{P}_p(\bbr^d)$ such that
\begin{equation}\label{eq33 : Theorem-1}
\rho_n(t) \,\mbox{(up to a subsequence) \,narrowly converges to}\, \rho(t) \,\mbox{as} \, n \rightarrow \infty, \quad \forall \ 0\leq t \leq T.
\end{equation}

Due to the lower semicontinuity of $p$-Wasserstein distance $W_p(\cdot, \cdot)$ with respect to the narrow convergence, we have
\begin{equation}\label{eq34 : Theorem-1}
 W_p(\rho(s),\rho(t)) \leq   C (t-s)^{\frac{p-1}{p}},
\qquad \forall ~ 0\leq s\leq t\leq T,
\end{equation}
which concludes \eqref{KK-May7-60}.

Let us prove that the curve $\rho :[0,T] \mapsto \mathcal{P}_p(\bbr^d)$ is a solution of \eqref{E:Main}. First, we note%
%We note that the uniform bounds on $\| \rho_n\|_{L^{m-1}_x}$ and $p$-th moments given in \eqref{eq11 : Theorem-1} and \eqref{eq12 : Theorem-1} respectively imply
%$\rho(t) \in \mathcal{P}_p(\bbr^d) \cap L^{m-1}(\bbr^d)$ for all $t\in[0,T]$. Furthermore, those uniform bound together with the convergence in \eqref{eq7 : Theorem-1}
%gives us
\begin{equation}\label{eq32 : Theorem-1}
\begin{aligned}
\iint_{Q_T}  (V_n \cdot \nabla \varphi  ) \rho_n \,dx\,dt
\longrightarrow  \iint_{Q_T}  (V \cdot \nabla \varphi ) \rho \,dx\,dt, \,~\mbox{as} \ n\rightarrow \infty,
\end{aligned}
\end{equation}
for any $\varphi \in \calC_c^\infty(\bbr^d \times [0,T)) $. Indeed, we have
\begin{equation*}
\begin{aligned}
&\iint_{Q_T}  \big[(V_n \cdot \nabla \varphi ) \rho_n
- ( V \cdot \nabla \varphi ) \rho\big ] \,dx\,dt \\
&\quad = \iint_{Q_T}  \nabla \varphi \cdot \big (V_n- V\big )\rho_n  \,dx\,dt
+ \iint_{Q_T}  \nabla \varphi \cdot  V \bke{\rho_n- \rho}  \,dx\,dt
 = I + II.
\end{aligned}
\end{equation*}
 Let us first estimate $I$.  Due to \eqref{eq17 : Theorem-1}, we apply Lemma \ref{P:L_r1r2} and  Remark \ref{Remark : AL-1} to obtain
\begin{equation}\label{eq28 : Theorem-1}
%\begin{aligned}
\|(V_n-V)\rho_n \|_{L_{x,t}^{\frac{md+2}{md+1}}}
\leq \|V_n-V \|_{{\mathcal{S}}_{m,1}^{({q}_1, {q}_2)}} \| \rho_n\|_{L_{x,t}^{r_1, r_2}}
 \leq C %\left( \int_{\bbr^d}\rho_0(\log \rho_0 + \langle x\rangle^p) \,dx,~ \|V\|_{{\mathcal{S}}_{m,1}^{({q}_1, {q}_2)}}\right)
\|V_n-V \|_{\mathcal{S}_{m,1}^{(q_1,q_2)}} .
%\end{aligned}
\end{equation}
Hence, we have
\begin{equation}\label{eq13 : Theorem-1}
%\begin{aligned}
| I |\leq \|\nabla \varphi \|_{L_{x,t}^{md+2}} \|(V_n - V)\rho_n \|_{L_{x,t}^{\frac{md+2}{md+1}}}
\leq C \|V_n - V \|_{\mathcal{S}_{m,1}^{(q_1,q_2)}} \longrightarrow 0, \quad \mbox{as} \quad n\rightarrow \infty.
%\end{aligned}
\end{equation}
 Also, narrow convergence in \eqref{eq33 : Theorem-1} implies
$II$ converges to $0$ as $n \rightarrow \infty$.  Hence, this with \eqref{eq13 : Theorem-1} implies \eqref{eq32 : Theorem-1}.

 Next, we study the convergence of following nonlinear term
\begin{equation}\label{eq14 : Theorem-1}
\begin{aligned}
\iint_{Q_T}   (\rho_n)^m \Delta \varphi \,dx\,dt  \longrightarrow  \iint_{Q_T}  \rho^m \, \Delta \varphi   \,dx\,dt.
\end{aligned}
\end{equation}
Due to \eqref{E:rho:1} in Remark \ref{Remark : AL-1}, we have
\begin{equation}\label{eq27 : Theorem-1}
\begin{aligned}
\|\rho_n^m \|_{L_{x,t}^{\frac{md+2}{md}}} %& \leq C\left(\int_{\bbr^d} \rho_{0,n} (\log \rho_{0,n} + \langle x \rangle^p)\,dx, ~ \|V_n\|_{\mathcal{S}_{m,1}^{(q_1,q_2)}} \right)\\
& \leq C.
\end{aligned}
\end{equation}
This implies that $\rho_n^m$(up to a subsequence) has a weak limit $\eta \in L^{\frac{md+2}{md}}\left( \bar{Q}_T \right)$. That is
\begin{equation}\label{eq23 : Theorem-1}
\begin{aligned}
\iint_{Q_T}   (\rho_n)^m \Delta \varphi  ~\,dx\,dt
 \longrightarrow  \iint_{Q_T}   \eta \Delta \varphi  ~\,dx\,dt.
\end{aligned}
\end{equation}
 Furthermore, we combine Proposition \ref{Proposition : AL-1} with \eqref{eq27 : Theorem-1} to have
\begin{equation*} %\label{eq26 : Theorem-1}
\| \partial_t\rho_n\|_{W_x^{-1,\frac{md+2}{md+1}}L_t^{\frac{md+2}{md+1}}} + \|\rho_n \|_{W_x^{1,\frac{md+2}{d+1}}L_t^{\frac{md+2}{d+1}}}
 \leq C .%\left (\|V \|_{\mathcal{S}_{m,1}^{(q_1,q_2)}}, ~ \int_{\bbr^d} (\log \rho_0 + \langle x \rangle ^p)dx \right).
\end{equation*}
By applying Lemma \ref{AL}, we have $\rho_n$ (up to a subsequence) which converges in $L^{\frac{md+2}{d+1}}(0, T; L_{\loc}^{r}(\bbr^d))$ for
$r < \frac{d(md+2)}{d(d-m)+(d-2)}$ (see Remark~\ref{R:AL-1}).
That implies $\rho_n$ (up to a subsequence) converges almost everywhere in $\bar{Q}_T$.
Now, we recall that $\rho_n(t)$ narrowly converges to $\rho(t)$ for any $t \in [0,T]$.
We put these two facts together and conclude $\rho_n$  converges to $\rho$ almost everywhere in  $\bar{Q}_T$. Hence,  $\rho_n^m$ also pointwise converges to $\rho^m$.
Combining this with \eqref{eq23 : Theorem-1}, we conclude \eqref{eq14 : Theorem-1}.
%\begin{equation}\label{eq17 : Theorem-1}
%\begin{aligned}
%\iint_{Q_T}   (\rho_n)^m \Delta \varphi  \,dx\,dt
% \longrightarrow  \iint_{Q_T}   \Delta \varphi \rho^m \,dx\,dt.
%\end{aligned}
%\end{equation}
We plug \eqref{eq33 : Theorem-1}, \eqref{eq32 : Theorem-1}, \eqref{eq14 : Theorem-1} into \eqref{eq21 : Theorem-1}, and have \eqref{KK-May7-40}.
%\begin{equation*}
%\begin{aligned}
%\iint_{Q_T}  \left [ \partial_t\varphi +(V \cdot \nabla \varphi  )\right]\rho  + \rho^m \Delta \varphi  \,dx\,dt =  - \int_{\bbr^d}\varphi(x,0) \rho_0(x) \,dx.
%\end{aligned}
%\end{equation*}
%Also, the lower semicontinuity of $W_p$ with respect to narrowly convergence and the estimate \eqref{eq30 : Theorem-1} give us \eqref{KK-May7-60}.

 Next, we prove \eqref{T1:apriori-log}.  From \eqref{eq17 : Theorem-1}, we have
\begin{equation}\label{eq24 : Theorem-1}
\| \nabla \rho_n^{\frac{m}{2}}\|_{L_{x,t}^2} \leq C.
\end{equation}
As before, with the almost everywhere convergence of $\rho_n$ to $\rho$, this gives us that $\nabla \rho_n^{\frac{m}{2}}$ weakly converges to $\nabla \rho^{\frac{m}{2}}$
 in $L^2(\bar{Q}_{T})$. Hence, we have
\begin{equation}\label{eq18 : Theorem-1}
\begin{aligned}
\|\nabla \rho^{\frac{m}{2}} \|_{L_{x,t}^{2}} %& \leq C\left(\int_{\bbr^d} \rho_{0,n} (\log \rho_{0,n} + \langle x \rangle^p)\,dx, ~ \|V_n\|_{\mathcal{S}_{m,1}^{(q_1,q_2)}} \right)\\
& \leq C.
\end{aligned}
\end{equation}
Moreover, \eqref{W-p} and \eqref{eq18 : Theorem-1} give us
\begin{equation}\label{eq19 : Theorem-1}
\iint_{Q_T}\left |\frac{\nabla \rho^m}{\rho} \right|^p \rho \,dx \,dt  \leq C.
\end{equation}
Next, we note
\begin{equation}\label{eq22 : Theorem-1}
\begin{aligned}
\iint_{Q_T} |V_n|^p \rho_n \,dx \,dt \longrightarrow  \iint_{Q_T} |V|^p \rho \,dx \,dt, \quad \mbox{as}\quad n \rightarrow \infty.
\end{aligned}
\end{equation}
Indeed, we have
\begin{equation*}
%\begin{aligned}
\iint_{Q_T}  \left(|V_n|^p \rho_n - |V|^p \rho \right) \,dx\,dt
 = \iint_{Q_T}   \left(|V_n|^p  - |V|^p \right)\rho_n  \,dx\,dt
+ \iint_{Q_T}  |V|^p \left(\rho_n - \rho \right)  \,dx\,dt
 = I + II.
%\end{aligned}
\end{equation*}
 Let us first estimate $I$.  We exploit \eqref{speed00.5} in Lemma \ref{P:W-p} to have
\begin{equation*}
%\begin{aligned}
\iint_{Q_T} \left| V_n-V\right |^{{\lambda_1}} \rho_n \,dx \,dt
\leq  \int_0^T \|V_n-V \|_{L_x^{q_1}}^{{\lambda_1}} \|\nabla \rho_n^{\frac{m}{2}} \|_{L_x^2}^{\frac{2(1-\theta)}{m}},  %\\
%\end{aligned}
\end{equation*}
where $\theta$ and $q_2$ are defined in \eqref{speed00} and \eqref{speed01} for $q=1$, respectively. Notice that $\frac{{\lambda_1}}{q_2}+\frac{1-\theta}{m}=1 $, and hence we have
\begin{equation}\label{eq35-1 : Theorem-1}
%\begin{aligned}
\iint_{Q_T} \left| V_n-V\right |^{{\lambda_1}} \rho_n \,dx \,dt \leq  \|V_n-V \|_{\mathcal{S}_{m,1}^{(q_1,q_2)}}^{{\lambda_1}} \|\nabla \rho_n^{\frac{m}{2}} \|_{L_{x,t}^2}^{\frac{2(1-\theta)}{m}}
\leq C \|V_n-V \|_{\mathcal{S}_{m,1}^{(q_1,q_2)}}^{{\lambda_1}},  %\\
%\end{aligned}
\end{equation}
where we exploit \eqref{eq24 : Theorem-1} in the second inequality. For any $1< p< {\lambda_1}$, from \eqref{speed03}, we have
\begin{equation}\label{eq35 : Theorem-1}
%\begin{aligned}
| I | \leq  C \|  |V_n-V|^p\rho_n \|_{L_{x,t}^1}
\leq  C \|  |V_n-V|^{{\lambda_1}}\rho_n \|_{L_{x,t}^1}^{\frac{p}{{\lambda_1}}}
\leq C \|V_n - V \|_{\mathcal{S}_{m,1}^{(q_1,q_2)}}^p \longrightarrow 0, \quad \mbox{as} \quad n\rightarrow \infty.
%\end{aligned}
\end{equation}
 Also, narrow convergence in \eqref{eq33 : Theorem-1} implies $II$ converges to $0$ as $n \rightarrow \infty$. This with \eqref{eq35 : Theorem-1} gives us \eqref{eq22 : Theorem-1}
 which implies
 \begin{equation}\label{eq36 : Theorem-1}
\iint_{Q_T} |V|^p \rho \,dx \,dt  \leq C.
 \end{equation}
%Next, from \eqref{eq17 : Theorem-1}, we have
%\begin{equation*}\label{eq8 : Theorem-1}
%\begin{aligned}
%\sup_{0\leq t \leq T} &\int_{\mathbb{R}^d\times \{t\}} \rho_n \left( \abs{\log \rho_n} + \langle x \rangle^p \right) \,dx  \leq  C.
%\end{aligned}
%\end{equation*}
%where the constant $C=C\left( m,p,d, T, \|V\|_{\mathcal{S}_{m, 1}^{(q_1,q_2)}}, \int_{\mathbb{R}^d} \rho_{0} \left(\log \rho_{0} + \langle x \rangle^p \right) \,dx\right)$.
 We remind that the entropy and $p$-th moment are lower semicontinuous with respect to the narrow convergence  (refer \cite[Remark 9.3.8 and (5.1.15)]{ags:book}). Hence, from \eqref{eq17 : Theorem-1}, we have for any $t\in [0, T]$
\begin{equation}\label{eq15 : Theorem-1}
\begin{aligned}
\int_{\mathbb{R}^d\times \{t\}} \rho \left( \log \rho + \langle x \rangle^p \right) \,dx
  &\leq \liminf_{n \rightarrow \infty} \int_{\mathbb{R}^d\times \{t\}} \rho_n \left( \log \rho_n + \langle x \rangle^p \right) \,dx
  \leq C.
\end{aligned}
\end{equation}
We also remind that there exists a constant $c>0$ (independent of $\rho$) such that
\begin{equation}\label{eq16 : Theorem-1}
\begin{aligned}
\int_{\mathbb{R}^d} | \min \{\rho \log \rho,~0 \} | \,dx \leq c \int_{\mathbb{R}^d}  \rho  \langle x \rangle^p  \, dx.
\end{aligned}
\end{equation}
Combining \eqref{eq15 : Theorem-1} and \eqref{eq16 : Theorem-1}, we have
\begin{equation}\label{eq20 : Theorem-1}
\begin{aligned}
\sup_{0\leq t \leq T} &\int_{\mathbb{R}^d\times \{t\}} \rho \left( \abs{\log \rho} + \langle x \rangle^p \right) \,dx \leq  C.
\end{aligned}
\end{equation}
We combine \eqref{eq18 : Theorem-1}, \eqref{eq19 : Theorem-1}, \eqref{eq36 : Theorem-1} and \eqref{eq20 : Theorem-1} to get \eqref{T1:apriori-log}.
This completes the proof.

{\it Proof of (ii)} :  Since $\mathcal{P}_a(\bbr^d) \subset \mathcal{P}_b(\bbr^d)$ if $a \geq b$, we note  $\rho_0 \in \mathcal{P}_{{\lambda_1}}(\bbr^d)$. Then, we exploit the result of $(ii)$ with $p={\lambda_1}$ to
have  a solution $\rho $ of satisfying \eqref{KK-May7-40}, \eqref{T1:apriori-log-1} and \eqref{KK-May7-60-1} with a constant
$C=C(\|V\|_{\mathcal{S}_{m, 1}^{(q_1,q_2)}}, \int_{\mathbb{R}^d} \rho_{0} \left(\log \rho_{0} + \langle x \rangle^{{\lambda_1}} ) \,dx\right)$.
Since $\int_{\mathbb{R}^d} \rho_{0} \langle x \rangle^{{\lambda_1}} \,dx \leq \int_{\mathbb{R}^d} \rho_{0} \langle x \rangle^{p} \,dx$, we may choose the constant
$C=C(\|V\|_{\mathcal{S}_{m, 1}^{(q_1,q_2)}}, \int_{\mathbb{R}^d} \rho_{0} (\log \rho_{0} + \langle x \rangle^{p} ) \,dx )$. This concludes $(ii)$. \qed

\subsubsection{Proof of Theorem \ref{T:log-div-free}}
Similar to the reason in the proof of Theorem \ref{Theorem-1} $(ii)$, we only focus on the case $1< p \leq {\lambda_1}$.
%For $m>1$ and $p\in (1,  {\lambda_1}]$, let $\rho_{0,n}$ and  $V_n$  be same as in the proof of Theorem \ref{Theorem-1}.
As in the proof of Theorem \ref{Theorem-1}, we have $\rho_n\in AC(0,T; \mathcal{P}_2(\bbr^d))$ which is a regular  solution of \eqref{eq5 : Theorem-1} satisfying \eqref{eq21 : Theorem-1} and \eqref{eq - regular}.
%\begin{equation}\label{eq1 : T:log-div-free}
%\begin{aligned}
%\iint_{Q_T}  \left [\partial_t\varphi + (\nabla
%\varphi \cdot V_n)\right ]\rho_n
% + (\rho_n)^m \Delta \varphi  \,dx\,dt
% =  - \int_{\bbr^d}\varphi(0) \rho_{0,n}\,dx,
%\end{aligned}
%\end{equation}
%for each $\varphi\in \calC_c^\infty( \bbr^d \times [0,T))$ and
%\begin{equation}\label{eq - regular}
%\sup_{0 \leq t \leq T}\int_{\bbr^d \times \{ t\}}\rho_n \,dx + \iint_{Q_T} \left|\nabla \rho_{n}^{\frac{m}{2}} \right|^2 \,dx\,dt < \infty.
%\end{equation}

Since $V$ holds $\nabla \cdot V=0$
and \eqref{T1:V-log-divfree}, $V_n$ is also divergence-free and satisfies \eqref{T1:V-log-divfree} for all $n \in \mathbb{N}$.
Then, from Proposition \ref{P:log-energy} $(iii)$ and Proposition \ref{P:log-energy-speed} $(iii)$, we have \eqref{eq17 : Theorem-1}.
%\begin{equation}\label{eq2 : T1-log-div-free}
%\begin{aligned}
%\sup_{0\leq t \leq T} \int_{\mathbb{R}^d\times \{t\}} \rho_n \left( \abs{\log \rho_n} + \langle x \rangle^p \right) \,dx
%+ \iint_{Q_T} \left[\abs{\nabla \rho_{n}^{\frac{m}{2}}}^2 + \left(\abs{\frac{\nabla \rho_n^m}{\rho_n}}^p+|V_n|^{p}\right ) \rho_n\right] \,dx\,dt   \le C,
% \end{aligned}
%\end{equation}
%where % $C=C  (\|V_n\|_{\mathcal{S}_{m,1}^{(q_1,q_2)}}, \int_{\bbr^d}(\rho_{0,n} \log \rho_{0,n} +\rho_{0,n}\langle x \rangle^p ) )$, and hence
%$C=C (\|V\|_{\mathcal{S}_{m,1}^{(q_1,q_2)}}, \int_{\bbr^d}(\rho_{0} \log \rho_{0} +\rho_{0}\langle x \rangle^p ))$.

As in the proof of Theorem \ref{Theorem-1}, we exploit the estimate \eqref{eq17 : Theorem-1} and note that there
exists a curve $\rho:[0,T] \mapsto \mathcal{P}_p(\bbr^d)$ satisfying \eqref{eq33 : Theorem-1} and \eqref{eq34 : Theorem-1}.

We note that Remark \ref{Remark : AL-1}, \ref{R:AL-1}, and Proposition \ref{Proposition : AL-1} follow from Proposition \ref{P:log-energy}, and also note that Lemma \ref{P:W-p} holds from
\eqref{eq - regular} and the assumption \eqref{T1:V-log-divfree} on $V$. %We also note that Remark \ref{Remark : AL-1}, Proposition \ref{Proposition : AL-1} and Lemma \ref{P:W-p} hold.
 Once  Remark \ref{Remark : AL-1}, \ref{R:AL-1}, Proposition \ref{Proposition : AL-1} and Lemma \ref{P:W-p} are available, we follow exactly same steps in the proof of Theorem \ref{Theorem-1}
 to conclude \eqref{KK-May7-40} and \eqref{T1:apriori-log} from \eqref{eq21 : Theorem-1} and \eqref{eq17 : Theorem-1}, respectively.

 Let us be more precise. In the proof of Theorem \ref{Theorem-1}, \eqref{E:Vrho:1-1} in Remark \ref{Remark : AL-1} gives \eqref{eq28 : Theorem-1}
%\begin{equation}\label{eq3 : T1-log-div-free}
%\begin{aligned}
%\|  (V_n-V)\rho_n \|_{L_{x,t}^{\frac{md+2}{md+1}}}  &\leq C \|V_n-V \|_{\mathcal{S}_{m,1}^{(q_1,q_2)}}   \longrightarrow 0, \quad \mbox{as} \quad n\rightarrow \infty.
%\end{aligned}
%\end{equation}
which plays a key role to have \eqref{eq32 : Theorem-1}.
%\begin{equation}\label{eq4 : T1-log-div-free}
%\begin{aligned}
%\iint_{Q_T}  (V_n \cdot \nabla \varphi ) \rho_n  \,dx\,dt \longrightarrow \iint_{Q_T}  (V \cdot \nabla \varphi ) \rho \,dx\,dt\quad
%\mbox{as} \quad n\rightarrow \infty.
%\end{aligned}
%\end{equation}
Furthermore, for $1< m \leq 2$, Proposition
\ref{Proposition : AL-1} and Remark~\ref{R:AL-1} with Aubin-Lions lemma give \eqref{eq14 : Theorem-1},
%\begin{equation*}\label{eq5 : T1-log-div-free}
%\begin{aligned}
%\iint_{Q_T}  (\rho_n)^m \Delta \varphi    \,dx\,dt \longrightarrow \iint_{Q_T}   \rho^m \Delta \varphi   \,dx\,dt\quad
%\mbox{as} \quad n\rightarrow \infty,
%\end{aligned}
%\end{equation*}
which together with \eqref{eq32 : Theorem-1} provides crucial ingredients to have \eqref{KK-May7-40}.
On the other hand, we exploit \eqref{speed00.5} in Lemma \ref{P:W-p} to have \eqref{eq35-1 : Theorem-1}
%\begin{equation*}\label{eq6 : T1-log-div-free}
%\begin{aligned}
%\iint_{Q_T} \left| V_n-V\right |^{{\lambda_1}} \rho_n \,dx \,dt &\leq   C \|V_n-V \|_{\mathcal{S}_{m,1}^{(q_1,q_2)}}^{{\lambda_1}}\longrightarrow 0, \quad \mbox{as} \quad n\rightarrow \infty,  %\\
%\end{aligned}
%\end{equation*}
and this gives \eqref{eq36 : Theorem-1}. Having \eqref{eq36 : Theorem-1}, it follows  \eqref{T1:apriori-log} and this completes the proof.
\qed

\subsubsection{Proof of Theorem \ref{T:log-tilde}}
The proof is completed in the exactly same as in Theorem
\ref{Theorem-1} or Theorem \ref{T:log-div-free}. The only difference is that,
instead of  \eqref{eq7 : Theorem-1}, we choose a sequence of vector
fields $V_n\in \calC^\infty(\bar{Q}_T) $ such that
\begin{equation*} %\label{eq7 : T:log-tilde}
\lim_{n \rightarrow \infty}\|V_n - V \|_{\tilde{\mathcal{S}}_{m,1}^{(\tilde{q}_1,\tilde{q}_2)} } =0.
\end{equation*}
And also, in stead of \eqref{eq28 : Theorem-1}, we have
\begin{equation}\label{eq3 : T:log-tilde}
%\begin{aligned}
\|  (V_n-V)\rho_n \|_{L_{x,t}^{\frac{md+2}{md+1}}}
\leq C \|V_n-V \|_{\mathcal{S}_{m,1}^{(q_1,q_2)}}
 \leq C \|V_n-V \|_{\tilde{\mathcal{S}}_{m,1}^{(\tilde{q}_1,\tilde{q}_2)}}  \longrightarrow 0, \quad \mbox{as} \quad n\rightarrow \infty,
%\end{aligned}
\end{equation}
where $C=C (\|V\|_{\tilde{\mathcal{S}}_{m,1}^{(\tilde{q}_1,\tilde{q}_2)}}, \int_{\bbr^d}(\rho_{0} \log \rho_{0} +\rho_{0}\langle x \rangle^p )\,dx )$. Here, $q_1:=\frac{d \tilde{q}_1}{d-\tilde{q}_1}, ~ q_2:=\tilde{q}_2$.
We exploit \eqref{norm-V-gradV} in the second inequality of \eqref{eq3 : T:log-tilde} because the assumption \eqref{T:V-tilde-log-energy} guarantees $1<\tilde{q}_1<d$ (see Remark~\ref{R:S-tildeS} $(ii)$).

\subsection{Existence for case: $\rho_0 \in L^{q}(\mathbb{R}^d), \ q>1$}
Here we prove theorems in Section~\ref{SS:m greater 2}: Theorem~\ref{Theorem-2-a} and \ref{Theorem-2-b}.

\subsubsection{Proof of Theorem~\ref{Theorem-2-a}}
%Let $1<m \leq 2$ and $q>1$.
%Since both theorems can be proved in exactly the same way, we only show the proof of Theorem \ref{Theorem-2-a}.  %For $m, q >1$, let us assume $m-q+1 \geq 0$ holds that there are two cases, either $1<m \leq 2$ \& $q>1$ (Theorem~\ref{Theorem-2-a}) or $m>2$ \& $q\geq m-1$ (Theorem~\ref{Theorem-2-aa}).

Suppose $\rho_0\in  \mathcal{P}(\bbr^d) \cap L^{q}(\bbr^d)$ and $V \in  \mathcal{S}_{m,q}^{(q_1,q_2)} $ satisfies \eqref{T2a:V}.
  %with $\int_{\bbr^d} \rho_0^{m-1} \,dx < \infty$.
By exploiting mollification, we may choose a sequence of vector fields $V_n\in \calC^\infty(\bar{Q}_T) $ such that
\begin{equation}\label{eq7 : Theorem-2-a}
\lim_{n \rightarrow \infty}\|V_n - V \|_{\mathcal{S}_{m,q}^{(q_1,q_2)} } =0.
\end{equation}
Using truncation, mollification and normalization, we have a sequence of functions
$\rho_{0,n}\in \calC_c^\infty(\bbr^d)\cap \mathcal{P}(\bbr^d)$ satisfying
\begin{equation}\label{eq9 : Theorem-2-a}
\lim_{n \rightarrow \infty}\|\rho_{0,n} - \rho_0 \|_{L^q(\bbr^d)} =0.
%, \quad
%\lim_{n \rightarrow \infty} \int_{\bbr^d} (\rho_{0,n})^{q} \,dx =
%\int_{\bbr^d} (\rho_0)^{q} \,dx.
\end{equation}
As in the proof of Theorem \ref{Theorem-1}, we have $\rho_n\in AC(0,T; \mathcal{P}_2(\bbr^d))$ a regular  solution of \eqref{eq5 : Theorem-1} satisfying \eqref{eq21 : Theorem-1} %a regular solution of
%\begin{equation*}\label{eq5 : Theorem-2-a}
%\begin{cases}
%    \partial_t \rho_n =   \nabla \cdot( \nabla (\rho_n)^m  - V_n\rho_n)\\
%    \rho_n(0)=\rho_{0,n},
% \end{cases}
%\end{equation*}
%that is,
%\begin{equation}\label{eq21 : Theorem-2-a}
%\begin{aligned}
%\iint_{Q_T} \{ \left [\partial_t\varphi + V_n \cdot \nabla
%\varphi \right ]\rho_n
% + (\rho_n)^m \Delta \varphi \} \,dx\,dt
% =  - \int_{\bbr^d}\varphi(0) \rho_{0,n}\,dx,
%\end{aligned}
%\end{equation}
%for each $\varphi\in \calC_c^\infty( \bbr^d \times [0,T))$
and
\begin{equation*} %\label{eq - regular - q}
\sup_{0 \leq t \leq T}\int_{\bbr^d \times \{ t\}} (\rho_n)^q \,dx + \iint_{Q_T} \left|\nabla \rho_{n}^{\frac{q+m-1}{2}} \right|^2 \,dx\,dt < \infty.
\end{equation*}

%Since $q \geq m-1$ and $q>1$,
From Proposition \ref{P:Lq-energy} $(i)$, we have
\begin{equation}\label{eq12 : Theorem-2-a}
\begin{aligned}
\sup_{0\leq t \leq T} &\int_{\mathbb{R}^d\times \{t\}}  (\rho_n)^q   \,dx
 + \frac{2qm(q-1)}{(q+m-1)^2} \iint_{Q_T} \left |\nabla \rho_n^{\frac{q+m-1}{2}}\right |^2  \,dx\,dt \leq C,%\\
%& \leq  C_0\left( m,p,d, T, \|V_n\|_{\mathcal{S}_{m, q}^{(q_1,q_2)}},
%\int_{\mathbb{R}^d}  \left( (\rho_{0,n})^q  \right) \,dx\right),
\end{aligned}\end{equation}
where $C= C( \|V_n\|_{\mathcal{S}_{m, q}^{(q_1,q_2)}}, \|\rho_{0,n}\|_{L^q(\mathbb{R}^d)} )$.
Due to \eqref{eq7 : Theorem-2-a} and \eqref{eq9 : Theorem-2-a}, we may choose
$C= C( \|V\|_{\mathcal{S}_{m, q}^{(q_1,q_2)}}, \|\rho_{0}\|_{L^q(\mathbb{R}^d)} )$.

From \eqref{eq12 : Theorem-2-a}, we note that $\rho_n$ is bounded in $L^{r_1, r_2}_{x,t}(\bar{Q}_T)$ where $(r_1, r_2)$ are satisfying \eqref{q-r1r2}.
Hence, there exists a subsequence (by abusing notation) $\rho_n$ weakly converges to
$\rho \in L^{r_1, r_2}_{x,t}(\bar{Q}_T) $ (in particular, it is weak* convergence if $r_1=\infty$ or $r_2 = \infty$).

We claim that $\rho$ is a solution of \eqref{E:Main}. First, we have
\begin{equation}\label{eq32 : Theorem-2-a}
\begin{aligned}
\iint_{Q_T}  (V_n \cdot \nabla \varphi  ) \rho_n
\,dx\,dt \longrightarrow  \iint_{Q_T}  (V \cdot \nabla
\varphi) \rho \,dx\,dt, ~~~\mbox{as}~~ n\rightarrow \infty,
\end{aligned}
\end{equation}
for any $\varphi \in \calC_c^\infty(\bbr^d \times [0,T)) $. Indeed, we let
\begin{equation*}
\begin{aligned}
I: = \iint_{Q_T}  \nabla \varphi \cdot \big (V_n-
V\big )\rho_n  \,dx\,dt \quad \text{and} \quad  II:= \iint_{Q_T}  \nabla \varphi \cdot  \bke{V\rho_n- V\rho} ~ \,dx\,dt.
\end{aligned}
\end{equation*}
From Lemma~\ref{L-compact}, we have
\begin{equation}\label{eq13 : Theorem-2-a}
\begin{aligned}
I &\leq \left(\iint_{Q_T} \left|(V_n-V) \rho_n^q \nabla \varphi \right| \,dxdt \right)^{\frac 1q} \left(\iint_{Q_T} \abs{ (V_n-V)  \nabla \varphi} \,dxdt \right)^{1-\frac 1q} \\
&\leq \left(\iint_{Q_T} \abs{ (V_n-V) \rho_n^q  \nabla \varphi }^{r} \,dxdt \right)^{\frac{1}{qr}} \left(\iint_{Q_T} \abs{\nabla \varphi} \,dxdt \right)^{\frac 1q(1-\frac 1r)} \left(\iint_{Q_T} \abs{ (V_n-V) \nabla \varphi} \,dxdt \right)^{1-\frac 1q} \\
& \leq C(\varphi) \|V_n-V\|_{\mathcal{S}_{m,q}^{(q_1,q_2)}}^{\frac 1q} \|\rho_n\|_{L^{r_1, r_2}_{x,t}} \|V_n-V\|_{\mathcal{S}_{m,q}^{(q_1,q_2)}}^{1-\frac 1q}\\
& \leq C(\varphi) \|V_n-V\|_{\mathcal{S}_{m,q}^{(q_1,q_2)}} \|\rho_n\|_{L^{r_1, r_2}_{x,t}} \longrightarrow 0, \quad \mbox{as}\quad n\rightarrow \infty.
\end{aligned}
\end{equation}
Also, weak convergence of $\rho_n$ implies that $II$ converges
to $0$ as $n \rightarrow \infty$. This together with \eqref{eq13 : Theorem-2-a} gives us \eqref{eq32 : Theorem-2-a}.

Next, we study the convergence of following nonlinear term
\begin{equation}\label{eq14 : Theorem-2-a}
\begin{aligned}
\iint_{Q_T}  (\rho_n)^m  \Delta \varphi \,dx\,dt
\longrightarrow  \iint_{Q_T}   \rho^m \Delta \varphi
\,dx\,dt.
\end{aligned}
\end{equation}
Due to \eqref{E:rho:q-1} in Remark \ref{Remark : AL-2}, we have
\begin{equation}\label{eq17 : Theorem-2-a}
\begin{aligned}
\| \rho_n^{m}\|_{L_{x,t}^{\frac{d(q+m-1)+2q}{md}}} &\leq C, %\left(
%\|\rho_{0,n}\|_{L^{q}}, ~ \|V_n\|_{\mathcal{S}_{m,q}^{(q_1,q_2)}}\right) \\
%& \leq C \left( \|\rho_{0}\|_{L^{q}}, ~ \|V\|_{\mathcal{S}_{m,q}^{(q_1,q_2)}}\right).
\end{aligned}
\end{equation}
where $C= C( \|V\|_{\mathcal{S}_{m, q}^{(q_1,q_2)}}, \|\rho_{0}\|_{L^q(\mathbb{R}^d)} )$.
This implies $(\rho_n)^m$(up to a subsequence) has a weak limit $\eta \in L^{\frac{d(q+m-1)+2q}{md}}( Q_T )$.
In other words, we have
\begin{equation}\label{eq16 : Theorem-2-a}
\begin{aligned}
\iint_{Q_T}   (\rho_n)^m \Delta \varphi  \,dx\,dt
 \longrightarrow  \iint_{Q_T}   \eta \Delta \varphi  \,dx\,dt.
\end{aligned}
\end{equation}
 Furthermore, combining Proposition \ref{Proposition : AL-2}$(i)$ and \eqref{eq17 : Theorem-2-a}, we have
 \begin{equation}\label{eq15 : Theorem-2-a}
\|\partial_t (\rho_n)^q \|_{W_x^{-1,\frac{d(q+m-1)+2q}{d(q+m-1)+q}}L_t^1} + \|\nabla (\rho_n)^q \|_{L_{x,t}^{\frac{d(q+m-1)+2q}{q(d+1)}}}
\leq C.
%\left( \|V\|_{\mathcal{S}_{m,q}^{(q_1,q_2)}}, \int_{\bbr^d}((\rho_0)^q+\langle x \rangle^q \rho_0) \,dx \right)
\end{equation}

Combining \eqref{eq15 : Theorem-2-a} and Lemma \ref{AL}, we have $(\rho_n)^q$ (up to a subsequence) which converges in $L^{\frac{d(q+m-1)+2q}{q(d+1)}}(0, T; L_{\loc}^{r}(\bbr^d))$ for
$r < p=\frac{d^2(q+m+1)+2dq}{q(d^2-2) - d(m-1)}$ where $p$ is defined in \eqref{rho_2_2}, Remark~\ref{R:AL-3}.
Hence, $\rho_n$ and  $(\rho_n)^q$ also converge almost everywhere in $\bar{Q}_T$.
Recalling $\rho_n$ weakly converges to $\rho$, we can replace $\eta$ by $\rho^m$ in \eqref{eq16 : Theorem-2-a} which gives us \eqref{eq14 : Theorem-2-a}.
%
%Now, we recall that $\rho_n(t)$ narrowly converges to $\rho(t)$ for any
%$t \in [0,T]$. Furthermore,{\color{red} Aubin-Lions
%\eqref{AL:greater_2:iii}} gives us that $(\rho_n)^{m-1}$ converges
%in $L_x^q L_t^1$ for $1 \leq q < \frac{2d}{d-2}$ which means
%$(\rho_n)^{m-1}$ (up to a subsequence) converges almost everywhere
%in $[0,T] \times \bbr^d$. We put these two facts together and
%conclude $\rho_n$  converges to $\rho$ almost everywhere in $[0,T]
%\times \bbr^d$. Hence,  $(\rho_n)^m$ also pointwise converges to
%$\rho^m$ and, with \eqref{eq16 : Theorem-2-a}, we have
%\begin{equation}\label{eq17 : Theorem-2-a}
%\begin{aligned}
%\iint_{Q_T}   (\rho_n)^m \Delta \varphi   \,dx\,dt
% \longrightarrow  \iint_{Q_T}   \rho^m \Delta \varphi   \,dx\,dt.
%\end{aligned}
%\end{equation}
We plug \eqref{eq32 : Theorem-2-a} and
\eqref{eq14 : Theorem-2-a} into \eqref{eq21 : Theorem-1}, and get \eqref{KK-May7-40}.
%\begin{equation}\label{eq22 : Theorem-2-a}
%\begin{aligned}
%\iint_{Q_T}  \left \{ [\partial_t\varphi + V \cdot \nabla
%\varphi \right ]\rho   + \rho^m \Delta \varphi \}  \,dx\,dt
%=  - \int_{\bbr^d}\varphi(x,0) \rho_0(x) \,dx,
%\end{aligned}
%\end{equation}
%for any $\varphi\in \calC_c^\infty( \bbr^d \times [0,T))$.

Now, we investigate \eqref{T:Energy-Lq}. Using \eqref{eq12 : Theorem-2-a}, weak convergence, and following similarly the
proof of Theorem \ref{Theorem-1}, we have
\begin{equation}\label{eq24 : Theorem-2-a}
\begin{aligned}
\sup_{0\leq t\leq T}\int_{\bbr^d \times \{t\}} \rho^q \,dx + \iint_{Q_T} \abs{\nabla \rho^{\frac{q+m-1}{2}}}^2 \,dxdt  \leq C.
\end{aligned}
\end{equation}
Finally, we combine \eqref{KK-May7-40} and \eqref{eq24 : Theorem-2-a} to complete the proof.% for the case $\hat{q}=q$.
%Finally, we combine \eqref{eq26 : Theorem-2-a}, \eqref{eq25 :
%Theorem-2-a}, \eqref{eq27 : Theorem-2-a} and \eqref{eq28 :
%Theorem-2-a} to conclude \eqref{T:Energy-Speed-Lq}. This completes the
%proof.
\qed

\subsubsection{Proof of Theorem~\ref{Theorem-2-b}}

 {\it For the case$(i)$} :  Similar to the reason in the proof of Theorem \ref{Theorem-1}, if $p>{\lambda_q}$ then $\rho_0 \in \mathcal{P}_{{\lambda_q}}(\bbr^d)$. Therefore, the case (b) follows from (a).
 Hence, we only consider the case (a) i.e. for $1< p\leq {\lambda_q}$.
 Assume that $V \in  \mathcal{S}_{m,q}^{(q_1,q_2)} $ satisfies \eqref{T2:V}.
  %with $\int_{\bbr^d} \rho_0^{m-1} \,dx < \infty$.
As in the proof of Theorem \ref{Theorem-1}, we have a sequence of vector fields $V_n\in \calC^\infty(Q_T) $ such that
\begin{equation}\label{eq7 : Theorem-2-b}
\lim_{n \rightarrow \infty}\|V_n - V \|_{\mathcal{S}_{m,q}^{(q_1,q_2)} } =0,
\end{equation}
and a sequence of functions $\rho_{0,n}\in \calC_c^\infty(\bbr^d)\cap \mathcal{P} (\bbr^d)$ satisfying
\begin{equation*} %\label{eq9 : Theorem-2-b}
\lim_{n \rightarrow \infty} W_p(\rho_{0,n} , \rho_0) =0, \qquad
\lim_{n \rightarrow \infty} \|\rho_{0,n} - \rho_0 \|_{L^q(\bbr^d)} =0.
\end{equation*}
Similar to the proof of Theorem \ref{Theorem-1}, we have $\rho_n\in AC(0,T;
\mathcal{P}_2(\bbr^d))$, a regular  solution of \eqref{eq5 : Theorem-1} satisfying \eqref{eq21 : Theorem-1}.
%\begin{equation*}\label{eq5 : Theorem-2-b}
%\begin{cases}
%    \partial_t \rho_n =   \nabla \cdot( \nabla (\rho_n)^m  - V_n\rho_n),\\
%    \rho_n(0)=\rho_{0,n},
%    \end{cases}
%\end{equation*}
%that is,
%\begin{equation}\label{eq21 : Theorem-2-b}
%\begin{aligned}
%\iint_{Q_T} \{ \left [\partial_t\varphi + V_n \cdot \nabla
%\varphi\right ]\rho_n
% + (\rho_n)^m \Delta \varphi  \}\,dx\,dt
% =  - \int_{\bbr^d}\varphi(0) \rho_{0,n}\,dx,
%\end{aligned}
%\end{equation}
%for each $\varphi\in \calC_c^\infty( \bbr^d \times [0,T))$ and satisfies \eqref{eq - regular - q}.

Note that $\frac{2m-1}{m-1} < \frac{q+m-1}{m-1}$ for $q>m$. Hence, under the assumption \eqref{T2:V} on $V$,
we note that Proposition \ref{P:Lq-energy} $(i)$ and  Proposition \ref{P:Energy-speed} $(i)$ are available, and we have
\begin{equation}\label{eq11 : Theorem-2-b}
%\begin{aligned}
\sup_{0\leq t \leq T} \int_{\mathbb{R}^d\times \{t\}}  \{ (\rho_n)^q + \rho_n\langle x \rangle^p \} \,dx
 + \iint_{Q_T} \{\left |\nabla \rho_n^{\frac{q+m-1}{2}}\right |^2  + \left(\abs{\frac{\nabla \rho_n^m}{\rho_n}}^p+|V_n|^{p}\right ) \rho_n  \}\,dx\,dt \leq C,
%\end{aligned}
\end{equation}
where the constant %$C= C\left( \|V_n\|_{\mathcal{S}_{m, q}^{(q_1,q_2)}}, \int_{\mathbb{R}^d}  \left( (\rho_{0,n})^q + \rho_{0,n}\langle x \rangle^p \right) \,dx\right)$.
%Due to \eqref{eq7 : Theorem-2-b} and \eqref{eq9 : Theorem-2-b}, we may choose
$C= C (  \|V\|_{\mathcal{S}_{m, q}^{(q_1,q_2)}}, \int_{\mathbb{R}^d}  \left( \rho_{0}^q + \rho_{0}\langle x \rangle^p \right) \,dx )$.
%Similarly, from \eqref{Lq-energy}, we have
%\begin{equation}\label{eq12 : Theorem-2-b}
%\begin{aligned}
%\sup_{0\leq t \leq T} &\int_{\mathbb{R}^d\times \{t\}}  (\rho_n)^q   \,dx
% + \frac{qm(q-1)}{(q+m-1)^2} \iint_{Q_T} \left |\nabla \rho_n^{\frac{q+m-1}{2}}\right |^2  \,dx\,dt \leq C_0,%\\
%%& \leq  C_0\left( m,p,d, T, \|V_n\|_{\mathcal{S}_{m, q}^{(q_1,q_2)}},
%%\int_{\mathbb{R}^d}  \left( (\rho_{0,n})^q  \right) \,dx\right),
%\end{aligned}\end{equation}
%where the constant $ C_0=C_0\left( m,p,d, T, \|V\|_{\mathcal{S}_{m, q}^{(q_1,q_2)}}, \int_{\mathbb{R}^d}   \rho_{0}^q   \,dx\right)$.
%%depending on $\int_{\mathbb{R}^d}  \left( (\rho_{0})^q + \rho_{0}\langle x \rangle^p \right) \,dx$
%% and $\|V\|_{\mathcal{S}_{m,q}^{(q_1,q_2)} }$.

As in \eqref{eq30 : Theorem-1}, we exploit Lemma \ref{representation of AC curves} with \eqref{eq11 : Theorem-2-b} and have
\begin{equation*} %\label{eq30 : Theorem-2-b}
W_p(\rho_n(s),\rho_n(t))  \leq  C (t-s)^{\frac{p-1}{p}}, \qquad
\forall ~ 0\leq s\leq t\leq T.
\end{equation*}
Similar to \eqref{eq33 : Theorem-1} and \eqref{eq34 : Theorem-1}, there exists a curve
$\rho:[0,T] \mapsto \mathcal{P}_p(\bbr^d)$ such that, as $n
\rightarrow \infty$ (up to subsequence)
\begin{equation}\label{eq33 : Theorem-2-b}
\rho_n(t) ~~\mbox{narrowly converges to}~~ \rho(t) \quad \mbox{and}
\quad W_p(\rho(s),\rho(t)) \leq   C (t-s)^{\frac{p-1}{p}}, \qquad
\forall ~ 0\leq s\leq t\leq T.
\end{equation}

Furthermore, the curve $\rho :[0,T] \mapsto \mathcal{P}_p(\bbr^d)$ is a solution of \eqref{E:Main}. Indeed, in the proof of Theorem \ref{Theorem-2-a}, we exploited
Remark \ref{Remark : AL-2} and Proposition \ref{Proposition : AL-2} (i), and showed following convergence \eqref{eq32 : Theorem-2-a} and \eqref{eq14 : Theorem-2-a}.
By plugging \eqref{eq33 : Theorem-2-b}, \eqref{eq32 : Theorem-2-a} and
\eqref{eq14 : Theorem-2-a} into \eqref{eq21 : Theorem-1}, we get \eqref{KK-May7-40}.
%\begin{equation*}\label{eq22 : Theorem-2-b}
%\begin{aligned}
%\iint_{Q_T} \{ \left [\partial_t\varphi + (\nabla
%\varphi \cdot V)\right ]\rho   + \Delta \varphi \,  \rho^m \} \,dx\,dt
%=  - \int_{\bbr^d}\varphi(x,0) \rho_0(x) \,dx,
%\end{aligned}
%\end{equation*}
%for each $\varphi\in \calC_c^\infty( \bbr^d \times [0,T))$.

Next, we investigate \eqref{T:Energy-Speed-Lq}. We first note that \eqref{eq11 : Theorem-2-b}, \eqref{eq33 : Theorem-2-b} and the lower semicontinuity of $L^q$-norm w.r.t
the narrow convergence give us \eqref{eq24 : Theorem-2-a}.
%\begin{equation}\label{eq26 : Theorem-2-b}
%\sup_{0\leq t\leq T}\int_{\bbr^d \times \{t\}} \rho^q \,dx + \iint_{Q_T} \abs{\nabla \rho^{\frac{q+m-1}{2}}}^2 \,dxdt \leq C.
%\end{equation}
% As in the proof of Theorem \ref{Theorem-2-a}, we also have
%\begin{equation}\label{eq24 : Theorem-2-b}
%\begin{aligned}
%\| \nabla \rho^{\frac{q+m-1}{2}}\|_{L_{x,t}^2} \leq C.
%\end{aligned}
%\end{equation}
%Combining \eqref{eq26 : Theorem-2-b} and \eqref{eq24 : Theorem-2-b}, we conclude \eqref{T:Energy-Lq}.
Moreover, \eqref{W-p} and \eqref{eq24 : Theorem-2-a}  give us
\begin{equation}\label{eq25 : Theorem-2-b}
%\begin{aligned}
\iint_{Q_T}\left |\frac{\nabla \rho^m}{\rho} \right|^p \rho \,dx \,dt
\leq \| \nabla \rho^{\frac{q+m-1}{2}}\|_{L_{x,t}^2}^2
+ C\left( \sup_{0\leq t\leq T}\int_{\bbr^d \times \{t\}} \rho^q \,dx\right)
\leq C.
%\end{aligned}
\end{equation}

Next, we investigate
\begin{equation}\label{eq27 : Theorem-2-b}
\begin{aligned}
\iint_{Q_T} |V|^p \rho \,dx \,dt \leq C.
\end{aligned}
\end{equation}
First, for the case $1<q\leq m$, we exploit \eqref{speed00.5} to have
\begin{equation}\label{eq29 : Theorem-2-b}
\begin{aligned}
\iint_{Q_T} \left| V_n-V\right |^{{\lambda_q}} \rho_n \,dx \,dt &\leq  c(\|\rho_n \|_{L_{x,t}^{q,\infty}})\int_0^T \|V_n-V \|_{L_x^{q_1}}^{{\lambda_q}} \|\nabla \rho_n^{\frac{q+m-1}{2}}
\|_{L_x^2}^{\frac{2(1-\theta)}{q+m-1}} \,dt,  %\\
\end{aligned}
\end{equation}
where $\theta$ and $q_2$ are defined in \eqref{speed00} and \eqref{speed01}, respectively. Notice that $\frac{{\lambda_q}}{q_2}+\frac{1-\theta}{q+m-1}=1 $, and hence we have
\begin{equation}\label{eq34 : Theorem-2-b}
%\begin{aligned}
\iint_{Q_T} \left| V_n-V\right |^{{\lambda_q}} \rho_n \,dx \,dt \leq  \|V_n-V \|_{L_{x,t}^{q_1, q_2}}^{{\lambda_q}} \|\nabla \rho_n^{\frac{q+m-1}{2}} \|_{L_{x,t}^2}^{\frac{2(1-\theta)}{q+m-1}}
\leq C \|V_n-V \|_{L_{x,t}^{q_1,q_2}}^{{\lambda_q}},
%\end{aligned}
\end{equation}
where we exploit \eqref{eq11 : Theorem-2-b} in the second inequality.\\
Second, for the case $q > m$, we note ${\lambda_q}=2$. Combining \eqref{eq29 : Theorem-2-b} and \eqref{eq34 : Theorem-2-b}, we have
\begin{equation*} %\label{eq35 : Theorem-2-b}
\begin{aligned}
\iint_{Q_T} \left| V_n-V\right |^{{\lambda_q}} \rho_n \,dx \,dt 
&\leq C \|V_n-V \|_{L_{x,t}^{q_1, q_2}}^{{\lambda_q}},  %\\
\end{aligned}
\end{equation*}
where 
$C=C (\|\rho_0\|_{L^q(\mathbb{R}^d)}, ~ \|V\|_{\mathcal{S}_{m,q}^{(q_1,q_2)}} )$.
Either  $1<q\leq m$ or $q>m$, we have
\begin{equation}\label{eq36 : Theorem-2-b}
\begin{aligned}
\iint_{Q_T} \left| V_n-V\right |^{{\lambda_q}} \rho_n \,dx \,dt 
&\leq C \|V_n-V \|_{\mathcal{S}_{m,q}^{(q_1,q_2)}}^{{\lambda_q}}\longrightarrow  0 \quad \mbox{as}\quad n \rightarrow \infty.  %\\
\end{aligned}
\end{equation}
Now, as in the proof of Theorem \ref{Theorem-1}, we exploit \eqref{eq36 : Theorem-2-b} and conclude
\begin{equation*}
\begin{aligned}
\iint_{Q_T} |V_n|^p \rho_n \,dx \,dt \longrightarrow  \iint_{Q_T} |V|^p \rho \,dx \,dt \quad \mbox{as}\quad n \rightarrow \infty,
\end{aligned}
\end{equation*}
which implies \eqref{eq27 : Theorem-2-b}.

Since $p$-th moment is lower semicontinuous w.r.t the narrow convergence, \eqref{eq11 : Theorem-2-b} gives us
\begin{equation*}\label{eq28 : Theorem-2-b}
\sup_{0\leq t\leq T}\int_{\bbr^d \times \{t\}} \rho \langle x\rangle^p \,dx \leq C.
\end{equation*}
Finally, we combine \eqref{eq24 : Theorem-2-a}, \eqref{eq25 : Theorem-2-b}, \eqref{eq27 : Theorem-2-b} and \eqref{eq28 :
Theorem-2-b} to conclude \eqref{T:Energy-Speed-Lq}.

\noindent {\it For the case$(ii)$} : Let $q>m$ and assume $V$ satisfies \eqref{T2:V-embedding}.  Then, from Proposition \ref{P:Energy-embedding} $(i)$, there exist
$q^{\ast} \in (1, q)$ and $q_{2}^{\ast} = q_{2}^{\ast}(q_1) \in [2, q_2)$ such that $ V \in \mathcal{S}_{m, q^\ast}^{(q_1, q_2^{\ast})}$ satisfying \eqref{T2:V}.
Hence, the same conclusions hold as in Theorem~\ref{Theorem-2-b} $(i)$. This completes the
proof.
\qed

\subsection{Existence for case: $\nabla\cdot V=0$ and $\rho_0 \in L^{q}(\mathbb{R}^d), \ q>1$}
Here we prove theorems in Section~\ref{SS:divergence-free}: Theorem~\ref{Theorem-4a} and \ref{Theorem-4}.
\subsubsection{Proof of Theorem \ref{Theorem-4a}}
The proof can be completed almost the same way as in the proof of Theorem \ref{Theorem-2-a}. Let us just point out the differences between the proofs of two theorems.

The first difference  is that Proposition \ref{P:Lq-energy} gives us, due to the divergence-free assumption $\nabla \cdot V=0$
\begin{equation*} %\label{eq1 : Theorem-4a}
\begin{aligned}
\sup_{0\leq t \leq T} &\int_{\mathbb{R}^d\times \{t\}}  (\rho_n)^q   \,dx
 + \frac{qm(q-1)}{(q+m-1)^2} \iint_{Q_T} \left |\nabla \rho_n^{\frac{q+m-1}{2}}\right |^2  \,dx\,dt \leq \int_{\mathbb{R}^d}  \rho_{n,0}^q   \,dx,%\\
%& \leq  C_0\left( m,p,d, T, \|V_n\|_{\mathcal{S}_{m, q}^{(q_1,q_2)}},
%\int_{\mathbb{R}^d}  \left( (\rho_{0,n})^q  \right) \,dx\right),
\end{aligned}\end{equation*}
instead of \eqref{eq12 : Theorem-2-a}. This induces \eqref{T:Energy-divfree}.

The second difference lies in obtaining the estimates of temporal and spatial derivatives when $V$ satisfies \eqref{T4:V-energy}. For the case $q > \max\{1, \frac m2\}$, we exploit Proposition \ref{Proposition : AL-2} and have for $r\in (1, q]$,
\begin{equation}\label{eq 2 : Theorem-4a}
\|\partial_t (\rho_n)^q \|_{W_x^{-1,\frac{d(q+m-1)+2q}{d(q+m-1)+q}}L_t^1} +\|\nabla (\rho_n)^q \|_{L_{x,t}^{\frac{d(q+m-1)+2q}{q(d+1)}}} \cdot \delta_{\{q \geq m-1\}}+ \|\nabla (\rho_n)^q \|_{L_{x,t}^{\frac{2d(q+m-1)+4q}{d(3q-r)+2q}}} \cdot \delta_{\{\frac m 2 < q < m-1\}}
\leq C.
\end{equation}
Hence,
$(\rho_n)^q$ (up to a subsequence) converges in (see Remark~\ref{R:AL-3})
\[
\begin{cases}
    L^{\frac{d(q+m-1)+2q}{q(d+1)}}(0, T; L^r_{\loc}(\bbr^d))\  \text{ for } \ r< \frac{d^2(q+m+1)+2dq}{q(d^2-2) - d(m-1)}, & \text{ if } q \geq m-1 \vspace{2 mm} \\
    L^{\frac{2d(q+m-1)+4q}{d(3q-r)+2q}}(0, T; L^r_{\loc}(\bbr^d)) \ \text{ for } \ r< \frac{2d^2(q+m-1)+4dq}{d^2(3q-r)-2d(m-1)-4q}, & \text{ if } \frac m2 < q < m-1.
\end{cases}
\]
This completes the proof. \qed

\subsubsection{Proof of Theorem \ref{Theorem-4}}
{\it For the case $(i)$} :  The proof can be completed almost the same way as in Theorem \ref{Theorem-2-b} $(i)$.
 As before, let us just point out the differences between the proofs of two theorems.
Firstly, to have \eqref{eq11 : Theorem-2-b}, we exploit Proposition \ref{P:Lq-energy} $(iii)$  and  Proposition \ref{P:Energy-speed} $(iii)$ instead of
Proposition \ref{P:Lq-energy} $(i)$ and  Proposition \ref{P:Energy-speed} $(i)$, respectively.

Secondly, for the strong convergence of $(\rho_n)^q$, we follow the exactly same steps as in the proof of Theorem \ref{Theorem-4a}. More precisely,
for $q\geq m-1$, we exploit Proposition \ref{Proposition : AL-2} $(i)$  and have \eqref{eq15 : Theorem-2-a}.
On the other hand, for the case $\frac{m}{2} \leq q < m-1$, we exploit Proposition \ref{Proposition : AL-2} $(ii)$ and Remark~\ref{R:AL-3} $(ii)$ to have \eqref{eq 2 : Theorem-4a}.

\noindent {\it For the case $(ii)$} : Refer Proposition \ref{P:Energy-embedding} $(iii)$ to have constants
$q^{\ast} \in (1, q)$ and $q_{2}^{\ast} = q_{2}^{\ast}(q_1) \in [2, q_2)$ such that $ V \in \mathcal{S}_{m, q^\ast}^{(q_1, q_2^{\ast})}$ satisfying \eqref{T4:V-divfree}.
Hence, the same conclusions hold as in Theorem~\ref{Theorem-4} $(i)$.
\qed

\subsection{Existence for case: $\nabla V$ in sub-scaling classes }
We prove theorems in Section~\ref{SS:tilde Serrin}: Theorem~\ref{Theorem-5a} and \ref{Theorem-5}.
\subsubsection{Proof of Theorem \ref{Theorem-5a}}
The proof is almost the same as the proof of Theorem \ref{Theorem-2-a}.  Again as before, let us just point
out the differences between the proofs of two theorems. Firstly, instead of \eqref{eq7 : Theorem-2-a},
we choose a sequence of vector fields $V_n\in \calC^\infty(Q_T) $ such that
\begin{equation}\label{eq7 : T:log-tilde}
\lim_{n \rightarrow \infty}\|V_n - V \|_{\tilde{\mathcal{S}}_{m,q}^{(\tilde{q}_1,\tilde{q}_2)} } =0.
\end{equation}
Secondly, we obtain the same conclusion as in \eqref{eq13 : Theorem-2-a} by supressing $\|V_n-V \|_{\mathcal{S}_{m,q}^{(q_1,q_2)}} \leq c\|V_n - V \|_{\tilde{\mathcal{S}}_{m,q}^{(\tilde{q}_1,\tilde{q}_2)}}$ by Remark~\ref{R:S-tildeS}.
%\begin{equation}\label{eq3 : Theorem-5a}
%\begin{aligned}
%\| \rho_n (V_n-V)\|_{L_{x,t}^r } &\leq C ( \|\rho_0\|_{L^{q}},~ \|V\|_{\tilde{\mathcal{S}}_{m,q}^{({q}_1, {q}_2)}})\|V_n-V \|_{\mathcal{S}_{m,q}^{(q_1,q_2)}}  \\
%& \leq C \|V_n-V \|_{\tilde{\mathcal{S}}_{m,q}^{(\tilde{q}_1,\tilde{q}_2)}}  \longrightarrow 0, \quad \mbox{as} \quad n\rightarrow \infty,
%\end{aligned}
%\end{equation}
%where $ r=\frac{d(q+m-1)+2q}{md+q}$, $q_1=\frac{d \tilde{q}_1}{d-\tilde{q}_1}$, and $q_2=\tilde{q}_2$.
%We note that the assumption \eqref{T5a:V-tilde} guarantees $1<\tilde{q}_1<d$.
%Hence we exploit \eqref{norm-V-gradV} in the second inequality of \eqref{eq3 : Theorem-5a}.
Thirdly, we exploit Proposition \ref{P:Lq-energy} $(ii)$ instead of Proposition \ref{P:Lq-energy} $(i)$.
Lastly, for the strong convergence of $(\rho_n)^q$, we follow the exactly same steps as in the proof of Theorem \ref{Theorem-4a}. This completes the proof.
\qed

\subsubsection{Proof of Theorem \ref{Theorem-5}}
{\it For the case (i)} :  Similar to the reason in the proof of Theorem \ref{Theorem-2-b} $(i)$,
we only consider the case (a) i.e. for $1< p\leq {\lambda_q}$. As expected, the proof can be done similarly to that of Theorem \ref{Theorem-2-b} $(i)$.
We just point out differences.

Instead of \eqref{eq7 : Theorem-2-b},
we choose a sequence of vector fields $V_n\in \calC^\infty(\bar{Q}_T) $ satisfying \eqref{eq7 : T:log-tilde}.
%\begin{equation*}\label{eq7 : T:log-tilde}
%\lim_{n \rightarrow \infty}\|V_n - V \|_{\tilde{\mathcal{S}}_{m,1}^{(\tilde{q}_1,\tilde{q}_2)} } =0.
%\end{equation*}
%Besides \eqref{eq3 : Theorem-5a},
Then we have
\begin{equation*} %\label{eq4 : Theorem-5}
%\begin{aligned}
\iint_{Q_T} \left| V_n-V\right |^{{\lambda_q}} \rho_n \,dx \,dt 
\leq C \|V_n-V \|_{\mathcal{S}_{m,q}^{(q_1,q_2)}}^{{\lambda_q}} \\
\leq C \|V_n-V \|_{\tilde{\mathcal{S}}_{m,q}^{(\tilde{q}_1,\tilde{q}_2)}}^{{\lambda_q}}\longrightarrow  0 \quad \mbox{as}\quad n \rightarrow \infty,   %\\
%\end{aligned}
\end{equation*}
instead of \eqref{eq36 : Theorem-2-b} for  $q_1=\frac{d \tilde{q}_1}{d-\tilde{q}_1}$ and $q_2=\tilde{q}_2$ (see Remark~\ref{R:S-tildeS}). %We note that the assumption \eqref{T5:V-tilde} guarantees $1<\tilde{q}_1<d$.
%Hence we exploit \eqref{norm-V-gradV} in the second inequality of \eqref{eq4 : Theorem-5}.
Let us also point out that we exploit Proposition~\ref{P:Lq-energy} $(ii)$ and \ref{P:Energy-speed} $(ii)$
instead of Proposition~\ref{P:Lq-energy} $(i)$ and \ref{P:Energy-speed} $(i)$, respectively.  Lastly, for the strong convergence of $(\rho_n)^q$,
we consider again the two cases separately. That is, for $q\geq m-1$, we exploit Proposition \ref{Proposition : AL-2} $(i)$ to have \eqref{eq15 : Theorem-2-a}. Also, for the case $\frac{m}{2} \leq q < m-1$,
we exploit Proposition \ref{Proposition : AL-2} $(ii)$ and Remark~\ref{R:AL-3} $(ii)$ to have \eqref{eq 2 : Theorem-4a}.

\noindent {\it For the case (ii)} : Let $q>m$ and assume $V$ satisfies \eqref{T5:V-tilde-embedding}.  Then, from Proposition \ref{P:Energy-embedding} $(ii)$, there exist
$q^{\ast} \in (1, q)$ and $q_{2}^{\ast} = q_{2}^{\ast}(q_1) \in (\frac{2+q^{\ast}_{m,d}}{1+q^{\ast}_{m,d}},  \tilde{q}_2]$
such that $ V \in \tilde{\mathcal{S}}_{m, q^\ast}^{(\tilde{q}_1, \tilde{q}_2^{\ast})}$ satisfying \eqref{T5:V-tilde}.
 Hence, the same conclusions hold as in Theorem~\ref{Theorem-5} $(i)$. \qed
%As in the proof of Theorem \ref{Theorem-1}, we exploit the estimate \eqref{eq2 : T:log-tilde} and note that there
% \qed

%%%%%%%%%%%%%%%%%%%%%%%%%%%%%%%%%%%%%%%%%%%%%%%%%%%%%%%%%%%%%%%%%%%%%%%%%%%%%%5
%%%%%%%%%%%%%%%%%%%%%%%%%%%%%%%%%%%%%%%%%%%%%%%%%%%%%%%%%%%%%%%%%%%%%%%%%%5%%%%

\section{Proofs of uniqueness and an application }\label{PME-KS-eq}

In this section, we present proofs of Theorem~\ref{Corollary : Uniqueness} and Theorem \ref{thm-jengchik50}.
First, we recall a technical lemma, which is useful for the uniqueness result.

\begin{lemma}\label{Lemma : Uniqueness} \cite[Lemma~2.4.1]{Kim-thesis}
For $i=1,2$, let $\rho_i \in AC(0, T ; \mathcal{P}_2(\bbr^d))$ and $w_i \in L^2(0, T  ; L^2(\rho_i))$ be the velocity for $\rho_i$, i.e
\begin{equation*}
\partial_t \rho_i +\nabla \cdot (w_i \rho_i)=0.
\end{equation*}
We set $f(t):=\frac{1}{2} W_2^2(\rho_1(t),\rho_2(t))$. Then, $f\in W^{1,2}(0,T)$ and
\begin{equation*}
f '(t) \leq \int_{\bbr^d \times \bbr^d} \langle w_1(x,t)-w_2(y,t),x-y\rangle \,d \gamma_t(x,y), \quad \gamma_t \in \Gamma_0(\rho_1, \rho_2), 
\end{equation*}
for a.e $t \in (0,T)$.
\end{lemma}

By exploiting above lemma, the next proposition provides a uniqueness result in the class of absolutely continuous curves in $\mathcal{P}_2 (\bbr^d)$.
\begin{proposition} \label{thm-uniqueness}
Let $\rho_0 \in L^q(\bbr^d) \cap \mathcal{P}_2(\bbr^d)$ for $q\geq 1$. Suppose
that $\rho_i \in AC(0,T: \mathcal{P}_2(\bbr^d))$ for $i=1,2$ are $L^q$-weak solutions of \eqref{E:Main}
%\begin{equation}\label{eq 4 : uniqueness}
%\partial_t \rho + \nabla \cdot (-\nabla \rho^m +V \rho)=0,\quad \rho(\cdot, 0)=\rho_0
%\end{equation} \footnote{May simply quote \eqref{E:Main} ???? }
satisfying
\begin{equation}\label{eq 11 : uniqueness}
 \iint_{Q_T} \left ( \left |\frac{\nabla \rho_i^m}{\rho_i} \right |^2 +|V|^2\right )\rho_i \,dx\,dt < \infty, \quad \text{for} \quad i = 1,2.
\end{equation}
Assume further that $\nabla V \in L_{x,t}^{\infty, 1} (Q_T)$. Then they are identically the same, i.e. $\rho_1 =\rho_2$.
\end{proposition}

\begin{proof}
%{\bf Proof of Theorem \ref{thm-uniqueness} :}
%\begin{theorem}\label{thm-uniqueness}
%Assume $\nabla V \in L_x^{\infty} L_t^1(\bar{Q}_T)$.
For $i=1, 2$, let $\rho_i \in AC(0,T: \mathcal{P}_2(\bbr^d))$ be a $L^q$-weak solution of
\begin{equation}\label{eq 10 : uniqueness}
\partial_t \rho_i + \nabla \cdot (-\nabla \rho_i^m +V \rho_i)=0, \quad \rho_i(\cdot, 0)=\rho_0
\end{equation}
satisfying \eqref{eq 11 : uniqueness}.
%\begin{equation}\label{eq 11 : uniqueness}
%\iint_{Q_T}
%\iint_{Q_T}\left ( \left |\frac{\nabla \rho_i^m}{\rho_i} \right |^2 +|V|^2\right )\rho_i \,dx\,dt < \infty.
%\end{equation}
%
%\end{theorem}
%
%\begin{proof}
Let $w_i:=-\frac{\nabla \rho_i^m}{\rho_i} +V$.
From \eqref{eq 11 : uniqueness} and \eqref{eq 10 : uniqueness}, we note that $w_i$ is the velocity for $\rho_i$ and satisfies $w_i \in L^2(0,T;L^2(\rho_i))$.
We define $f : [0, T] \rightarrow (-\infty, \infty)$ by $f(t)=\frac{1}{2} W_2^2(\rho_1(t), \rho_2(t))$.  Then, from Lemma \ref{Lemma : Uniqueness}, we have $f\in W^{1,2}(0,T)$ and
\begin{equation}\label{eq 9 : uniqueness}
\begin{aligned}
f '(t) %& \leq \int_{\bbr^d \times \bbr^d} \langle w_i(x)-w_2(y), x-y\rangle \,d \gamma_t(x,y) \quad \gamma_t \in \Gamma_0(\rho_1, \rho_2)\\
& \leq \int_{\bbr^d \times \bbr^d} %\langle \left (-\frac{\nabla \rho_1^m}{\rho_1} +V \right )(x,t)
%- \left(- \frac{\nabla \rho_2}{\rho_2}  + V\right )(y,t), x-y\rangle
\langle w_1(x,t)
- w_2(y,t), x-y\rangle \,d \gamma_t(x,y) \\
&= \int_{\bbr^d \times \bbr^d} \langle -\frac{\nabla \rho_1^m}{\rho_1}(x,t) + \frac{\nabla \rho_2^m}{\rho_2}(y,t), x-y\rangle \,d \gamma_t(x,y)
 + \int_{\bbr^d \times \bbr^d} \langle V(x,t) -V(y,t), x-y\rangle \,d \gamma_t(x,y) \\
&= I + II, \qquad \text{ for a.e. } t\in (0,T).
\end{aligned}
\end{equation}
%for a.e $t \in (0,T)$.%(refer Lemma 2.4.1 of \cite{Kim-thesis} ).

To estimate $I$, we first note $\rho_i^m \in L^1(0,T;  W^{1,1}(\bbr^d) )$ for $i=1,2$. Indeed,
\begin{equation*} %\label{eq 7 : uniqueness}
\begin{aligned}
\iint_{Q_T} \left |\nabla \rho_i^m(x,t) \right | \,dx\,dt &= \iint_{Q_T} \frac{ |\nabla \rho_i^m |}{\rho_i^{\frac{1}{2}}}\rho_i^{\frac{1}{2}} \,dx\,dt
\leq \left ( \iint_{Q_T} \frac{|\nabla \rho_i^m |^2}{\rho_i} \,dx\,dt\right )^\frac{1}{2} \left ( \iint_{Q_T} \rho_i \,dx\,dt\right )^\frac{1}{2}\\
&= T^{\frac{1}{2}} \left ( \iint_{Q_T} \left |\frac{\nabla \rho_i^m }{\rho_i}\right|^2\rho_i \,dx\,dt\right )^\frac{1}{2} <\infty.
\end{aligned}
\end{equation*}
%Combining \ref{eq 11 : uniqueness} and \eqref{eq 7 : uniqueness}, we have $\rho_i^m \in L^1(0,T;  W^{1,1}(\bbr^d) )$.
Especially, this means $\rho_i^m(\cdot, t) \in W^{1,1}(\bbr^d )$ for a.e $t \in [0,T]$ and $i=1,2$.

Next, we define $\mathcal{F} : \mathcal{P}_2(\bbr^d) \rightarrow (-\infty, \infty ]$ by
\begin{equation*} %\label{F}
\mathcal{F}(\mu)=
\begin{cases}
    \frac{1}{m} \int_{\bbr^d} \rho^m \,dx, & \text{ if }  \mu =\rho \mathcal{L}^d, \vspace{2 mm}\\
    +\infty,  & \text{ otherwise},
    \end{cases}
    \end{equation*}
    where $\mathcal{L}^d$ is the Lebesgue measure in $\bbr^d$.
Then, $\mathcal{F}$ is geodesically convex in $\mathcal{P}_2(\bbr^d)$(refer  Chapter 9 in \cite{ags:book}). Furthermore, if $\rho^m \in W^{1,1}_{loc} (\bbr^d) $ then
$\frac{\nabla \rho^m}{\rho} \in \partial_- \mathcal{F}(\rho)$ (refer Section 10.4 in \cite{ags:book}). Hence, for a.e $t\in [0,T]$ and $i=1,2$,
we have $\frac{\nabla (\rho_i(t))^m}{\rho_i(t)} \in \partial_- \mathcal{F}(\rho_i(t))$ which means
\begin{equation}\label{eq 1 : uniqueness}
\begin{aligned}
\mathcal{F}(\rho_2(t)) \geq \mathcal{F}(\rho_1(t)) + \int_{\bbr^d \times \bbr^d} \langle \frac{\nabla \rho_1^m}{\rho_1}(x,t) , y-x \rangle \, d \gamma_t(x,y),
\quad ~ \gamma_t \in \Gamma_0(\rho_1(t), \rho_2(t))
\end{aligned}
\end{equation}
and
\begin{equation}\label{eq 2 : uniqueness}
\begin{aligned}
\mathcal{F}(\rho_1(t)) \geq \mathcal{F}(\rho_2(t)) + \int_{\bbr^d \times \bbr^d} \langle \frac{\nabla \rho_2^m}{\rho_2}(x,t) , y-x \rangle \, d \tilde{\gamma}_t(x,y),
\quad ~ \tilde{\gamma}_t \in \Gamma_0(\rho_2(t), \rho_1(t)),
\end{aligned}
\end{equation}
where we note that $\Gamma_0(\rho_1(t), \rho_2(t))$ and $\Gamma_0(\rho_2(t), \rho_1(t))$ consist of single element due to the fact that
$\rho_1(t)$ and $\rho_2(t)$ are absolutely continuous with respect to the Lebesgue measure.
By adding \eqref{eq 1 : uniqueness} and \eqref{eq 2 : uniqueness}, we have
\begin{equation}\label{eq 3 : uniqueness}
\begin{aligned}
0 \geq  \int_{\bbr^d \times \bbr^d} \langle \frac{\nabla \rho_1^m}{\rho_1}(x,t) - \frac{\nabla \rho_2^m}{\rho_2}(y,t) , y-x \rangle \, d \gamma_t(x,y) = I.
\end{aligned}
\end{equation}

For the estimation of $II$, we exploit $\nabla V \in L_{x,t}^{\infty, 1} (\bar{Q}_T)$ to have
\begin{equation}\label{eq 8 : uniqueness}
%\begin{aligned}
II \leq  \|\nabla V(\cdot, t) \|_{L_x^\infty} \int_{\bbr^d \times \bbr^d} | x-y|^2 \,d \gamma_t(x,y)
  = \|\nabla V(\cdot, t) \|_{L_x^\infty} W_2^2(\rho_1(t),\rho_2(t)). %\quad for \quad  a.e ~~ t \in [0,T].
%\end{aligned}
\end{equation}

We plug \eqref{eq 3 : uniqueness} and \eqref{eq 8 : uniqueness} into \eqref{eq 9 : uniqueness} to have
$$f '(t) \leq \|\nabla V(\cdot, t) \|_{L_x^\infty} f(t) \quad for \quad  a.e ~~ t \in [0,T], $$
which gives us
\begin{equation*} %\label{eq 6 : uniqueness}
W_2^2(\rho_1(t), \rho_2(t)) \leq \| \nabla V\|_{L_{x,t}^{\infty,1}} W_2^2(\rho_1(0), \rho_2(0))  \quad \forall~ t \in [0,T].
\end{equation*}
Since $\rho_1(0)=\rho_2(0)=\rho_0$, we have $\rho_1(t)=\rho_2(t)$ for all $t \in [0,T]$. This completes the proof.
\end{proof}

%\begin{cor}\label{Corollary : Uniqueness}
%In Theorem \ref{Theorem-2-b}, Theorem \ref{Theorem-4} and Theorem \ref{Theorem-5}, we impose one more assumption
%$\nabla V \in L_{x,t}^{\infty,1}(\bar{Q}_T)$. If $\rho_0 \in L^q(\bbr^d) \cap \mathcal{P}_p(\bbr^d)$ for $q \geq m$ and $p \geq 2$ then there exists a unique solution of
%\begin{equation}\label{eq 10 : uniqueness}
%\partial_t \rho + \nabla \cdot (-\nabla \rho^m +V \rho)=0, \qquad \rho(\cdot, 0)=\rho_0
%\end{equation}
%such that
%\begin{equation}\label{eq 11 : uniqueness}
%\sup_{0\leq t \leq T} \int_{\mathbb{R}^d\times \{t\}}   \rho^q \,dx + \iint_{Q_T} \left |\nabla \rho^{\frac{q+m-1}{2}} \right |^2 \,dx\,dt < \infty.
%\end{equation}
%
%
%\end{cor}
With the aid of Proposition \ref{thm-uniqueness}, we are ready to prove Theorem \ref{Corollary : Uniqueness}.

\begin{pfthm2.20}:
For $i=1,~2$, let  $\rho_i$  be a $L^m$-weak solution of \eqref{eq 10 : uniqueness} such that $\rho_i : [0, T] \mapsto \mathcal{P}(\bbr^d)$ is narrowly continuous.
%Since $q \geq m$, we note that $\rho_i$ is a $L^m$-weak solution.
Hence, we have
\begin{equation}\label{eq 13 : uniqueness}
\|\rho_i \|_{L^\infty(0,T;L^m(\bbr^d))} + \|\nabla
\rho_i^{\frac{2m-1}{2}} \|_{L^2([0,T] \times \bbr^d)} < \infty,
\quad \mbox{for} ~~ i=1,~2.
\end{equation}
We recall that $V$ satisfies \eqref{V-p-moment : uniqueness} which is \eqref{V-p-moment} for $q=m$ and ${\lambda_q}=2$. Once we have \eqref{eq 13 : uniqueness} and \eqref{V-p-moment}, we exploit Lemma \ref{P:W-p} and Remark \ref{Remark : moment estimate} to have
the speed estimate and the moment estimate. More precisely, we have \eqref{W-p}, \eqref{V-p} and \eqref{p-moment : q-weak} for $p=2$.
Hence, exploiting Gr\"{o}nwall's inequality as in the proof of Proposition \ref{P:Energy-speed}, we get
\begin{equation*}
\sup_{0\leq t \leq T} \int_{\mathbb{R}^d\times \{t\}}   \rho_i
\langle x \rangle^2 \,dx + \iint_{Q_T} \left ( \left |\frac{\nabla
\rho_i^m}{\rho_i} \right |^2 + |V|^2\right )\rho_i \,dx\,dt <
\infty,
\end{equation*}
which implies $\rho_i \in AC(0,T ; \mathcal{P}_2(\bbr^d))$ for
$i=1,2$.
Due to the assumption $\nabla V \in L_{x,t}^{\infty,1}(Q_T)$, we exploit Proposition \ref{thm-uniqueness} to conclude the uniqueness which completes the proof.
%\begin{equation}
%\begin{aligned}
%W_2^2(\rho_1(t), \rho_2(t)) \leq \| \nabla V\|_{L_{x,t}^{\infty,1} }
%W_2^2(\rho_1(0), \rho_2(0)) =0  \quad \forall~ t \in [0,T],
%\end{aligned}
%\end{equation}
%which completes the proof.
\end{pfthm2.20}

{\bf Proof of Corollary \ref{Corollary : Uniqueness-2}}:
Since $\rho_0 \in L^q(\bbr^d) \cap \mathcal{P}_p(\bbr^d)$ for $p\geq 2$ and $q \geq m$, we have $\rho_0 \in L^m(\bbr^d) \cap \mathcal{P}_2(\bbr^d)$.
Due to $q \geq m$, we first note ${\lambda_q}=2$. We also note the following relations in sub-scaling classes $ \mathfrak{S}_{m,q}^{(q_1, q_2)} \subset  \mathfrak{S}_{m,m}^{(q_1, q_2)} $ and
$ \tilde{\mathfrak{S}}_{m, q}^{(\tilde{q}_1, \tilde{q}_2)}\subset  \tilde{\mathfrak{S}}_{m, m}^{(\tilde{q}_1, \tilde{q}_2)} $.
Exploiting these inclusions, as in the proof of Proposition \ref{P:Energy-speed}, we have that if $V$ satisfy the hypothesis
in one of Theorem~\ref{Theorem-2-b}, \ref{Theorem-4} and \ref{Theorem-5} then $V$ satisfies \eqref{V-p-moment : uniqueness} which is \eqref{V-p-moment} for $q=m$.
Hence, we exploit Theorem \ref{Corollary : Uniqueness} and complete the proof.
\qed

Next, we provide the proof of Theorem \ref{thm-jengchik50}.

\begin{pfthm2.21}: As mentioned before, we need to prove only the case $\max\{\frac{2d-3}{d}, \,\frac{2d}{d+2}\}<m\le 1+\frac{(d-1)(d-2)}{d^2}$ for $d>2$ due to previous result in \cite{Freitag}.
We first compute the $p$-th moment estimate, where  $p>1$ is determined later. Direct computations using H\"{o}lder inequality show that
\begin{equation}\label{KS-moment10}
\begin{aligned}
\frac{d}{dt}\int_{\R^d} \rho \langle x\rangle^p \, dx & =\int_{\R^d} \rho^m \Delta \langle x\rangle^p \, dx-\int_{\R^d} \rho \nabla c\nabla \langle x\rangle^p \, dx \\
&\le \int_{\R^d} \rho^m \Delta \langle x\rangle^p \, dx+\epsilon \int_{\R^d}\rho\abs{\nabla c}^p \, dx+ C_{\epsilon}\int_{\R^d} \rho \langle x\rangle^p \, dx.
\end{aligned}
\end{equation}
For the first term in \eqref{KS-moment10}, we have (see also \eqref{J-1})
\[
\int_{\R^d} \rho^m \Delta \langle x\rangle^p \, dx\le C\norm{\rho}_{L^m_x}^m
\le C\norm{\rho}^{m\mu}_{L^1_x}\norm{\rho^{\frac{m}{2}}}^{2(1-\mu)}_{L^{\frac{2d}{d-2}}_{x}}\le C\norm{\nabla \rho^{\frac{m}{2}}}_{L^2_x}^{2(1-\mu)},
\]
where $\mu=\frac{2}{d(m-1)+2}$.
On the other hand, the second term,  $\int_{\R^d}\rho\abs{\nabla c}^p \, dx$, can be controlled as follows:
\[
\int_{\R^d}\rho\abs{\nabla c}^p \, dx\le \norm{\rho}_{L^a_x}\norm{\abs{\nabla c}^p}_{L^b_x}
=\norm{\rho}_{L^a_x}\norm{\nabla c}^p_{L^{bp}_{x}}\le C\norm{\rho}_{L^a_x}\norm{\nabla^2 c}^p_{L^a_x},
\]
where $a=\frac{(p+1)d}{p+d}$ and $b=\frac{(p+1)d}{p(d-1)}$.
Therefore, we have for $\alpha=p+1$ and $\beta=\frac{p+1}{p}$
\[
\int_0^T\int_{\R^d}\rho\abs{\nabla c}^p \, dxdt\le C\int_0^T\norm{\rho}_{L^a_x}\norm{\nabla^2 c}^p_{L^a_x}\,dt
\le C\norm{\rho}_{L^{a, \alpha}_{x,t}}\norm{\nabla^2 c}^p_{L^{a, p\beta}_{x, t}}\le C\norm{\rho}^{p+1}_{L^{a, p+1}_{x,t}},
\]
where we used the maximal regularity of $c$, i.e. $\norm{\nabla^2 c}_{L^{a, p+1}_{x,t}}\lesssim \norm{\rho}_{L^{a, p+1}_{x,t}}$ (here we assume that the initial data $c_0(x)=0$ for simplicity, since it is enough to add $\norm{c_0}_{W^{2,a}_x}$ in the righthand side, otherwise).
Noting that
\[
\norm{\rho}_{L^a_x}\le \norm{\rho}^{\theta}_{L^1_x}\norm{\rho^{\frac{m}{2}}}_{L^{\frac{2d}{d-2}}_x}^{\frac{2(1-\theta)}{m}}, \quad \text{for } \  1-\theta=\frac{pm(d-1)}{(p+1)(md-d+2)},
\]
it follows that
\[
\int_0^T \norm{\rho(t)}^{p+1}_{L^a_x} \,dt \le C\int_0^T \norm{\nabla \rho^{\frac{m}{2}}}_{L^{2}_x}^{\frac{2(p+1)(1-\theta)}{m}}\,dt.
\]
To set $\frac{(p+1)(1-\theta)}{m}\le 1$, here we choose  $p\in \left(1, \frac{md-d+2}{d-1}\right]$  since $m > \frac{2d-3}{d}$. We can see that
any $p\in \left(1, \frac{md-d+2}{d-1}\right]$ is less than ${\lambda_1}=1+\frac{1}{md-d+1}$, since $m<1+\frac{(d-1)(d-2)}{d^2}$ (in fact, it is valid as long as $m<\frac{2(d-1)}{d}$).
Now integrating \eqref{KS-moment10} in time and combining estimates above, we obtain
\begin{equation}\label{ks-moment20}
\begin{aligned}
\int_{\R^d} \rho(\cdot,t) \langle x\rangle^p \, dx
&\le \int_{\R^d} \rho_0 \langle x\rangle^p \, dx + C\int_0^t\norm{\nabla \rho^{\frac{m}{2}}}_{L^2_x}^{2(1-\mu)}\,d\tau \\
&\quad +C\epsilon\int_0^t \norm{\nabla \rho^{\frac{m}{2}}}_{L^{2}_x}^{\frac{2(p+1)(1-\theta)}{m}}\,d\tau +C_{\epsilon}\int_0^t\int_{\R^d} \rho \langle x\rangle^p \, dxd\tau,
\end{aligned}
\end{equation}
for any $t\in (0, T]$.
Combining estimates \eqref{KS-moment5} and \eqref{ks-moment20}, and using Gr\"{o}nwall's inequality, we can control $\int_{\R^d} \rho|\log \rho| \,dx$ via \eqref{entropy-log}, which
implies immediately that
$\rho\in L^{1, \infty}_{x,t}\cap L^{\frac{md}{d-2}, m}_{x, t}$,
and therefore
\[
\rho\in L^{a,b}_{x,t}, \quad \text{where}\quad \frac{d}{a}+\frac{2+d(m-1)}{b}=d, \quad 1\le a\le \frac{md}{d-2}.
\]
In particular, in case $b=m$, it follows that
$\rho\in L^{a, m}_{x, t}$ with $a=\frac{md}{d-2}$.
Due to maximal regularity, one can see that
$c_t, \Delta c\in L^{a, m}_{x,t}$.
We claim that
\[
\nabla c\in L^{q^{(1)}_1}_x L^{\frac{m}{m-1}}_t \quad \text{for} \quad  q^{(1)}_1=\frac{md}{d+2-3m}.
\]
Indeed,
recalling that $\rho\in L_{x}^{\frac{md}{d-2}}L^m_t$, it follows that for $q^{(1)}_1\ge \frac{md}{d-2}$
\[
\norm{\nabla c}_{L^{q^{(1)}_1}_x}\le \norm{\int_0^t (\nabla \Gamma * \rho) (\tau) d\tau}_{L^{q^{(1)}_1}_x}
\leq C \int_0^t (t-\tau)^{-\frac{d}{2}(\frac{d-2}{md}-\frac{1}{q^{(1)}_1})-\frac{1}{2}}\norm{\rho(\tau)}_{L_{x}^{\frac{md}{d-2}}} \,d\tau.
\]
Again using the Young's inequality, we have
\[
\norm{\nabla c}_{L^{q^{(1)}_1}_xL^{\frac{m}{m-1}}_t} \leq C \norm{\rho}_{L^{\frac{md}{d-2}, m}_{x,t}},\quad \text{for} \quad \frac{md}{d-2}\le q^{(1)}_1\le \frac{md}{d+2-3m}.
\]

We note that $\nabla c\in S^{(q^{(1)}_1, \frac{m}{m-1})}_{m, q^{(1)}}$ with $q^{(1)}=\frac{d(m-1)}{d-2m}$. More precisely,
\[
\frac{d}{q^{(1)}_1}+\frac{(m-1)(2+q^{(1)}_{m,d})}{m}=1+q^{(1)}_{m,d},\quad \text{where} \quad q_{m, d}^{(1)}=\frac{d(m-1)}{q^{(1)}}.
\]
In fact, we can see that if $m>\frac{2d}{d+2}$, then $q^{(1)}>1$.
We denote
\[
q^{(1)}=\frac{d(m-1)}{d-2m}, \quad q^{(1)}_1=\frac{md}{d+2-3m},\quad \text{and} \quad  q^{(1)}_2=\frac{m}{m-1}.
\]
Now we derive an iterative formula.
Suppose that $q^{(k)}>1$ is given and $\nabla c\in S^{(q^{(k)}_1, q^{(k)}_2)}_{m, q^{(k)}}$.
Then,
$\rho$ is upgraded as follows, thanks to Proposition~\ref{P:Lq-energy}:
\[
\rho\in L^{q^{(k)}, \infty}_{x,t}\cap L^{\frac{d(m+q^{(k)}-1)}{d-2}, m+q^{(k)}-1}_{x,t}.
\]
Next we show that
\[
\nabla c\in L^{q^{(k+1)}_1}_x L^{\frac{m+q^{(k)}-1}{m-1}}_t, \quad \text{for}\quad q^{(k+1)}_1=\frac{d(m+q^{(k)}-1)}{d-3m+3-q^{(k)}}.
\]
Indeed, reminding that $\rho\in L^{\frac{d(m+q^{(k)}-1)}{d-2}, m+q^{(k)}-1}_{x,t}$, it follows that for $q^{(k+1)}_1\ge \frac{d(m+q^{(k)}-1)}{d-2}$
\[
\norm{\nabla c}_{L^{q^{(k+1)}_1}_x}\le \norm{\int_0^t (\nabla \Gamma * \rho) (\tau) d\tau}_{L^{q^{(k+1)}_1}_x}
\leq  C \int_0^t (t-\tau)^{-\frac{d}{2}(\frac{d-2}{d(m+q^{(k)}-1)}-\frac{1}{q^{(k+1)}_1})-\frac{1}{2}}\norm{\rho(\tau)}_{L_{x}^{\frac{d(m+q^{(k)}-1)}{d-2}}}\,d\tau.
\]
Again via the Young's inequality, we have
\[
\norm{\nabla c}_{L^{q^{(k+1)}_1}_xL^{\frac{m+q^{(k)}-1}{m-1}}_t} \leq C \norm{\rho}_{L^{\frac{d(m+q^{(k)}-1)}{d-2}}_xL^{m+q^{(k)}-1}_t}, \quad \text{for}\quad \frac{d(m+q^{(k)}-1)}{d-2}\le q^{(k+1)}_1\le \frac{d(m+q^{(k)}-1)}{d-3m+3-q^{(k)}}.
\]
Now we get
\[
\nabla c\in S^{(q^{(k+1)}_1, \frac{m+q^{(k)}-1}{m-1})}_{m, q^{(k+1)}} \quad \text{for}\quad  q^{(k+1)}=\frac{d(m-1)q^{(k)}}{d-2m+2-2q^{(k)}}.
\]
Then it is direct that $q^{(k+1)}>q^{(k)}$, because this is equivalent to
\[
\frac{d(m-1)}{d-2m+2-2q^{(k)}}>1 \quad \Longleftrightarrow \quad q^{(k)}>d+1-\frac{m(d+2)}{2}.
\]
We note that $q^{(1)}> 1$ as long as $m>\frac{2d}{d+2}$.
We stop this process until $d-3m+3-q^{(k)}\le 0$ (it can be achieved in a finite step, since
$q^{(k)}$ is strictly increasing and it cannot converge).
Since $q^{(k)} \geq d-3m+3$, we note that
\[
\frac{m+q^{(k)}-1}{m-1}\ge  \frac{d-2m+2}{m-1}\ge 2,
\]
for $m \leq \frac{d+2}{4}$. This is true for $m\le \frac{2(d-1)}{d}$ because $\frac{2(d-1)}{d}< \frac{d+4}{4}$, which is equivalent to $d^2-4d+8>0$.
This completes the proof of Theorem \ref{thm-jengchik50}.
\end{pfthm2.21}

%Summing up, we obtain the following result.
%\begin{theorem}\label{thm-jengchik50}
%Let $m>\frac{2d}{d+2}$, $T\in (0, \infty)$ and $p\in \left(1, \frac{md-d+2}{d-1}\right]$. Suppose that $\rho_0\in {\mathcal P}_p(\R^d)\cap L^{\infty}(\R^d)$ and $c_0\in( L^1\cap W^{2,\infty})(\R^d)$. Then, a pair of solution $(\rho, c)$ for the equations \eqref{milchim-10}-\eqref{milchim-11} becomes bounded.
%\end{theorem}

%\begin{remark}
%If $\rho_0$ is more regular, e.g. H\"{o}lder continuous, then $\rho$ becomes  H\"{o}lder continuous for $t\in ),T]$. Since its proof is rather straightforward once solutions become bounded, we omit the details.
%\end{remark}

%%
\begin{appendices}

\section{H\"{o}lder continuity of homogeneous PME}\label{Appendix:Holder}

Here we concern the homogeneous form of \eqref{E:Main}
\begin{equation}\label{H-PME}
 \partial_t u =   \Delta u^m \  \text{ in } \ Q_{T}   \quad \text{ for }\ m > 1,
\end{equation}
for $Q_{T}:= \mathbb{R}^d \times (0, T]$, $d\geq 2$, $0<T<\infty$ with a nonnegative initial data $u(x,0)=u_0 (x)$, $x\in \mathbb{R}^d$.
In \cites{CF78, CF79}, Caffarelli and Friedman proved the modulus of continuity of $u$ in the entire half space, if $u_0 (x)$ is H\"{o}lder continuous. In \cite{DF85, DF85A}, DiBenedetto and Friedman determined H\"{o}lder exponent and constants quantitatively in the interior only depending upon the data. For $d=1$, the H\"{o}lder exponent is known explicitly as $\min\{1, \frac{1}{m-1}\}$ by Aronson \cites{Aron70a, Aron70b}.

The H\"{o}lder continuity up to initial boundary is known, for exmaple, see \cite[Section~6]{DB83}, as the H\"{o}lder exponent is depending upon the data and the modulus of continuity of the initial data. More specifically, for $u_0 (x) \in C^{\alpha_0}(\mathbb{R}^d)$, a bounded nonnegative weak solution of \eqref{H-PME} is uniformly H\"{o}lder continuous up to the initial boundary with the exponent $\alpha = \alpha (m, d,\alpha_0)\in (0,1)$ and the coefficient constant $c = c (m,d,\|u\|_{L^{\infty}(\bar{Q}_{T})})$. For more systemetic arguments, we refer \cite{DB93, DGV12} for parabolic $p$-Laplace equation and for both $p$-Laplace and porous medium equations.

Although the uniform H\"{o}lder continuity of a bounded nonnegative weak solution of \eqref{H-PME} is well-known, the dependency of the H\"{o}lder exponent with respect to the initial H\"{o}lder exponent seems not specified from previous works mentioned above. Therefore, our aim is to more clearly demonstrate the relation of interior and initial H\"{o}lder regularity. More precisely, in Theorem~\ref{T:Boundary_Holder}, we clarify the dependencies of uniform H\"{o}lder exponent as the minimum of the intertior H\"{o}lder exponent and given H\"{o}lder exponent of the initial data, that is, $\alpha = \min\{\alpha^\ast, \alpha_0\}$ where $\alpha^\ast = \alpha^\ast (m,d)$ is determined from the interior H\"{o}lder result in Theorem~\ref{T:interior_Holder} and $\alpha_0 \in (0,1)$ is given from $u_0 \in \mathcal{C}^{\alpha_0}(\bbr^d)$.
To prove Theorem~\ref{T:Boundary_Holder}, we mainly modify arguments, the boundary H\"{o}lder continuity of parabolic $p$-Laplace equations up to the initial time, by DiBenedetto in \cite[Theorem~III]{DB93}. Also typical arguments, for example, in \cite{DB83, DF85, CHKK17, HZ21}, are used to prove local H\"{o}lder continuity of \eqref{H-PME} quantitatively.

Let $K_{r} (x_0)$ denote the $d-$dimensional cube centered at $x_0 \in \mathbb{R}^d$ and wedge $2r$ for $r>0$, that is,
\[
K_{r}(x_0):= \{ x\in \mathbb{R}^d : \max_{1\leq i \leq d}|x^{i}-x_0^{i}| < r \}, \quad \text{for} \quad x=(x^1, \cdots, x^d).
\]
Besides let us denote two types of parabolic cyliners in $Q_{T}$, for $r>0$ and $0 < t_0 < t_0 + r^2 \leq T$,
\[
Q_r := K_r (x_0) \times (t_0, t_0 +r^2 ] \quad \text{and} \quad
Q_{r}^{0} := K_r (x_0) \times [0, r^2].
\]

First, we introduce the interior H\"{o}lder continuity results in \cite{DF85, DF85A, DB83, DGV12}. When the divergence form of drift term is combined, we refer results in \cite{CHKK17, HZ21}.
\begin{theorem}\label{T:interior_Holder}
Let $u$ be a weak solution of \eqref{H-PME} in $Q_{1} \subset Q_{T}$. Then there exist $\alpha^\ast = \alpha^\ast (m,d) \in (0,1)$ and $c = c (m, d, \|u\|_{L^{\infty}(Q_{1})})$, such that, $u$
 is uniformly H\"{o}lder continuous in $Q_{1/2}$ with the exponent $\alpha^\ast$ and the coefficient $c$; that is, $ \| u\|_{\calC^{\alpha^\ast} (Q_{1/2})} \leq c$.
\end{theorem}

The main theorem is the following: the H\"{o}lder continuity up to the initial time.
\begin{theorem}\label{T:Boundary_Holder_local}
Let $u$ be a weak solution of \eqref{H-PME} in $Q_{1}^{0}\subset Q_{T}$ with
$u_0 \in \calC^{\alpha_0} (K_{1})$. Then there exists $\tilde{\alpha} = \min\{\alpha_0, \alpha^\ast\}$ where $\alpha^\ast = \alpha^\ast (m,d) \in (0,1)$ is given in Theorem~\ref{T:interior_Holder}, and $c  = c (m, d, \|u\|_{L^{\infty}(Q_{1}^{0})})$, such that, $u$ is uniformly H\"{o}lder continuous in $Q_{1/2}^{0}$ with the exponent $\alpha$ and the coefficient $c$; that is, $ \| u\|_{\calC^{\tilde{\alpha}} (Q_{1/2}^0)} \leq c$.
\end{theorem}

Theorem~\ref{T:Boundary_Holder} is a direct consequence of Theorem~\ref{T:interior_Holder} and Theorem~\ref{T:Boundary_Holder_local}.

\textbf{The proof of Theorem~\ref{T:Boundary_Holder}}: The global H\"{o}lder estimate up to $t=0$, is given by the standard covering arguments based on both interior and boundary estimates in Theorem~\ref{T:interior_Holder} and Theorem~\ref{T:Boundary_Holder_local}. \qed

It remains to prove Theorem~\ref{T:Boundary_Holder_local}.
In the first subsection, we provide energy and logarithmic type estimates up to the initial time which are fundamental to carry further arguments in the following subsection to prove Theorem~\ref{T:Boundary_Holder_local}.

%------
\subsection*{Energy estimates}

In this section, we apply the idea in \cite[Section~II.4.iii]{DB93} and local energy estimates in \cite[Proposition~4.1 \& 4.2]{CHKK17} to obtain energy estimates of \eqref{H-PME1} up to the initial time.

%Let $K_{r} (x_0)$ denote the $d-$dimensional cube centered at $x_0 \in \mathbb{R}^d$ and wedge $2r$ for $r>0$, that is, for $x=(x^1, \cdots, x^d)$,
%$$K_{r}(x_0):= \{ x\in \mathbb{R}^d : \max_{1\leq i \leq d}|x^{i}-x_0^{i}| < r \}.$$

We consider a weak solution of \eqref{H-PME} that takes the initial data $u_0$ such that
   \begin{equation}\label{u0_h}
   \left( u_{h} (\cdot, 0) - u_0{}\right)_{\pm} \to 0 , \quad \text{as } h\to 0 \quad \text{in $L^2_{\loc}$ sense,}
   \end{equation}
   where $u_{h}$ indicates the Stekelov average that
   \[
   u_{h} (x, t) = \frac{1}{h} \int_{t}^{t+h} u(x, \tau) \,d\tau, \quad \text{for all}\quad  0\leq t \leq T-h,
   \]
   which converges to $u$ as $h\to 0$.

For some $k>0$, let us denote $(u - k)_{+} = \max\{ u-k, 0\}$ and $ (u - k)_{-} = \max\{k-u, 0\}$.
\begin{proposition}\label{P:energy}
Let $u$ be a nonnegative bounded weak solution of \eqref{H-PME}. For $r>0$, $0<t<T$, suppose that $\zeta$ is a smooth cut-off function in $K_{r}(x_0)\times [0, t]$, which is $0 \leq \zeta \leq 1$, independent of $t$, and vanishing on the lateral boundary $\partial K_{r}(x_0)$.
Let $\mu_{\pm}$ be positive constants such that
\[
\mu_{+} \geq \esssup_{K_{r}(x_0)\times [0, t]} u \quad \text{and} \quad \mu_{-} \leq \essinf_{K_{r}(x_0)\times [0, t]} u.
\]
Moreover, for every level $k > 0$, suppose that
   \begin{equation}\label{H:k}
   \begin{cases}
   \mu_{+} - k > \esssup_{K_{r} (x_0)} u_0, & \quad \text{for}\quad  (u - \mu_{+} - k)_{+}, \\
   \mu_{-} + k < \essinf_{K_{r} (x_0)} u_0, & \quad \text{for}\quad  (u - \mu_{-} + k)_{-}.
   \end{cases}
   \end{equation}
   Then the following estimate holds
    \begin{equation}\label{H:local_energy}
        \begin{aligned}
            \sup_{0 \leq \tau \leq t}  \int_{K_r (x_0)} \left( u (\cdot, \tau) - \mu_{\pm} \pm k\right)_{\pm}^{2} \zeta^2 \,dx
            & + m \int_{0}^{t} \int_{K_{r}(x_0)} u^{m-1} \left| \nabla \left( u - \mu_{\pm} \pm k\right)_{\pm}\right|^2 \zeta^2 \,dx\,d\tau \\
            &\leq 16m \int_{0}^{t} \int_{K_{r}(x_0)} u^{m-1} \left( u - \mu_{\pm} \pm k\right)_{\pm}^{2} |\nabla \zeta|^2 \,dx\,d\tau.
        \end{aligned}
    \end{equation}
 Moreover, it holds that, for any $\tau \in (0, t]$,
    \begin{equation}\label{H:log_energy}
        \begin{aligned}
            \int_{K_r (x_0)} \Psi_{\pm}^{2}(u (\cdot, \tau)) \zeta^2 \,dx
            &+ 2m \int_{0}^{t} \int_{K_{r}(x_0)} u^{m-1} |\nabla u|^2 \Psi_{\pm} \left|\Psi'_{\pm}(u)\right|^2 \zeta^2 \,dx\,d\tau \\
            &\leq 16 m \int_{0}^{t} \int_{K_{r}(x_0)} u^{m-1} \Psi_{\pm}^{2}(u) |\nabla \zeta|^2 \,dx\,d\tau,
        \end{aligned}
    \end{equation}
    where, for $\delta \in (0, 1)$,
    \[
    \Psi_{\pm} (u) = \ln^{+} \left[\frac{k}{(1+\delta)k - \left( u - \mu_{\pm} \pm k\right)_{\pm}}\right].
    \]

\end{proposition}

\begin{proof}
A weak solution formulation of \eqref{H-PME} rephrased in terms of Stekelov averages is given
\[
\int_{0}^{t} \int_{K_{r}(x_0)} \left[ \partial_t u_h \varphi + \nabla u_h^{m} \nabla \varphi \right]\,dx\,dt = 0,
\]
for all $0<t<T-h$ and for all $\varphi \in W_{o}^{1,p} (\Omega) \cap L^{\infty}_{x}(\Omega)$, $\varphi \geq 0$ (see \cite[Section~II.2.i]{DB93}).

To obtain \eqref{H:local_energy}, we test $\varphi = \pm 2\left( u_{h} - \mu_{\pm} \pm k\right)_{\pm} \zeta^2$. Then
\begin{equation*}
    \begin{aligned}
\int_{0}^{t} \int_{K_{r}(x_0)} \partial_t u_h \varphi\,dx\,dt
&= \int_{0}^{t} \int_{K_{r}(x_0)} \partial_t \left[\left( u_{h} - \mu_{\pm} \pm k\right)_{\pm} \zeta \right]^2 \,dx\,dt \\
&= \int_{K_{r}(x_0)\times\{t\}} \left( u_{h} - \mu_{\pm} \pm k\right)_{\pm}^2 \zeta^2 \,dx\,dt - \int_{K_{r}(x_0)} \left( u_{h} (\cdot, 0) - \mu_{\pm} \pm k\right)_{\pm}^2 \zeta^2 \,dx\,dt,
\end{aligned}
\end{equation*}
in which the second term vanishes as $h \to 0$ because of \eqref{u0_h} and \eqref{H:k}. %Handling $\iint \nabla u_h^{m} \nabla \varphi \,dx\,dt$ follows the computations in \cite[Proposition~4.1]{CHKK17}.
We skip the computation of $\iint \nabla u_h^{m} \nabla \varphi \,dx\,dt$ because it is similar to \cite[Proposition~4.1]{CHKK17}. The combination of above estimates yields \eqref{H:local_energy}.

Now to obtain \eqref{H:log_energy}, we note that, from the definition of $\Psi_{\pm} (u)$, it follows that
\[
\Psi_{\pm}(u) = 0 \quad \text{whenever $ (u- \mu_{\pm}\pm k)_{\pm} = 0$. }
\]
Then \eqref{H:k} provides that
\[
\int_{K_{r} (x_0)} \Psi^{2} (u_h) (x, 0) \zeta^2 \,dx \to 0, \quad \text{as $h\to 0$}.
\]
Therefore, \eqref{H:log_energy} follows by carrying almost the same computations as in \cite[Proposition~4.2]{CHKK17}.
\end{proof}

%-----
\subsection*{Proof of Theorem~\ref{T:Boundary_Holder_local}.} %\label{Appendix:A2}

Here, we prove Theorem~\ref{T:Boundary_Holder_local} by modifying methodologies in \cite[Section~III]{DB93}.
Because of translation invariant property, we assume $x_0 = 0$ without loss of generality. First, we construct a parabolic cyliner
\begin{equation}\label{Q0}
Q_{2R, R^{2-\epsilon}} = K_{2R}(0)\times [0, R^{2-\epsilon}], \quad \text{for}\quad R>0,
\end{equation}
where $\epsilon = \epsilon(m)$ is a positive number determined later.  Moreover, let us set
\begin{equation}\label{mu-omega}
\mu_{+} := \esssup_{Q_{2R, R^{2-\epsilon}}} u , \qquad \mu_{-} := \essinf_{Q_{2R, R^{2-\epsilon}}} u,
\quad \text{and} \quad \omega := \essosc_{Q_{2R, R^{2-\epsilon}}} u = \mu_{+} - \mu_{-}.
%\end{aligned}
\end{equation}
Now we construct the intrinsic parabolic cylinder such that
\[
Q_{R, a_0 R^{2}} := K_{R}(0) \times [0, a_0 R^2], \quad \text{where} \quad \frac{1}{a_0}=\left(\frac{\omega}{4}\right)^{m-1}.
\]
These parabolic boxes are lying on the bottom of $\mathbb{R}^d\times \mathbb{R}_{+}$.

We establish the decay of oscillation in the next proposition,  which is similar to \cite[Proposition~III.11.1]{DB93}.
\begin{proposition}\label{P:main}
There exist constants $\epsilon_0 = \epsilon_{0} (m,d,\epsilon) \in (0,1)$ for $\epsilon$ given in \eqref{Q0}, $ \eta_1 =  \eta_1 (m,d)\in (0,1)$, and $ \eta_2 =  \eta_2 (m,d) \in (0,1)$ which satisfy the following.
Let us construct the sequences with $R_0 = R$ and $\omega_0 = \omega$:
\begin{equation}\label{omega-n}
R_n =  {\eta_2}^{n} R, \quad \omega_{n+1} = \max\{ \eta_1 \omega_n, \, 4 R_n^{\epsilon_0}\}, \quad n=1,2,\ldots .
\end{equation}
Also, construct the family of intrinsic parabolic cylinders:
\[
Q_n := K_{R_n} \times [0, \left(\frac{\omega_{n}}{4}\right)^{1-m} R_n^2], \quad n=1,2,\ldots.
\]
that satisfies $Q_{n+1} \subset Q_{n}$.
Then the following holds: for all $n=0,1,2,\ldots$,
\begin{equation}\label{main-oscillation}
\essosc_{Q_n} u \leq \max\{\omega_n,  \, 2 \essosc_{K_{R_n}} u_0\}.
\end{equation}
\end{proposition}

\begin{proof}
For any $R>0$, let us start from $Q_0 := Q_{2R, R^{2-\epsilon}}$ in \eqref{Q0} where $\epsilon \in (0,1)$ to be determined later.
For $\mu_{\pm}$ and $\omega$ given in \eqref{mu-omega}, we construct intrinsically scaled parabolic cylinders
\[
Q_{R}^{\omega} := K_{R}(0) \times [0, \left(\frac{\omega}{4}\right)^{1-m} R^2].
\]

Then there are two cases: either
\begin{equation}\label{Q_Q0}
Q_{R}^{\omega} \subset Q_0 \quad \text{which gives} \quad \omega > 4 R^{\frac{\epsilon}{m-1}},
\end{equation}
or
\begin{equation}\label{epsilon_0}
Q_0 \subset Q_{R}^{\omega} \quad \text{which gives} \quad \omega \leq 4 R^{\frac{\epsilon}{m-1}}.
\end{equation}
If \eqref{epsilon_0} holds, then \eqref{omega-n} follows directly with $\epsilon_0 = \frac{\epsilon}{m-1}$.

Now let us assume \eqref{Q_Q0} and  set
\[
\mu^{+}_{0} = \esssup_{K_R} u_0 , \quad \mu^{-}_{0}= \essinf_{K_R} u_0, \quad \omega_0 = \essosc_{K_R} u_0 = \mu^{+}_{0} - \mu^{-}_{0}.
\]
Then there are two inequalities to cosider:
\begin{equation}\label{mu_mu0}
\mu^{+} - \frac{\omega}{4} \leq \mu^{+}_{0} \quad \text{and} \quad \mu^{-} + \frac{\omega}{4} \geq \mu^{-}_{0}.
\end{equation}

If both inequalities in \eqref{mu_mu0} hold, then the subtraction of the second from the first gives
\begin{equation*} %\label{essosc_u0}
\essosc_{Q_{R}^{\omega}} u \leq 2 \essosc_{K_R} u_0,
\end{equation*}
which holds \eqref{main-oscillation} and there is nothing to prove.

If the second inequality in \eqref{mu_mu0} is violated, then we choose $k =\frac{\omega}{4}$ and it gives
\[
\mu^{-} + k < \mu_{0}^{-},
\]
which satisfies \eqref{H:k}. Therefore, we have energy estimates \eqref{H:local_energy} and \eqref{H:log_energy} for $ (u - \mu^{-} - k)_{-}$. With these estimates, we are able to carry the quantitative method such as the expansion of positivity and DeGiorgi iteration as in \cite{CHKK17}. We obtain that there exists $\delta_1 = \delta_1 (m,d) \in (0,1)$ such that
\[
\essinf_{Q_{R/8}^{\omega/4}}u \geq \mu_{-} + \delta_1 \omega, \quad \text{for} \quad Q_{R/8}^{\omega/4} := K_{R/8}\times [0, \left(\frac{\omega}{4}\right)^{1-m}\left(\frac{R}{8}\right)^2].
\]

If the first inequality in \eqref{mu_mu0} fails, then we have \eqref{H:local_energy} and \eqref{H:log_energy} for $(u - \mu^{+}+k)_{+}$ because \eqref{H:k} holds $k=\frac{\omega}{4}$. Then again the same arguments in \cite{CHKK17} deduce that there exists $\delta_2 = \delta_2 (m,d) \in (0,1)$ such that
\[
\esssup_{Q_{R/8}^{\omega/4}}u \leq \mu_{+} - \delta_2 \omega.
\]

From above, when \eqref{mu_mu0} fails, we choose $ \eta_1 = \max\{ 1 - \delta_1, \ 1-\delta_2\}$ to obtain
\[
\essosc_{Q_{R/8}^{\omega/4}} u \leq \eta_1 \omega,
\]
which parallels to Lemma~III.11.1 in \cite{DB93}.

Finally, we construct nested and shrinking family of parabolic cylinders. First, let us set
\begin{equation}\label{lambda}
    {\eta_2} \leq \frac{1}{8} \, {\eta_1}^{\frac{m-1}{2}}, \quad \text{for} \quad {\eta_1} = \max\{ 1 - \delta_1, \ 1-\delta_2\} \in (0,1).
\end{equation} 
This provides
\[
Q_1:=Q_{{\eta_2} R}^{{\eta_1} \omega / 4} \subset Q_{R/8}^{\omega/4} \subset Q_{R}^{\omega/4} \subset Q_{R}^{\omega/4}=: Q_0.
\]
Then by repeating the same iteration, we complete the proof.
\end{proof}

From Proposition~\ref{P:main}, we derive the H\"{o}lder continuity in the next lemma.
\begin{lemma}\label{L:main}
Let $u$ be a nonnegative weak solution of \eqref{H-PME} with $u_0 \in \calC^{\alpha_0}(\mathbb{R}^d)$. Then there exist $\epsilon_0 = \epsilon_0 (m) \in (0,1)$, $\alpha_i = \alpha_i (m,d) \in (0,1)$ for $i=1,2$, and $c= c(m,d) >1$, such that, for any $R>0$ and $0< r \leq R$, it holds
    \begin{equation*} %\label{Holder}
    \essosc_{Q_r^0} u \leq c \left( \omega + R^{\epsilon_0} + R^{\alpha_0}\right) \left(\frac{r + (\omega/4)^{\frac{m-1}{2}} r}{R}\right)^{\tilde{\alpha}},
    \end{equation*}
    for $\tilde{\alpha} = \min\{\alpha_0, \alpha_1, \alpha_2, \epsilon_0/2 \}$.
\end{lemma}

\begin{proof}
For ${\eta_1}, {\eta_2} \in (0,1)$ given in Proposition~\ref{P:main}, let us set $R_n = {\eta_2}^n R$ and $\omega_n = {\eta_1}^n \omega$. For any $0 < \rho \leq R$, there exist nonnegative integers $k$ and $l$ such that
\begin{equation}\label{rho:space}
    R_{k+1} < r \leq R_{k}
\end{equation}
and
\begin{equation}\label{rho:time}
    \left(\frac{\omega_{l+1}}{4}\right)^{1-m}R_{l+1}^{2} < r^2 \leq \left(\frac{\omega_{l}}{4}\right)^{1-m}R_{l}^{2},
\end{equation}
which means $Q_{r}$ belongs to either $Q_k$ or $Q_l$. Therefore, by Proposition~\ref{P:main}, it implies
\[
\essosc_{Q_r} u \leq \max\{ \omega_{k}, \ \omega_{l}\},
\]
and, moreover, it gives (because $\essosc_{K_{R_n}} u_0 \leq R_n^{\alpha_0}$ by the H\"{o}lder continuity of $u_0$)
\[
\omega_{n+1} \leq {\eta_1} \omega_n + 4R_n^{\epsilon_0} + 2 R_n^{\alpha_0}.
\]

First, we deduce from the iteration that
\begin{equation*}
%\begin{aligned}
\omega_n \leq {\eta_1}^n \omega + 4 \left(\sum_{i=1}^{n}{\eta_1}^{n-i}{\eta_2}^{\epsilon_0 (i-1)}\right)R^{\epsilon_0} + 2\left(\sum_{i=1}^{n}{\eta_1}^{n-i}{\eta_2}^{\alpha_0 (i-1)}\right)R^{\alpha_0}
= I_{n} + II_{n} + III_{n}.
\end{equation*}
Then we consider two cases separately.

If \eqref{rho:space} holds, then we deduce the following, from the LHS of \eqref{rho:space},
    \[
    {\eta_2}^{k+1}  < \frac{r}{R} \quad \Longrightarrow \quad k+1 < \log_{{\eta_2}}r/R = \log_{{\eta_1}} (r/R)^{\alpha_1}, \quad \text{for} \quad \alpha_1 = \left|\log_{{\eta_2}} {\eta_1} \right|.
    \]
    Therefore, we easily obtain
    \begin{equation}\label{I_k}
    {\eta_1}^{k} < {\eta_1}^{-1}\left(\frac{r}{R}\right)^{\alpha_1} \quad \Longrightarrow \quad I_{k} \leq {\eta_1}^{-1}\omega\left(\frac{r}{R}\right)^{\alpha_1}.
    \end{equation}
 Recalling \eqref{lambda}, without loss of generality, let us fix
    \begin{equation}\label{lambda_fixed}
        {\eta_2} = {\eta_1}^{m-1} < \frac{1}{8}{\eta_1}^{\frac{m-1}{2}} \quad \Longrightarrow \quad {\eta_1} = {\eta_2}^{\frac{1}{m-1}}.
    \end{equation}
    Then now we observe that
    \begin{equation*}
    %\begin{aligned}
    II_{k} = 4 \left(\sum_{i=1}^{k}{\eta_1}^{k-i}{\eta_2}^{\epsilon_0 (i-1)}\right)R^{\epsilon_0}
    = 4 \left(\sum_{i=1}^{k}{\eta_2}^{\frac{1}{m-1}(k-i)}{\eta_2}^{\epsilon_0 (i-1)}\right)R^{\epsilon_0}
    \leq 4 \left(\sum_{i=1}^{k}{\eta_2}^{\epsilon_0 (k-1)}\right)R^{\epsilon_0},
    %\end{aligned}
    \end{equation*}
    by choosing $\epsilon_0 = \min \{ 1/2, \ 1/(m-1) \}$.
    Therefore, it yields
    \begin{equation}\label{II_k}
        II_{k} \leq 4 {\eta_2}^{-\epsilon_0} k ({\eta_2}^{k} R )^{\epsilon_0} \leq 4 {\eta_2}^{-\epsilon_0} \left(\frac{R}{r}\right)^{\epsilon_0 /2} \left(\frac{r}{{\eta_2}}\right)^{\epsilon_0}= 4{\eta_2}^{-2\epsilon_0} R^{\epsilon_0} \left(\frac{r}{R}\right)^{\epsilon_0 / 2},
    \end{equation}
    by \eqref{rho:space} and by choosing $k \in \mathbb{N}$ such that
    \begin{equation}\label{k}
    k < {\eta_2}^{-k \epsilon_0 /2} \leq \left(\frac{R}{r}\right)^{\epsilon_0 /2}.
    \end{equation}
   The second inequality of \eqref{k} comes form the RHS of \eqref{rho:space}, and the first inequality of \eqref{k} is true in general by choosing $\frac{1}{{\eta_2}}$ large enough.
  Let us fix
    \begin{equation}\label{alpha0-i}
    \tilde{\alpha} = \min \{ \alpha_0, \alpha_1, \epsilon_0 / 2\},
    \end{equation}
    and then it follows
    \begin{equation}\label{III_k}
        \begin{aligned}
        III_{k} &= 2\left(\sum_{i=1}^{k}{\eta_1}^{k-i}{\eta_2}^{\alpha_0 (i-1)}\right)R^{\alpha_0}
        \leq 2\left(\sum_{i=1}^{k}{\eta_2}^{\epsilon_0 (k-i)}{\eta_2}^{\alpha_0 (i-1)}\right)R^{\alpha_0} \\
        &\leq 2 {\eta_2}^{-\epsilon_0} \left(\sum_{i=1}^{k}{\eta_2}^{(\epsilon_0 -\tilde{\alpha}) (k-i +1)}\right) ({\eta_2}^{k} R)^{\tilde{\alpha}}
        < \frac{2 {\eta_2}^{-2\epsilon_0}}{{\eta_2}^{-\epsilon_0 /2} - 1} R^{\alpha_0} \left(\frac{r}{R}\right)^{\tilde{\alpha}},
        \end{aligned}
    \end{equation}
    because, by the LHS of \eqref{rho:space},
    \[
    \sum_{i=1}^{k}{\eta_2}^{(\epsilon_0 -\tilde{\alpha}) (k-i +1)} < \frac{1}{{\eta_2}^{\alpha_0 - \epsilon_0} - 1} \leq \frac{1}{{\eta_2}^{- \epsilon_0 /2} - 1},
    \quad \text{and}\quad {\eta_2}^{-\tilde{\alpha}} < {\eta_2}^{-\epsilon_0}.
    \]
 Finally, the summation of \eqref{I_k}, \eqref{II_k}, \eqref{III_k}, and the choice of $\alpha_0$ in \eqref{alpha0-i} imply that
    \begin{equation*}
    \omega_{l} \leq {\eta_1}^{-1}\omega\left(\frac{r}{R}\right)^{\alpha_1} + 4{\eta_2}^{-2\epsilon_0} R^{\epsilon_0} \left(\frac{r}{R}\right)^{\epsilon_0 / 2}
    +  \frac{2 {\eta_2}^{-2\epsilon_0}}{{\eta_2}^{-\epsilon_0 /2} - 1} R^{\alpha_0} \left(\frac{r}{R}\right)^{\tilde{\alpha}}
    \leq c \left( \omega + R^{\epsilon_0} + R^{\alpha_0}\right) \left(\frac{r}{R}\right)^{\tilde{\alpha}}.
    \end{equation*}

  If \eqref{rho:time} hold, then let us rewrite \eqref{rho:time} in the following:
    \begin{equation}\label{rho:time2}
        \left({\eta_1}^{1-m} {\eta_2}^2 \right)^{l+1} R^2 < \left(\frac{\omega}{4}\right)^{m-1} r^2 \ (=:\tilde{r}^2) \leq \left({\eta_1}^{1-m} {\eta_2}^2 \right)^{l} R^2.
    \end{equation}
Let us set
    \[
    \tilde{{\eta_2}} = {\eta_1}^{\frac{1-m}{2}}{\eta_2}, \quad \text{and} \quad \tilde{r}^2 =  \left(\frac{\omega}{4}\right)^{m-1} r^2.
    \]
    Then \eqref{rho:time2} is now simplified as
    \begin{equation*}
        \tilde{{\eta_2}}^{l+1} R < \tilde{r} \leq \tilde{{\eta_2}}^l R.
    \end{equation*}
Without loss of generality, assume that $\tilde{r} \leq R$.
    Moreover, the choice of ${\eta_2}$ as in \eqref{lambda_fixed} gives
    \[
    {\eta_1}^{1-m} {\eta_2}^{2} = {\eta_1}^{m-1} = {\eta_2}  \quad \text{and} \quad \tilde{{\eta_2}} = {\eta_2}^{3/2}.
    \]
   Therefore, the first inequality in \eqref{rho:time2} yields
        \begin{equation*} %\label{I_l}
    l+1 < \log_{{\eta_1}} \left(\tilde{r}/R\right)^{\alpha_2} \quad \Longrightarrow \quad I_l \leq {\eta_1}^{-1} \omega \left(\frac{\tilde{r}}{R}\right)^{\alpha_2}, \quad \text{for}\quad \alpha_2 = \abs{\log_{\tilde{{\eta_2}}} {\eta_1}}.
    \end{equation*}
Then we carry similar analysis for $II_l$ and $III_l$ as handling $II_k$ in \eqref{II_k} and $III_k$ as in \eqref{III_k} with $\tilde{\alpha} = \min\{\alpha_0, \alpha_2, \frac{\epsilon_0}{2}\}$ to have
\begin{equation*}
    \omega_{l} \leq {\eta_1}^{-1} \omega \left(\frac{\tilde{r}}{R}\right)^{\alpha_2}+ 4 \tilde{{\eta_2}}^{-2\epsilon_0} R^{\epsilon_0}\left(\frac{\tilde{r}}{R}\right)^{\epsilon_0 /2}+ \frac{2 {\eta_2}^{-2\epsilon_0}}{{\eta_2}^{-\epsilon_0 /2} - 1}   R^{\alpha_0}\left(\frac{\tilde{r}}{R} \right)^{\tilde{\alpha}}
    \leq c \left( \omega + R^{\epsilon_0} + R^{\alpha_0}\right) \left(\frac{\tilde{r}}{R} \right)^{\tilde{\alpha}}.
    \end{equation*}

Therefore, we complete the proof by choosing $\tilde{\alpha} = \min\{\alpha_0, \alpha_1, \alpha_2, \frac{\epsilon_0}{2} \}$ and by combining all estimates above.
\end{proof}

Once we have Lemma~\ref{L:main}, Theorem~\ref{T:Boundary_Holder_local} follows from standard computations (refer \cites{DB93, CHKK17, HZ21}).

%------------------------------------------------------------------------------------------------------

\section{Figure supplements}\label{Appendix:fig}

\begin{itemize}
\item Figures 11 and 12 (\textit{Theorem~\ref{Theorem-2-b} for case $1<m\leq 2$, $q>1$}): Refer Remark~\ref{R:Theorem2-b} $(iii)$ and Figure 4-$(i)$. %When $d$ gets larger, then there are two cases.
As $d$ increases, the point $\textbf{b}$ approaches closer to the origin and $\textbf{A}$ may locate on the right hand side of $\textbf{B}$ and $\textbf{b}$.

% Figure 11
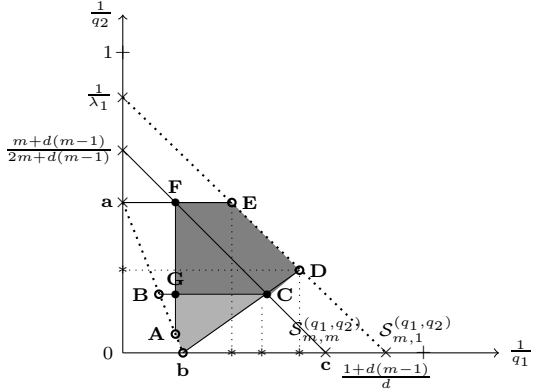
\begin{center}
\begin{tikzpicture}[domain=0:16]

%\draw[very thin, color=gray](0, 0) grid (6, 5);

\draw (0, -1) node[right] { \scriptsize \textit{Figure 4-(ii)}. $\max\{2,\frac{2m}{(2m-1)(m-1)}\} < d \leq \frac{2m}{m-1}$.};

\fill[fill= lgray]
(0.7, 0.78) -- (0.7, 0.25) -- (0.8,0) -- (1.92, 0.78);

\fill[fill= gray]
(0.7, 2)--(0.7, 0.78) -- (1.97,0.78) -- (2.35, 1.1) -- (1.45, 2)--(0.7, 2);

\draw[->] (0,0) node[left] {\scriptsize $0$}
-- (5,0) node[right] {\scriptsize $\frac{1}{q_1}$};
\draw[->] (0,0) -- (0,4.5) node[left] { \scriptsize $\frac{1}{q_2}$};

\draw (0,4) node{\scriptsize $+$} node[left]{\scriptsize $1$} ;
\draw (4,0) node{\scriptsize $+$} ;

\draw[very thin]
(0, 2) -- (1.45, 2) ;
\draw[thick](1.45, 2) circle(0.05) node[right] {\scriptsize \textbf{E}};

\draw (0, 2) -- (1.45, 2);

\draw (0.8, 0) -- (2.35, 1.1);

\draw[thick] (2.35, 1.1) circle(0.05) node[right] {\scriptsize \textbf{D}};
\draw[very thin]
 (0.48, 0.78)
-- (1.92, 0.78) node{\scriptsize $\bullet$} node[right]{\scriptsize \textbf{C}} ;

\draw[thick](0.48, 0.78) circle(0.05) node[left] {\scriptsize \textbf{B}};

\draw (0.7, 0.78) node{\scriptsize $\bullet$} node[above]{\scriptsize \textbf{G}}  ;

\draw[very thin]
(0.7, 0.3)
 --(0.7, 2) node{\scriptsize $\bullet$} node[above] {\scriptsize \textbf{F}};

\draw[thick](0.7, 0.25) circle(0.05) node[left] {\scriptsize \textbf{A}};

% S
\draw[thick, dotted] (0,2) node{\scriptsize $\times$}  node[left] {\scriptsize $\textbf{a}$}
 -- (0.8, 0)  node[below] {\scriptsize $\textbf{b}$};
%\draw[thick] (0,2) circle(0.05);
\draw[thick] (0.8, 0) circle(0.05);

% S_1
\draw[thick, dotted]
(0,3.4) node {\scriptsize $\times$} node[left] {\scriptsize $\frac{1}{{\lambda_1}}$}
-- (3.5, 0) node {\scriptsize $\times$} node[below] {\scriptsize $\frac{1+d(m-1)}{d}$};
\draw (3.9, 0.3) node{\scriptsize $\mathcal{S}_{m,1}^{(q_1, q_2)}$};

% S_m
\draw (0,2.7) node {\scriptsize $\times$} node[left] {\scriptsize $\frac{m+d(m-1)}{2m+d(m-1)}$}
    -- (2.7, 0) node {\scriptsize $\times$} node[below] {\scriptsize \textbf{c}};
\draw (2.7, 0.3) node{\scriptsize $\mathcal{S}_{m,m}^{(q_1, q_2)}$};

% dotted line
%\draw[thin, dotted] (0.4, 0) node{\scriptsize $*$} node[below] {\scriptsize $\frac{m-1}{2m}$} --(0.4, 2);

%\draw[thin, dotted] (0, 1.3) node{\scriptsize $*$}  -- (0.4, 1.3);
%\draw[thin, dotted] (0, 0.63) node{\scriptsize $*$} node[left]{\scriptsize $\frac{m-1}{2m-1}$}  -- (0.8, 0.63);
%\draw[thin, dotted] (0.8, 0) node{\scriptsize $*$}  -- (0.8, 0.63);

\draw[thin, dotted] (1.85, 0) node{\scriptsize $*$}  -- (1.85, 0.63);

\draw[thin, dotted] (2.35, 0) node{\scriptsize $*$}  -- (2.35, 1.1);
\draw[thin, dotted] (0, 1.1) node{\scriptsize $*$}  -- (2.35, 1.1);

\draw[thin, dotted] (1.45, 0) node{\scriptsize $*$}  -- (1.45, 2);
%\draw[thin, dotted] (0, 2) node{\scriptsize $*$}  -- (1.45, 2);

%\draw (6, 4) node[right] {\scriptsize $\overline{\textbf{ab}} = \mathcal{S}_{m, \infty}^{(q_1, q_2)}$, $\textbf{a} = (0, \frac 12)$, $\textbf{b} = (\frac 1d, 0)$ };
%\draw (7, 4) node[right] {\scriptsize $\textbf{a} = (0, \frac 12)$, $\textbf{b} = (\frac 1d, 0)$ };

\draw(7, 4) node[right]{\scriptsize $\overline{\textbf{ab}}, \textbf{A}-\textbf{F}, \textbf{c}$: same as in Fig. 4-$(i)$};
\draw (7, 3.5) node[right] {\scriptsize $\mathcal{R} (\textbf{GCDEF})$: scaling invariant class of \eqref{T2:V} };
\draw (7, 3) node[right] {\scriptsize $\mathcal{R} (\textbf{GAbC})$: scaling invariant class of \eqref{T2:V-embedding} };
%\draw (6, 1.5) node[right] {\scriptsize $\textbf{A} - \textbf{F}$: same as in Fig. 4-i};
\draw (7, 2.5) node[right] {\scriptsize $\textbf{G} = (\frac{m-1}{2m}, \frac{m-1}{2m-1})$ };

\end{tikzpicture}
\end{center}

% Figure 12
\begin{center}
\begin{tikzpicture}[domain=0:16]

% second figure

\draw (0, -1) node[right] { \scriptsize \textit{Figure 4-(iii)}. $ d > \frac{2m}{m-1}$.};

\fill[fill= lgray]
(0.75, 1) -- (0.75, 0.2) -- (1.75,1);

\fill[fill= gray]
(0.75, 2)--(0.75, 1) -- (1.75,1) -- (2.15, 1.3) -- (1.45, 2)--(0.75, 2);

\draw[->] (0,0) node[left] {\scriptsize $0$}
-- (5,0) node[right] {\scriptsize $\frac{1}{q_1}$};
\draw[->] (0,0) -- (0,4.5) node[left] { \scriptsize $\frac{1}{q_2}$};

\draw (0,4) node{\scriptsize $+$} node[left]{\scriptsize $1$} ;
\draw (4,0) node{\scriptsize $+$} ;

\draw[very thin]
(0, 2) -- (1.45, 2) ;
\draw[thick](1.45, 2) circle(0.05) node[right] {\scriptsize \textbf{E}};

\draw (0, 2) -- (1.45, 2);

\draw (0.5, 0) -- (2.15, 1.3);

\draw[thick] (2.15, 1.3) circle(0.05) node[right] {\scriptsize \textbf{D}};
\draw[very thin]
 (0.25, 1)
-- (1.75, 1) node{\scriptsize $\bullet$} node[right]{\scriptsize \textbf{C}} ;

\draw[thick](0.25, 1) circle(0.05) node[below] {\scriptsize \textbf{B}};

\draw (0.75, 1) node{\scriptsize $\bullet$} node[above]{\scriptsize \textbf{G}}  ;
\draw (0.75, 0.2) node{\scriptsize $\bullet$} node[above]{\scriptsize \textbf{H}}  ;

\draw[very thin]
(0.75, 0)
 --(0.75, 2) node{\scriptsize $\bullet$} node[above] {\scriptsize \textbf{F}};

\draw[thick](0.75, 0) circle(0.05) node[below] {\scriptsize \textbf{A}};

% S
\draw[thick, dotted] (0,2) node{\scriptsize $\times$} node[left] {\scriptsize $\textbf{a}$}
 -- (0.5, 0) node{\scriptsize $\times$}  node[below] {\scriptsize $\textbf{b}$};
%\draw[thick] (8,2) circle(0.05);
%\draw[thick] (8.5,0) circle(0.05);

% S_1
\draw[thick, dotted]
(0,3.4) node {\scriptsize $\times$} node[left] {\scriptsize $\frac{1}{{\lambda_1}}$}
-- (3.5, 0) node {\scriptsize $\times$} node[below] {\scriptsize $\frac{1+d(m-1)}{d}$};
\draw (3.9, 0.3) node{\scriptsize $\mathcal{S}_{m,1}^{(q_1, q_2)}$};

% S_m
\draw (0,2.8) node {\scriptsize $\times$} node[left] {\scriptsize $\frac{m+d(m-1)}{2m+d(m-1)}$}
    -- (2.7, 0) node {\scriptsize $\times$} node[below] {\scriptsize \textbf{c}};
\draw (2.7, 0.3) node{\scriptsize $\mathcal{S}_{m,m}^{(q_1, q_2)}$};

% dotted line
%\draw[thin, dotted] (0.4, 0) node{\scriptsize $*$} node[below] {\scriptsize $\frac{m-1}{2m}$} --(0.4, 2);

%\draw[thin, dotted] (0, 1.3) node{\scriptsize $*$}  -- (0.4, 1.3);
%\draw[thin, dotted] (0, 0.63) node{\scriptsize $*$} node[left]{\scriptsize $\frac{m-1}{2m-1}$}  -- (0.8, 0.63);
%\draw[thin, dotted] (0.8, 0) node{\scriptsize $*$}  -- (0.8, 0.63);

%\draw[thin, dotted] (1.85, 0) node{\scriptsize $*$}  -- (1.85, 0.63);

\draw[thin, dotted] (10.35, 0) node{\scriptsize $*$}  -- (10.35, 1.1);
%\draw[thin, dotted] (0, 1.1) node{\scriptsize $*$}  -- (2.35, 1.1);

\draw[thin, dotted] (9.45, 0) node{\scriptsize $*$}  -- (9.45, 2);
%\draw[thin, dotted] (0, 2) node{\scriptsize $*$}  -- (1.45, 2);

%\draw (8, 4) node[right] {\scriptsize $\overline{\textbf{ab}} = \mathcal{S}_{m, \infty}^{(q_1, q_2)}$, $\textbf{a} = (0, \frac 12)$, $\textbf{b} = (\frac 1d, 0)$ };
%\draw (7, 4) node[right] {\scriptsize $\textbf{a} = (0, \frac 12)$, $\textbf{b} = (\frac 1d, 0)$ };

\draw(7, 4) node[right]{\scriptsize $\overline{\textbf{ab}}, \textbf{A}-\textbf{F}, \textbf{c}$: same as in Fig. 4-$(i)$};
\draw (7, 3.5) node[right] {\scriptsize $\mathcal{R} (\textbf{GCDEF})$: scaling invariant class of \eqref{T2:V} };
\draw (7, 3) node[right] {\scriptsize $\mathcal{R} (\textbf{GHC})$: scaling invariant class of \eqref{T2:V-embedding}};

%\draw (8, 1.5) node[right] {\scriptsize $\textbf{A}- \textbf{F}$: same as in Fig. 4-i };
\draw (7, 2.5) node[right] {\scriptsize $\textbf{G} = (\frac{m-1}{2m}, \frac{m-1}{2m-1})$};
\draw (7, 2) node[right] {\scriptsize $\textbf{H} = (\frac{m-1}{2m}, \frac{(m-1)d-2m}{2m(d-2)})$ };
%\draw (7, 2) node[right] {\scriptsiz}

\end{tikzpicture}
\end{center}

%%%%%%%%%%%%%%%%

\item \emph{Theorem~\ref{Theorem-2-b} for case $m>2$, $q\geq m-1$}: Refer Remark~\ref{R:Theorem2-b} $(iv)$ and Figure 5-$(i)$.
%As $d$ increases, the point $\textbf{b}$ approaches closer to the origin and $\textbf{A}$ may locate on the right hand side of $\textbf{b}$.

% Figure 12
\begin{center}
\begin{tikzpicture}[domain=0:16]

\draw (0, -1) node[right] { \scriptsize \textit{Figure 5-(ii)}. $m>2$, $q\geq m-1$, $d >\frac{2m}{m-1}$.};

\fill[fill= lgray]
(0.85, 1.27)-- (0.85, 0.38) -- (1.5, 1.27);

\fill[fill= gray]
(0.85, 2) -- (0.85,1.27) -- (1.5, 1.27) --(2,2);

\draw[->] (0,0) node[left] {\scriptsize $0$} -- (5,0) node[right] {\scriptsize $\frac{1}{q_1}$};
\draw[->] (0,0) -- (0,4.5) node[left] { \scriptsize $\frac{1}{q_2}$};

\draw (0,4) node{\scriptsize $+$} node[left] {\scriptsize $1$} ;
\draw (4,0) node{\scriptsize $+$};

\draw
(0, 2) node[left] {\scriptsize \textbf{a}} -- (2, 2) ;

\draw[thick, dotted] (0,2) node{\scriptsize $\times$} node[left] {\scriptsize $\textbf{a}$}
 -- (0.6, 0)node{\scriptsize $\times$} node[below] {\scriptsize $\textbf{b}$};
% \draw[thick] (0.6, 0) circle(0.05);
%\draw[thick] (0,2) circle(0.05);

\draw(2, 2) node{\scriptsize $\bullet$} node[above] {\scriptsize \textbf{D}};

\draw(0.85,2) node{\scriptsize $\bullet$} node[above]{\scriptsize \textbf{E}}
--(0.85, 0) node[below]{\scriptsize \textbf{A}} ;
\draw[thick] (0.85, 0) circle(0.05);

\draw (0.23, 1.27) node[below]{\scriptsize \textbf{B}}
--(1.5, 1.27)node{\scriptsize $\bullet$}node[right]{\scriptsize \textbf{C}};
\draw[thick] (0.23, 1.27) circle(0.05);

\draw(0.85, 1.27) node{\scriptsize $\bullet$} node[above]{\scriptsize \textbf{G}};

\draw(0.85, 0.38) node{\scriptsize $\bullet$}node[right]{\scriptsize \textbf{H}};

\draw (0.6, 0)  -- (2, 2);

%\draw[very thin] (0, 2)-- (1.65, 1.13);

\draw (0,3.5) node {\scriptsize $\times$} node[left] {\scriptsize $\frac{1+d}{2+d}$}
-- (4.6, 0) node {\scriptsize $\times$} node[below] {\scriptsize $\frac{1+d}{d}$} ;
\draw (4.5, 1) node {\scriptsize $\mathcal{S}_{m,m-1}^{(q_1, q_2)}$} ;

\draw (0,2.9) node {\scriptsize $\times$} node[left] {\scriptsize $\frac{m+d(m-1)}{md}$}
 -- (2.7, 0) node {\scriptsize $\times$} node[below] {\scriptsize $\frac{m+d(m-1)}{2m+d(m-1)}$};
\draw (3.2, 0.5) node{\scriptsize $\mathcal{S}_{m,m}^{(q_1, q_2)}$};

\draw (7, 4) node[right] {\scriptsize $\overline{\textbf{ab}}, \textbf{B}-\textbf{E}$: same as in Fig. 5-$(i)$};
\draw (7, 3.5) node[right] {\scriptsize $\mathcal{R}(\textbf{GCDE})$: scaling invariant class of \eqref{T2:V} };
\draw (7, 3) node[right] {\scriptsize $\mathcal{R}(\textbf{GHC})$: scaling invariant class of \eqref{T2:V-embedding} };
\draw (7, 2.5) node[right] {\scriptsize $\textbf{A}= (\frac{m-1}{2m},0)$  };
\draw (7, 2) node[right] {\scriptsize $\textbf{G}, \textbf{H}$: same as in Fig. 4-$(iii)$ };

\end{tikzpicture}
\end{center}

\item \emph{Theorem~\ref{Theorem-4a} for case $m>2$, $q > \frac m2$}: Refer Remark~\ref{R:Theorem-4a} $(iii)$ and Figure 6-$(i)$.

When $d=2$, the line $\overline{\textbf{CH}}$ is excluded from the region $\mathcal{R}(\textbf{ABCHI})$.

% Figure 13
\begin{center}
\begin{tikzpicture}[domain=0:16]

%\draw[very thin, color=gray](0, 0) grid (6, 5);

\draw (0, -1) node[right] { \scriptsize \textit{Figure 6-(ii)}. $m>2 $, $q> \frac m2$, $d>2$.};

\fill[fill= lgray]
(0, 2) -- (1.1,0) -- (2.5, 0) -- (3.25, 0.72) -- (0,3) -- (0, 2);

\draw[->] (0,0) node[left] {\scriptsize $0$}
-- (6.5,0) node[right] {\scriptsize $\frac{1}{q_1}$};
\draw[->] (0,0) -- (0,4.5) node[left] { \scriptsize $\frac{1}{q_2}$};

\draw (0,4) node{\scriptsize $+$} node[left]{\scriptsize $1$} ;
%\draw (4,0) node{\scriptsize $+$} node[below]{\scriptsize $1$};
%\draw (4,0) node{\scriptsize $\bullet$} node[below]{\scriptsize \textbf{E}} ;

\draw[thick] (0, 2) circle(0.05) node[left] {\scriptsize \textbf{A}};
\draw (0.9, 0.5) node{\scriptsize $\mathcal{S}_{m,\infty}^{(q_1, q_2)}$};

% S
\draw[thick, dotted] (0,2) -- (1.1, 0) node[below] {\scriptsize $\textbf{B}$};
\draw[thick] (1.1, 0) circle(0.05) ;

% S_1
\draw [dotted]
(0,3.6)  node[left] {\scriptsize \textbf{E}}
-- (5.7, 0) node{\scriptsize $\times$}  node[below]{\scriptsize $\frac{1+d(m-1)}{d}$} ;
\draw[thick] (0, 3.6) circle(0.05) ;
%\draw[thick] (5, 0) circle(0.05) ;
\draw (1.6, 3) node{\scriptsize $\mathcal{S}_{m,1}^{(q_1, q_2)}$};

% S_q
\draw[thick, dotted] (0,3)
    -- (4.3, 0)   node {\scriptsize $\times$} node[below]{\scriptsize $\frac{m+2d(m-1)}{md}$};
\draw (2, 2) node{\scriptsize $\mathcal{S}_{m,\frac m2}^{(q_1, q_2)}$};
\draw[thick] (0, 3) circle(0.05) node[left]{\scriptsize \textbf{I}};
\draw[thick] (3.23, 0.75) circle(0.05);
\draw (3.23, 0.75) node[right] {\scriptsize \textbf{H}};

% S_q^\ast
\draw (0,2.6) node {\scriptsize $\bullet$}  node[left]{\scriptsize \textbf{f}}
    -- (2.5, 0) node {\scriptsize $\bullet$}  node[below] {\scriptsize \textbf{C}};
\draw (2, 1) node{\scriptsize $\mathcal{S}_{m,q^\ast}^{(q_1, q_2)}$};

% cutting line
\draw[thick] (3.7, 1.25) circle(0.05) node[right]{\scriptsize \textbf{D}} ;
\draw (2.5, 0) -- (3.7,1.25);

\draw (7, 4) node[right] {\scriptsize $\overline{\textbf{AB}} = \mathcal{S}_{m, \infty}^{(q_1, q_2)}$, $\textbf{A} = (0, \frac 12)$, $\textbf{B} = (\frac 1d, 0)$ };
%\draw (7, 4) node[right] {\scriptsize $\textbf{a} = (0, \frac 12)$, $\textbf{b} = (\frac 1d, 0)$ };
%\draw (7, 3.5) node[right] {\scriptsize  $\textbf{g} = (\frac{1}{d+2}, \frac{1}{d+2})$ };

\draw (7, 3.5) node[right] {\scriptsize $\mathcal{R} (\textbf{ABCHI})$: scaling invariant class of \eqref{T4:V-energy}};
\draw (7, 3) node[right] {\scriptsize \textbf{C,D,E,f}: same as in Fig. 6-$(i)$ };

\draw (7,2.5) node[right] {\scriptsize $\overline{\textbf{IH}}$ is on $\mathcal{S}_{m,m/2}^{(q_1, q_2)}$};
\draw (7, 2) node[right] {\scriptsize  $\textbf{H} = \textbf{g}_{\{q=m/2\}}$ where $\textbf{g}$ is same in Fig. 6-$(i)$};
\draw (7, 1.5) node[right] {\scriptsize $\textbf{I} = (0, \frac{m+2d(m-1)}{2m+2d(m-1)})$};

\end{tikzpicture}
\end{center}

%%%%%

\item \emph{Theorem~\ref{Theorem-4} for case $1<m\leq 2$, $q>1$}: Refer Remark~\ref{R:Theorem-4}  $(ii)$ and Figure 7-$(i)$.

% Figure 14
\begin{center}
\begin{tikzpicture}[domain=0:16]

%\draw[very thin, color=gray](0, 0) grid (6, 5);

\draw (0, -1) node[right] { \scriptsize \textit{Figure 7-(ii)}. $\max\{ 2, \frac{2m}{(2m-1)(m-1)}\} < d \leq \frac{2m}{m-1}$.};

%\fill[fill= lgray]
%(0, 2) -- (1.1,0) -- (3.7, 0) -- (0, 3.6) -- (0, 2);

\fill[fill= lgray]
(0, 2) -- (1.1,0) -- (2.3, 0) -- (2.3, 1.35) --(0,3.6)--(0, 2);

\fill[fill= gray]
(0.75, 1.05) -- (1.9,1.05) -- (2.3, 1.35) -- (0, 3.6)--(0.75,2);

%\fill[fill= lgray]
%(0, 2) -- (1.46,0.64) -- (1.85, 1.6) -- (1.45, 2);

\draw[->] (0,0) node[left] {\scriptsize $0$}
-- (5,0) node[right] {\scriptsize $\frac{1}{q_1}$};
\draw[->] (0,0) -- (0,4.5) node[left] { \scriptsize $\frac{1}{q_2}$};

\draw (0,4) node{\scriptsize $+$} node[left]{\scriptsize $1$} ;
%\draw (4,0) node{\scriptsize $+$} ;

% S
\draw[thick, dotted] (0,2) -- (1.1, 0)  node[below] {\scriptsize $\textbf{b}$};
\draw[thick] (1.1, 0) circle(0.05) ;

% S_1
\draw [thick, dotted]
(0,3.6)  node[left] {\scriptsize $\frac{1}{{\lambda_1}}$} node[right] {\scriptsize \textbf{E}}
-- (3.7, 0)  node[below]{\scriptsize \textbf{g}};
\draw[thick] (0, 3.6) circle(0.05) ;
\draw[thick] (3.7, 0) circle(0.05) ;
\draw (4, 0.5) node{\scriptsize $\mathcal{S}_{m,1}^{(q_1, q_2)}$};

% S_m
\draw (0,2.6) node {\scriptsize $\times$} node[left] {\scriptsize $\frac{m+d(m-1)}{2m+d(m-1)}$}
    -- (3.2, 0) node {\scriptsize $\times$} node[below] {\scriptsize \textbf{i}};
\draw (2.3, 0.5) node{\scriptsize $\mathcal{S}_{m,m}^{(q_1, q_2)}$};

% line AF
\draw (0.75, 0.6) node[left]{\scriptsize \textbf{A}} -- (0.75, 2);
\draw[thick] (0.75, 0.6) circle(0.05);

% line EF
\draw (0.75, 2) node{\scriptsize $\bullet$} node[right] {\scriptsize \textbf{F}};
\draw[dotted](0, 2)  node[left] {\scriptsize \textbf{a}}--(0.75, 2);
\draw[thick] (0, 2) circle(0.05) ;
\draw (0.75, 2) -- (0, 3.6) ;

% line BC
\draw (0.5,1.05) node[left]{ \scriptsize $\textbf{B}$}
--(1.9, 1.05)node{\scriptsize $\bullet$}node[right]{ \scriptsize $\textbf{C}$};
\draw[thick] (0.5, 1.05) circle(0.05);

% line CD
\draw (1.9, 1.05) -- (2.3, 1.35) node[right] {\scriptsize \textbf{D}};
\draw[thick] (2.3, 1.35) circle(0.05) ;

% point G
\draw (0.75,1.05) node{\scriptsize $\bullet$} node[right]{\scriptsize $\textbf{G}$};

%\draw[thin, dotted] (0.5, 0) node{\scriptsize $*$}  -- (0.5, 1.1);
%\draw[thin, dotted] (0, 1.1) node{\scriptsize $*$}  -- (0.5, 1.1);

%\draw[thin, dotted] (0, 0.85) node{\scriptsize $*$}  -- (1.75, 0.85);
%\draw[thin, dotted] (0.6, 0) node{\scriptsize $*$}  -- (0.6, 0.85);

%\draw[thin, dotted] (1.75, 0) node{\scriptsize $*$}--(1.75, 0.85);

\draw[thin, dotted] (2.3, 0)  node[below]{\scriptsize \textbf{h}}  -- (2.3, 1.35) ;
\draw[thick] (2.3, 0) circle(0.05) ;

%\draw[thin, dotted] (0, 1.35) node{\scriptsize $*$}  -- (2.3, 1.35) ;

\draw (7, 4) node[right] {\scriptsize $\overline{\textbf{ab}} = \mathcal{S}_{m, \infty}^{(q_1, q_2)}$, $\textbf{a} = (0, \frac 12)$, $\textbf{b} = (\frac 1d, 0)$ };

\draw (7, 3.5) node[right] {\scriptsize $\mathcal{R} (\textbf{GCDEF})$: scaling invariant class of \eqref{T4:V-divfree} };
\draw (7, 3) node[right] {\scriptsize $\mathcal{R} (\textbf{abhDE})$: scaling invariant class of \eqref{T4:V-divfree-embedding} };

\draw (7, 2.5) node[right] {\scriptsize $\textbf{A-F, g-i}$: same as in Fig. 7-$(i)$};
\draw (7, 2) node[right] {\scriptsize $\textbf{G} = (\frac{m-1}{2m}, \frac{m-1}{2m-1})$};

\end{tikzpicture}
\end{center}

% Figure 15
\begin{center}
\begin{tikzpicture}[domain=0:16]

\draw (0, -1) node[right] { \scriptsize \textit{Figure 7-(iii)}. $d > \frac{2m}{m-1}$.};

%\fill[fill= lgray]
%(0, 2) -- (1.1,0) -- (3.7, 0) -- (0, 3.6) -- (0, 2);

\fill[fill= lgray]
(0, 2) -- (0.5,0) -- (2.3, 0) -- (2.3, 1.35) --(0,3.6)--(0, 2);

\fill[fill= gray]
(0.75, 1.05) -- (1.9,1.05) -- (2.3, 1.35) -- (0, 3.6)--(0.75,2);

%\fill[fill= lgray]
%(0, 2) -- (1.46,0.64) -- (1.85, 1.6) -- (1.45, 2);

\draw[->] (0,0) node[left] {\scriptsize $0$}
-- (5,0) node[right] {\scriptsize $\frac{1}{q_1}$};
\draw[->] (0,0) -- (0,4.5) node[left] { \scriptsize $\frac{1}{q_2}$};

\draw (0,4) node{\scriptsize $+$} node[left]{\scriptsize $1$} ;
%\draw (12,0) node{\scriptsize $+$} ;

% S
\draw[thick, dotted] (0,2) -- (0.5, 0)  node[below] {\scriptsize $\textbf{b}$};
\draw[thick] (0.5, 0) circle(0.05) ;

% S_1
\draw [thick, dotted]
(0,3.6)  node[left] {\scriptsize $\frac{1}{{\lambda_1}}$} node[right] {\scriptsize \textbf{E}}
-- (3.7, 0)  node[below]{\scriptsize \textbf{g}};
\draw[thick] (0, 3.6) circle(0.05) ;
\draw[thick] (3.7, 0) circle(0.05) ;
\draw (4, 0.5) node{\scriptsize $\mathcal{S}_{m,1}^{(q_1, q_2)}$};

% S_m
\draw (0,2.6) node {\scriptsize $\times$}
    -- (3.2, 0) node {\scriptsize $\times$} node[below] {\scriptsize \textbf{i}};
\draw (2.3, 0.5) node{\scriptsize $\mathcal{S}_{m,m}^{(q_1, q_2)}$};

% line AF
\draw (0.75, 0) node{\scriptsize $\bullet$} node[below]{\scriptsize \textbf{A}} -- (0.75, 2);

% line EF
\draw (0.75, 2) node{\scriptsize $\bullet$} node[right] {\scriptsize \textbf{F}};
\draw[dotted](0, 2)  node[left] {\scriptsize \textbf{a}}--(0.75, 2);
\draw[thick] (0, 2) circle(0.05) ;
\draw (0.75, 2) -- (0, 3.6) ;

% line BC
\draw (0.25,1.05)  node[left]{ \scriptsize $\textbf{B}$}
--(1.9, 1.05)node{\scriptsize $\bullet$}node[right]{ \scriptsize $\textbf{C}$};
\draw[thick] (0.25, 1.05) circle(0.05);

% line CD
\draw (1.9, 1.05) -- (2.3, 1.35) node[right] {\scriptsize \textbf{D}};
\draw[thick] (2.3, 1.35) circle(0.05) ;

% point G
\draw (0.75,1.05) node{\scriptsize $\bullet$} node[right]{\scriptsize $\textbf{G}$};

%\draw[thin, dotted] (0.5, 0) node{\scriptsize $*$}  -- (0.5, 1.1);
%\draw[thin, dotted] (0, 1.1) node{\scriptsize $*$}  -- (0.5, 1.1);

%\draw[thin, dotted] (0, 0.85) node{\scriptsize $*$}  -- (1.75, 0.85);
%\draw[thin, dotted] (0.6, 0) node{\scriptsize $*$}  -- (0.6, 0.85);

%\draw[thin, dotted] (1.75, 0) node{\scriptsize $*$}--(1.75, 0.85);

\draw[thin, dotted] (2.3, 0)  node[below]{\scriptsize \textbf{h}}  -- (2.3, 1.35) ;
\draw[thick] (2.3, 0) circle(0.05) ;

%\draw[thin, dotted] (0, 1.35) node{\scriptsize $*$}  -- (2.3, 1.35) ;

\draw (7, 4) node[right] {\scriptsize $\overline{\textbf{ab}} = \mathcal{S}_{m, \infty}^{(q_1, q_2)}$, $\textbf{a} = (0, \frac 12)$, $\textbf{b} = (\frac 1d, 0)$ };

\draw (7, 3.5) node[right] {\scriptsize $\mathcal{R} (\textbf{GCDEF})$: scaling invariant class of \eqref{T4:V-divfree}};
\draw (7, 3) node[right] {\scriptsize $\mathcal{R} (\textbf{abhDE})$: scaling invariant class of \eqref{T4:V-divfree-embedding}};

\draw (7, 2.5) node[right] {\scriptsize $\textbf{A-F, g-i}$: same as in Fig. 7-$(i)$};
\draw (7, 2) node[right] {\scriptsize $\textbf{G} = (\frac{m-1}{2m}, \frac{m-1}{2m-1})$};
%\draw (6, 2.5) node[right] {\scriptsize $\textbf{C} = (\frac{1+d(m-1)}{d(2m-1)}, \frac{m-1}{2m-1})$, $\textbf{D} = (\frac{1+d(m-1)}{md}, \frac{m-1}{{\lambda_1} m})$   };
%\draw (6, 2) node[right] {\scriptsize  $\textbf{E} = (0, \frac{1}{{\lambda_1}})$, $\textbf{F} = (\frac{m-1}{2m}, \frac 12)$ };
%\draw (6, 1.5) node[right] {\scriptsize $\textbf{g} = (\frac{1+d(m-1)}{d}, 0)$};
%\draw (6, 1) node[right] {\scriptsize  $\textbf{h} = (\frac{1+d(m-1)}{md}, 0)$, $\textbf{i} = (\frac{m + d(m-1)}{md}, 0)$  };

\end{tikzpicture}
\end{center}

\item \emph{Theorem~\ref{Theorem-4} for case $m>2$, $q\geq \frac m2$}: Refer Remark~\ref{R:Theorem-4} $(ii)$ and Figure 7-$(i)$. We note that $\frac{2m}{(2m-1)(m-1)} < 2$ for $m>2$, thus it is enough to consider two cases.

% Figure 16
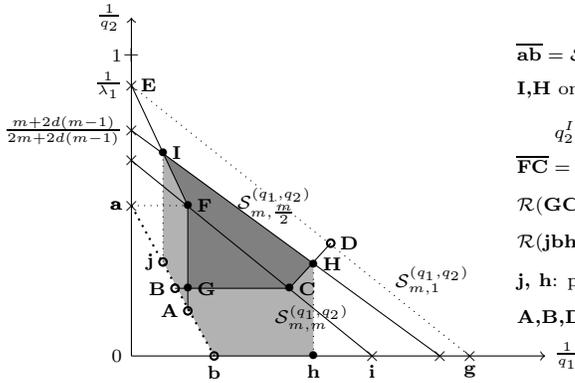
\begin{center}
\begin{tikzpicture}[domain=0:16]

%\draw[very thin, color=gray](0, 0) grid (6, 5);

\draw (-0.5, -1) node[right] { \scriptsize \textit{Figure 8-(i)}.  $m>2 $, $q\geq \frac m2$, $2< d \leq \frac{2m}{m-1}$. };

\fill[fill= lgray]
(0.42, 2.7) -- (0.42, 1.25) -- (1.1, 0) -- (2.42, 0) -- (2.42, 1.23);

\fill[fill= gray]
(0.42, 2.7) -- (0.75, 2) -- (0.75,0.9) -- (2.1, 0.9) -- (2.42, 1.23);

\draw[->] (0,0) node[left] {\scriptsize $0$}
-- (5.5,0) node[right] {\scriptsize $\frac{1}{q_1}$};
\draw[->] (0,0) -- (0,4.5) node[left] { \scriptsize $\frac{1}{q_2}$};

\draw (0,4) node{\scriptsize $+$} node[left]{\scriptsize $1$} ;
%\draw (4,0) node{\scriptsize $+$} ;
%\draw (4,0) node{\scriptsize $\bullet$} node[below]{\scriptsize \textbf{E}} ;

% S
\draw[thick, dotted] (0,2) node{\scriptsize $\times$} node[left] {\scriptsize \textbf{a}} -- (1.1, 0) node[below] {\scriptsize $\textbf{b}$};
\draw[thick] (1.1, 0) circle(0.05) ;

% S_1
\draw [dotted]
(0,3.6) node{\scriptsize $\times$}  node[left] {\scriptsize $\frac{1}{{\lambda_1}}$} node[right] {\scriptsize $\textbf{E}$}
-- (4.5, 0) node{\scriptsize $\times$} node[below]{\scriptsize \textbf{g}} ;
%\draw[thick] (0, 3.6) circle(0.05) ;
%\draw[thick] (4.5, 0) circle(0.05) ;
\draw (4, 1) node{\scriptsize $\mathcal{S}_{m,1}^{(q_1, q_2)}$};

% S_q
\draw (0,3) node {\scriptsize $\times$} node[left] {\scriptsize $\frac{m+2d(m-1)}{2m+2d(m-1)}$ }
    -- (4.1, 0) node {\scriptsize $\times$};
%\draw[thick] (4.4, 0) circle(0.05) ;
\draw (1.9, 2) node{\scriptsize $\mathcal{S}_{m,\frac m2}^{(q_1, q_2)}$};
\draw (0.42, 2.7) node{\scriptsize $\bullet$} node[right] {\scriptsize $\textbf{I}$} ;
\draw[thin, dotted] (0.42, 2.7) --(0.42, 1.25) node[left] {\scriptsize $\textbf{j}$} ;
\draw[thick] (0.42, 1.25) circle(0.05) ;
\draw (2.42, 1.23) node{\scriptsize $\bullet$} node[right] {\scriptsize $\textbf{H}$};

% S_m
\draw (0,2.6) node {\scriptsize $\times$}
    -- (3.2, 0) node {\scriptsize $\times$} node[below] {\scriptsize \textbf{i}};
\draw (2.4, 0.5) node{\scriptsize $\mathcal{S}_{m,m}^{(q_1, q_2)}$};

% line AF
\draw (0.75, 0.6) node[left]{\scriptsize \textbf{A}} -- (0.75, 2);
\draw[thick] (0.75, 0.6) circle(0.05) ;

% line EF
\draw (0.75, 2) node{\scriptsize $\bullet$} node[right] {\scriptsize \textbf{F}};
\draw[dotted](0, 2)  node[left] {\scriptsize \textbf{a}}--(0.75, 2);
\draw (0.75, 2) -- (0, 3.6) ;

% line BC
\draw (0.58,0.9) node[left]{ \scriptsize $\textbf{B}$}
--(2.1, 0.9)node{\scriptsize $\bullet$}node[right]{ \scriptsize $\textbf{C}$};
\draw[thick] (0.58, 0.9) circle(0.05) ;

% line CD
\draw (2.1, 0.9) -- (2.65, 1.5) node[right] {\scriptsize \textbf{D}};
\draw[thick] (2.65, 1.5) circle(0.05) ;

% point G
\draw (0.75,0.9) node{\scriptsize $\bullet$} node[right]{\scriptsize $\textbf{G}$};

%\draw[thin, dotted] (0.5, 0) node{\scriptsize $*$}  -- (0.5, 1.1);
%\draw[thin, dotted] (0, 1.1) node{\scriptsize $*$}  -- (0.5, 1.1);

%\draw[thin, dotted] (0, 0.85) node{\scriptsize $*$}  -- (1.75, 0.85);
%\draw[thin, dotted] (0.6, 0) node{\scriptsize $*$}  -- (0.6, 0.85);

%\draw[thin, dotted] (1.75, 0) node{\scriptsize $*$}--(1.75, 0.85);

\draw[thin, dotted] (2.42, 0) node{\scriptsize $\bullet$} node[below]{\scriptsize \textbf{h}}  -- (2.42, 1.23) ;
%\draw[thick] (2.3, 0) circle(0.05) ;

\draw (5, 4) node[right] {\scriptsize $\overline{\textbf{ab}} = \mathcal{S}_{m, \infty}^{(q_1, q_2)}$, $\textbf{a} = (0, \frac 12)$, $\textbf{b} = (\frac 1d, 0)$ };
\draw (5, 3.5) node[right] {\scriptsize $\textbf{I,H}$ on $\mathcal{S}_{m, \frac m2}^{(q_1, q_2)}$ where $\textbf{I} = (\frac{1}{q_1^{I}}, \frac{1}{q_2^I})$, $\textbf{H} = (\frac{1}{q_1^{H}}, \frac{1}{q_2^H})$};
\draw (5.5, 3) node[right] {\scriptsize $q_2^I = p_{m/2}$, $q_2^H =\frac{p_{m/2}(3m-2)}{3m-4}$ for $p_{m/2} = 1+\frac{d(m-2)+m}{2d(m-1)+m}$ };

\draw (5, 2.5) node[right] {\scriptsize $\overline{\textbf{FC}} = \mathcal{S}_{m, m}^{(q_1, q_2)}$, $\textbf{F, C, i}$: same as in Figure 7-$(i)$};

\draw (5, 2) node[right] {\scriptsize $\mathcal{R} (\textbf{GCHIF})$: scaling invariant class of \eqref{T4:V-divfree}};
\draw (5, 1.5) node[right] {\scriptsize $\mathcal{R} (\textbf{jbhHI})$: scaling invariant class of \eqref{T4:V-divfree-embedding} };
\draw (5, 1) node[right] {\scriptsize $\textbf{j, h}$: projections of $\textbf{I, H}$ onto $\mathcal{S}_{m, \infty}^{(q_1, q_2)}$ and $1/q_1$-axis, respectively };
\draw (5, 0.5) node[right] {\scriptsize $\textbf{A,B,D,E,G,g}$: same as in Fig. 7-$(ii)$};
%\draw (5, 1.5) node[right] {\scriptsize $\textbf{E} = (1,0)$, $\textbf{H} = (1, \frac{m+d(m-2)}{2m+2d(m-1)})$};
%\draw (5, 1) node[right] {\scriptsize $\textbf{F} = (1, \frac{1+d(m-2)}{2+d(m-1)})$};

\end{tikzpicture}
\end{center}

% Figrue 18
\begin{center}
\begin{tikzpicture}[domain=0:16]

% second figure

\draw (-0.5, -1) node[right] { \scriptsize \textit{Figure 8-(ii)}.  $ m > 2$, $q \geq \frac{m}{2}$, $d > \frac{2m}{m-1}$. };

\fill[fill= lgray]
(0.42, 2.7) -- (0.42, 0.5) -- (0.55, 0) -- (2.42, 0) -- (2.42, 1.23);

\fill[fill= gray]
(0.42, 2.7) -- (0.75, 2) -- (0.75,0.9) -- (2.1, 0.9) -- (2.42, 1.23);

\draw[->] (0,0) node[left] {\scriptsize $0$}
-- (5.5,0) node[right] {\scriptsize $\frac{1}{q_1}$};
\draw[->] (0,0) -- (0,4.5) node[left] { \scriptsize $\frac{1}{q_2}$};

\draw (0,4) node{\scriptsize $+$} node[left]{\scriptsize $1$} ;
%\draw (4,0) node{\scriptsize $+$} ;
%\draw (4,0) node{\scriptsize $\bullet$} node[below]{\scriptsize \textbf{E}} ;

\draw(0, 2) node{\scriptsize $\times$} node[left] {\scriptsize \textbf{a}};

% S
\draw[thick, dotted] (0,2) -- (0.55, 0)  node[below] {\scriptsize $\textbf{b}$};
\draw[thick] (0.55, 0) circle(0.05) ;

% S_1
\draw [dotted]
(0,3.6) node{\scriptsize $\times$} node[left] {\scriptsize $\frac{1}{{\lambda_1}}$} node[right] {\scriptsize $\textbf{E}$}
-- (4.5, 0) node{\scriptsize $\times$}  node[below]{\scriptsize \textbf{g}} ;
%\draw[thick] (0, 3.6) circle(0.05) ;
%\draw[thick] (4.5, 0) circle(0.05) ;
\draw (4, 1) node{\scriptsize $\mathcal{S}_{m,1}^{(q_1, q_2)}$};

% S_q
\draw (0,3) node {\scriptsize $\times$} node[left] {\scriptsize $\frac{m+2d(m-1)}{2m+2d(m-1)}$ }
    -- (4.1, 0) node {\scriptsize $\times$};
%\draw[thick] (4.4, 0) circle(0.05) ;
\draw (1.9, 2) node{\scriptsize $\mathcal{S}_{m,\frac m2}^{(q_1, q_2)}$};
\draw (0.42, 2.7) node{\scriptsize $\bullet$} node[right] {\scriptsize $\textbf{I}$} ;
\draw[thin, dotted] (0.42, 2.7) --(0.42, 0.5)  node[left] {\scriptsize $\textbf{j}$} ;
\draw[thick] (0.42, 0.5) circle(0.05) ;
\draw (2.42, 1.23) node{\scriptsize $\bullet$} node[right] {\scriptsize $\textbf{H}$};

% S_m
\draw (0,2.6) node {\scriptsize $\times$}
    -- (3.2, 0) node {\scriptsize $\times$} node[below] {\scriptsize \textbf{i}};
\draw (2.4, 0.5) node{\scriptsize $\mathcal{S}_{m,m}^{(q_1, q_2)}$};

% line AF
\draw (0.75, 0) node{\scriptsize $\bullet$} node[below]{\scriptsize \textbf{A}} -- (0.75, 2);

% line EF
\draw (0.75, 2) node{\scriptsize $\bullet$} node[right] {\scriptsize \textbf{F}};
\draw[dotted](0, 2)  node[left] {\scriptsize \textbf{a}}--(0.75, 2);
%\draw[thick] (0, 2) circle(0.05) ;
\draw (0.75, 2) -- (0, 3.6) ;

% line BC
\draw (0.28,0.9)  node[left]{ \scriptsize $\textbf{B}$}
--(2.1, 0.9)node{\scriptsize $\bullet$}node[right]{ \scriptsize $\textbf{C}$};
\draw[thick] (0.28, 0.9) circle(0.05) ;

% line CD
\draw (2.1, 0.9) -- (2.65, 1.5) node[right] {\scriptsize \textbf{D}};
\draw[thick] (2.65, 1.5) circle(0.05) ;

% point G
\draw (0.75,0.9) node{\scriptsize $\bullet$} node[right]{\scriptsize $\textbf{G}$};

%\draw[thin, dotted] (0.5, 0) node{\scriptsize $*$}  -- (0.5, 1.1);
%\draw[thin, dotted] (0, 1.1) node{\scriptsize $*$}  -- (0.5, 1.1);

%\draw[thin, dotted] (0, 0.85) node{\scriptsize $*$}  -- (1.75, 0.85);
%\draw[thin, dotted] (0.6, 0) node{\scriptsize $*$}  -- (0.6, 0.85);

%\draw[thin, dotted] (1.75, 0) node{\scriptsize $*$}--(1.75, 0.85);

\draw[thin, dotted] (2.42, 0) node{\scriptsize $\bullet$} node[below]{\scriptsize \textbf{h}}  -- (2.42, 1.23) ;
%\draw[thick] (2.3, 0) circle(0.05) ;

\draw (7, 4) node[right] {\scriptsize $\overline{\textbf{ab}} = \mathcal{S}_{m, \infty}^{(q_1, q_2)}$, $\textbf{a,b, I, H, F, C, i}$: same as in Fig. 8-$(i)$ };
%\draw (5, 3.5) node[right] {\scriptsize $\overline{\textbf{IH}} = \mathcal{S}_{m, \frac m2}^{(q_1, q_2)}$, $\textbf{I} = (*, \frac {1}{p_{m/2}})$, $\textbf{H} = (**, \frac{3m-4}{p_{m/2}(3m-2)})$ };
%\draw (5, 3) node[right] {\scriptsize $\overline{\textbf{FC}} = \mathcal{S}_{m, m}^{(q_1, q_2)}$, $\textbf{F, C, i}$: same as in Figure 7-(i)};

\draw (7, 3.5) node[right] {\scriptsize $\mathcal{R} (\textbf{GCHIF})$: scaling invariant class of \eqref{T4:V-divfree}};
\draw (7, 3) node[right] {\scriptsize $\mathcal{R} (\textbf{jbhHI})$: scaling invariant class of \eqref{T4:V-divfree-embedding} };
\draw (7, 2.5) node[right] {\scriptsize $\textbf{j, h}$: same as in Fig. 8-$(i)$ };
\draw (7, 2) node[right] {\scriptsize $\textbf{A,B,D,E,G,g}$: same as in Fig. 7-$(ii)$};
%\draw (5, 1.5) node[right] {\scriptsize $\textbf{E} = (1,0)$, $\textbf{H} = (1, \frac{m+d(m-2)}{2m+2d(m-1)})$};
%\draw (5, 1) node[right] {\scriptsize $\textbf{F} = (1, \frac{1+d(m-2)}{2+d(m-1)})$};

\end{tikzpicture}
\end{center}

\item \emph{Theorem~\ref{Theorem-5a} and \ref{Theorem-5} for case $m>2$, $q> \frac m2$}: Refer Remark~\ref{R:Theoerm-5} and Figure 9-$(i)$.

% Figure 19
\begin{center}
\begin{tikzpicture}[domain=0:16]

\draw (0, -1) node[right] {\scriptsize \textit{Figure 9-(ii)}.  $m> 2$, $q \geq \frac m2$.};

\fill[fill= lgray]
(1.27, 0.8) -- (1.6,0) -- (2.4, 0.8) ;

\fill[fill= gray]
(0.8, 2) -- (1.27, 0.8)-- (2.4,0.8) -- (2.87,1.33) --(0.8, 3.25);

\draw[->] (0,0) -- (6,0) node[right] {\scriptsize $\frac{1}{\tilde{q}_1}$};
\draw[->] (0, 0) -- (0, 5) node[left] {\scriptsize $\frac{1}{\tilde{q}_2}$};

\draw[dotted] (0, 4) -- ( 5.5 , 0) node{\scriptsize $\times$} node[below] {\scriptsize $\frac{2+d(m-1)}{d}$};
\draw[thick] (0, 4) circle(0.05) node[left]{\scriptsize $\textbf{a}$};

\draw (5, 1) node {\scriptsize $\tilde{\mathcal{S}}_{m,1}^{(\tilde{q}_1, \tilde{q}_2)}$};

\draw (0, 4) -- (3, 0) node{\scriptsize $\times$};
\draw (3, 0.5) node {\scriptsize $\tilde{\mathcal{S}}_{m,m}^{(\tilde{q}_1, \tilde{q}_2)}$};
\draw (2.5, 2.2) node {\scriptsize $\tilde{\mathcal{S}}_{m,\frac m2}^{(\tilde{q}_1, \tilde{q}_2)}$};

\draw[thick, dotted] (0,4)
-- (1.6, 0)  node[below] {\scriptsize $\textbf{b}$} ;
\draw[thick] (1.6, 0) circle(0.05);

\draw[very thin]
(0.8, 0) node{\scriptsize$*$} node[below] {\scriptsize $\frac 1d$}
-- (0.8, 4) node[above]{\scriptsize $\tilde{q}_1 = d$} ;
\draw[thick] (0.8, 2) circle(0.05) node[left]{\scriptsize \textbf{A}};
\draw[thick] (0.8, 2.94) circle(0.05) node[left]{\scriptsize \textbf{f}} ;
\draw[thick] (0.8, 3.43) circle(0.05) node[right]{\scriptsize \textbf{E}} ;

%\draw[thin, dotted] (0,0) -- (2.3,2.3);
\draw (1.6, 0) -- (3.2,1.68) ;
\draw[thick] (3.2, 1.68) circle(0.05)  node[right]{\scriptsize \textbf{D}};

\draw (1.27, 0.8) node[left] {\scriptsize \textbf{B}} -- (2.4,0.8) node{\scriptsize $\bullet$} node[right]{\scriptsize \textbf{C}};
\draw[thick] (1.27, 0.8) circle(0.05);

\draw[thin, dotted] (0, 2) node{\scriptsize $+$} node[left] {\scriptsize $\frac 12$}
-- (0.8, 2);
\draw[thin, dotted] (0, 3.25) node{\scriptsize$*$} -- (0.8, 3.25);
%\draw[thin, dotted] (2.4, 0) node{\scriptsize$*$} -- (2.4, 0.8);
\draw[thin, dotted] (0, 3.25) node{\scriptsize$*$} node[left] {\scriptsize $\frac{m+2d(m-1)}{2(m+d(m-1))}$}  -- (0.8, 3.25);

\draw[very thin] (4, 0) node{\scriptsize$+$}
-- (4, 2) node[above] {\scriptsize{$\tilde{q}_1 = 1$}};
\draw (4, 1.1) circle(0.05) node[right] {\scriptsize \textbf{g}} ;

%\draw[thin, dotted] (0, 0.8) node{\scriptsize$*$}  -- (2.4, 0.8);
\draw[thin, dotted] (0, 1.33) node{\scriptsize$*$}  -- (2.87,1.33);
\draw[thin, dotted] (2.87,0) node{\scriptsize$*$}  -- (2.87,1.33);

%\draw[thin, dotted] (0, 1.1) node{\scriptsize$*$} node[left] {\scriptsize $\frac{2+d(m-2)}{2+d(m-1)}$} -- (4, 1.1);

% S_{m, m/2}
\draw[thick, dotted] (0,4) -- (4.3, 0) node{\scriptsize $\times$} node[below] {\scriptsize $\textbf{h}$} ;
\draw[thick] (2.87,1.33) circle(0.05) node[right] {\scriptsize $\textbf{H}$};
\draw[thick] (0.8, 3.25) circle(0.05) node[left] {\scriptsize $\textbf{I}$};

\draw (7, 4) node[right] {\scriptsize $\textbf{a,b, A-E, f}$: same as in Fig. 9-$(i)$};
%\draw (7, 4) node[right] {\scriptsize $\textbf{a} = (0, 1)$, $\textbf{b} = (\frac 2d, 0)$ };
\draw (7, 3.5) node[right] {\scriptsize $\mathcal{R} (\textbf{AbHI})$: scaling invariant class of \eqref{T5a:V-tilde}. };
\draw (7, 3) node[right] {\scriptsize $\mathcal{R} (\textbf{ABCHI})$: scaling invariant class of \eqref{T5:V-tilde}. };
\draw(7,2.5) node[right] {\scriptsize $\mathcal{R} (\textbf{bCB})$: $(\frac{1}{\tilde{q}_1}, \frac{1}{\tilde{q}_2})$: scaling invariant class of \eqref{T5:V-tilde-embedding}. };

\draw (7, 2) node[right] {\scriptsize $\overline{\textbf{ah}}= \tilde{\mathcal{S}}_{m, m/2}^{(\tilde{q}_1, \tilde{q}_2)}$, $h=(\frac{2m+2d(m-1)}{md},0)$  };
\draw (7, 1.5) node[right] {\scriptsize $\textbf{H} = (\frac{2(m+4)(m+d(m-1))}{m(3m+2)}, \frac{2(m-1)}{3m+2})$};
%\draw (7, 1) node[right] {\scriptsize $\textbf{I} = (\frac 1d, \frac{m+2d(m-1)}{2(m+d(m-1))})$};
%\draw (7, 1) node[right] {\scriptsize $\textbf{D} = (\frac{2+d(m-1)}{md}, \frac{m-1}{m})$, $\textbf{E} = (\frac 1d, \frac{1}{{\lambda_1}})$   };
%\draw (7, 0.5) node[right] {\scriptsize $\textbf{f} = (\frac 1d, \frac{m+d(m-1)}{2m+d(m-1)})$, $\textbf{g}= (1, \frac{2+d(m-2)}{2+d(m-1)})$  };

\end{tikzpicture}
\end{center}

\end{itemize}
%%%%%%%%%%%%%%%%%%%%%%%%%%%%%%%%%%%%%%%%%%%%%

\subsection*{Embedding}

The following figures illustrate strategies of applying embedding arguments in the temporal varaiable in the second part os Theorem~\ref{Theorem-2-b}, \ref{Theorem-4}, and \ref{Theorem-5}.
Unfortunately, embedding arguments in spatial variables do not work because $\bbr^d$ is unbounded. If spatial embedding arguments are applicable, then, for example in Figure 4-(e), it gives a way to include the region $\mathcal{R}(\textbf{aAF})$ by searching $(q_1^\ast, q_2)$ in $\mathcal{R}(\textbf{ABCDEF})$ by decreasing $q_1^\ast < q_1$ for some $q^\ast \in (1, q)$.
We are preparing a parallel paper constructing the existence results of \eqref{E:Main} in $\Omega \times (0, T]$ for $\Omega$ bounded in $\bbr^{d}$.

\begin{itemize}
\item  \emph{Theorem~\ref{Theorem-2-b} $(ii)$ for case $1<m\leq 2$, $q>1$}:
Starting a pair $(q_1, q_2)$ in $\mathcal{R}(\textbf{bCB})$, one can  decrease $q_2$ to $q_{2}^\ast$ until $(q_1, q_{2}^{\ast})$ hits the line $\overline{\textbf{BC}}$ or any pair belonging $\mathcal{R}(\textbf{ABCDEF})$. Then there exists  $q^\ast \in (1, q)$ in which $(q_1, q_{2}^{\ast})$ lies on $\mathcal{S}_{m,q^\ast}^{(q_1, q_2^\ast)}$. Then we apply the existence results in Theorem~\ref{Theorem-2-b} $(i)$ for $1 < p \leq {\lambda_{q^\ast}}$.  When $ \max\{2,\frac{2m}{(2m-1)(m-1)}\} < d \leq \frac{2m}{m-1}$ or $d > \frac{2m}{m-1}$, we repeat the same strategy with Figure 4-$(ii)$, $(iii)$, respectively.

% Figure 20
\begin{center}
\begin{tikzpicture}[domain=0:16]

%\draw[very thin, color=gray](0, 0) grid (6, 5);

\draw (0, -1) node[right] { \scriptsize \textit{Figure 4-(e)}. $1<m\leq 2$, $q>1$, $2<d \leq \max\{2, \frac{2m}{(2m-1)(m-1)}\}$.};

\fill[fill= lgray]
(0.8, 0.63) -- (1.2,0) -- (1.85, 0.63);

\fill[fill= gray]
(0.4, 1.3)--(0.8, 0.63) -- (1.85,0.63) -- (2.35, 1.1) -- (1.45, 2)--(0.4, 2);

\draw[->] (0,0) %node[left] {\scriptsize $0$}
-- (5,0) node[right] {\scriptsize $\frac{1}{q_1}$};
\draw[->] (0,0) -- (0,4.5) node[left] { \scriptsize $\frac{1}{q_2}$};

\draw (0,4) node{\scriptsize $+$} node[left]{\scriptsize $1$} ;
\draw (4,0) node{\scriptsize $+$} ;

\draw[very thin]
(0, 2) -- (1.45, 2) ;
\draw[thick](1.45, 2) circle(0.05) node[right] {\scriptsize \textbf{E}};

\draw[thin] (0, 2) -- (1.45, 2);

\draw[thin] (1.2, 0) -- (2.35, 1.1);

\draw[thick] (2.35, 1.1) circle(0.05) node[right] {\scriptsize \textbf{D}};
\draw[very thin]
 (0.8, 0.63) node[left]{\scriptsize \textbf{B}}
-- (1.85, 0.63) node{\scriptsize $\bullet$} node[right]{\scriptsize \textbf{C}} ;
\draw[thick] (0.8,0.63) circle(0.05);

\draw[very thin]
(0.4, 1.3)  node[below] {\scriptsize \textbf{A}}
 --(0.4, 2) node{\scriptsize $\bullet$} node[above] {\scriptsize \textbf{F}};
 \draw[thick] (0.4,1.3) circle(0.05);

%S
\draw[thick, dotted] (0,2)  node[left] {\scriptsize $\textbf{a}$}
 -- (1.2, 0)  node[below] {\scriptsize $\textbf{b}$};
\draw[thick] (0,2) circle(0.05);
\draw[thick] (1.2,0) circle(0.05);

% S_1
\draw[thick, dotted]
(0,3.4) node {\scriptsize $\times$} node[left] {\scriptsize $\frac{1}{{\lambda_1}}$}
-- (3.5, 0) node {\scriptsize $\times$} node[below] {\scriptsize $\frac{1+d(m-1)}{d}$};
\draw (3.9, 0.3) node{\scriptsize $\mathcal{S}_{m,1}^{(q_1, q_2)}$};

\draw[thick, ->] (1.35,0.3)  -- (1.35, 0.5) ;
\draw[thick] (1.35,0.16) node{\scriptsize $\bullet$} node[right]{\scriptsize \textbf{g}};
\draw[dotted] (0, 0.15) node{\scriptsize $*$} node[left]{\scriptsize $\frac{1}{q_2}$} -- (1.35, 0.15);
\draw[thick, dashed] (0, 2.15) -- (1.45, 0);
%\draw[dred] (0.8,2.5) node{\scriptsize $\mathcal{S}_{m, q}^{(q_1, q_2)}$} ;

\draw[thick] (1.35,0.63) node{$\star$} node[above] {\scriptsize \textbf{h}} ;
%\draw[blue] (2.3,0.3) node{\scriptsize $\mathcal{S}_{m, q^{\ast}}^{(q_1, q^{\ast}_2)}$} ;
\draw[dotted] (0, 0.63) node{\scriptsize $*$} node[left]{\scriptsize $\frac{1}{q^{\ast}_2}$} -- (1.35, 0.63);
\draw[thick] (0,2.3) -- (1.85, 0) ;

\draw (7, 4) node[right] {\scriptsize $\overline{\textbf{ab}}, \textbf{A}-\textbf{F}$: same as in Fig. 4-$(i)$ };
\draw (7, 3.5) node[right] {\scriptsize $\mathcal{R} (\textbf{ABCDEF})$: scaling invariant class of  \eqref{T2:V}};
\draw (7, 3) node[right] {\scriptsize $\mathcal{R} (\textbf{bCB})$: scaling invariant class of \eqref{T2:V-embedding} };
\draw (7,2.5) node[right] {$\bullet$ \scriptsize $ g = (\frac{1}{q_1}, \frac{1}{q_2})$ on $\mathcal{S}_{m, q}^{(q_1, q_2)}$ (dashed line)} ;
%\draw[blue] (7,0.8) node[right] {\scriptsize blue line:  $\mathcal{S}_{m, q^\ast}^{(q_1, q_2^\ast)}$ where $q^\ast < q$} ;
\draw (7,2) node[right] {$\star$ \scriptsize  $ h = (\frac{1}{q_1}, \frac{1}{q^{\ast}_2})$ on $\mathcal{S}_{m, q^\ast}^{(q_1, q_2^\ast)}$ (thick line)} ;
\draw (9,1.5) node[right] { \scriptsize  for $q^{\ast}_2 < q_2$ and $q^\ast < q$} ;
%\draw[blue] (6,1.5) node[right] { \scriptsize for $q^{\ast}_2 < q_2$, $q^\ast < q$} ;
%\draw[blue] (7,0.2) node[right] {\scriptsize $b_1 = (0, \frac{1+q^\ast_{m,d}}{2+q^\ast_{m,d}})$, $b_2 = (\frac{1+q^\ast_{m,d}}{d}$, 0)} ;

\end{tikzpicture}
\end{center}

\item \emph{Theorem~\ref{Theorem-2-b} $(ii)$ for case $m > 2$, $q\geq m-1$}: Refer Figure 5-$(i)$. The same strategy works to Figure 5-$(ii)$.

% Figure 21
\begin{center}
\begin{tikzpicture}[domain=0:16]

%\draw[very thin, color=gray](0, 0) grid (6, 5);

\draw (0, -1) node[right] { \scriptsize \textit{Figure 5-(e)}. $m>2$, $q\geq m-1$, $ 2 < d \leq \frac{2m}{m-1}$.};

\fill[fill= lgray]
(0.7, 1.15)-- (0.7, 0.6) -- (1, 0) -- (1.58,1.15);

\fill[fill= gray]
(0.7, 2) -- (0.7,1.15) -- (1.58, 1.15) --(2,2);

\draw[->] (0,0) node[left] {\scriptsize $0$} -- (5,0) node[right] {\scriptsize $\frac{1}{q_1}$};
\draw[->] (0,0) -- (0,4.5) node[left] { \scriptsize $\frac{1}{q_2}$};

\draw (0,4) node{\scriptsize $+$} node[left] {\scriptsize $1$} ;
\draw (4,0) node{\scriptsize $+$};

\draw
(0, 2) node[left] {\scriptsize \textbf{a}} -- (2, 2) ;

\draw[thick, dotted] (0,2)  node[left] {\scriptsize $\textbf{a}$}
 -- (1, 0)   node[below] {\scriptsize $\textbf{b}$};
\draw[thick] (0,2) circle(0.05);
\draw[thick] (1,0) circle(0.05);

\draw(2, 2) node{\scriptsize $\bullet$} node[above] {\scriptsize \textbf{D}};

\draw(0.7,2) node{\scriptsize $\bullet$} node[above]{\scriptsize \textbf{E}}
--(0.7, 0.6) node[left]{\scriptsize \textbf{A}} ;
\draw[thick] (0.7,0.6) circle(0.05);

\draw (0.43, 1.15) node[left]{\scriptsize \textbf{B}}
--(1.58, 1.15)node{\scriptsize $\bullet$}node[right]{\scriptsize \textbf{C}};
\draw[thick] (0.43, 1.15) circle(0.05);

\draw(0.7, 1.15) node{\scriptsize $\bullet$} node[above]{\scriptsize \textbf{G}};

\draw (1, 0)  -- (2, 2);

%\draw[very thin] (0, 2)-- (1.65, 1.13);
\draw (1.15,0.65) node{\scriptsize $\bullet$} node[left]{\scriptsize $\textbf{g}$};
\draw[dotted] (0, 0.65) node{\scriptsize $*$} node[left] {\scriptsize $\frac{1}{q_2}$}
-- (1.15, 0.65);
\draw[thick, ->] (1.15,0.75)-- (1.15, 1.05)  ;
\draw[thick, dashed] (0, 2.3)--(1.6, 0);

\draw (1.15, 1.15) node{ $\star$} node[above]{\scriptsize $\textbf{h}$ };
\draw[dotted] (0, 1.15) node{\scriptsize $*$} node[left] {\scriptsize $\frac{1}{q^{\ast}_2}$} -- (1.15, 1.15);
\draw[thick] (0, 2.6) -- (2.1, 0) ;
\draw (2.1, 0.65) node{\scriptsize { $\mathcal{S}_{m, q^\ast}^{(q_1,q_2)}$}} ;

\draw (0,3.5) node {\scriptsize $\times$} node[left] {\scriptsize $\frac{1+d}{2+d}$}
-- (4.6, 0) node {\scriptsize $\times$} node[below] {\scriptsize $\frac{1+d}{d}$} ;
\draw (4.5, 0.65) node {\scriptsize $\mathcal{S}_{m,m-1}^{(q_1, q_2)}$} ;

%\draw (0,2.7) node {\scriptsize $\times$} node[left] {\scriptsize $\frac{m+d(m-1)}{md}$}
% -- (2.8, 0) node {\scriptsize $\times$} node[below] {\scriptsize $\frac{m+d(m-1)}{2m+d(m-1)}$};
%\draw (3.1, 0.3) node{\scriptsize $\mathcal{S}_{m,m}^{(q_1, q_2)}$};

\draw (7, 4) node[right] {\scriptsize $\textbf{a, b, A-E}$: same as in Fig. 5-$(i)$ };
%\draw (6, 4.5) node[right] {\scriptsize $\overline{\textbf{ab}}, \textbf{A}-\textbf{E}$: same as in Fig. 4-(b-i)};
\draw (7, 3.5) node[right] {\scriptsize $\mathcal{R}(\textbf{GCDE})$: scaling invariant class of \eqref{T2:V} };
\draw (7, 3) node[right] {\scriptsize $\mathcal{R}(\textbf{GAbC}))$: scaling invariant class of \eqref{T2:V-embedding} };
\draw (7,2.5) node[right] {$\bullet$ \scriptsize $ g = (\frac{1}{q_1}, \frac{1}{q_2})$ on $\mathcal{S}_{m, q}^{(q_1, q_2)}$ (dashed line)} ;
%\draw[blue] (7,0.8) node[right] {\scriptsize blue line:  $\mathcal{S}_{m, q^\ast}^{(q_1, q_2^\ast)}$ where $q^\ast < q$} ;
\draw (7,2) node[right] {$\star$ \scriptsize  $ h = (\frac{1}{q_1}, \frac{1}{q^{\ast}_2})$ on $\mathcal{S}_{m, q^\ast}^{(q_1, q_2^\ast)}$ (thick line)} ;
\draw (9,1.5) node[right] {\scriptsize for $q^{\ast}_2 < q_2$, $q^\ast < q$} ;

\end{tikzpicture}
\end{center}

%%%%%%%

\item \emph{Theorem~\ref{Theorem-4} $(ii)$ for case $1<m\leq 2$, $q>1$}: With Figure 7-(i), we explain strategy of searching $q_2^\ast$ and $q^\ast$. There are two regions $\mathcal{R} ( \textbf{aAFE})$ if $ \frac{2m}{m-1}< q_1 < \infty$ and $\mathcal{R} ( \textbf{BbhDC})$ if $\frac{md}{1+d(m-1)} < q_1 < d(2m-1)$, in which we find $q_2^\ast$ such that $(q_1, q_2^\ast)$ belongs to $\mathcal{R} (\textbf{ABCDEF})$. Then we are able to search matching $q^\ast \in (1, q]$ and able to apply Theorem~\ref{Theorem-4} $(i)$.
%However, in the region $\mathcal{R} ( \textbf{Dhg})$ for $q_1 \leq \frac{md}{1+d(m-1)}$, we cannot find $q^\ast_2$ to apply results in Theorem~\ref{Theorem-4} $(i)$.
%apply our strategy becuase it relies on embedding in temporal variable.
    Also, we can apply the same strategy with Figure 7-$(ii)$, $(iii)$.

% Figure 21
\begin{center}
\begin{tikzpicture}[domain=0:16]

%\draw[very thin, color=gray](0, 0) grid (6, 5);

\draw (0, -1) node[right] { \scriptsize \textit{Figure 7-(e)}.  $1<m\leq 2$, $q>1$, $2 < d \leq \max\{2,\frac{2m}{(2m-1)(m-1)}\}$.};

\fill[fill= lgray]
(0, 2) -- (1.1,0) -- (2.6, 0) -- (2.6, 1.1) --(0,3.6)--(0, 2);

\fill[fill= gray]
(0.4, 1.25) -- (0.6,0.85) -- (1.42, 0.85) -- (2.6, 1.1) -- (0, 3.6)--(0.4,2);

%\fill[fill= lgray]
%(0, 2) -- (1.46,0.64) -- (1.85, 1.6) -- (1.45, 2);

\draw[->] (0,0) node[left] {\scriptsize $0$}
-- (5,0) node[right] {\scriptsize $\frac{1}{q_1}$};
\draw[->] (0,0) -- (0,4.5) node[left] { \scriptsize $\frac{1}{q_2}$};

\draw (0,4) node{\scriptsize $+$} node[left]{\scriptsize $1$} ;
\draw (4,0) node{\scriptsize $+$} ;

\draw[dotted]
(0, 2)  node[left] {\scriptsize \textbf{a}}--(0.4, 2);

%\draw[very thin]
%(0, 0) -- (1.8, 1.8);

\draw[thin]
(1.42, 0.85) -- (2.6, 1.1) node[right] {\scriptsize \textbf{D}};
\draw[thick] (2.6, 1.1) circle(0.05) ;
\draw[thick] (2.6, 0) circle(0.05) node[below]{\scriptsize \textbf{h}};

%\draw (0, 2.7) -- (2.6, 0);

\draw[dotted] (2.6, 0) -- (2.6, 1.9) node[above]{\scriptsize $q_1 = \frac{md}{1+d(m-1)}$};

\draw[thin]
(0.4, 2) -- (0, 3.6) ;

%\draw
%(0, 2)  -- (1.42, 0.85) ;

\draw[thin] (0.6,0.85) node[left]{ \scriptsize $\textbf{B}$}
--(1.42, 0.85)node{\scriptsize $\bullet$}node[right]{ \scriptsize $\textbf{C}$};
\draw[thick] (0.6, 0.85) circle(0.05);

% S
\draw[thick, dotted] (0,2) -- (1.1, 0) node[below] {\scriptsize $\textbf{b}$};
\draw[thick] (0, 2) circle(0.05) ;
\draw[thick] (1.1, 0) circle(0.05);

% S_1
\draw [thick, dotted]
(0,3.6)  node[left] {\scriptsize $\frac{1}{{\lambda_1}}$} node[right] {\scriptsize \textbf{E}}
-- (3.7, 0) node[below]{\scriptsize \textbf{g}} ;
\draw[thick] (0, 3.6) circle(0.05) ;
\draw[thick] (3.7, 0) circle(0.05) ;
\draw (4, 0.5) node{\scriptsize $\mathcal{S}_{m,1}^{(q_1, q_2)}$};

% S_m
%\draw (0,2.45) node {\scriptsize $\times$} node[left] {\scriptsize $\frac{m+d(m-1)}{2m+d(m-1)}$}
 %   -- (2.2, 0) node {\scriptsize $\times$} node[below] {\scriptsize $\frac{m+d(m-1)}{md}$};
%\draw (2.6, 0.3) node{\scriptsize $\mathcal{S}_{m,m}^{(q_1, q_2)}$};

\draw[thin] (0.4, 1.25)  node[left]{\scriptsize \textbf{A}} -- (0.4, 2);
\draw[thick] (0.4, 1.25) circle(0.05);

\draw[thin] (0.4, 2) node{\scriptsize $\bullet$} node[right] {\scriptsize \textbf{F}};

\draw (0.15,2.4) node{\scriptsize $\bullet$} node[below]{\scriptsize $\textbf{i}$};
\draw[thick, ->] (0.15,2.55)-- (0.15, 2.8)  ;
%\draw[dred] (0,2.5) -- (2.3,0);
\draw[dotted] (0,2.4) node{\scriptsize $*$}node[left]{\scriptsize $\frac{1}{q_2}$} -- (0.15,2.4);
\draw[thick, dashed] (0, 2.6)--(2.4, 0);
\draw[thick] (0, 3.1)--(3, 0);

\draw (0.15, 2.95) node{$\star$} node[right]{\scriptsize $\textbf{j}$};
%\draw[blue] (0,3.1) -- (3.2,0);
\draw[dotted] (0,2.95) node{\scriptsize $*$}node[left]{\scriptsize $\frac{1}{q_2^\ast}$} -- (0.15,2.95);

\draw (2.05,0.38) node{\scriptsize $\bullet$} node[right]{\scriptsize $\textbf{k}$};
\draw[thick, ->] (2.05,0.5)-- (2.05, 0.85)  ;
\draw[dotted] (0,0.38) node{\scriptsize $*$}node[left]{\scriptsize $\frac{1}{q_2}$} -- (2.05,0.38);

\draw (2.05, 0.97) node{$\star$} node[above]{ \scriptsize $\textbf{l}$};
\draw[dotted] (0,0.97) node{\scriptsize $*$}node[left]{\scriptsize $\frac{1}{q_2^\ast}$} -- (2.05,0.97);

\draw (7, 4) node[right] {\scriptsize $\overline{\textbf{ab}} = \mathcal{S}_{m, \infty}^{(q_1, q_2)}$, $\textbf{a} = (0, \frac 12)$, $\textbf{b} = (\frac 1d, 0)$ };
\draw (7, 3.5) node[right] {\scriptsize $\mathcal{R} (\textbf{ABCDEF})$: scaling invariant class of \eqref{T4:V-divfree} };

\draw (7, 3) node[right] {\scriptsize $\textbf{A}-\textbf{g}$: same as in Figure 7-$(i)$.  };
\draw (7, 2.5) node[right] {\scriptsize $\textbf{h}= (\frac{1+d(m-1)}{md}, 0)$ on $\mathcal{S}_{m, \frac{md}{d-1}}^{(q_1,q_2)}$};

%\draw (7, 2.5) node[right] {\scriptsize $\mathcal{R} (\textbf{abhDE} )$:  $(\frac{1}{q_1},\frac{1}{q_2})$ satisfying \eqref{T4:V-divfree-embedding}.} ;

%\draw[dred] (7,2) node[right] {\scriptsize red lines:  $\mathcal{S}_{m, q}^{(q_1, q_2)}$} ;
%\draw[dred] (7,1.5) node[right] {\scriptsize $r_1 = (0, \frac{1+q_{m,d}}{2+q_{m,d}})$, $r_2 = (\frac{1+q_{m,d}}{d}$, 0)} ;
\draw (7,2) node[right] {$\bullet$ \scriptsize $ i, k = (\frac{1}{q_1}, \frac{1}{q_2})$ on $\mathcal{S}_{m, q}^{(q_1, q_2)}$ (dashed line) } ;

%\draw[blue] (7,0.8) node[right] {\scriptsize blue lines:  $\mathcal{S}_{m, q^\ast}^{(q_1, q_2^\ast)}$ where $q^\ast < q$} ;
\draw(7,1.5) node[right] {$\star$ \scriptsize $ j, l = (\frac{1}{q_1}, \frac{1}{q^{\ast}_2})$ on $\mathcal{S}_{m, q^\ast}^{(q_1, q_2^\ast)}$ (thick line)} ;
\draw(9,1) node[right] {\scriptsize  for  $q^{\ast}_2   < q_2$ and $q^\ast < q$} ;

\end{tikzpicture}
\end{center}

%%%%%%%%%%%%%%%%%%%%

\item \emph{Theorem~\ref{Theorem-4} $(ii)$ for case $m > 2$, $q\geq \frac m2$}: With Figure 8-(i), we explain strategy of searching $q_2^\ast$ and $q^\ast$. There are two regions $\mathcal{R} ( \textbf{jAFI})$ and $\mathcal{R} ( \textbf{GAbhHCG})$, in which we find $q_2^\ast$ such that $(q_1, q_2^\ast)$ belongs to $\mathcal{R} (\textbf{FGCHI})$ to apply Theorem~\ref{Theorem-4} $(i)$.
%However, in the region $\mathcal{R} ( \textbf{Hhg})$, we cannot apply our strategy becuase it relies on embedding in temporal variable.
Also, we can apply the same strategy with Figure 8-$(ii)$.

% Figure 23
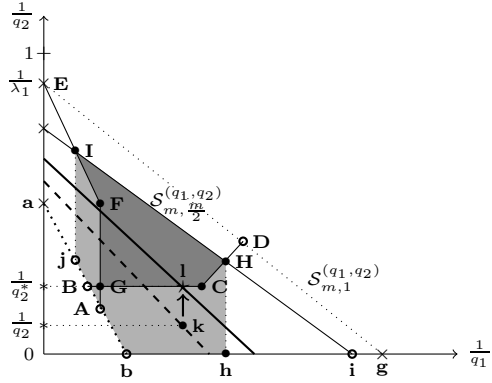
\begin{center}
\begin{tikzpicture}[domain=0:16]

%\draw[very thin, color=gray](0, 0) grid (6, 5);

\draw (0, -1) node[right] { \scriptsize \textit{Figure 8-(e)}.  $m>2 $, $q\geq \frac m2$, $2< d \leq \frac{2m}{m-1}$. };

\fill[fill= lgray]
(0.42, 2.7) -- (0.42, 1.25) -- (1.1, 0) -- (2.42, 0) -- (2.42, 1.23);

\fill[fill= gray]
(0.42, 2.7) -- (0.75, 2) -- (0.75,0.9) -- (2.1, 0.9) -- (2.42, 1.23);

\draw[->] (0,0) node[left] {\scriptsize $0$}
-- (5.5,0) node[right] {\scriptsize $\frac{1}{q_1}$};
\draw[->] (0,0) -- (0,4.5) node[left] { \scriptsize $\frac{1}{q_2}$};

\draw (0,4) node{\scriptsize $+$} node[left]{\scriptsize $1$} ;

% S
\draw[thick, dotted] (0,2) node{\scriptsize $\times$} node[left]{\scriptsize $\textbf{a}$}  -- (1.1, 0) node[below] {\scriptsize $\textbf{b}$};
\draw[thick] (1.1, 0) circle(0.05) ;

% S_1
\draw [dotted]
(0,3.6) node{\scriptsize $\times$} node[left] {\scriptsize $\frac{1}{{\lambda_1}}$} node[right] {\scriptsize $\textbf{E}$}
-- (4.5, 0) node{\scriptsize $\times$} node[below]{\scriptsize \textbf{g}} ;
%\draw[thick] (0, 3.6) circle(0.05) ;
%\draw[thick] (4.5, 0) circle(0.05) ;
\draw (4, 1) node{\scriptsize $\mathcal{S}_{m,1}^{(q_1, q_2)}$};

% S_q
\draw (0,3) node {\scriptsize $\times$}
    -- (4.1, 0) node[below]{\scriptsize \textbf{i}};
\draw[thick] (4.1, 0) circle(0.05) ;
\draw (1.9, 2) node{\scriptsize $\mathcal{S}_{m,\frac m2}^{(q_1, q_2)}$};
\draw (0.42, 2.7) node{\scriptsize $\bullet$} node[right] {\scriptsize $\textbf{I}$} ;
\draw[thin, dotted] (0.42, 2.7) --(0.42, 1.25) node[left] {\scriptsize $\textbf{j}$} ;
\draw[thick] (0.42, 1.25) circle(0.05);
\draw (2.42, 1.23) node{\scriptsize $\bullet$} node[right] {\scriptsize $\textbf{H}$};

% S_m
%\draw (0,2.6) node {\scriptsize $\times$} node[left] {\scriptsize $\frac{m+d(m-1)}{2m+d(m-1)}$}
%\draw (3.2, 0) node {\scriptsize $\times$} node[below] {\scriptsize \textbf{i}};
%\draw (2.4, 0.5) node{\scriptsize $\mathcal{S}_{m,m}^{(q_1, q_2)}$};

% line AF
\draw (0.75, 0.6) node[left]{\scriptsize \textbf{A}} -- (0.75, 2);
\draw[thick] (0.75, 0.6) circle(0.05);

% line EF
\draw (0.75, 2) node{\scriptsize $\bullet$} node[right] {\scriptsize \textbf{F}};
%\draw (0, 2)  node[left] {\scriptsize \textbf{a}};
%\draw[thick] (0, 2) circle(0.05) ;
\draw (0.75, 2) -- (0, 3.6) ;

% line BC
\draw (0.58,0.9) node[left]{ \scriptsize $\textbf{B}$}
--(2.1, 0.9)node{\scriptsize $\bullet$}node[right]{ \scriptsize $\textbf{C}$};
\draw[thick] (0.58, 0.9) circle(0.05);

% line CD
\draw (2.1, 0.9) -- (2.65, 1.5) node[right] {\scriptsize \textbf{D}};
\draw[thick] (2.65, 1.5) circle(0.05) ;

% point G
\draw (0.75,0.9) node{\scriptsize $\bullet$} node[right]{\scriptsize $\textbf{G}$};

%\draw[thin, dotted] (0.5, 0) node{\scriptsize $*$}  -- (0.5, 1.1);
%\draw[thin, dotted] (0, 1.1) node{\scriptsize $*$}  -- (0.5, 1.1);

%\draw[thin, dotted] (0, 0.85) node{\scriptsize $*$}  -- (1.75, 0.85);
%\draw[thin, dotted] (0.6, 0) node{\scriptsize $*$}  -- (0.6, 0.85);

%\draw[thin, dotted] (1.75, 0) node{\scriptsize $*$}--(1.75, 0.85);

\draw[thin, dotted] (2.42, 0) node{\scriptsize $\bullet$} node[below]{\scriptsize \textbf{h}}  -- (2.42, 1.23) ;
%\draw[thick] (2.3, 0) circle(0.05) ;

%\draw (0.6,1.7) node{\scriptsize $\bullet$} node[left]{\scriptsize $\textbf{i}$};
%\draw[thick, ->] (0.6,1.8)-- (0.6, 2.2)  ;
%\draw[dred] (0,2.5) -- (2.3,0);
%\draw[dotted] (0,2.4) node{\scriptsize $*$}node[left]{\scriptsize $\frac{1}{q_2}$} -- (0.15,2.4);
\draw[thick, dashed] (0, 2.3)--(2.2, 0);
\draw[thick] (0, 2.6)--(2.8, 0);

%\draw (0.6, 2.35) node{$\star$} node[left]{\scriptsize $\textbf{j}$};
%\draw[blue] (0,3.1) -- (3.2,0);
%\draw[dotted] (0,2.95) node{\scriptsize $*$}node[left]{\scriptsize $\frac{1}{q_2^\ast}$} -- (0.15,2.95);

\draw (1.85,0.38) node{\scriptsize $\bullet$} node[right]{\scriptsize $\textbf{k}$};
\draw[thick, ->] (1.85,0.5)-- (1.85, 0.8)  ;
\draw[dotted] (0,0.38) node{\scriptsize $*$}node[left]{\scriptsize $\frac{1}{q_2}$} -- (1.85,0.38);

\draw (1.85, 0.9) node{$\star$} node[above]{ \scriptsize $\textbf{l}$};
\draw[dotted] (0,0.9) node{\scriptsize $*$}node[left]{\scriptsize $\frac{1}{q_2^\ast}$} -- (1.85,0.9);

\draw (7, 4) node[right] {\scriptsize $\textbf{a,b,A-I,h-j}$ :same as in Fig. 8-$(i)$ };

\draw (7, 3.5) node[right] {\scriptsize $\mathcal{R} (\textbf{jbhHI})$: scaling invariant class of \eqref{T4:V-divfree} };

\draw (7,3) node[right] {$\bullet$ \scriptsize $ k = (\frac{1}{q_1}, \frac{1}{q_2})$ on $\mathcal{S}_{m, q}^{(q_1, q_2)}$ (dashed line) } ;

%\draw[blue] (7,0.8) node[right] {\scriptsize blue lines:  $\mathcal{S}_{m, q^\ast}^{(q_1, q_2^\ast)}$ where $q^\ast < q$} ;
\draw(7,2.5) node[right] {$\star$ \scriptsize $ l = (\frac{1}{q_1}, \frac{1}{q^{\ast}_2})$ on $\mathcal{S}_{m, q^\ast}^{(q_1, q_2^\ast)}$ (thick line)} ;
\draw(7,2) node[right] {\scriptsize  for  $q^{\ast}_2   < q_2$ and $q^\ast < q$} ;

\end{tikzpicture}
\end{center}

%%%%

\item \emph{Theorem~\ref{Theorem-5} $(ii)$}:
The following figure visualizes the idea of searching $q^\ast$ and $\tilde{q}^\ast_2$ with Figure 9-$(i)$. Also we can repeat the same strategy with Figure 9-$(ii)$.

% Figure 24
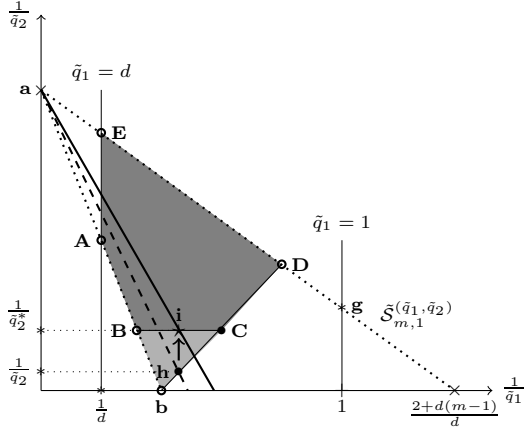
\begin{center}
\begin{tikzpicture}[domain=0:16]

\draw (0, -1) node[right] {\scriptsize \textit{Firgure 9-(e)}. $1<m\leq 2$, $q>1$.};

\fill[fill= lgray]
(1.27, 0.8) -- (1.6,0) -- (2.4, 0.8) ;

\fill[fill= gray]
(0.8, 2) -- (1.27, 0.8)-- (2.4,0.8) -- (3.2, 1.68) --(0.8, 3.43);

\draw[->] (0,0) -- (6,0) node[right] {\scriptsize $\frac{1}{\tilde{q}_1}$};
\draw[->] (0, 0) -- (0, 5) node[left] {\scriptsize $\frac{1}{\tilde{q}_2}$};

\draw[thick, dotted] (0, 4) node{\scriptsize $\times$} node[left]{\scriptsize $\textbf{a}$} -- ( 5.5 , 0) node{\scriptsize $\times$} node[below] {\scriptsize $\frac{2+d(m-1)}{d}$};
%\draw[thick] (0, 4) circle(0.05) node[left]{\scriptsize $\textbf{a}$};

\draw (5, 1) node {\scriptsize $\tilde{\mathcal{S}}_{m,1}^{(\tilde{q}_1, \tilde{q}_2)}$};

%\draw (0, 4) -- (3, 0) node{\scriptsize $\times$} node[below] {\scriptsize $\frac{2m+ d(m-1)}{md}$};
%\draw (3.3, 0.5) node {\scriptsize $\tilde{\mathcal{S}}_{m,m}^{(\tilde{q}_1, \tilde{q}_2)}$};

\draw[thick, dotted] (0,4)
-- (1.6, 0)  node[below] {\scriptsize $\textbf{b}$} ;
\draw[thick] (1.6, 0) circle(0.05);

\draw[very thin]
(0.8, 0) node{\scriptsize$*$} node[below] {\scriptsize $\frac 1d$}
-- (0.8, 4) node[above]{\scriptsize $\tilde{q}_1 = d$} ;
\draw[thick] (0.8, 2) circle(0.05) node[left]{\scriptsize \textbf{A}};
%\draw[thick] (0.8, 2.94) circle(0.05) node[left]{\scriptsize \textbf{f}} ;
\draw[thick] (0.8, 3.43) circle(0.05) node[right]{\scriptsize \textbf{E}} ;

%\draw[thin, dotted] (0,0) -- (2.3,2.3);
\draw[thin] (1.6, 0) -- (3.2,1.68) ;
\draw[thick] (3.2, 1.68) circle(0.05)  node[right]{\scriptsize \textbf{D}};

\draw[thin] (1.27, 0.8)  node[left] {\scriptsize \textbf{B}} -- (2.4,0.8) node{\scriptsize $\bullet$} node[right]{\scriptsize \textbf{C}};
\draw[thick] (1.27, 0.8) circle(0.05);

%\draw[thin, dotted] (0, 2) node{\scriptsize $+$} node[left] {\scriptsize $\frac 12$}
%-- (0.8, 2);
%\draw[thin, dotted] (0, 2.94) node{\scriptsize$*$} -- (0.8, 2.94);
%\draw[thin, dotted] (2.4, 0) node{\scriptsize$*$} -- (2.4, 0.8);

%\draw[thin, dotted] (0, 3.43) node{\scriptsize$*$} node[left] {\scriptsize $\frac{1+d(m-1)}{2+d(m-1)}$}  -- (0.8, 3.43);

\draw[very thin] (4, 0) node{\scriptsize$+$} node[below] {\scriptsize $1$}
-- (4, 2) node[above] {\scriptsize{$\tilde{q}_1 = 1$}};
\draw (4, 1.1) node{\scriptsize $*$} node[right] {\scriptsize \textbf{g}} ;

\draw (1.83,0.25) node{\scriptsize $\bullet$} node[left]{\scriptsize \textbf{h}};
\draw[dotted] (0, 0.25) node{\scriptsize $*$} node[left] {\scriptsize $\frac{1}{\tilde{q}_2}$} -- (1.75, 0.25);
\draw[thick, ->] (1.83,0.4)-- (1.83, 0.7)  ;
\draw[thick, dashed](0,4) -- (1.95,0);

\draw (1.83, 0.8) node{$\star$} node[above]{\scriptsize \textbf{i}};
\draw[dotted] (0, 0.8) node{\scriptsize $*$}  -- (1.7, 0.8);
\draw (0,1) node[left] {\scriptsize $\frac{1}{\tilde{q}^{\ast}_2}$};
\draw[thick] (0, 4) -- (2.3, 0) ;
%\draw (2.7, 0.3) node{\scriptsize { \color{blue} $\tilde{\mathcal{S}}_{m, \tilde{q}^\ast}^{(\tilde{q}_1,\tilde{q}^{\ast}_2)}$}} ;

%\draw[thin, dotted] (0, 2) node{\scriptsize $+$} node[left] {\scriptsize $\frac 12$}
%-- (0.8, 2);
%\draw[thin, dotted] (0, 2.94) node{\scriptsize$*$} -- (0.8, 2.94);
%\draw[thin, dotted] (2.4, 0) node{\scriptsize$*$} -- (2.4, 0.8);

%\draw[thin, dotted] (0, 3.43) node{\scriptsize$*$} node[left] {\scriptsize $\frac{1+d(m-1)}{2+d(m-1)}$}  -- (0.8, 3.43);

%\draw[very thin] (4, 0) node{\scriptsize$+$} node[below] {\scriptsize $1$}
%-- (4, 2) node[above] {\scriptsize{$\tilde{q}_1 = 1$}};
%\draw (4, 1.1) circle(0.05) node[right] {\scriptsize \textbf{f}} ;

%\draw[thin, dotted] (0, 0.8) node{\scriptsize$*$}  -- (2.4, 0.8);
%\draw[thin, dotted] (0, 1.68) node{\scriptsize$*$}  -- (3.2, 1.68);
%\draw[thin, dotted] (3.2, 0) node{\scriptsize$*$}  -- (3.2, 1.68);

%\draw[thin, dotted] (0, 1.1) node{\scriptsize$*$} node[left] {\scriptsize $\frac{2+d(m-2)}{2+d(m-1)}$} -- (4, 1.1);

\draw (7, 4) node[right] {\scriptsize $\overline{\textbf{ab}} = \tilde{\mathcal{S}}_{m, \infty}^{(\tilde{q}_1, \tilde{q}_2)}$, $\textbf{a} = (0, 1)$, $\textbf{b} = (\frac 2d, 0)$};
%\draw (7, 4) node[right] {\scriptsize $\textbf{a} = (0, 1)$, $\textbf{b} = (\frac 2d, 0)$ };

\draw (7, 3.5) node[right] {\scriptsize $\mathcal{R} (\textbf{ABCDE})$: scaling invariant class of \eqref{T5:V-tilde} };
\draw (7, 3) node[right] {\scriptsize $\textbf{A}-\textbf{E}$: same as in Figure 9-$(i)$.  };
\draw(7, 2.5) node[right] {\scriptsize $\mathcal{R} (\textbf{bCB})$: scaling invariant class of \eqref{T5:V-tilde-embedding} };

\draw (7,2) node[right] {$\bullet$ \scriptsize $ h = (\frac{1}{\tilde{q}_1}, \frac{1}{\tilde{q}_2})$ on $\tilde{\mathcal{S}}_{m, q}^{(\tilde{q}_1, \tilde{q}_2)}$ (dashed line)}  ;
%\draw[dred] (7,1.5) node[right] {\scriptsize $r_1 = (0, \frac{1+q_{m,d}}{2+q_{m,d}})$, $r_2 = (\frac{1+q_{m,d}}{d}$, 0)} ;
%\draw[dred] (7,1.5) node[right] {$\star$ \scriptsize $ h = (\frac{1}{\tilde{q}_1}, \frac{1}{\tilde{q}_2})$ } ;

\draw (7,1.5) node[right] { $\star$ \scriptsize $ i = (\frac{1}{\tilde{q}_1}, \frac{1}{\tilde{q}^{\ast}_2})$ on  $\tilde{\mathcal{S}}_{m, q^\ast}^{(\tilde{q}_1, \tilde{q}_2^\ast)}$ (thick line)} ;
\draw (9,1) node[right] {\scriptsize for $\tilde{q}^{\ast}_2   < \tilde{q}_2$ and $q^\ast < q$} ;
\end{tikzpicture}
\end{center}
\end{itemize}

\end{appendices}

%%%%%%%%%%%%%%%%%%%%%%%%%%%%%%%%%%%%%%%%%%%%%%%%%%%%%%%%%%%%%%%%%%%%%%%%%%%%%%%%%%%%%%%%%%%%%%%%%%%%%%%%%%%%%%%%%%%%%%%%%%%%%%%%%%%%%%%%%%%%%%%%%%%%%%%%%%%%%%%%%%%%%%%%%%%%%%%%%%%%%%%%%%%%%%%%%%%%%%%%%%%%%%%%%%%%%%%%%%%%%%%%%%%%%%%%%%%%%%%%%%%%%%%%%%%%%%%%%%%%%%%%%%%%%%%%%

\subsection*{Acknowledgements}
We thank the anonymous referee for careful reading and helpful suggestions, which helped improve the clarity of the paper. S. Hwang's work is partially supported by funding for the academic research program of Chungbuk National University in 2023 and NRF-2022R1F1A1073199. K. Kang's work is partially supported by NRF-2019R1A2C1084685. H. Kim's work is partially supported by NRF-2021R1F1A1048231 and NRF-2018R1D1A1B07049357.

%%%%%%%%%%%%%%%%%%%%%%%%%%%%%%%%
%\nocite{*}

%%%%%%%%%%%%%%%%%%%%%%%%%%%

\end{document}